\documentclass[10pt,reqno]{amsart}
\usepackage[margin=1in]{geometry}

\usepackage{amsmath,amsthm,marvosym, wasysym,bbold,tensor,mdframed}
\usepackage{amssymb}
\usepackage{amsfonts}
\usepackage{mathscinet}
\usepackage[dvipsnames]{xcolor}
\usepackage{tikz-cd}
\usepackage{tikz,keyval,pgfsys,graphicx}
\usepackage{mathtools, enumerate, combelow}
\usepackage[shortlabels]{enumitem}

\usetikzlibrary{arrows, decorations.pathreplacing}
\usepgflibrary{arrows}

\newtheorem{thm}{\normalfont\scshape Theorem}[section]
\newtheorem*{main}{\normalfont\scshape Main Theorem}
\newtheorem{prop}[thm]{\normalfont\scshape Proposition}
\newtheorem{lem}[thm]{\normalfont\scshape Lemma}
\newtheorem{cor}[thm]{\normalfont\scshape Corollary}

\theoremstyle{definition}
\newtheorem{defn}[thm]{\normalfont\scshape Definition}

\theoremstyle{remark}
\newtheorem{rem}[thm]{Remark}
\theoremstyle{remark}

\theoremstyle{remark}

\allowdisplaybreaks

\begin{document}

\binoppenalty=10000
\relpenalty=10000

\numberwithin{equation}{section}

\newcommand{\uuu}{\mathfrak{u}}
\newcommand{\ccc}{\mathfrak{c}}
\newcommand{\qqq}{\mathfrak{q}}
\newcommand{\ddd}{\mathfrak{d}}
\newcommand{\QQ}{\mathbb{Q}}
\newcommand{\ZZ}{\mathbb{Z}}
\newcommand{\Hilb}{\mathrm{Hilb}}
\newcommand{\CC}{\mathbb{C}}
\newcommand{\PP}{\mathcal{P}}
\newcommand{\Hom}{\mathrm{Hom}}
\newcommand{\Rep}{\mathrm{Rep}}
\newcommand{\MM}{\mathfrak{M}}
\newcommand{\gl}{\mathfrak{gl}}
\newcommand{\ppp}{\mathfrak{p}}
\newcommand{\VV}{\mathcal{V}}
\newcommand{\NN}{\mathbb{N}}
\newcommand{\OO}{\mathcal{O}}
\newcommand{\HH}{\mathcal{H\kern-.44em H}}
\newcommand{\git}{/\kern-.35em/}
\newcommand{\KK}{\mathbb{K}}
\newcommand{\Proj}{\mathrm{Proj}}
\newcommand{\quot}{\mathrm{quot}}
\newcommand{\core}{\mathrm{core}}
\newcommand{\Sym}{\mathrm{Sym}}
\newcommand{\FF}{\mathbb{F}}
\newcommand{\Span}[1]{\mathrm{span}\{#1\}}
\newcommand{\UTor}{U_{\qqq,\ddd}(\ddot{\mathfrak{sl}}_\ell)}
\newcommand{\Sss}{\mathcal{S}}

\title{Wreath Macdonald polynomials as eigenstates}
\author{Joshua Jeishing Wen}
\keywords{Macdonald polynomials, quantum toroidal algebra, shuffle algebra}
\subjclass[2020]{Primary: 05E05, 81R10; Secondary: 33D52, 81R12.}
\address{Fakult\"{a}t f\"{ur} Mathematik, Universit\"{a}t Wien, Vienna, Austria}
\email{joshua.jeishing.wen@univie.ac.at}
\maketitle

\begin{abstract}
We show that the wreath Macdonald polynomials for $\ZZ/\ell\ZZ\wr\Sigma_n$, when naturally viewed as elements in the vertex representation of the quantum toroidal algebra $U_{\qqq,\ddd}(\ddot{\mathfrak{sl}}_\ell)$, diagonalize its horizontal Heisenberg subalgebra. 
Our proof makes heavy use of shuffle algebra methods, and we also obtain a new proof of existence of wreath Macdonald polynomials.
\end{abstract}

\section{Introduction}

Macdonald polynomials are distinguished bigraded characters of the symmetric groups $\Sigma_n$ defined by a pair of triangularity conditions along with a specified normalization: for $\lambda$ a partition of $n$, the \textit{transformed} Macdonald polynomial $H_\lambda$ is characterized by (cf. \cite{Haiman})
\begin{enumerate}
\item $H_\lambda\otimes\sum_{i=0}^n(-q)^i\left[\bigwedge^i\CC^n\right]$ lies in the span of the simple modules $[V_\mu]$ for $\mu\ge\lambda$;
\item $H_\lambda\otimes\sum_{i=0}^n(-t)^{-i}\left[\bigwedge^i\CC^n\right]$ lies in the span of the simple modules $[V_\mu]$ for $\mu\le\lambda$;
\item the coefficient of the trivial module in $H_\lambda$ is 1
\end{enumerate}
Here, $\CC^n$ is the reflection representation of $\Sigma_n$. 
From this definition alone, one may be surprised by their ubiquity in mathematics; for example, they have appeared in enumerative geometry, knot theory, quantum algebra, and probability. 
Many of these connections are branches of an initial link to integrable systems. 
Translating characters into symmetric functions via the Frobenius characteristic, one can view the polynomials instead as some basis of $\Lambda_{q,t}$, the twice-deformed ring of symmetric functions. 
In this setting, they distinguish themselves in an \textit{a priori} very different way: the Macdonald polynomials diagonalize the \textit{Macdonald operators}, a commuting family of difference operators that are the Hamiltonians of the quantum trigonometric Ruijsenaars-Schneider (tRS) integrable system. 
This paper is concerned with a generalization of this picture to what are known as the \textit{wreath} Macdonald polynomials.

The wreath Macdonald polynomials were proposed by Haiman in \cite{Haiman} as a generalization of the definition above from $\Sigma_n$ to the wreath product $\ZZ/\ell\ZZ\wr\Sigma_n$. 
Simple representations of $\ZZ/\ell\ZZ\wr\Sigma_n$ are instead indexed by $\ell$-tuples of partitions (cf. \cite{Mac} Chapter I, Appendix II). 
As a result, the $\ell$-core and $\ell$-quotient decomposition of an ordinary partition plays a key role: one can peel away contiguous strips of length $\ell$ from a partition $\lambda$ until one is left with $\core(\lambda)$, and the strips that are peeled away can be recorded in an $\ell$-tuple of partitions $\quot(\lambda)$. 
We review these notions in Section \ref{Bosons}. 
The wreath product $\ZZ/\ell\ZZ\wr\Sigma_n$ also has a natural reflection representation $\mathfrak{h}_n\cong\CC^n$, and for $\lambda$ such that the sizes of the components of $\quot(\lambda)$ sum up to $n$, the wreath Macdonald polynomial $H_\lambda$ is characterized by 
\begin{enumerate}
\item $H_\lambda\otimes\sum_{i=0}^n(-q)^i\left[\bigwedge^i\mathfrak{h}_n^*\right]$ lies in the span of the simple modules $[V_{\quot(\mu)}]$ where $\core(\mu)=\core(\lambda)$ and $\mu\ge\lambda$;
\item $H_\lambda\otimes\sum_{i=0}^n(-t)^{-i}\left[\bigwedge^i\mathfrak{h}_n^*\right]$ lies in the span of the simple modules $[V_{\quot(\mu)}]$ where $\core(\mu)=\core(\lambda)$ and $\mu\le\lambda$;
\item the coefficient of the trivial module in $H_\lambda$ is 1.
\end{enumerate}
One can view each $\ell$-core as giving an ordering on $\ell$-tuples of partitions with which to define our triangularity conditions. 
Thus, for each $\ell$-core, this definition gives a basis of the representation ring of $\ZZ/\ell\ZZ\wr\Sigma_n$.

Similar to the classical Macdonald polynomials, it is not obvious from the definition that the wreath Macdonald polynomials exist. 
Nearly a decade after Haiman's proposal, Bezrukavnikov and Finkelberg proved their existence as well as an analogue of Macdonald positivity \cite{BezFink}. 
As far as we are aware, no published results on the subject have appeared since. 
In this paper, we prove the second fundamental fact about wreath Macdonald polynomials: that they diagonalize a large commutative algebra of operators. 
We conjecture that this algebra can be identified with the commutative algebra of Hamiltonians for some generalization of the quantum tRS system---we discuss this in more detail below.

\subsection{Statement of the main theorem} The aforementioned commutative algebra sits inside a larger structure: the quantum toroidal algebra $U_{\qqq,\ddd}(\ddot{\mathfrak{sl}}_\ell)$. 
To see why this larger structure is natural to consider in this setting, let us revisit \cite{BezFink}. 
Generalizing Haiman's seminal proof of the Macdonald positivity conjecture, the authors of \textit{loc. cit.} construct $\ZZ/\ell\ZZ\wr\Sigma_n$-equivariant bundles on cyclic Nakajima quiver varieties whose fibers at torus-fixed points are representations of $\ZZ/\ell\ZZ\wr\Sigma_n$ satisfying the definition of the wreath Macdonald polynomials. 
The main takeaway for us is that this matches wreath Macdonald polynomials with fixed-point classes in torus-equivariant K-theory of cyclic quiver varieties. 
On these K-theory groups, Varagnolo and Vasserot had previously constructed an action of $U_{\qqq,\ddd}(\ddot{\mathfrak{sl}}_\ell)$ \cite{VVCyclic}.

The quantum toroidal algebra contains two copies of the quantum affine algebra $U_\qqq(\dot{\mathfrak{gl}}_\ell)$, called the vertical and horizontal subalgebras. 
Each of them in turn contains a rank $\ell$ Heisenberg subalgebra, which we also call vertical and horizontal. 
In the construction of Varagnolo-Vasserot, it is obvious that the fixed point classes diagonalize the \textit{vertical} Heisenberg subalgebra. 
Work of Nagao \cite{NagaoK} identifies this K-theoretic module with the $q$-deformed fermionic Fock space $\mathcal{F}$ of Kashiwara-Miwa-Stern \cite{KMS}. 
Thus, we can shave off the geometry and say that in this fermionic module, one has a commutative algebra of operators and its diagonal basis, although it is not clear how to directly situate the representation theory of $\ZZ/\ell\ZZ\wr\Sigma_n$ in this picture.

On the other hand, from work \cite{FJW} of I. Frenkel, Jing, and Wang, a natural home for the wreath product representations is the vertex representation $W$ of $U_{\qqq,\ddd}(\ddot{\mathfrak{sl}}_\ell)$, which is like a bosonic Fock space. 
Recall that similar to the representation theory of the groups $\Sigma_n$, there is a \textit{wreath} Frobenius characteristic relating representation rings of the groups $\{\ZZ/\ell\ZZ\wr\Sigma_n\}_{n\ge 1}$ for fixed $\ell$ and the ring $\Lambda^{\otimes\ell}$ (also reviewed in Section \ref{Bosons}). 
Now, as a vector space, $W\cong\Lambda_{q,t}^{\otimes\ell}\otimes\CC[Q]$, where $Q$ is the root lattice of $\mathfrak{sl}_\ell$. 
One can use the root lattice to index $\ell$-cores, so this extra tensor factor makes  $W$ a natural home for the wreath Macdonald polynomials. 
As usual, we write elements of $\CC[Q]$ using the exponentiated basis $\{e^\alpha:\alpha\in Q\}$.

From the definition of the action on $W$, it is far from obvious that there exists an eigenbasis for a commutative subalgebra as large as the vertical Heisenberg subalgebra. 
However, Tsymbaliuk \cite{Tsym} has shown that $\mathcal{F}$ and $W$ are twisted isomorphic, where the twist here is a conceptually beautiful but formulaically complicated automorphism of Miki \cite{Miki} that switches the vertical and horizontal subalgebras. 
Therefore, since the vertical Heisenberg subalgebra is diagonalized in $\mathcal{F}$, the \textit{horizontal} Heisenberg subalgebra must be diagonalized on $W$. 
We can now state our main theorem:

\begin{main}
The wreath Macdonald polynomials form an eigenbasis for the horizontal Heisenberg subalgebra of the quantum toroidal algebra of $\mathfrak{sl}_\ell$. 
\end{main}

\noindent The diagonal basis in $\mathcal{F}$ is indexed by partitions and we denote it by $\{|\lambda\rangle\}$. 
Our theorem matches $|\lambda\rangle$ and $H_\lambda$ up to an unknown nonzero constant. 
We also note that our proof gives a new proof of existence of wreath Macdonald polynomials.


\subsection{Strategy of the proof} There is no simple way to study elements of the horizontal Heisenberg subalgebra, and our strategy and methods are highly indirect as a result. 
Our overall plan is to translate the pair of triangularity conditions to the fermionic module $\mathcal{F}$ and show that they are satisfied by $\left\{ |\lambda\rangle \right\}$. 
Before that, we must translate them from wreath product representations into bosons. 
As in the classical Macdonald case, we can rewrite the tensor product appearing in the definition as a `plethystic substitution', i.e. as a ring automorphism on $\Lambda_{q,t}^{\otimes\ell}$ defined by a linear map on power sum generators. 
By inverting these automorphisms, we can express conditions (1) and (2) as saying that $H_\lambda$ spans the intersection of two subspaces: one obtained by multiplying certain combinations of plethystically-transformed, colored \textit{complete} symmetric functions to $1\otimes e^{\core(\lambda)}$ and another obtained by multiplying plethystically-transformed, colored \textit{elementary} symmetric functions instead (here, by color we mean the tensorand in $\Lambda_{q,t}^{\otimes\ell}$). 
Thus, we can build up these two subspaces via multiplication by certain bosonic generators.

In $W$, these bosonic multiplications come from the action of \textit{vertical} Heisenberg elements. 
Therefore, in $\mathcal{F}$, they must come from \textit{horizontal} Heisenberg elements. 
We are stuck again with analyzing elements of the horizontal Heisenberg subalgebra, but on the fermionic side, the problem is amenable to shuffle algebra methods. 
The shuffle algebra is a certain space of symmetric rational functions endowed with an exotic product, and vaguely speaking, these functions are meant to model correlation functions of $\UTor$. 
By work of Negu\cb{t} \cite{NegutTor}, the shuffle algebra is isomorphic to a certain part of $\UTor$, and for our purposes, we can find shuffle elements corresponding to the horizontal Heisenberg elements of interest.

Before addressing how we do so, let us first discuss why this is a good idea. 
By Proposition IV.8 of \cite{NegutCyc} (Proposition \ref{ShuffleFock} in this paper), a shuffle element $F$ acts on the basis $\{|\lambda\rangle\}$ by adding certain boxes to the partition and appending a coefficient obtained by, roughly speaking, evaluating $F$ at the $(q,t)$-weights of the added boxes. 
A consequence of this is that one can determine that certain matrix elements must vanish by considering the zeros of $F$. 
Naively, to prove the theorem, we can try to show that when multiplying by the appropriate horizontal Heisenberg elements to $|\core(\lambda)\rangle$ to obtain the subspace for condition (1), the matrix elements for $|\mu\rangle$ vanish when $\mu\not\ge\lambda$ and likewise for condition (2), the matrix elements for $|\mu\rangle$ vanish when $\mu\not\le\lambda$. 

To find these shuffle elements, we adapt work of Feigin and Tsymbaliuk, specifically Sections 3 and 4 of \cite{FeiTsym}. 
First, we give a characterization of shuffle elements corresponding to the negative half of the horizontal Heisenberg subalgebra, which is where our elements of interest live. 
For technical reasons, in order to do this, we need a shuffle presentation of some set of generators of the Heisenberg subalgebra. 
A suitable generating set is given by vacuum-to-vacuum matrix elements of $L$-operators in the vertex representations, since we can find the corresponding shuffle elements by computing the vacuum correlation functions of those representations. 
We are then able to identify the negative half of the horizontal Heisenberg subalgebra with a subspace of functions satisfying certain limit conditions.

On this subspace of functions, we define two \textit{Gordon filtrations}, cf. also \cite{TsymBook} 3.2.4-3.2.5. 
These are filtrations defined in terms of certain evaluations: if a function vanishes on more evaluations, it lies deeper within the filtration. 
Using the known shuffle presentations of the $L$-operators from the previous paragraph, we can actually translate these evaluation functionals in terms of bosons in the horizontal Heisenberg subalgebra. 
This allows us to translate the filtrations as well, and from here, we can see that the plethystically-transformed, colored complete and elementary symmetric functions we care about each lie in a one-dimensional piece of the filtration. 
On the shuffle algebra side, it is not too hard to find shuffle elements spanning each of these one-dimensional pieces, giving us the desired shuffle presentations up to constants. 

Examining the zeros of these shuffle elements, one may be disappointed by how little they impose on the newly added boxes. 
We end the paper with some combinatorial results on partitions necessary to convert these weak conditions into the triangularity results needed to prove the theorem. 
Due to the specificity of our desired results, we have been unable to find suitable references for this part of the paper. 
Therefore, we do not know if our arguments are novel. 

Strictly speaking, our results on the quantum toroidal algebra and the shuffle algebra only apply to $\ell\ge 3$. 
For $\ell=2$, the definitions of the quantum toroidal algebra, its vertex representation, and the shuffle algebra are different but we do not expect our results to change significantly.
We recommend Section 5 of \cite{FeiTsym} as a reference for the necessary alterations in this case.

\subsection{Further directions} Our work allows wreath Macdonald theory to make contact with methods from quantum algebra, and we expect this interaction to continue bearing fruit.

For example, following \cite{FeiTsymK}, our shuffle elements can be used to produce wreath Pieri rules.
Recall that the wreath Macdonald polynomials are generalizations of \textit{transformed} Macdonald polynomials.
One can write down a definition of an analogue of \textit{ordinary} Macdonald polynomials.
The two differ by a plethystic transformation as well as by a renormalization, and finding the renomalization term is also an interesting problem.
The primary challenge for finding Pieri rules would then be producing a \textit{manicured} formula for them as in the classical case.

The quantum algebraic structure also allows a systematic study of degenerations.
One aspect of wreath Macdonald theory that may be strange for a symmetric function theorist is that we automatically jump to the double-deformed case without regarding analogues of Jack and Hall-Littlewood polynomials.
The `Jack degeneration' of the quantum toroidal algebra is the \textit{affine Yangian}.
By investigating a similar eigenbasis in the analogue of the fermionic Fock space, Uglov has defined certain $\mathrm{Jack}(\mathfrak{gl}_\ell)$ polynomials \cite{Ugl} (cf. \cite{Kod}).
He also shows that his polynomials diagonalize the quantum Hamiltonians for the spin Calogero-Moser system, hinting at a similar relation between the wreath Macdonald polynomials and the spin tRS system. 

Following our diagonalization result, a natural direction to pursue is to find a structure analagous to the double affine Hecke algebra (DAHA).
In the usual Macdonald case, the quantum toroidal algebra of $\mathfrak{gl}_1$ can be realized as a stabilization of the spherical DAHAs for $GL_n$ as $n$ goes to infinity (cf. \cite{SchiffVass}).
Whatever the correct answer for the wreath case should be, we expect a similar relation with the higher rank quantum toroidal algebra.
We suspect that structures uncovered by Chalykh and Fairon in their study of multiplicative quiver varieties \cite{CF2} have something to do with this.
The work of Chalykh-Fairon shows that multiplicative quiver varieties for the Jordan quiver are phase spaces for the spin tRS system. 
Quantizing these spaces, one should obtain analogues of the spherical DAHA for this system.
We also expect these quantizations to coincide with quantum K-theoretic Coulomb branches for the cyclic quiver (cf. \cite{BEF}, \cite{FinkTsym}).


\subsection{Outline of the paper} Section 2 begins with a review of the representation theory of wreath products.
We then go over the partition combinatorics necessary to define the wreath Macdonald polynomials ($\ell$-cores, $\ell$-quotients, etc.).
Finally, we introduce Haiman's definition and rewrite it so that a wreath Macdonald polynomial is characterized by spanning the intersection of two subspaces.

Section 3 introduces the quantum toroidal algebra and its structures.
Besides covering the zoo of subalgebras and representations in play, we also review factorizations of $R$-matrices in both the affine and toroidal cases.
This will be crucial for computing matrix elements of $L$-operators. 

Section 4 is the technical heart of the paper.
Here, we define the shuffle algebra and review its relationship to $\UTor$.
We then carry out the constructions and arguments as outlined in the introduction with the main result being the shuffle presentation of the horizontal Heisenberg elements of interest.

Section 5 contains the combinatorial arguments on partitions necessary to prove the theorem.
The results here may be of independent interest to box-stacking enthusiasts. 

We close with two appendices that only make contact with \ref{FactorProof} from the main body of the paper.
Appendix \ref{DualVertexApp} contains details about the dual vertex representation.
Appendix \ref{PosMode} outlines a version of the arguments from Section 4 for the opposite half of the toroidal algebra.

\subsection{Acknowledgements} I'd like to thank Philippe di Francesco, Rinat Kedem, Andrei Negu\cb{t}, and Oleksandr Tsymbaliuk for tremendously helpful conversations as well as for inviting me to speak about this work at their respective seminars.
I also want to thank Daniel Orr and Mark Shimozono for our collaboration, which gave this work a much-needed second wind. 
Most of all, I want to thank Tom Nevins for helping me navigate what mathematicians are and what a paper is.
Finally, I'd like to thank the anonymous referee for their patience and precision in helping me correct numerous mistakes.
This work was supported by a Gene H. Golub Fellowship, a Louis C. Hack Fellowship, a James D. Hogan Memorial Scholarship, as well as a Tondeur Dissertation Prize, all from the UIUC Department of Mathematics.
Additionally, this work received support from NSF grants DMS-1502125 and DMS-1802094 as well as the NSF-RTG grant ``Algebraic Geometry and Representation Theory at Northeastern University'' (DMS-1645877).

\subsection{Notation} A \textit{partition} $\lambda=(\lambda_1, \lambda_2,\ldots)$ is a nonincreasing list of nonnegative integers:
\[
\lambda_1\ge \lambda_2\ge\cdots
\]
The highest index $k$ where $\lambda_k\not=0$ is called its \textit{length} and denoted by $\ell(\lambda)$.
The \textit{size} is denoted by $|\lambda|:=\lambda_1+\cdots+\lambda_{\ell(\lambda)}$.
Another notation for $\lambda$ we may use is
\[\lambda=(1^{m_1}2^{m_2}\ldots)\]
where $m_i$ is multiplicity with which $i$ appears in $\lambda$.
The transposed partition is denoted by $\tensor[^t]{\lambda}{}$.
We will also make frequent use of vectors of partitions.
Subscripts will always index the part of a partition while superscripts will always index the component of such a vector.
Finally, $\lambda\ge\mu$ denotes dominance order: for all $k$,
\[
\lambda_1+\cdots+\lambda_k\ge\mu_1+\cdots+\mu_k
\]

As usual, for an integer $n$ and variable $\qqq$, $[n]_\qqq$ denotes the quantum number
\[[n]_\qqq=\frac{\qqq^n-\qqq^{-n}}{\qqq-\qqq^{-1}}\]
The factorial is then defined for $n>0$ as
\[[n]_\qqq!:=\prod_{i=1}^n[i]_\qqq\]

To save space, we will occasionally need to index products of noncommuting elements.
For this, we use the notation
\[\overset{\curvearrowright}{\prod_{i=1}^n}a_i=a_1\cdots a_n\]
to denote the product ordered from left to right according to the index.
Similarly, we use the notation
\[\overset{\curvearrowleft}{\prod_{i=1}^n}a_i=a_n\cdots a_1\]
for the product in the opposite order.

\section{Bosons}\label{Bosons}

\subsection{Representation theory of wreath products}\label{Wreath} 
Throughout this subsection, $\Gamma$ will be a finite group, $\Gamma^*$ will denote the set of its irreducible complex representations, and $\Gamma_*$ will denote its conjugacy classes. 
We let
\[\ell:=|\Gamma^*|=|\Gamma_*|\]
Our presentation closely follows Chapter I, Appendix II of \cite{Mac}. 
We direct the interested reader to this classic reference for any details and proofs.

\subsubsection{Wreath products} 
The wreath product $\Gamma_n:=\Gamma\wr\Sigma_n$ is by definition the semi-direct product
\[\Gamma^n\rtimes\Sigma_n\]
where the action is given by permuting the $n$ copies of $\Gamma$. 
One can concretely realize this group as the set of $n\times n$ permutation matrices with `entries in $\Gamma$'. 
We will instead just view elements of $\Gamma_n$ as pairs $(\vec{g},\sigma)$, where $\vec{g}=(g_1,\ldots,g_n)\in\Gamma^n$ and $\sigma\in\Sigma_n$ . 

\subsubsection{Conjugacy classes}\label{Conj} 
Recall that the conjugacy class of an element $\sigma\in\Sigma_n$ is determined by its cycle type, and thus conjugacy classes of $\Sigma_n$ are indexed by partitions of $n$. 
On the other hand, for an element $(\vec{g},\sigma)\in\Gamma_n$, we consider for each cycle $z=(i_1\ldots i_s)$ of $\sigma$ its \textit{cycle product} $g_{i_1}\cdots g_{i_s}\in\Gamma$. 
For each $c\in\Gamma_*$, we can gather together the cycles of $\sigma$ whose cycle product lies in $c$ and assign a partition $\lambda^c$ to $c$ in the natural way. 
Notice that
\[\sum_{c\in\Gamma_*}|\lambda^c|=n\]
We call a vector of partitions $(\lambda^c)_{c\in\Gamma_*}$ an \textit{$\ell$-multipartition} (or multipartition if it is clear from context), and if the sizes of its components sum to $n$, we say it is an \textit{$\ell$-multipartition of $n$}. 
In the way outlined above, the conjugacy classes of $\Gamma_n$ are indexed by $\ell$-multipartitions of $n$.

\subsubsection{Irreducible representations} 
As a consequence of \ref{Conj}, we can index the irreducible representations of $\Gamma_n$ by $\ell$-multipartitions of $n$. 
For such a multipartition $\vec{\lambda}=\left(\lambda^\gamma\right)_{\gamma\in\Gamma^*}$, we can concretely realize this representation in the following way (cf. \cite{JKerb}). 
Let $I_\gamma$ by the corresponding irreducible $\Gamma$-module for $\gamma\in\Gamma^*$. 
The tensor power $I_{\gamma}^{\otimes |\lambda^\gamma|}$ is a $\Gamma_{|\lambda^\gamma|}$-module in the following natural way:
\[(\vec{g},\sigma)\cdot v_1\otimes\cdots\otimes v_{|\lambda^\gamma|}=g_1v_{\sigma^{-1}(1)}\otimes\cdots\otimes g_{|\lambda_\gamma|}v_{\sigma^{-1}(|\lambda^\gamma|)}\]
On the other hand, if we let $V_{\lambda^\gamma}$ be the irreducible $\Sigma_{|\lambda^\gamma|}$-module corresponding to $\lambda^\gamma$, we can endow it with $\Gamma_{|\lambda^\gamma|}$-module structure by having the $\Gamma$-factors act trivially. 
The irreducible representation $V_{\vec{\lambda}}$ of $\Gamma_n$ corresponding to $\vec{\lambda}$ is then realized by the induced representation
\[V_{\vec{\lambda}}\cong\mathrm{Ind}_{\prod_{\gamma\in\Gamma^*}\Gamma_{|\lambda^\gamma|}}^{\Gamma_n}
\left\{
\bigotimes_{\gamma\in\Gamma^*}\left[
V_{\lambda^\gamma}\otimes I_\gamma^{\otimes|\lambda^\gamma|}\right]
\right\}
\]
For example, when $\lambda^{\gamma}=(n)$ and all the other entries of $\vec{\lambda}$ are empty, then 
\[V_{\vec{\lambda}}\cong I_\gamma^{\otimes n}\] 
Similarly, when $\lambda^\gamma=(1^n)$ and all the other entries are empty, 
\[V_{\vec{\lambda}}\cong\mathrm{sign}\otimes I_\gamma^{\otimes n}\]

\subsubsection{Recollections from symmetric function theory}\label{Sym} 
In the representation theory of symmetric groups, one learns that there is great utility in considering all the symmetric groups together. 
We would like to take a similar approach for wreath products with $\Gamma$ fixed, and thus we first review the classical story. 

By `all the symmetric groups together', we mean the following: if we let $R_n=\CC\otimes\Rep(\Sigma_n)$, we can consider the direct sum
\[R=\bigoplus_nR_n\]
We can view $R$ as a ring via the \textit{induction product}. 
Namely, for $\Sigma_n\curvearrowright V$ and $\Sigma_m\curvearrowright W$, 
\[V*W:=\mathrm{Ind}_{\Sigma_n\times\Sigma_m}^{\Sigma_{n+m}}(V\times W)\]
This ring in turn can be modeled by the ring $\Lambda$  of symmetric functions `in infinitely many variables'.
In $\Lambda$, we have the following generating sets and bases:
\begin{enumerate}
\item the \textit{power sums} $p_n$ and the basis given by 
\[p_\lambda=p_{\lambda_1}\cdots p_{\lambda_k};\]
\item the \textit{elementary symmetric functions} $e_n$ and the basis given by 
\[e_\lambda=e_{\tensor[^t]{\lambda}{}_1}\cdots e_{ \tensor[^t]{\lambda}{}_k}\]
(note the transpose);
\item the \textit{complete symmetric functions} $h_n$ and the basis $\{h_\lambda\}$ defined similarly to $\{p_{\lambda}\}$;
\item the basis of \textit{Schur functions} $\{s_\lambda\}$.
\end{enumerate}
A nice feature of the theory is how these bases and generating sets are related to each other. 
For example, the first three families share the following fundamental relationship:
\begin{align*}
\sum_{r\ge 0}h_rz^r&=\exp\left(\sum_{r>0}p_r\frac{z^r}{r}\right)\\
\sum_{r\ge 0}(-1)^re_rz^r&=\exp\left(-\sum_{r>0}p_r\frac{z^r}{r}\right)
\end{align*}
We will also make use of the following ``triangularity'' relating $\left\{ s_\lambda \right\}$ with $\left\{ e_\lambda \right\}$ and $\left\{ h_\lambda \right\}$:

\begin{prop}[cf. I.6 of \cite{Mac}]\label{SchurTriangle}
There exist coefficients $\left\{ a_{\lambda\mu} \right\}$ and $\left\{ b_{\lambda\mu} \right\}$ such that
\[
s_\lambda=\sum_{\mu\le\lambda}a_{\lambda\mu}e_\mu=\sum_{\mu\ge\lambda}b_{\lambda\mu}h_{\mu}
\]
\end{prop}

Returning to representation theory, note that the irreducible representations (irrep for short) of $\Sigma_n$ are indexed by partitions of $n$. 
Letting $V_\lambda$ denote this corresponding irrep, we have the classic result:
\begin{thm}
The linear map $R\rightarrow\Lambda$ induced by
\[[V_\lambda]\mapsto s_\lambda\]
is a ring isomorphism.
\end{thm}
\noindent We call this map the \textit{Frobenius characteristic}. 
Of note is that under this map, $h_n$ and $e_n$ correspond to the trivial and sign representations of $\Sigma_n$, respectively. 
Finally, for $\lambda=(1^{m_1}2^{m_2}\ldots)$, if we let
\[z_\lambda=\prod_i i^{m_i}m_i!\]
then the indicator function for the conjugacy class corresponding to $\lambda$ is mapped to $z_\lambda^{-1} p_\lambda$ (cf. formula (I.7.2) in \cite{Mac}). 

\subsubsection{Wreath Frobenius characteristic}\label{WreathFrob} 
Mirroring \ref{Sym}, we set
\[R_n(\Gamma):=\CC\otimes\Rep(\Gamma_n)\]
and consider
\[R(\Gamma):=\bigoplus_nR_n(\Gamma)\]
as a ring under the analogous induction product. 
For the analogue of symmetric functions, we first view $\Lambda$ as a single free boson:
\[\Lambda=\CC[p_r]_{r\ge 1}\]
In the wreath case, we take $\ell=|\Gamma_*|$ free bosons instead:
\[\Lambda(\Gamma):=\CC[p_r(c)]_{r\ge 1, c\in\Gamma_*}\]

Similar to \ref{Sym}, these new power sums will be closely related to indicator class functions. 
On the other hand, to access irreducible representations, we will need new generators indexed by $\gamma\in\Gamma^*$.
We abuse notation and conflate $\gamma$ with its character, i.e the class function given by taking the trace in the irreducible representation $\gamma$.
With this set, we define
\begin{equation}
p_r(\gamma):=\sum_{c\in\Gamma_*}\frac{\gamma(c)}{|\mathrm{Stab}(c)|}p_r(c)
\label{ConjGen}
\end{equation}

For each $\gamma\in\Gamma^*$, we can define $h_r(\gamma)$, $e_r(\gamma)$, and $s_\lambda(\gamma)$ by writing them in terms of the $p_r(\gamma)$ as one writes $h_r, e_r,$ and $s_\lambda$ in terms of the $p_r$. 
For a multipartition $\vec{\lambda}$, we can then define the \textit{multi-Schur function}:
\[s_{\vec{\lambda}}:=\prod_{\gamma\in\Gamma^*}s_{\lambda^\gamma}(\gamma)\]
Our analogue of the Frobenius characteristic is the following:
\begin{thm}
The linear map $R(\Gamma)\rightarrow\Lambda(\Gamma)$ induced by
\[[V_{\vec{\lambda}}]\mapsto s_{\vec{\lambda}}\]
is a ring isomorphism.
\end{thm}
\noindent Again, of note is that $h_r(\gamma)$ corresponds to $I_\gamma^{\otimes r}$ and $e_r(\gamma)$ corresponds to $\mathrm{sign}\otimes I_\gamma^{\otimes r}$. 
For a multipartition $\vec{\lambda}=(\lambda^c)_{c\in\Gamma_*}$, the indicator function for the class corresponding to $\vec{\lambda}$ is mapped to
\[\prod_{c\in\Gamma_*}\frac{p_{\lambda^c}(c)}{z_{\lambda^c}|\mathrm{Stab}(c)|^{\ell(\lambda^c)}}\]
(cf. (6.2) of Chapter I, Appendix B of \cite{Mac}).

\subsection{Core-quotient decomposition} 
In order to define the wreath Macdonald polynomials, we will need to review some combinatorics relating ordinary partitions and $\ell$-multipartitions. 
The core-quotient decomposition of a partition is integral to the definition of wreath Macdonald polynomials. 
We believe that it is nicely viewed in terms of Maya diagrams, which will also play a crucial role in Section \ref{Combi}. 
Our presentation borrows much from \cite{Nagao1}. 
For the rest of this section, we fix $\ell\ge 1$. 

\subsubsection{Young diagrams}\label{Young} 
We will view the \textit{Young diagram} of a partition $\lambda$ as the set of $(a,b)\in\ZZ^2$ such that $1\le a\le\lambda_b$. 
A \textit{node} of $\lambda$ is a point in its Young diagram. 
In visual representations of Young diagrams, we will replace the nodes with boxes and follow the French convention. 
For example, below is the Young diagram of the partition $(4,4,2)$:
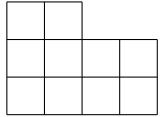
\begin{equation*}
\begin{tikzpicture}[scale=.5]
\draw (0,0)--(4,0);;
\draw (0,1)--(4,1);;
\draw (0,2)--(4,2);;
\draw (0,3)--(2,3);;
\draw (0,0)--(0,3);;
\draw (1,0)--(1,3);;
\draw (2,0)--(2,3);;
\draw (3,0)--(3,2);;
\draw (4,0)--(4,2);;
\end{tikzpicture}
\end{equation*}
\noindent These conventions will inform any visual language we may use with regards to Young diagrams (e.g. above, north, etc.). 
For a partition $\mu\subset\lambda$, we will also consider the \textit{skew diagram} $\lambda\backslash\mu$.
Even if $\mu\not\subset\lambda$, we will abbreviate $\lambda\backslash\mu:=\lambda\backslash(\mu\cap\lambda)$.

The \textit{content} of a node $(a,b)$ is $b-a$. 
This quantity marks the diagonal on which the node sits and increases towards the northwest. 
For $i\in\ZZ/\ell\ZZ$, this node is called an \textit{$i$-node} if $b-a\equiv i\hbox{ mod }\ell$. 

An \textit{$\ell$-strip} of a Young diagram is a connected subset of $\ell$ nodes on the outer rim whose removal leaves behind another Young diagram. 
For us, an $\ell$-strip `begins' at its northwestern-most node and `ends' at its southeastern-most node, and we use this to totally order the nodes within the strip (i.e. in \textit{decreasing} content).
An \textit{$\ell$-core} is a partition whose Young diagram contains no $\ell$-strips. 
The $\ell$-core of a partition $\lambda$ is the partition obtained after successively removing $\ell$-strips from $\lambda$ until one cannot do so anymore.
There may be many ways to exhaustively remove $\ell$-strips from $\lambda$, but all sequences of removals yield the same $\ell$-core, which we denote by $\core(\lambda)$.
For example, Figure \ref{fig:CoreEx} shows how to obtain the 3-core of $(4,4,2)$.

\begin{figure}[h]
\centering
\begin{tikzpicture}[scale=.5]
\draw (0,0)--(1,0)--(1,1)--(0,1)--(0,0);;
\draw (1,0)--(2,0)--(2,1)--(1,1)--(1,0);;
\draw (2,0)--(3,0)--(3,1)--(2,1)--(2,0);;
\draw[fill=lightgray] (3,0)--(4,0)--(4,1)--(3,1)--(3,0);;
\draw (0,1)--(1,1)--(1,2)--(0,2)--(0,1);;
\draw (1,1)--(2,1)--(2,2)--(1,2)--(1,1);;
\draw[fill=lightgray] (2,1)--(3,1)--(3,2)--(2,2)--(2,1);;
\draw[fill=lightgray] (3,1)--(4,1)--(4,2)--(3,2)--(3,1);;
\draw (0,2)--(1,2)--(1,3)--(0,3)--(0,2);;
\draw (1,2)--(2,2)--(2,3)--(1,3)--(1,2);;
\draw (5, 1.5) node {$\rightarrow$};;
\draw (6,0)--(7,0)--(7,1)--(6,1)--(6,0);;
\draw (7,0)--(8,0)--(8,1)--(7,1)--(7,0);;
\draw (8,0)--(9,0)--(9,1)--(8,1)--(8,0);;
\draw (6,1)--(7,1)--(7,2)--(6,2)--(6,1);;
\draw[fill=lightgray] (7,1)--(8,1)--(8,2)--(7,2)--(7,1);;
\draw[fill=lightgray] (6,2)--(7,2)--(7,3)--(6,3)--(6,2);;
\draw[fill=lightgray] (7,2)--(8,2)--(8,3)--(7,3)--(7,2);;
\draw (10, 1.5) node {$\rightarrow$};;
\draw (11,0)--(12,0)--(12,1)--(11,1)--(11,0);;
\draw (12,0)--(13,0)--(13,1)--(12,1)--(12,0);;
\draw (13,0)--(14,0)--(14,1)--(13,1)--(13,0);;
\draw (11,1)--(12,1)--(12,2)--(11,2)--(11,1);;
\end{tikzpicture}
\caption{The 3-core of $(4,4,2)$.}
\label{fig:CoreEx}
\end{figure}

\subsubsection{Maya diagrams} 
A \textit{Maya diagram} is a map $m:\ZZ\rightarrow\{\pm 1\}$ such that
\[m(j)=\left\{\begin{array}{ll}
-1 & \hbox{for }j\gg0\\
1 & \hbox{for }j\ll 0
\end{array}\right.\]
We can visually represent a Maya diagram by a string of white and black beads indexed by $\ZZ$, where the bead for $j\in\ZZ$ is white if $m(j)=-1$ and black if $m(j)=1$. 
For reasons apparent in \ref{YM}, our integers will increase towards the left:
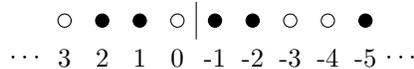
\begin{equation*}
\begin{tikzpicture}[scale=.5]
\draw (-4, 0) node {$\cdots$};;
\draw (-3,0) node {3};;
\draw (-3,1) circle (5pt);;
\draw (-2,0) node {2};;
\draw[fill=black] (-2,1) circle (5pt);;
\draw (-1,0) node {1};;
\draw[fill=black] (-1,1) circle (5pt);;
\draw (0,0) node {0};;
\draw (0,1) circle (5pt);;
\draw (.5,.5)--(.5,1.5);;
\draw (1,0) node {-1};;
\draw[fill=black] (1,1) circle (5pt);;
\draw (2,0) node {-2};;
\draw[fill=black] (2,1) circle (5pt);;
\draw (3,0) node {-3};;
\draw (3,1) circle (5pt);;
\draw (4,0) node {-4};;
\draw (4,1) circle (5pt);;
\draw (5,0) node {-5};;
\draw[fill=black] (5,1) circle (5pt);;
\draw (6,0) node {$\cdots$};;
\end{tikzpicture}
\end{equation*}
\noindent Here, the ellipsis on the left signifies that the beads are all white after 3 and the ellipsis on the right signifies that the beads are all black after -5. 
The notch between 0 and -1 is called the \textit{central line} of the Maya diagram. 
We call the Maya diagram where all beads left of the central line are white and all beads right of the central line are black the \textit{vacuum diagram}.

We define the \textit{charge} of a Maya diagram as
\[c(m)=|\{j\ge 0:m(j)=1\}|-|\{j<0: m(j)=-1\}|\]
Visually, we can think about this as follows. 
The beads left of the central line will tend to be white while the beads right of the central line will tend to be black. 
Thus, white beads on the right and black beads on the left are exceptional, and the charge is the difference between the number of exceptions on the left (black) with those on the right (white). 
The charge of the Maya diagram given above is zero.
%

\subsubsection{Young-Maya correspondence}\label{YM} There is a natural bijection between Young diagrams and Maya diagrams of charge zero. 
To start, for a partition $\lambda$ and integer $j$, let $n_j(\lambda)$ denote the number of nodes in the Young diagram of $\lambda$ with content $j$. 
Observe that 
\[n_{j+1}(\lambda)-n_j(\lambda)=\left\{\begin{array}{rl}
-1\hbox{ or }0 & \hbox{for }j\ge 0\\
1\hbox{ or }0 & \hbox{for }j<0
\end{array}\right. \]
We define the corresponding Maya diagram as
\[m(\lambda)(j)=\left\{\begin{array}{rll}
1 & j\ge 0,& n_{j+1}(\lambda)-n_j(\lambda)=-1\\
-1 & j\ge 0, & n_{j+1}(\lambda)-n_j(\lambda)=0\\
-1 & j<0, & n_{j+1}(\lambda)-n_j(\lambda)=1\\
1 & j<0, & n_{j+1}(\lambda)-n_j(\lambda)=0
\end{array}\right.\]

This construction has the following very transparent visual interpretation. 
First, we tilt the Young diagram by 45 degrees to follow the Russian convention and draw lines marking the level sets for the content. 
We index the gap between the content $j$ and $j+1$ lines with $j$ so that the central line lines up with content zero line. 
The piece of the outer rim in each gap has either slope 1 or -1, and that is what our Maya diagram assigns to the corresponding integers. 
For the remaining integers, we assign the appropriate limiting values (-1 on the left and 1 on the right). 
Figure \ref{fig:YoungMaya} depicts this construction for $\lambda=(4,4,2)$.
\begin{figure}[h]
\centering
\begin{tikzpicture}[scale=.5]
\draw[thick] (-2.5,5)--(-.5,7)--(.5,6)--(2.5,8)--(4.5,6);;
\draw (0.5,2)--(4.5,6);;
\draw (-.5,3)--(3.5,7);;
\draw (-1.5,4)--(2.5,8);;
\draw (-2.5,5)--(-.5,7);;
\draw (0.5,2)--(-2.5,5);;
\draw (1.5,3)--(-1.5,6);;
\draw (2.5,4)--(-.5,7);;
\draw (3.5,5)--(1.5,7);;
\draw (4.5,6)--(2.5,8);;
\draw (-4, 0) node {$\cdots$};;
\draw (-3,0) node {3};;
\draw (-3,1) circle (5pt);;
\draw (-2,0) node {2};;
\draw[fill=black] (-2,1) circle (5pt);;
\draw (-1,0) node {1};;
\draw[fill=black] (-1,1) circle (5pt);;
\draw (0,0) node {0};;
\draw (0,1) circle (5pt);;
\draw (.5,.5)--(.5,1.5);;
\draw (1,0) node {-1};;
\draw[fill=black] (1,1) circle (5pt);;
\draw (2,0) node {-2};;
\draw[fill=black] (2,1) circle (5pt);;
\draw (3,0) node {-3};;
\draw (3,1) circle (5pt);;
\draw (4,0) node {-4};;
\draw (4,1) circle (5pt);;
\draw (5,0) node {-5};;
\draw[fill=black] (5,1) circle (5pt);;
\draw (6,0) node {$\cdots$};;
\draw[dashed] (0.5,2)--(0.5,9);;
\draw[dashed] (1.5,2)--(1.5,9);;
\draw[dashed] (2.5,2)--(2.5,9);;
\draw[dashed] (3.5,2)--(3.5,9);;
\draw[dashed] (4.5,2)--(4.5,9);;
\draw[dashed] (-0.5,2)--(-0.5,9);;
\draw[dashed] (-1.5,2)--(-1.5,9);;
\draw[dashed] (-2.5,2)--(-2.5,9);;
\end{tikzpicture}
\caption{The Young-Maya correspondence for $(4,4,2)$}
\label{fig:YoungMaya}
\end{figure}

\begin{prop}
The construction outlined above gives a bijection between Young diagrams and charge zero Maya diagrams.
\end{prop}


\begin{rem}\label{Rem11}
We can see that the Maya diagram encodes the outer rim of the Young diagram and therefore has simple interpretations for the addition and removal of nodes. 
Specifically:
\begin{enumerate}
\item Adding a node corresponds to the following switch on adjacent beads:
\begin{equation*}
\begin{tikzpicture}[scale=.5]
\draw (0,1) circle (5pt);;
\draw[fill=black] (1,1) circle (5pt);;
\draw (3,1) node {$\rightarrow$};;
\draw (6,1) circle (5pt);;
\draw[fill=black] (5,1) circle (5pt);;
\end{tikzpicture}
\end{equation*}
\noindent Removing a node corresponds to the opposite switch:
\begin{equation*}
\begin{tikzpicture}[scale=.5]
\draw[fill=black]  (0,1) circle (5pt);;
\draw (1,1) circle (5pt);;
\draw (3,1) node {$\rightarrow$};;
\draw[fill=black] (6,1) circle (5pt);;
\draw (5,1) circle (5pt);;
\end{tikzpicture}
\end{equation*}

\item Adding and removing an $\ell$-strip correspond to the same moves done to beads $\ell$ apart:

\begin{equation*}
\begin{tikzpicture}[scale=.5]
\draw (0,1) circle (5pt);;
\draw (0,0) node {$j+\ell$};;
\draw (1,1) node {$\cdots$};;
\draw[fill=black] (2,1) circle (5pt);;
\draw (2,0) node {$j$};;
\draw (4,1) node {$\rightarrow$};;
\draw[fill=black] (6,1) circle (5pt);;
\draw (6,0) node {$j+\ell$};;
\draw (7,1) node {$\cdots$};;
\draw (8,1) circle (5pt);;
\draw (8,0) node {$j$};;
\end{tikzpicture}
\end{equation*}

\begin{equation*}
\begin{tikzpicture}[scale=.5]
\draw[fill=black] (0,1) circle (5pt);;
\draw (0,0) node {$j+\ell$};;
\draw (1,1) node {$\cdots$};;
\draw (2,1) circle (5pt);;
\draw (2,0) node {$j$};;
\draw (4,1) node {$\rightarrow$};;
\draw (6,1) circle (5pt);;
\draw (6,0) node {$j+\ell$};;
\draw (7,1) node {$\cdots$};;
\draw[fill=black] (8,1) circle (5pt);;
\draw (8,0) node {$j$};;
\end{tikzpicture}
\end{equation*}
\end{enumerate}
\end{rem}

\subsubsection{Cores and quotients}\label{Corequot} 
We have recalled in \ref{Young} that the $\ell$-core of a partition $\lambda$ is obtained by successively performing all possible $\ell$-strip removals. 
Recall as well that we denote this remaining $\ell$-core partition $\core(\lambda)$.
Roughly speaking, the \textit{$\ell$-quotient} $\quot(\lambda)$ of $\lambda$ is an $\ell$-multipartition that encodes the way $\ell$-strips are layered onto $\core(\lambda)$ to obtain $\lambda$.

From the Maya diagram $m(\lambda)$, the $\ell$-quotient is obtained from the \textit{quotient subdiagrams} 
\[m_i(\lambda)(j):=m(\lambda)(i+j\ell)\]
for $i=0,\ldots, \ell-1$. 
We call the beads in $m(\lambda)$ with indices congruent to $i$ mod $\ell$ the \textit{$i$-beads}.
Thus, $m_i(\lambda)$ consists of all the $i$-beads of $m(\lambda)$.
In general, $m_i(\lambda)$ will have a nontrivial charge $c_i$. 
If we perform the shift $j\mapsto j-c_i$ on the integer index for $m_i(\lambda)$, we will have a charge zero Maya diagram and hence a Young diagram $\lambda^i$. 
In Figure \ref{fig:QuotEx}, we compute the 3-quotient of $\lambda=(4,4,2)$. 
For each $m_i(\lambda)$, we use the indexing it inherits from $m(\lambda)$ so as to be clear where each bead comes from. 
Also, the solid line in the $m_i(\lambda)$ is the central line inherited from $m(\lambda)$ while the dashed line is the new central line after the shift by $c_i$.
\begin{figure}[h]
\centering
\begin{tikzpicture}[scale=.4]
\draw (-6,.5) node {$m(\lambda)$:};;
\draw (-4, 0) node {$\cdots$};;
\draw (-3,0) node {3};;
\draw (-3,1) circle (5pt);;
\draw (-2,0) node {2};;
\draw[fill=black] (-2,1) circle (5pt);;
\draw (-1,0) node {1};;
\draw[fill=black] (-1,1) circle (5pt);;
\draw (0,0) node {0};;
\draw (0,1) circle (5pt);;
\draw (.5,.5)--(.5,1.5);;
\draw (1,0) node {-1};;
\draw[fill=black] (1,1) circle (5pt);;
\draw (2,0) node {-2};;
\draw[fill=black] (2,1) circle (5pt);;
\draw (3,0) node {-3};;
\draw (3,1) circle (5pt);;
\draw (4,0) node {-4};;
\draw (4,1) circle (5pt);;
\draw (5,0) node {-5};;
\draw[fill=black] (5,1) circle (5pt);;
\draw (6,0) node {$\cdots$};;
\draw (-14,-5.5) node {$m_0(\lambda):$};;
\draw (-12,-6) node {$\cdots$};;
\draw (-11,-6) node {0};;
\draw (-11,-5) circle (5pt);;
\draw (-10,-6) node {-3};;
\draw (-10,-5) circle (5pt);;
\draw[dashed] (-9.5,-5.5)--(-9.5,-4.5);;
\draw (-9.5,-3) node {$\varnothing$};;
\draw (-9,-6) node {-6};;
\draw[fill=black] (-9,-5) circle (5pt);;
\draw (-8,-6) node {$\cdots$};;
\draw (-10.5,-5.5)--(-10.5,-4.5);;
\draw (-5,-5.5) node {$m_1(\lambda):$};;
\draw (-3,-6) node {$\cdots$};;
\draw (-2,-6) node {4};;
\draw[dashed] (-1.5,-5.5)--(-1.5,-4.5);;
\draw (-1.5,-3) node {$\varnothing$};;
\draw (-2,-5) circle (5pt);;
\draw (-1,-6) node {1};;
\draw[fill=black] (-1,-5) circle (5pt);;
\draw (-.5,-5.5)--(-.5,-4.5);;
\draw (0,-6) node {-2};;
\draw[fill=black] (0,-5) circle (5pt);;
\draw (1,-6) node {$\cdots$};;
\draw (4, -5.5) node {$m_2(\lambda):$};;
\draw (6,-6) node {$\cdots$};;
\draw (7,-6) node {5};;
\draw (7,-5) circle (5pt);;
\draw (8,-6) node {2};;
\draw[fill=black] (8,-5) circle (5pt);;
\draw (8.5,-5.5)--(8.5,-4.5);;
\draw (8.5,-4)--(9.5,-3)--(8.5,-2)--(7.5,-3)--(8.5,-4);;
\draw (9.5,-3)--(10.5,-2)--(9.5,-1)--(8.5,-2)--(9.5,-3);;
\draw (9,-6) node {-1};;
\draw[fill=black] (9,-5) circle (5pt);;
\draw (10,-6) node {-4};;
\draw (10,-5) circle (5pt);;
\draw (11,-6) node {-7};;
\draw[fill=black] (11,-5) circle (5pt);;
\draw (12,-6) node {$\cdots$};;
\end{tikzpicture}
\caption{The 3-quotient of $\lambda=(4,4,2)$.}
\label{fig:QuotEx}
\end{figure}
\noindent In Section \ref{Combi}, we will arrange the quotient diagrams $\{m_i(\lambda)\}$ `abacus style': one over the other in decreasing order starting from $i=\ell-1$. 
We show this in Figure \ref{fig:AbacusEx} for $(4,4,2)$ and $\ell=3$.
\begin{figure}[h]
\centering
\begin{tikzpicture}[scale=.5]
\draw (-5,-2.5) node {$m_0(\lambda):$};;
\draw (-2,-3) node {$\cdots$};;
\draw (-1,-3) node {0};;
\draw (-1,-2) circle (5pt);;
\draw (-0.5,-2.5)--(-0.5,-1.5);;
\draw (0,-3) node {-3};;
\draw (0,-2) circle (5pt);;
\draw (1,-3) node {-6};;
\draw[fill=black] (1,-2) circle (5pt);;
\draw (2,-3) node {$\cdots$};;
\draw (-5,-0.5) node {$m_1(\lambda):$};;
\draw (-3,-1) node {$\cdots$};;
\draw (-2,-1) node {4};;
\draw (-2,0) circle (5pt);;
\draw (-1,-1) node {1};;
\draw[fill=black] (-1,0) circle (5pt);;
\draw (-.5,-0.5)--(-.5,0.5);;
\draw (0,-1) node {-2};;
\draw[fill=black] (0,-0) circle (5pt);;
\draw (1,-1) node {$\cdots$};;
\draw (-5, 1.5) node {$m_2(\lambda):$};;
\draw (-3,1) node {$\cdots$};;
\draw (-2,1) node {5};;
\draw (-2,2) circle (5pt);;
\draw (-1,1) node {2};;
\draw[fill=black] (-1,2) circle (5pt);;
\draw (-.5, 1.5)--(-.5,2.5);;
\draw (0,1) node {-1};;
\draw[fill=black] (0,2) circle (5pt);;
\draw (1,1) node {-4};;
\draw (1,2) circle (5pt);;
\draw (2,1) node {-7};;
\draw[fill=black] (2,2) circle (5pt);;
\draw (3,1) node {$\cdots$};;
\end{tikzpicture}
\caption{The abacus for $\lambda=(4,4,2)$ and $\ell=3$.}
\label{fig:AbacusEx}
\end{figure}
\noindent Notice that we position the diagrams so that the corresponding solid lines are vertically aligned.

To obtain $\core(\lambda)$, we revert each $m_i(\lambda)$ to the vacuum diagram corresponding to the shifted central line and then make the corresponding changes to $m(\lambda)$. 
In the example of $\lambda=(4,4,2)$ and $\ell=3$, we only need to change 2 to white and -4 to black.
This is shown in Figure \ref{fig:MayaCore}.
\begin{figure}[h]
\centering
\begin{tikzpicture}[scale=.5]
\draw (-6,.5) node {$m(\lambda)$:};;
\draw (-4, 0) node {$\cdots$};;
\draw (-3,0) node {3};;
\draw (-3,1) circle (5pt);;
\draw (-2,0) node {2};;
\draw[fill=black] (-2,1) circle (5pt);;
\draw (-1,0) node {1};;
\draw[fill=black] (-1,1) circle (5pt);;
\draw (0,0) node {0};;
\draw (0,1) circle (5pt);;
\draw (.5,.5)--(.5,1.5);;
\draw (1,0) node {-1};;
\draw[fill=black] (1,1) circle (5pt);;
\draw (2,0) node {-2};;
\draw[fill=black] (2,1) circle (5pt);;
\draw (3,0) node {-3};;
\draw (3,1) circle (5pt);;
\draw (4,0) node {-4};;
\draw (4,1) circle (5pt);;
\draw (5,0) node {-5};;
\draw[fill=black] (5,1) circle (5pt);;
\draw (6,0) node {$\cdots$};;
\draw (0.5,-2) node {$\downarrow$};;
\draw (-3, -9) node {$\cdots$};;
\draw (-2,-9) node {2};;
\draw (-2,-8) circle (5pt);;
\draw (-1,-9) node {1};;
\draw[fill=black] (-1,-8) circle (5pt);;
\draw (0,-9) node {0};;
\draw (0,-8) circle (5pt);;
\draw (.5,-8.5)--(.5,-7.5);;
\draw (1,-9) node {-1};;
\draw[fill=black] (1,-8) circle (5pt);;
\draw (2,-9) node {-2};;
\draw[fill=black] (2,-8) circle (5pt);;
\draw (3,-9) node {-3};;
\draw (3,-8) circle (5pt);;
\draw (4,-9) node {-4};;
\draw[fill=black] (4,-8) circle (5pt);;
\draw (5,-9) node {$\cdots$};;
\draw (.5,-7)--(3.5,-4);;
\draw (-.5,-6)--(2.5,-3);;
\draw (-1.5, -5)--(-.5,-4);;
\draw (.5,-7)--(-1.5,-5);;
\draw (1.5,-6)--(-.5,-4);;
\draw (2.5,-5)--(1.5,-4);;
\draw (3.5,-4)--(2.5,-3);;
\end{tikzpicture}
\caption{Obtaining the 3-core from the Maya diagram of $\lambda=(4,4,2)$.}
\label{fig:MayaCore}
\end{figure}
\noindent 
Note then that the Maya diagrams $\{m_i(\core(\lambda))\}$ are the vacuum diagrams shifted by the charges $c_i$, so these charges entirely determine $\core(\lambda)$. 
Since $m(\lambda)$ has charge zero, it follows that
\[c_0+\cdots+c_{\ell-1}=0\]
We will view the vector $(c_0,\ldots,c_{\ell-1})$ as an element of the root lattice for $\mathfrak{sl}_\ell$. 
Abusing notation, we will denote both the root lattice vector and the partition by $\core(\lambda)$. 

\begin{prop}[cf. \cite{JamesComb}]
The core-quotient decomposition yields a bijection
\[\{\hbox{\rm partitions}\}\rightarrow\{\ell\hbox{\rm-cores}\}\times\{\ell\hbox{\rm-multipartitions}\}\]
\end{prop}

\begin{rem}\label{Rem12}
We make some observations about cores and quotients that will be used later and may help the reader build some intuition for them.
\begin{enumerate}
\item By Remark \ref{Rem11}, adding a single node to $\lambda^i$ amounts to adding an $\ell$-strip to $\lambda$ starting at an $i$-node. 
It is natural to build up a partition by adding not its individual boxes but its columns, and one may try a similar approach with $\ell$-quotients. 
With that in mind, suppose we add a new column of length $k$ to $\lambda^i$ that is no taller than any of the existing columns of $\lambda^i$. 
In $m_i(\lambda)$, this amounts to the switch

\begin{equation*}
\begin{tikzpicture}[scale=.5]
\draw (0,1) circle (5pt);;
\draw (0,0) node {$j+k$};;
\draw (1,1) circle (5pt);;
\draw (2,1) node {$\cdots$};;
\draw (3,1) circle (5pt);;
\draw (4,0) node {$j$};;
\draw[fill=black] (4,1) circle (5pt);;
\draw (6,1) node {$\rightarrow$};;
\draw (8,0) node {$j+k$};;
\draw[fill=black] (8,1) circle (5pt);;
\draw (9,1) circle (5pt);;
\draw (10,1) node {$\cdots$};;
\draw (11,1) circle (5pt);;
\draw (12,0) node {$j$};;
\draw (12,1) circle (5pt);;
\end{tikzpicture}
\end{equation*}
\noindent where all the beads in the middle are white.
This adds a $k\ell$-strip to $\lambda$, but because the middle beads above are all white, any middle $i$-node of this new strip must be \textit{below} the $(i+1)$-node preceding it in the strip (recall that we totally order the nodes within a strip from the northwest to southeast). 
For a general column, the move will be multiple iterations of this, resulting in multiple strips wherein the added $i$-nodes are always below any adjacent added $(i+1)$-node.

Similarly, adding a new row of length $k$ to $\lambda^i$ that is no longer than the existing rows of $\lambda^i$ looks like

\begin{equation*}
\begin{tikzpicture}[scale=.5]
\draw (0,1) circle (5pt);;
\draw (0,0) node {$j$};;
\draw[fill=black]  (1,1) circle (5pt);;
\draw (2,1) node {$\cdots$};;
\draw[fill=black] (3,1) circle (5pt);;
\draw (4,0) node {$j-k$};;
\draw[fill=black] (4,1) circle (5pt);;
\draw (6,1) node {$\rightarrow$};;
\draw (8,0) node {$j$};;
\draw[fill=black] (8,1) circle (5pt);;
\draw[fill=black] (9,1) circle (5pt);;
\draw (10,1) node {$\cdots$};;
\draw[fill=black] (11,1) circle (5pt);;
\draw (12,0) node {$j-k$};;
\draw (12,1) circle (5pt);;
\end{tikzpicture}
\end{equation*}
\noindent in $m_i(\lambda)$, where now all the middle beads are black.
In the newly added $k\ell$-strip of $\lambda$, any intermediate $i$-node must now sit to the \textit{right} of the $(i+1)$-node preceding it in the strip.
As with columns above, the general row addition consists of multiple connected components of this picture and thus the added $i$-node is always to the right of any adjacent added $(i+1)$-node

\item Adding exactly the same number of $i$-nodes to $\lambda$ for each $i\in\ZZ/\ell\ZZ$ does not change the core. 
To see this, note that by Remark \ref{Rem11}, adding an $i$-node changes a bead in $m_i(\lambda)$ from white to black. 
Adding an $(i+1)$-node changes a bead in $m_i(\lambda)$ from black to white. 
The other node additions do not affect the beads in $m_i(\lambda)$.
There are now two cases for this pair of bead swaps in $m_i(\lambda)$:
\begin{itemize}
\item the white-to-black swap and the black-to-white swap happen on the same side of the (solid) central line or
\item the two swaps happen on opposite sides of the central line
\end{itemize}
In the first case, the number of white beads in $m_i(\lambda)$ right of the central line and the number of black beads in $m_i(\lambda)$ left of the central line are unchanged.
On the other hand, the second case either increases both of these quantities by 1 or decreases them both by 1.
In all cases, the charge of $m_i(\lambda)$, which is the \textit{difference} between these two quantities, is unchanged.

\item For two partitions $\lambda$ and $\mu$, if $\lambda^i\ge\mu^i$ for all $i$, then $\lambda\ge\mu$. 
It is enough to consider the case where $\mu^i$ is obtained from $\lambda^i$ by a single northwest box slide while all the other components of the quotient are equal. 
In $m_i(\lambda)$, the box slide looks like
\begin{equation*}
\begin{tikzpicture}[scale=.5]
\draw (0,1) circle (5pt);;
\draw (0,0) node {$j+k$};;
\draw[fill=black]   (1,1) circle (5pt);;
\draw (2,1) node {$\cdots$};;
\draw[fill=black] (3,1) circle (5pt);;
\draw (4,0) node {$j$};;
\draw (4,1) circle (5pt);;
\draw (6,1) node {$\rightarrow$};;
\draw (8,0) node {$j+k$};;
\draw[fill=black] (8,1) circle (5pt);;
\draw (9,1) circle (5pt);;
\draw (10,1) node {$\cdots$};;
\draw (11,1) circle (5pt);;
\draw (12,0) node {$j$};;
\draw[fill=black]  (12,1) circle (5pt);;
\end{tikzpicture}
\end{equation*}
\noindent Here, the beads in the ellipses are arbitrary. 
Notice then that $k\ge 2$. 
Therefore, the corresponding $\ell$-strip in $\lambda$ and its new landing spot are completely disjoint---the content of the starting node of the old strip in $\lambda$ and the content of the ending node of the new strip in $\mu$ differ by at at least $\ell$. 
One can thus obtain $\mu$ from $\lambda$ via $\ell$ northwest box slides.
\end{enumerate}
\end{rem}

\subsection{Wreath Macdonald polynomials}\label{Macdonald} 
This subsection gives a reinterpretation of Haiman's definition of wreath Macdonald polynomials amenable to quantum algebraic methods. 
We will abuse notation by freely hopping between both sides of the Frobenius characteristic. 

\subsubsection{Specializing \ref{Wreath} to $\Gamma=\ZZ/\ell\ZZ$} 
We will index $(\ZZ/\ell\ZZ)_*$ using powers of a generator $c$. 
On the other hand, we will index $(\ZZ/\ell\ZZ)^*$ using additive notation: $i\in\ZZ/\ell\ZZ$ corresponds to the character
\[\gamma_i(c^j):=\zeta^{ij}\]
where $\zeta=e^\frac{2\pi\sqrt{-1}}{\ell}$. 
Observe that $h_n(i)$ corresponds to the character of $\Gamma_n$ where
\[(\vec{g},\sigma)\mapsto\left(\prod_{j=1}^n\gamma_i(g_j)\right)\]
and $e_n(i)$ corresponds to the character where
\[(\vec{g},\sigma)\mapsto\mathrm{sign}(\sigma)\left(\prod_{j=1}^n\gamma_i(g_j)\right)\]
Finally, we note that from (\ref{ConjGen}), the transition matrix from $\{p_n(c^i)\}$ to $\{p_n(i)\}$ for fixed $n$ is the following $\ell\times\ell$ Vandermonde matrix:
\[\begin{pmatrix}
1 & 1 & 1& \cdots & 1\\
1 & \zeta & \zeta^2& \cdots & \zeta^{\ell-1}\\
\vdots & \vdots &\vdots & &\vdots\\
1 & \zeta^k &\zeta^{2k} & \cdots & \zeta^{(\ell-1)k}\\
\vdots & \vdots & \vdots & &\vdots\\
1 & \zeta^{\ell-1} & \zeta^{2(\ell-1)} & \cdots & \zeta^{(\ell-1)^2}\\
\end{pmatrix}\]

\subsubsection{Bigraded characters} 
We will now work with the base-changed rings 
\begin{eqnarray*}
R_{q,t}(\ZZ/\ell\ZZ)&:=&\CC(q,t)\otimes R(\ZZ/\ell\ZZ)\\
\Lambda_{q,t}(\ZZ/\ell\ZZ)&:=&\CC(q,t)\otimes\Lambda(\ZZ/\ell\ZZ)
\end{eqnarray*}
where $q$ and $t$ are indeterminates. 
In order to define wreath Macdonald polynomials, we need to define the operators $-\otimes\bigwedge_{q}^\pm$ and $-\otimes\bigwedge_{t^{-1}}^\pm$. 
For each $n$, $\Gamma_n$ has a natural \textit{reflection representation} $\mathfrak{h}_n\cong\CC^n$ given by
\[(\vec{g},\sigma)\cdot(x_1,\ldots, x_n)=(\gamma_1(g_1)x_{\sigma^{-1}(1)},\ldots,\gamma_1(g_n)x_{\sigma^{-1}(n)})\]
For $\Gamma_n\curvearrowright V$ and a variable $s$, we define 
\begin{align*}
V\otimes\textstyle{\bigwedge_{s}^+}&:=\sum_{k=0}^n(-s)^kV\otimes\textstyle{\bigwedge^k}\mathfrak{h}_n\\
V\otimes\textstyle{\bigwedge_{s}^-}&:=\sum_{k=0}^n(-s)^kV\otimes\textstyle{\bigwedge^k}\mathfrak{h}_n^*
\end{align*}
and extend it linearly to all of $R_{q,t}(\ZZ/\ell\ZZ)$. 
We will abuse notation and also directly apply these maps to $\Lambda_{q,t}(\ZZ/\ell\ZZ)$.

\subsubsection{Bosons with sectors} 
We have already seen that $\Lambda_{q,t}(\ZZ/\ell\ZZ)$ has nice bases indexed by $\ell$-multipartitions. 
On the other hand, \ref{Corequot} tells us that $\ell$-cores are indexed by the root lattice $Q$ of $\mathfrak{sl}_\ell$. 
With that in mind, let $\CC[Q]$ denote the group algebra of $Q$. 
We will view root lattice vectors in terms of the basis $\{\alpha_1,\ldots, \alpha_{\ell-1}\}$, where $\alpha_i$ is the vector in $\ZZ^\ell$ with $1$ in its $i$th coordinate, $-1$ in its $(i+1)$th coordinate, and all other coordinates zero. 
Basis elements for $\CC[Q]$ will be denoted exponentially: $\{e^{\alpha}\}_{\alpha\in Q}$. 
Now, in $\Lambda_{q,t}(\ZZ/\ell\ZZ)\otimes\CC[Q]$, we can append the datum of an $\ell$-core to any of our bases in $\Lambda_{q,t}(\ZZ/\ell\ZZ)$. 
For an ordinary partition $\lambda$, we define the bases
\begin{align*}
\vec{s}_\lambda&:=s_{\quot(\lambda)}\otimes e^{\core(\lambda)}\\
\vec{e}_\lambda&:=e_{\quot(\lambda)}\otimes e^{\core(\lambda)}\\
\vec{h}_\lambda&:=h_{\quot(\lambda)}\otimes e^{\core(\lambda)}
\end{align*}


\subsubsection{Key definition}\label{WreathDef}
We will let $\mu\le_\ell\lambda$ denote that $\mu\le\lambda$ and $\core(\mu)=\core(\lambda)$.

\begin{defn}\label{MacDef}

The \textit{wreath Macdonald polynomial} $H_\lambda(q,t)\in\Lambda_{q,t}(\ZZ/\ell\ZZ)\otimes\CC[Q]$ is characterized by
\begin{enumerate}
\item $H_\lambda(q,t)\otimes\bigwedge_{q}^-\in\Span{\vec{s}_{\mu} :\mu\ge_\ell\lambda}$;
\item $H_\lambda(q,t)\otimes\bigwedge_{t^{-1}}^-\in\Span{\vec{s}_{\mu} :\mu\le_\ell\lambda}$;
\item the coefficient of the trivial representation of $\Gamma_{|\quot(\lambda)|}$ in $H_\lambda(q,t)$ is 1. 
\end{enumerate}

\end{defn}

\noindent We will often abbreviate $H_\lambda(q,t)$ by just $H_\lambda$.

\begin{rem} Our definition is equivalent to that of Haiman's in \cite{Haiman}, although it appears different at first. 
In \textit{loc. cit.}, conditions (1) and (2) are written with $\bigwedge_q^+$ and $\bigwedge_{t^{-1}}^+$, but Haiman's definition of content is the negative of ours. 
It follows from Proposition \ref{plethprop} below that acting on $\vec{s}_\lambda$, our $-\otimes\bigwedge_q^-$ is actually his $-\otimes\bigwedge_q^+$ and likewise for $q$ switched with $t^{-1}$.  
We are somewhat forced to rewrite the definition in this way because of our conventions for the quantum toroidal and shuffle algebras. 
\end{rem}

\subsubsection{Plethysm} 
We will try to make the definition above a little more transparent by writing the maps $-\otimes\bigwedge_q^-$ and $-\otimes\bigwedge_{t^{-1}}^-$ in terms of an analogue of plethystic substitution (cf. \cite{Haiman}). 
Specifically, these two maps are algebra  endomorphisms of $\Lambda_{q,t}(\ZZ/\ell\ZZ)$ defined simply in terms of the generators $p_n(i)$.

\begin{prop}\label{plethprop}
The map $-\otimes\bigwedge_q^-$ is the algebra endomorphism $\Phi_q$ of $\Lambda_{q,t}(\ZZ/\ell\ZZ)$ defined by
\[\Phi_q(p_n(i))=p_n(i)-q^np_n(i-1)\]
Similarly, $-\otimes\bigwedge_{t^{-1}}^-$ is given by
\[\Phi_{t^{-1}}(p_n(i))= p_n(i)-t^{-n}p_n(i-1)\]
\end{prop}

\begin{proof} 

It is enough to prove the $q$-statement. 
Since $-\otimes\bigwedge_q^-$ is diagonalized on indicator class functions, the skeptic should check that $\Phi_q$ has the same property. 
The indicator class functions are products of the $p_n(c^i)$, so it is enough to prove this for the $p_n(c^i)$. 
This follows from the following identity of $\ell\times\ell$ matrices
\begin{eqnarray*}
&&\begin{pmatrix}
1 & -q^n & 0 & \cdots & 0\\
0 & 1 & -q^n & \cdots & 0\\
\vdots &\ddots&\ddots & \ddots &\vdots\\
0 & \cdots &0 & 1 & -q^n \\
-q^n & 0 &\cdots & 0 & 1\\
\end{pmatrix}
\begin{pmatrix}
1 & 1 & 1& \cdots & 1\\
1 & \zeta & \zeta^2& \cdots & \zeta^{\ell-1}\\
\vdots & \vdots &\vdots & &\vdots\\
1 & \zeta^k &\zeta^{2k} & \cdots & \zeta^{(\ell-1)k}\\
\vdots & \vdots & \vdots & &\vdots\\
1 & \zeta^{\ell-1} & \zeta^{2(\ell-1)} & \cdots & \zeta^{(\ell-1)^2}\\
\end{pmatrix}\\
&=&\begin{pmatrix}
1 & 1 & 1& \cdots & 1\\
1 & \zeta & \zeta^2& \cdots & \zeta^{\ell-1}\\
\vdots & \vdots &\vdots & &\vdots\\
1 & \zeta^k &\zeta^{2k} & \cdots & \zeta^{(\ell-1)k}\\
\vdots & \vdots & \vdots & &\vdots\\
1 & \zeta^{\ell-1} & \zeta^{2(\ell-1)} & \cdots & \zeta^{(\ell-1)^2}\\
\end{pmatrix}
\begin{pmatrix}
1-q^n   & 0& \cdots & 0\\
0 & 1-\zeta q^n & \ddots&\vdots\\
\vdots & \ddots&\ddots&0\\
0&\cdots & 0&1-\zeta^{\ell-1}q^n\\
\end{pmatrix}
\end{eqnarray*}
and the invertibility of the Vandermonde matrix. 
Therefore, we just need to check that their eigenvalues match. 

Since the trivial representation is the sum of indicator class functions, it is enough then to check that $\Phi_q(h_n(0))=\sum_{k=0}^n(-q)^k\bigwedge^k\mathfrak{h}_n^*$. As is often the case, it is easier to consider all $n$ at the same time:
\begin{align*}
\sum_{n\ge 0}\Phi_q(h_n(0))z^n&=\exp\left(\sum_{n>0}\Phi_q(p_n(0))\frac{z^n}{n}\right)\\
&=\exp\left(\sum_{n>0}(p_n(0)-q^np_n(-1))\frac{z^n}{n}\right)\\
&=\left(\sum_{n\ge 0}h_n(0)z^n\right)\left(\sum_{n\ge 0}(-q)^ne_n(-1)z^n\right)\\
&=\sum_{n\ge 0}\left(\sum_{j+k=n} (-q)^kh_j(0)e_k(-1)\right)z^n
\end{align*}
Next, observe that $h_j(0)e_k(-1)$ corresponds to $\bigwedge^k\mathfrak{h}_n^*$ under the Frobenius characteristic. 
Specifically, let $I_-$ denote the irrep of $\ZZ/\ell\ZZ$ where the generator acts by $e^{-2\pi i/\ell}$ and let $I_0$ denote the trivial representation of $\ZZ/\ell\ZZ$.
By \ref{WreathFrob}, $h_j(0)e_k(-1)$ is induced from the following 1-dimensional representation of $\Gamma_j\times\Gamma_k\subset\Gamma_n$:
\begin{equation}
I_0^{\otimes j}\boxtimes\left( \mathrm{sign}\otimes I_-^{\otimes k} \right)
\label{ejk}
\end{equation}
Let $v^1,\ldots, v^n\in\mathfrak{h}_n^*$ be the standard (dual) basis.
Now recall Frobenius reciprocity: 
\begin{equation}
\Hom_{\Gamma_n}\left(\mathrm{Ind}_{\Gamma_j\times\Gamma_k}^{\Gamma_n} (V), W  \right)\cong\Hom_{\Gamma_j\times\Gamma_k}\left( V,\mathrm{Res}_{\Gamma_j\times\Gamma_k}^{\Gamma_n}(W) \right)
\label{FrobRep}
\end{equation}
The vector 
\[v^{n-k+1}\wedge\cdots\wedge v^n\in\textstyle\bigwedge^k\mathfrak{h}_n^*\]
induces a nonzero $\Gamma_j\times\Gamma_k$-homomorphism from the representation (\ref{ejk}) to the restriction of $\bigwedge^k\mathfrak{h}_n^*$, an object in the right hand side of (\ref{FrobRep}). 
Frobenius reciprocity then yields a $\Gamma_n$-homomorphism from the representation associated to $h_j(0)e_{k}(-1)$ to $\bigwedge^k\mathfrak{h}_n^*$.
Since $\mathfrak{h}_n^*$ is spanned by the $\Gamma_n$-orbit of $v^{n-k+1}\wedge\cdots\wedge v^n$, this map is surjective.
Finally, both representations have dimension $\binom{n}{k}$, so the map is an isomorphism.
\end{proof}

\subsubsection{A recharacterization} 
We continue onwards with our tidying up of Definition \ref{MacDef}. 
First, note that we can invert $\Phi_q$ and $\Phi_{t^{-1}}$ by explicit calculation:

\begin{lem}\label{InversePhi}
We have
\begin{align*}
\Phi_q^{-1}(p_n(i))&=\frac{p_n(i)+q^np_n(i-1)+q^{2n}p_n(i-2)+\cdots+q^{(\ell-1)n}p_n(i+1)}{1-q^{n\ell}}\\
\Phi_{t^{-1}}^{-1}(p_n(i))&=\frac{p_n(i)+t^{-n}p_n(i-1)+t^{-2n}p_n(i-2)+\cdots+t^{-(\ell-1)n}p_n(i+1)}{1-t^{-n\ell}}
\end{align*}

\end{lem}

Let us set
\begin{align}
\begin{split}
\hat{h}_n(i)&:=\Phi_q^{-1}(h_n(i))\\
\hat{h}_\lambda&:=\Phi_q^{-1}(h_{\quot(\lambda)})\otimes e^{\core(\lambda)}\\
\hat{e}_n(i)&:=\Phi_{t^{-1}}^{-1}(e_n(i))\\
\hat{e}_\lambda&:=\Phi_{t^{-1}}^{-1}(e_{\quot(\lambda)})\otimes e^{\core(\lambda)}
\end{split}
\label{HatBoson}
\end{align}
Recall from Proposition \ref{SchurTriangle} that the transition matrix between $\{s_\lambda\}$ and $\{e_\lambda\}$ is upper triangular with respect to dominance order while the one between $\{s_\lambda\}$ and $\{h_\lambda\}$ is lower triangular. 
Combining this with Remark \ref{Rem12}(3), we can alter Definition \ref{MacDef} in the following way:

\begin{prop}\label{AltDef}
The wreath Macdonald polynomial $H_\lambda$ is characterized by
\begin{enumerate}
\item $H_\lambda\in\Span{\hat{h}_\mu :\mu\ge_\ell\lambda}$;
\item $H_\lambda\in\Span{\hat{e}_\mu :\mu\le_\ell\lambda}$;
\item the coefficient of the trivial representation of $\Gamma_{|\quot(\lambda)|}$ in $H_\lambda$ is 1. 
\end{enumerate}
\end{prop}

\noindent In conclusion, we can characterize the line spanned by $H_\lambda$ as the intersection of a subspace built out of the generators $\{\hat{h}_n(i)\}$ and another built out of the generators $\{\hat{e}_n(i)\}$. In \ref{Redux}, we will give a cleaner presentation of $\hat{h}_n(i)$ and $\hat{e}_n(i)$ in terms of \textit{dual bosons}.

\section{Quantum toroidal algebra} 
Our presentation closely follows \cite{Tsym}. 
In this section, $\ell\ge 3$.

\subsection{Definitions}\label{TorDef} 
In this subsection, we will define the quantum toroidal algebra and its structures as a topological Hopf algebra. 
Let $\qqq$ and $\ddd$ be two variables and set $\mathbb{F}:=\CC(\qqq^\frac{1}{2},\ddd^\frac{1}{2})$. 
We will relate them to $(q,t)$ later on in \ref{Fock}.

\subsubsection{The algebra $\UTor$} 
For $i,j\in\ZZ/\ell\ZZ$, define $a_{i,j}$ and $m_{i,j}$ to be
\[a_{i,i}=2,\, a_{i,i\pm1}=-1,\, m_{i,i\pm 1}=\mp 1,\hbox{ and }a_{i,j}=m_{i,j}=0\hbox{ otherwise}\]
We then define
\[g_{i,j}(z):=\frac{\qqq^{a_{i,j}}z-1}{z-\qqq^{a_{i,j}}}\]
The \textit{quantum toroidal algebra} $U_{\qqq,\ddd}(\ddot{\mathfrak{sl}}_\ell)$ is a unital associative $\FF$-algebra generated by 
\[\{e_{i,k},f_{i,k},\psi_{i,k},\psi_{i,0}^{-1},\gamma^{\pm\frac{1}{2}}, \qqq^{\pm d_1},\qqq^{\pm d_2}\}_{i\in\ZZ/\ell\ZZ}^{k\in\ZZ}\]
To describe its relations, we piece together the generators indexed by $\ZZ$ into the currents
\begin{align*}
e_i(z)&:=\sum_{k\in\ZZ}e_{i,k}z^{-k}\\
f_i(z)&:=\sum_{k\in\ZZ}f_{i,k}z^{-k}\\
\psi_i^\pm(z)&:=\psi_{i,0}^{\pm 1}+\sum_{k>0}\psi_{i,\pm k}z^{\mp k}
\end{align*}
The relations then are:
\begin{gather*}
[\psi_i^\pm(z),\psi_j^\pm(w)]=0,\,\gamma^{\pm\frac{1}{2}}\hbox{ are central},\\
\psi_{i,0}^{\pm1}\psi_{i,0}^{\mp1}=\gamma^{\pm\frac{1}{2}}\gamma^{\mp\frac{1}{2}}=\qqq^{\pm d_1}\qqq^{\mp d_1}=\qqq^{\pm d_2}\qqq^{\mp d_2}=1,\\
\qqq^{d_1}e_i(z)\qqq^{-d_1}=e_i(\qqq z),\, \qqq^{d_1}f_i(z)\qqq^{-d_1}=f_i(\qqq z),\, \qqq^{d_1}\psi_i^\pm(z)\qqq^{-d_1}=\psi_i^\pm(\qqq z),\\
\qqq^{d_2}e_i(z)\qqq^{-d_2}=\qqq e_i(z),\, \qqq^{d_2}f_i(z)\qqq^{-d_2}=\qqq^{-1} f_i(z),\, \qqq^{d_2}\psi_i^\pm(z)\qqq^{-d_2}=\psi_i^\pm( z),\\
g_{i,j}(\gamma^{-1}\ddd^{m_{i,j}}z/w)\psi_i^{+}(z)\psi_j^{-}(w)=g_{i,j}(\gamma\ddd^{m_{i,j}} z/w)\psi_j^{-}(w)\psi_i^{+}(z),\\
e_i(z)e_j(w)=g_{i,j}(\ddd^{m_{i,j}}z/w)e_j(w)e_i(z),\\
f_i(z)f_j(w)=g_{i,j}(\ddd^{m_{i,j}}z/w)^{-1}f_j(w)f_i(z),\\
(\qqq-\qqq^{-1})[e_i(z),f_j(w)]=\delta_{i,j}\left(\delta(\gamma w/z)\psi_i^+(\gamma^{\frac{1}{2}}w)-\delta(\gamma z/w)\psi_i^-(\gamma^\frac{1}{2}z)\right),\\
\psi_i^\pm(z)e_j(w)=g_{i,j}(\gamma^{\pm\frac{1}{2}}\ddd^{m_{i,j}}z/w)e_j(w)\psi_i^\pm(z),\\
\psi_i^\pm(z)f_j(w)=g_{i,j}(\gamma^{\mp\frac{1}{2}}\ddd^{m_{i,j}}z/w)^{-1}f_j(w)\psi_i^\pm(z),\\
\Sym_{z_1,z_2}[e_i(z_1),[e_i(z_2),e_{i\pm1}(w)]_\qqq]_{\qqq^{-1}}=0,\,[e_i(z),e_j(w)]=0\hbox{ for }j\not=i,i\pm1,\\
\Sym_{z_1,z_2}[f_i(z_1),[f_i(z_2),f_{i\pm1}(w)]_\qqq]_{\qqq^{-1}}=0,\,[f_i(z),f_j(w)]=0\hbox{ for }j\not=i,i\pm1,
\end{gather*}
Here, $\delta(z)$ denotes the delta function
\[\delta(z)=\sum_{k\in\ZZ}z^k\]
and $[a,b]_\qqq=ab-\qqq ba$ is the $\qqq$-commutator. 

Observe that $\qqq^{d_1}$ and $\qqq^{d_2}$ each gives rise to a grading on $U_{\qqq,\ddd}(\ddot{\mathfrak{sl}}_\ell)$. 
We call the $\qqq^{-d_1}$ grading the \textit{homogenous grading} since it is the degree of a Fourier coefficient, e.g. $e_{i,n}$ has degree $n$. 
The $\qqq^{d_2}$ grading, on the other hand, counts the number of $e_i(z)$ components minus the number of $f_i(z)$ components. 
We call this latter grading the \textit{principal grading}.
Finally, note that the $\UTor$ contains two central elements: $\gamma^{\frac{1}{2}}$ and $\kappa:=\prod_{i=0}^{\ell-1}\psi_{i,0}$.
For $a,b\in\CC$, we say a representation of $\UTor$ has \textit{central charge} $(a,b)$ if $\gamma$ acts by the scalar $\qqq^a$ and $\kappa$ acts by the scalar $\qqq^b$.

Note that the currents $\psi_{i,0}^{\mp1}\psi_i^\pm(z)$ generate a Heisenberg algebra, to be explored further in \ref{Vertex}. 
We will also make frequent use of the Heisenberg generators $\{b_{i,k}\}_{k\not=0}$ given by
\[\psi_i^\pm(z)=\psi_{i,0}^{\pm 1}\exp\left(\pm(\qqq-\qqq^{-1})\sum_{k>0}b_{i,\pm k}z^{\mp k}\right)\]
Similarly, we will also need the elements $H_{i,0}$, where $\psi_{i,0}=\qqq^{H_{i,0}}$. 
We use these elements to make sense of $\qqq^{\omega}$ for any $\omega$ in the weight lattice of affine $\mathfrak{sl}_\ell$ (thus, strictly speaking, we shall work with a naturally extended version of $\UTor$). 

\subsubsection{Topological Hopf algebra structure}\label{Hopf} 
The general framework of Ding-Iohara \cite{DingIohara} allows us to endow $U_{\qqq,\ddd}(\ddot{\mathfrak{sl}}_\ell)$ with a topological Hopf algebra structure:
\begin{gather*}
\Delta(\psi_i^\pm(z))=\psi_i^\pm(\gamma_{(2)}^{\mp\frac{1}{2}}z)\otimes\psi_i^\pm(\gamma_{(1)}^{\pm\frac{1}{2}}z),\,\Delta(x)=x\otimes x\hbox{ for }x=\gamma^{\pm\frac{1}{2}},\,\qqq^{\pm d_1},\,\qqq^{\pm d_2},\\
\Delta(e_i(z))=e_i(\gamma_{(2)}^{-1}z)\otimes\psi_i^+(\gamma_{(2)}^{-\frac{1}{2}}z) +1 \otimes e_i(z),\\
\Delta(f_i(z))=\psi_i^-(\gamma_{(1)}^{-\frac{1}{2}}z)\otimes f_i(\gamma_{(1)}^{-1}z)+f_i(z)\otimes 1,\\
\epsilon(e_i(z))=\epsilon(f_i(z))=0,\,\epsilon(\psi_i^\pm(z))=1,\,\epsilon(x)=1\hbox{ for }x=\gamma^{\pm\frac{1}{2}},\,\qqq^{\pm d_1},\,\qqq^{\pm d_2},\\
S(e_i(z))=-e_i(\gamma z)\psi_i^+(\gamma^{\frac{1}{2}}z)^{-1},\,S(f_i(z))=-\psi_i^-(\gamma^{\frac{1}{2}}z)^{-1}f_i(\gamma z),\\
S(x)=x^{-1}\hbox{ for }x=\gamma^{\pm\frac{1}{2}},\,\qqq^{\pm d_1},\,\qqq^{\pm d_2},\, \psi_i^{\pm}(z)
\end{gather*}
where $\gamma_{(1)}=\gamma\otimes 1$ and $\gamma_{(2)}=1\otimes \gamma$.

\begin{rem}
Note that our coproduct differs from the one in \cite{FJMMRep} and \cite{Tsym}.
Namely, we attach $\psi_i^+(z)$ to $\Delta (e_i(z))$ and $\psi_i^-(z)$ to $\Delta (f_i(z))$.
There is some freedom in choosing the powers of $\gamma$ and whether the Cartan current appears in the left or right tensorand of $\Delta(e_i(z))$ and $\Delta(f_i(z))$.
We make this choice so that the bialgebra pairing defined below in Theorem \ref{TorPairDef} takes its particular form.
\end{rem}

\subsubsection{Bialgebra pairing}\label{Bipair} 
We define the following subalgebras of $U_{\qqq,\ddd}(\ddot{\mathfrak{sl}}_\ell)$:
\begin{enumerate}
\item $'\ddot{U}$: the subalgebra obtained by dropping $\qqq^{d_1}$;
\item $\ddot{U}'$: the subalgebra obtained by dropping $\qqq^{d_2}$;
\item $'\ddot{U}'$: the subalgebra obtained by dropping both $\qqq^{d_1}$ and $\qqq^{d_2}$;
\item $\ddot{U}^+$: the subalgebra generated by the currents $\{e_i(z)\}$;
\item $\ddot{U}^-$: the subalgebra generated by the currents $\{f_i(z)\}$;
\item $\ddot{U}^0$: the subalgebra generated by the currents $\{\psi_{i,0}^{\mp1}\psi_i^\pm(z)\}$;
\item $\ddot{U}^\ge$: the subalgebra generated by the currents $\{e_i(z), \psi_i^+(z)\}$ along with, $\gamma^{\pm\frac{1}{2}}$, $\qqq^{\pm d_1}$, and $\qqq^{\pm d_2}$;
\item $\ddot{U}^\le$: the subalgebra generated by the currents $\{f_i(z), \psi_i^-(z)\}$ along with $\gamma^{\pm\frac{1}{2}}$, $\qqq^{\pm d_1}$, and $\qqq^{\pm d_2}$.
\end{enumerate}
The last two are in fact sub-bialgebras. In the spirit of (1) and (2), we apply left or right primes to any subalgebras to denote that we have dropped $\qqq^{d_1}$ or $\qqq^{d_2}$, respectively. 

Recall that a pairing of bialgebras $\varphi(-,-): A\times B\rightarrow\FF$ satisfies
\begin{align}\label{BialgPairDef}
\begin{split}
\varphi\left(a, bb'\right)&=\varphi\left( \Delta_A(a),b\otimes b' \right)\\
\varphi\left( aa', b \right)&=\varphi\left( a\otimes a',\Delta_B^{\mathrm{op}}(b) \right)
\end{split}
\end{align}
Given this setup (two bialgebras with a bialgebra pairing), one can define the \textit{Drinfeld double} $D(A,B)$, which is generated by $A$ and $B$ and subject to the relation
\begin{equation}
\varphi(a_{[1]}, b_{[1]})a_{[2]}b_{[2]}=b_{[1]}a_{[1]}\varphi(a_{[2]}, b_{[2]})
\label{DrinDouble}
\end{equation}
Here, we have used Sweedler notation for the coproduct:
\[
\Delta(x)=x_{[1]}\otimes x_{[2]}
\]
We have the following structural result:

\begin{thm}[cf. \cite{NegutTor}, \cite{Tsym}]\label{TorPairDef}
There exists a unique bialgebra pairing $\varphi:\ddot{U}^{\ge}\times\ddot{U}^\le\rightarrow\FF$ such that
\begin{gather*}
\varphi(e_i(z),f_j(w))=\dfrac{\delta_{i,j}}{\qqq-\qqq^{-1}}\delta(z/w),\,\varphi(\psi_i^+(z),\psi_j^-(w))=g_{i,j}(\ddd^{m_{i,j}}z/w)\bigg|_{\|z\|\gg\|w\|},\\
\varphi(e_i(z),x^-)=\varphi(x^+,f_i(z))=0\hbox{ for }x^\pm=\psi_j^{\pm}(w), \gamma^\frac{1}{2},\qqq^{d_1},\qqq^{d_2},\\
\varphi(\psi_i^+(z),x)=\varphi(x,\psi_i^-(z))=1\hbox{ for }x=\gamma^\frac{1}{2},\qqq^{d_1},\\
\varphi(\psi_i^+(z),\qqq^{d_2})=\qqq,\,\varphi(\qqq^{d_2},\psi_i^-(z))=\qqq^{-1},\, \varphi(\qqq^{d_2},\qqq^{d_2})=\qqq^{\frac{\ell(\ell^2-1)}{12}},\\
\varphi(\gamma^{\frac{1}{2}},\qqq^{d_1})=\varphi(\qqq^{d_1},\gamma^{\frac{1}{2}})=\qqq^{-\frac{1}{2}},\\
\varphi(\gamma^\frac{1}{2},\qqq^{d_2})=\varphi(\qqq^{d_2},\gamma^\frac{1}{2})=\varphi(\gamma^\frac{1}{2},\gamma^\frac{1}{2})=\varphi(\qqq^{d_1},\qqq^{d_1})=\varphi(\qqq^{d_2},\qqq^{d_1})=1
\end{gather*}
The pairing $\varphi$ is nondegenerate.
Moreover, $U_{\qqq,\ddd}(\ddot{\mathfrak{sl}}_\ell)$ is isomorphic to the Drinfeld double of this pairing modulo the relations
\begin{gather*}
x\otimes 1-1\otimes x\hbox{ for }x=\gamma^{\pm\frac{1}{2}},\qqq^{\pm d_1},\qqq^{\pm d_2},\\
(\psi_{i,0}\otimes 1)(1\otimes \psi_{i,0}^{-1})=1
\end{gather*}

\end{thm}

\noindent Under this pairing, $b_{i,k}$ is orthogonal to the generators $\{ e_{j,n}, f_{j,n}, \gamma^{\pm\frac{1}{2}}, \qqq^{\pm d_1},\qqq^{\pm d_2} \}$.
Its pairing with the generators $\{\psi_{j,n}\}$ can be deduced from 
\begin{equation}
\varphi(b_{i,k},b_{j,-k'})=-\delta_{k,k'}\frac{[ka_{i,j}]_\qqq\ddd^{-km_{i,j}}}{k(\qqq-\qqq^{-1})} \hbox{ for }k>0
\label{HeisPair}
\end{equation}
We define the \textit{dual} Heisenberg generators $\{b_{i,k}^\perp\}_{k\not=0}\subset\ddot{U}^0$ by the property
\[\varphi(b_{i,k}^\perp,b_{j,-k'})=\varphi(b_{j,k'},b_{i,-k}^\perp)=\delta_{i,j}\delta_{k, k'} \hbox{ for }k>0\]

\begin{rem}\label{PerpRem}
Elements like $b_{i,k}^\perp$ were defined in Definition 4.3 of \cite{Tsym}, from which we drew the above definition.
However, in \textit{loc. cit.}, $h_{i,k}^\perp$ is defined so that it is a linear combination of $h_{j,-k}$.
We have instead elected to preserve homogeneous degree: our $b_{i,k}^\perp$ is a linear combination of $b_{j,k}$.
Note as well that our pairings are slightly different due to the different coproducts.
\end{rem}

\subsection{S-duality} 
One feature that makes $U_{\qqq,\ddd}(\ddot{\mathfrak{sl}}_\ell)$ deserve its name is that it contains two copies of the quantum affine algebra $U_\qqq(\dot{\mathfrak{gl}}_\ell)$ and hence has `two loops'. 
We review this and Miki's isomorphism interchanging the two copies.

\subsubsection{$U_\qqq(\dot{\mathfrak{sl}}_\ell)$ via Drinfeld-Jimbo} 
Recall that $U_\qqq(\dot{\mathfrak{sl}}_\ell)$ has two presentations. 
One is the classical Drinfeld-Jimbo presentation (cf. \cite{Lusztig}), wherein the algebra is generated by $\{E_i, F_i, K_i^\pm,D^{\pm1}\}_{i\in\ZZ/\ell\ZZ}$ and satisfies the relations
\begin{gather*}
D^{\pm1}D^{\mp1}=1,\, DK_iD^{-1}=K_i,\,DE_iD^{-1}= \qqq E_i,\, DF_iD^{-1}=\qqq^{-1}F_i,\\
K_i^{\pm1}K_i^{\mp1}=1,\,K_iK_j=K_jK_i,\,K_iE_jK_i=\qqq^{a_{i,j}}E_j,\,K_iF_jK_i=\qqq^{-a_{i,j}}F_j,\\
[E_i,F_j]=\delta_{i,j}\dfrac{K_i-K_i^{-1}}{\qqq-\qqq^{-1}},\\
\displaystyle\sum_{s=0}^{1-a_{i,j}}\dfrac{(-1)^s}{[s]_\qqq![1-a_{i,j}-s]_\qqq!}E_i^sE_jE_i^{1-a_{i,j}-s}=0,\\
\sum_{s=0}^{1-a_{i,j}}\dfrac{(-1)^s}{[s]_\qqq![1-a_{i,j}-s]_\qqq!}F_i^sF_jF_i^{1-a_{i,j}-s}=0
\end{gather*}
Note that the element
\[
C:= K_0\cdots K_{\ell-1}
\]
is central.

$U_\qqq(\dot{\mathfrak{sl}}_\ell)$ has a coproduct $\bar{\Delta}$ given by
\begin{gather*}
\bar{\Delta}(E_i)=E_i\otimes 1+K_i\otimes E_i, \bar{\Delta}(F_i)=F_i\otimes K_i^{-1}+1\otimes F_i\\
\bar{\Delta}(x)=x\otimes x\hbox{ for }x=K_i,D
\end{gather*}
Similar to the toroidal case, we define the following subalgebras:
\begin{enumerate}
\item $\dot{U}^\ge_{\mathfrak{sl}}$: the subalgebra generated by $\{E_i, K_i, D^{-1}\}$;
\item $\dot{U}^\le_{\mathfrak{sl}}$: the subalgebra generated by $\{F_i, K_i^{-1}, D\}$;
\end{enumerate}
Both are sub-bialgebras with respect to the coproduct $\bar{\Delta}$.
A bialgebra pairing $\bar{\varphi}_{\mathfrak{sl}}:\dot{U}^\ge_{\mathfrak{sl}}\times\dot{U}^\le_{\mathfrak{sl}}\rightarrow\CC(\qqq)$ can be defined by setting: 
\begin{gather*}
\bar{\varphi}_{\mathfrak{sl}}(E_i,F_j)=\frac{\delta_{i,j}}{\qqq-\qqq^{-1}},\, \bar{\varphi}_{\mathfrak{sl}}(K_i,K_j^{-1})=\qqq^{a_{i,j}}, \\
\bar{\varphi}_{\mathfrak{sl}}(K_i, D)=\bar{\varphi}_{\mathfrak{sl}}(D^{-1}, K_i^{-1})=\qqq\\
\bar{\varphi}_{\mathfrak{sl}}(E_i,K_j^{-1})=\bar{\varphi}_{\mathfrak{sl}}(K_i,F_j)=\bar{\varphi}_{\mathfrak{sl}}(E_i,D)=\bar{\varphi}_{\mathfrak{sl}}(D,F_j)=0
\end{gather*}
This pairing is nondegenerate.

\subsubsection{New Drinfeld presentation}
The other presentation of $U_\qqq(\dot{\mathfrak{sl}}_\ell)$ is the new Drinfeld realization in terms of currents. 
Here, the algebra is generated by
\[\{\bar{e}_{i,k},\bar{f}_{i,k},\bar{\psi}_{i,k},\bar{\psi}_{i,0}^{-1},C^{\pm 1}, D^{\pm 1}\}_{i=1,\ldots,\ell-1}^{k\in\ZZ}\]
As before, we fit them into the currents $\bar{e}_i(z)$, $\bar{f}_i(z)$, and $\bar{\psi}_i^\pm(z)$, the only difference being that we do not allow $i=0$. 
Our relations should look familiar:
\begin{gather*}
[\bar{\psi}_i^\pm(z),\bar{\psi}_j^\pm(w)]=0,\,C^{\pm 1}\hbox{ are central},\\
\bar{\psi}_{i,0}^{\pm1}\bar{\psi}_{i,0}^{\mp1}=C^{\pm 1}C^{\mp 1}=D^{\pm1}D^{\mp1}=1,\\
D\bar{e}_i(z)D^{-1}=\qqq\bar{e}_i(\qqq^{-n} z),\, D\bar{f}_i(z)D^{-1}=\qqq^{-1}\bar{f}_i(\qqq^{-n} z),\, D\bar{\psi}_i^\pm(z)D^{-1}=\bar{\psi}_i^\pm(\qqq^{-n} z),\\
g_{i,j}(C^{-1}z/w)\bar{\psi}_i^{+}(z)\bar{\psi}_j^{-}(w)=g_{i,j}(C z/w)\bar{\psi}_j^{-}(w)\bar{\psi}_i^{+}(z),\\
\bar{e}_i(z)\bar{e}_j(w)=g_{i,j}(z/w)\bar{e}_j(w)\bar{e}_i(z),\\
\bar{f}_i(z)\bar{f}_j(w)=g_{i,j}(z/w)^{-1}\bar{f}_j(w)\bar{f}_i(z),\\
(\qqq-\qqq^{-1})[\bar{e}_i(z),\bar{f}_j(w)]=\delta_{i,j}\left(\delta(C w/z)\bar{\psi}_i^+(Cw)-\delta(C z/w)\bar{\psi}_i^-(Cz)\right),\\
\bar{\psi}_i^+(z)\bar{e}_j(w)=g_{i,j}(z/w)\bar{e}_j(w)\bar{\psi}_i^+(z),\\
\bar{\psi}_i^-(z)\bar{e}_j(w)=g_{i,j}(C^{-1}z/w)\bar{e}_j(w)\bar{\psi}_i^-(z),\\
\bar{\psi}_i^+(z)\bar{f}_j(w)=g_{i,j}(C^{-1}z/w)^{-1}\bar{f}_j(w)\bar{\psi}_i^+(z),\\
\bar{\psi}_i^-(z)\bar{f}_j(w)=g_{i,j}(z/w)^{-1}\bar{f}_j(w)\bar{\psi}_i^-(z),\\
\Sym_{z_1,z_2}[\bar{e}_i(z_1),[\bar{e}_i(z_2),\bar{e}_{j}(w)]_\qqq]_{\qqq^{-1}}=0 \hbox{ if }a_{i,j}=-1,\\
[\bar{e}_i(z),\bar{e}_j(w)]=0\hbox{ if }a_{i,j}=0,\\
\Sym_{z_1,z_2}[\bar{f}_i(z_1),[\bar{f}_i(z_2),\bar{f}_{j}(w)]_\qqq]_{\qqq^{-1}}=0 \hbox{ if }a_{i,j}=-1,\\
[\bar{f}_i(z),\bar{f}_j(w)]=0\hbox{ if }a_{i,j}=0
\end{gather*}
Under this presentation, it has a new \textit{topological} coproduct $\bar{\Delta}_{Dr}$ defined similarly to \ref{Hopf} except $i\not=0$.

%
\subsubsection{Miki's automorphism} 
By \cite{VV}, the two algebra homomorphisms $h:U_\qqq(\dot{\mathfrak{sl}}_\ell)\rightarrow\, '\ddot{U}$ and $v:U_\qqq(\dot{\mathfrak{sl}}_\ell)\rightarrow \ddot{U}'$ defined by
\begin{align*}
&h:E_i\mapsto e_{i,0},\, F_i\mapsto f_{i,0},\, K_i\mapsto\psi_{i,0},\, D\mapsto \qqq^{d_2}\\
&v:\bar{e}_{i,k}\mapsto \ddd^{ik}e_{i,k},\, \bar{f}_{i,k}\mapsto \ddd^{ik}f_{i,k},\, \bar{\psi}_{i,k}\mapsto \ddd^{ik}\gamma^{k/2}\psi_{i,k},\, C\mapsto\gamma, D\mapsto \qqq^{-\ell d_1}\qqq^{\sum_{i=1}^{\ell-1}\frac{i(\ell-i)}{2}H_{i,0}}
\end{align*}
are injective (cf. \cite{Miki3} and \cite{Miki}). 
We call their images the \textit{horizontal} and \textit{vertical} $U_\qqq(\dot{\mathfrak{sl}}_\ell)$, respectively. 
Note that $v$ is a Hopf algebra map if we equip $U_\qqq(\dot{\mathfrak{sl}_\ell})$ with $\bar{\Delta}_{Dr}$.
On the other hand, $h$ is only an algebra map no matter which coproduct we choose, but the restriction of the bialgebra pairing $\varphi$ onto $h( U_\qqq(\dot{\mathfrak{sl}_\ell}) )$ gives the pairing $\bar{\varphi}_{\mathfrak{sl}}$. 
Also, $h(U_\qqq(\dot{\mathfrak{sl}}_\ell))$ lies entirely in homogeneous degree zero.

Let $\eta$ be the $\CC(\qqq)$-linear anti-involution of $'\ddot{U}'$ defined by
\begin{equation}
\begin{gathered}
\eta(\ddd)=\ddd^{-1}\\
\eta(e_{i,k})=e_{i,-k},\, \eta(f_{i,k})=f_{i,-k},\, \eta(b_{i,k})=-b_{i,-k},\\
\eta(\psi_{i,0})=\psi_{i,0}^{-1},\,\eta(\gamma^{\frac{1}{2}})=\gamma^{\frac{1}{2}}
\end{gathered}
\label{EtaDef}
\end{equation}
We have the beautiful construction of Miki \cite{Miki,Miki2}:

\begin{thm}\label{MikiAut}
There is an algebra embedding $\varpi:\,'\ddot{U}\rightarrow\ddot{U}'$ such that $\varpi\circ h=v$. 
When restricted to $'\ddot{U}'$, $\varpi$ is an automorphism with the property $\varpi^{-1}=\eta\circ\varpi\circ\eta$.
\end{thm}
\begin{proof}
Our $\varpi$ is denoted $\pi$ in \cite{Miki2}.
In \cite{Miki}, the author defines an automorphism $\psi$ of $'\ddot{U}'$ that satisfies $\psi^{-1}=\eta\circ\psi\circ\eta$.
The map $\pi$ is defined in Proposition 1 of \cite{Miki2} by extending $\psi^{-1}$ to $'\ddot{U}$ by specifying the image of $\qqq^{d_2}$.
\end{proof}

The map $\varpi$ becomes an isomorphism upon adjoining the $\ell$th root $\qqq^{d_2/\ell}$ to $'\ddot{U}$, and thus we call it \textit{Miki's automorphism}.
It exchanges the vertical and horizontal subalgebras and sends elements of positive \textit{principal} degree to elements of positive \textit{homogeneous degree}.
In general, it is quite difficult to explicitly compute the images of elements under both $\varpi$ and $\varpi^{-1}$.
The following is Proposition 2.6 of \cite{Tsym}:

\begin{prop}\label{MikiProp}
For $i\not=0$, we have
\begin{align}
\nonumber&
\begin{gathered}
\varpi(e_{i,0})=e_{i,0},\, \varpi(f_{i,0})=f_{i,0},\,\varpi(\psi_{i,0}^{\pm 1})=\psi_{i,0}^{\pm 1},\, \varpi(\psi_{0,0}^{\pm 1})=\gamma^{\pm 1}\psi_{0,0}^{\pm 1}\\
\varpi(\qqq^{d_2})=v(D)=\qqq^{-\ell d_1}\qqq^{\sum_{i=1}^{\ell-1}\frac{i(\ell-i)}{2}H_{i,0}}\\
\varpi(e_{0,-1})= (-\ddd)^\ell e_{0,1},\, \varpi(f_{0,1})=(-\ddd)^{-\ell} f_{0,-1}
\end{gathered}\\
\label{MikiE0}
\varpi(e_{0,0})&= \ddd\gamma\psi_{0,0}[ \cdots[f_{1,1},f_{2,0}]_\qqq,\cdots, f_{\ell-1,0}]_\qqq\\
\label{MikiF0}
\varpi(f_{0,0})&=\ddd^{-1}[e_{\ell-1,0},\cdots[e_{2,0},e_{1,-1}]_{\qqq^{-1}}\cdots]_{\qqq^{-1}}\psi_{0,0}^{-1}\gamma^{-1}\\
\label{PlusI}
\varpi(b_{i,1})&= (-1)^{i+1}\ddd^{-i}[ [\cdots[ [\cdots [f_{0,0},f_{\ell-1,0}]_\qqq,\cdots, f_{i+1,0}]_\qqq, f_{1,0}]_\qqq, \cdots, f_{i-1,0}]_\qqq , f_{i,0}]_{\qqq^2}\\
\label{MinusI}
\varpi(b_{i,-1})&= (-1)^{i+1}\ddd^i[e_{i,0}, [\cdots,[e_{1,0}, [e_{i+1,0},\cdots,[e_{\ell-1,0},e_{0,0}]_{\qqq^{-1}}\cdots]_{\qqq^{-1}}]_{\qqq^{-1}}\cdots]_{\qqq^{-1}}]_{\qqq^{-2}}\\
\label{Plus0}
\varpi(b_{0,1})&= (-1)^\ell \ddd^{-(\ell-1)}[ [\cdots [f_{1,1},f_{2,0}]_\qqq,\cdots, f_{\ell-1,0}]_\qqq, f_{0,-1}]_{\qqq^2}\\
\label{Minus0}
\varpi(b_{0,-1})&= (-1)^\ell \ddd^{\ell-1}[ e_{0,1},[e_{\ell-1,0},\cdots,[e_{2,0},e_{1,-1}]_{\qqq^{-1}}\cdots]_{\qqq^{-1}}]_{\qqq^{-2}}
\end{align}
\end{prop}

\subsubsection{Extension to $U_\qqq(\dot{\gl_\ell})$}
One can enhance the vertical $U_\qqq(\dot{\mathfrak{sl}}_\ell)$ into a copy of $U_\qqq(\dot{\mathfrak{gl}}_\ell)$ in the following way (cf. \cite{FJMMRep}). 
The elements $\{b_{0,k}^\perp\}_{k\not=0}\subset\ddot{U}^0$ commute with all the generators of $v(U_\qqq(\dot{\mathfrak{sl}}_\ell))$ besides $D$.
Moreover,
\[\langle b_{0,k}^\perp,b_{i,k}, \gamma : i\not=0, k\not=0\rangle=\ddot{U}^0\]
The subalgebra obtained by adjoining $\{b_{0,k}^\perp\}$ to $v(U_\qqq(\dot{\mathfrak{sl}}_\ell))$ is isomorphic to $U_\qqq(\dot{\mathfrak{gl}}_\ell)$. 
We call it the \textit{vertical $U_\qqq(\dot{\mathfrak{gl}}_\ell)$} and denote it by $U_\qqq^v(\dot{\mathfrak{gl}}_\ell)$. 
Likewise, we will call $\ddot{U}^0$ the \textit{vertical Heisenberg subalgebra}. 
We will extend $v$ accordingly, and we can extend $\bar{\Delta}_{Dr}$ so that $v$ still respects the coproduct.
Using $\varpi$, we can adjoin $\{\varpi^{-1}(b_{0,k}^\perp)\}$ to $h(U_\qqq(\dot{\mathfrak{sl}}_\ell))$ to obtain what we call the \textit{horizontal $U_\qqq(\dot{\mathfrak{gl}}_\ell)$} and denote it by $U_\qqq^h(\dot{\mathfrak{gl}}_\ell)$. 
We call the subalgebra $\varpi^{-1}(\ddot{U}^0)$ the \textit{horizontal Heisenberg subalgebra}. 
Similar to the vertical case, we extend $h$ to $U_\qqq^h(\dot{\mathfrak{gl}}_\ell)$. 

This extension may only seem natural from the perspective of the new Drinfeld presentation of $U_\qqq(\dot{\mathfrak{sl}}_\ell)$.
On the other hand, the algebra $U_\qqq(\dot{\mathfrak{gl}}_\ell)$ also has an \textit{RTT presentation} (cf. \cite{DingFrenk}).
Consider the following subalgebras of $U_\qqq^h(\dot{\mathfrak{gl}}_\ell)$:
\begin{enumerate}
\item $\dot{U}^\ge_{\mathfrak{gl}}$: the subalgebra generated by $h(\dot{U}^{\ge}_{\mathfrak{sl}})$ and $\{ \varpi^{-1}(b_{0,k}^\perp) \}_{k>0}$;
\item $\dot{U}^\le_{\mathfrak{gl}}$: the subalgebra generated by $h(\dot{U}^{\le}_{\mathfrak{sl}})$ and $\{ \varpi^{-1}(b_{0,-k}^\perp) \}_{k<0}$;
\end{enumerate}
From the RTT presentation, one can naturally extend the coproduct $\bar{\Delta}$ to a coproduct on $U_\qqq^h(\dot{\mathfrak{gl}}_\ell)$, which we will also denote by $\bar{\Delta}$. The subalgebras $\dot{U}^\ge_{\mathfrak{gl}}$ and $\dot{U}^\le_{\mathfrak{gl}}$ are sub-bialgebras and the bialgebra pairing $\bar{\varphi}_{\mathfrak{sl}}$ can be extended to a nondegenerate bialgebra pairing $\bar{\varphi}_{\mathfrak{gl}}:\dot{U}^\ge_{\mathfrak{gl}}\times\dot{U}^\le_{\mathfrak{gl}}\rightarrow\CC(\qqq)$.
Moreover, $U_\qqq^h(\dot{\mathfrak{gl}}_\ell)$ can be obtained as a quotient of the Drinfeld double of this pairing.
We omit details, which can be found in \cite{NegutTor}.
The key points for us are that (cf. Remark 2.23 of \textit{loc. cit.})
\begin{equation}
\bar{\varphi}_{\mathfrak{gl}}\left(\varpi^{-1}(b_{0,k}^\perp),h(\dot{U}_{\mathfrak{sl}}^{\le})\right)=\bar{\varphi}_{\mathfrak{gl}}\left(h(\dot{U}_{\mathfrak{sl}}^\ge),\varpi^{-1}(b_{0,-k}^\perp)\right)=0
\label{glorth}
\end{equation}
and
\begin{align}
\begin{split}
\bar{\Delta}(\varpi^{-1}(b_{0,k}^\perp))&=\varpi^{-1}(b_{0,k}^\perp)\otimes 1+\kappa^k\otimes \varpi^{-1}(b_{0,k}^\perp)\\
\bar{\Delta}(\varpi^{-1}(b_{0,-k}^\perp))&=\varpi^{-1}(b_{0,-k}^\perp)\otimes \kappa^{-k}+1\otimes \varpi^{-1}(b_{0,-k}^\perp)
\end{split}
\label{PerpCoprod}
\end{align}
for $k>0$.

\begin{prop}
For $k,k'>0$, we have:
\begin{equation}
\bar{\varphi}_{\mathfrak{gl}}\left(\varpi^{-1}(b_{0,k}^\perp),\varpi^{-1}(b_{0,-k'}^\perp)\right)=\varphi(b_{0,-k'}^\perp,b_{0,k}^\perp)
\label{GLHeisPair}
\end{equation}
\end{prop}

\begin{proof}
We can combine (\ref{glorth}), (\ref{PerpCoprod}), and the Drinfeld double relation (\ref{DrinDouble}) to understand the left hand side of (\ref{GLHeisPair}) in terms of a commutator.
Namely, this is zero if $k\not=k'$, and in the case $k=k'$, we have 
\begin{equation}
\left[\varpi^{-1}(b_{0,k}^\perp),\varpi^{-1}(b_{0,-k}^\perp)\right]=(\kappa^k-\kappa^{-k})\bar{\varphi}_{\mathfrak{gl}}\left(\varpi^{-1}(b_{0,k}^\perp),\varpi^{-1}(b_{0,-k}^\perp)\right)
\label{gldoub}
\end{equation}
However, noting $\gamma=\varpi(\kappa)$, the proposition follows from applying $\varpi$ to (\ref{gldoub}) and similarly studying the Drinfeld double relation, this time for $\UTor$ and $\varphi$.
\end{proof}

\subsection{\textit{R}-matrices}\label{Rmatrix} 
Here, we will review the universal $R$-matrices of $U_\qqq(\dot{\mathfrak{sl}}_\ell)$, $U_\qqq(\dot{\mathfrak{gl}}_\ell)$, and $U_{\qqq,\ddd}(\ddot{\mathfrak{sl}}_\ell)$.
All three are the canonical tensors for the bialgebra pairings on their respective algebras.
While the theory and applications of $R$-matrices is a vast subject, we will only have a very specific use for them in \ref{Loperators}, and thus our presentation is very sparse and technical.
Specifically, we will focus on their factorizations.

\subsubsection{$R$-matrix of $U_\qqq(\dot{\mathfrak{sl}}_\ell)$}\label{AffineR} 
The universal $R$-matrix $\bar{\mathcal{R}}_{\mathfrak{sl}}$ sits inside a suitable completion of $\dot{U}^{\ge}_\mathfrak{sl}\otimes\dot{U}^\le_\mathfrak{sl}$. 
In terms of the Drinfeld generators, the $R$-matrix has a nice factorization (cf. \cite{Damiani}):
\[\bar{\mathcal{R}}_\mathfrak{sl}=(1+\bar{\mathcal{R}}^+)\bar{\mathcal{R}}_\mathfrak{sl}^{0}(1+\bar{\mathcal{R}}^-)\qqq^{\bar{t}_\infty}\]
Let us comment on the factors from left to right:
\begin{enumerate}
\item $\bar{\mathcal{R}}^+$ is a sum of tensors whose first tensorands are products of the coefficients of the currents $\{\bar{e}_i(z)\}$ and whose second tensorands are products of the coefficients of the currents $\{\bar{f}_i(z)\}$.
Although we will not use this fact, we note that the weight of the first tensorands will always be a positive affine root.
\item For $\bar{\mathcal{R}}_\mathfrak{sl}^0$, we need the elements $\{\bar{b}_{i,k}\}$ for $k\not=0$ and $i=1,\ldots\ell-1$ defined by
\[\bar{\psi}^\pm_i(z)=\bar{\psi}_{i,0}^{\pm 1}\exp\left(\pm(\qqq-\qqq^{-1})\sum_{k>0}\bar{b}_{i,\pm k}z^{\mp k}\right)\]
For each positive integer $k$, let $M_k$ denote the following `Cartan-like' matrix for the finite-type root lattice $A_{\ell-1}$:
\[
\left[M_k\right]_{ij}:= \frac{[ka_{i,j}]_\qqq}{k(\qqq-\qqq^{-1})} 
\]
where $i,j=1,\ldots, \ell-1$.
This is the pairing matrix for the elements $\{\bar{b}_{i,\pm k}\}$, and it is invertible since $\bar{\varphi}_\mathfrak{sl}$ is nondegenerate.
We then have
\[\bar{\mathcal{R}}_\mathfrak{sl}^0=\exp\left(\sum_{i,j=1}^{\ell-1}\sum_{k>0}\left[\bar{M}_k^{-1}\right]_{ij} C^{-\frac{k}{2}}\bar{b}_{i,k}\otimes C^{\frac{k}{2}} \bar{b}_{j,-k}\right)\]
The powers of $C$ come from the comparison between Drinfeld generators and the \textit{root generators}, cf. Definition 3.9 of \cite{BeckBraid}.
\item Similar to but opposite of $\bar{\mathcal{R}}^+$, $\bar{\mathcal{R}}^-$ is a sum of tensors whose first tensorands are products of the coefficients of the currents $\{\bar{f}_i(z)\}$ and whose second tensorands are products of the coefficients of the currents $\{\bar{e}_i(z)\}$.
Again, we mention but will not use the fact that the weights of the first tensorands will always be a positive affine root.
\item To define $\qqq^{\bar{t}_\infty}$, first define $H_{i,0}$ for $i\not=0$ such that $\bar{\psi}_{i,0}=\qqq^{H_{i,0}}$, $\bar{c}$ so that $C=\qqq^{\bar{c}}$, and $\bar{d}$ so that $D=\qqq^{\bar{d}}$. 
Letting $(\tilde{a}_{i,j})$ be the inverse of the Cartan matrix of type $A_{\ell-1}$, we set
\[\Lambda_i:=\sum_{j=1}^{\ell-1}\tilde{a}_{i,j}H_{j,0}\]
We then define
\[\bar{t}_\infty:=c\otimes d+d\otimes c+\sum_{i=1}^{\ell-1}H_{i,0}\otimes\Lambda_i\]
\end{enumerate}

\subsubsection{\textit{R}-matrix of $U_\qqq(\dot{\mathfrak{gl}}_\ell)$}\label{AffineGLR}
Now, the universal $R$-matrix $\bar{\mathcal{R}}_\mathfrak{gl}$ is an element in a completion of $\dot{U}_\mathfrak{gl}^\ge\otimes\dot{U}_\mathfrak{gl}^\le$.
All we need to do is augment $\bar{\mathcal{R}}_{\mathfrak{sl}}$ with a factor corresponding to the additional Heisenberg generators.
Doing this yields
\[
\bar{\mathcal{R}}_\mathfrak{gl}=(1+\bar{\mathcal{R}}^+)\bar{\mathcal{R}}^0_{\mathfrak{gl}}(1+\bar{\mathcal{R}}^-)\qqq^{\bar{t}_\infty}
\]
Here, $\bar{\mathcal{R}}^+$, $\bar{\mathcal{R}}^-$, and $\qqq^{t_\infty}$ are as in \ref{AffineR}.
The novelty occurs in $\bar{\mathcal{R}}^0_{\mathfrak{gl}}$, where we use (\ref{glorth}) and (\ref{GLHeisPair}):
\begin{equation}
\bar{\mathcal{R}}^0_{\mathfrak{gl}}=\bar{\mathcal{R}}^0_{\mathfrak{sl}}\exp\left( \sum_{k>0}\varphi(b_{0,k}^\perp,b_{0,-k}^\perp)^{-1} \varpi^{-1}(b_{0,k}^\perp)\otimes \varpi^{-1}(b_{0,-k}^\perp)\right)
\label{GLRMatrix}
\end{equation}

%
%


\subsubsection{\textit{R}-matrix of $\UTor$} 
The bialgebra pairing $\varphi$ of $\UTor$ is degenerate if we drop one of the grading elements $\left\{ \qqq^{d_1},\qqq^{d_2} \right\}$.
However, this can be fixed by setting an appropriate central element $\{\gamma,\kappa\}$ equal to $1$.
Specifically, we will drop $\qqq^{d_2}$ and quotient by $\kappa-1$.
Let $[\ddot{U}^\ge]'_{\kappa=1}$ and $[\ddot{U}^\le]'_{\kappa=1}$ denote the bialgebras quotiented by the ideals generated by $(\kappa-1)$ and $\left( \kappa^{-1}-1 \right)$, respectively.
The following is Theorem 2.1(c) of \cite{Tsym}:

\begin{thm}
The pairing $\varphi$ restricts and descends to a nondegenerate pairing $\varphi:[\ddot{U}^\ge]'_{\kappa=1}\times[\ddot{U}^\le]'_{\kappa=1}\rightarrow\FF$.
\end{thm}
\noindent
Thus, this restricted and descended pairing has a canonical tensor $\mathcal{R}$, which sits inside a completion of $[\ddot{U}^\ge]'_{\kappa=1}\otimes[\ddot{U}^\le]'_{\kappa=1}$.
It is the \textit{universal $R$-matrix}.

In \cite{NegutTor}, the author provides a factorization of $\mathcal{R}$, of which we use a coarse corollary. 
Like before, we define $c$ so that $\gamma=\qqq^c$ and then define
\[t_\infty:=c\otimes d_1+d_1\otimes c+\sum_{i=1}^{\ell-1}H_{i,0}\otimes\Lambda_i\]

\begin{thm}\label{Factor1}
$\mathcal{R}$ has a factorization of the form
\[\mathcal{R}=(1+\mathcal{R}^-)h\left[(1+\bar{\mathcal{R}}^+)\bar{\mathcal{R}}_\mathfrak{gl}^0(1+\bar{\mathcal{R}}^-)\right](1+\mathcal{R}^{++})\qqq^{t_{\infty}}\]
where:
\begin{itemize}
\item the first tensorands in $\mathcal{R}^-$ have negative homogeneous degree;
\item the second tensorands in $\mathcal{R}^-$ have positive homogeneous degree;
\item the first tensorands in $\mathcal{R}^{++}$ have positive homogeneous degree;
\item and the second tensorands in $\mathcal{R}^{++}$ have negative homogeneous degree.
\end{itemize}
\end{thm}

\noindent No original ingredients go into the proof of this theorem---we use elements of \cite{NegutTor} and \cite{Tsym}.
However, we will present the proof in \ref{FactorProof} as it requires the language of \textit{shuffle algebras}, the main topic of Section \ref{Shuffle}.
We will not use Theorem \ref{Factor1} until \ref{LOpMiki}.

We conclude with the following computation:


\begin{prop}\label{Factor2}
We have
\begin{align*}
h(\bar{\mathcal{R}}_\mathfrak{gl}^0)&= \exp\left(\sum_{i=0}^{\ell-1}\sum_{k>0}\varpi^{-1}(b_{i,k}^\perp)\otimes \varpi^{-1}(b_{i,-k})\right)\\
&=\exp\left(\sum_{i=0}^{\ell-1}\sum_{k>0}\varpi^{-1}(b_{i,k})\otimes \varpi^{-1}(b_{i,-k}^\perp)\right)
\end{align*}
\end{prop}

\begin{proof}
Observe that
\begin{align}
h(\bar{\mathcal{R}}_\mathfrak{sl}^0)&=\varpi^{-1} v(\bar{\mathcal{R}}_\mathfrak{sl}^0)\nonumber\\
&=\exp\left(\sum_{i,j=1}^{\ell-1}\sum_{k>0}\left[ M_k^{-1} \right]_{ij} \varpi^{-1}v(C^{-\frac{k}{2}}\bar{b}_{i,k})\otimes\varpi^{-1}v(C^{\frac{k}{2}} \bar{b}_{j,-k})\right)\nonumber\\
&=\exp\left(\sum_{i,j=1}^{\ell-1}\sum_{k>0}\left[ M_k^{-1} \right]_{ij} \varpi^{-1}(\ddd^{ik}b_{i,k})\otimes\varpi^{-1}(\ddd^{-jk}b_{j,-k})\right)\label{HeisRMatrixCalc}
\end{align}
If we set $D_k$ to be the $(\ell-1)\times(\ell-1)$ diagonal matrix with $\ddd^{ik}$ in the $i$th diagonal entry, then we can rewrite (\ref{HeisRMatrixCalc}) as
\[
=\exp\left(\sum_{i,j=1}^{\ell-1}\sum_{k>0}\left[ D_kM_k^{-1}D_k^{-1} \right]_{ij} \varpi^{-1}(b_{i,k})\otimes\varpi^{-1}(b_{j,-k})\right)
\]
Now note that $\ddd^{ik-jk}=\ddd^{km_{i,j}}$ when $j=i,i\pm 1$, and thus $D_kM_kD_k^{-1}$ is the pairing matrix for $\varphi$ restricted to $\left\{ b_{i,\pm k} \right\}_{i=1}^{\ell-1}$.
The proposition follows by using (\ref{GLRMatrix}) to consider $h( \bar{\mathcal{R}}_{\mathfrak{gl}}^0 )$, block matrix inversion, and the invariance of the canonical tensor with respect to choices of bases.
\end{proof}

\subsubsection{Quasi $R$-matrix of $\UTor$}\label{QuasiR}
As we have seen above, the degeneracy of $\varphi$ adds some complications.
We can sidestep this by restricting $\varphi$ to $\ddot{U}^+\times\ddot{U}^-$, which will be nondegenerate.
Therefore, the restricted pairing has a canonical tensor $\mathcal{R}_\circ$ living in a suitable completion of $\ddot{U}^+\otimes\ddot{U}^-$.
We call $\mathcal{R}_\circ$ a \textit{quasi $R$-matrix}.
Since $\ddot{U}^\pm$ are not bialgebras, $\mathcal{R}_\circ$ does not interact well with the coproduct (e.g. satisfy quantum Yang-Baxter equations).

Observe that $\ddot{U}^+$ is orthogonal under $\varphi$ to $\kappa^{-1}-1\in\ddot{U}^{\le}$, and thus we can view it as a subalgebra of $\ddot{U}^\ge_{\kappa=1}$.
Similarly, $\ddot{U}^-$ is a subalgebra of $\ddot{U}^\le_{\kappa=1}$.
The proof of Theorem 3.5 yields the following:
\begin{cor}\label{QuasiRFact}
$\mathcal{R}_\circ$ has a factorization of the form
\[\mathcal{R}_\circ=(1+\mathcal{R}^-)h\left[(1+\bar{\mathcal{R}}^+)\bar{\mathcal{R}}_\mathfrak{gl}^0(1+\bar{\mathcal{R}}^-)\right](1+\mathcal{R}^+)\]
where:
\begin{itemize}
\item $\mathcal{R}^-$ is as in Theorem \ref{Factor1};
\item the first tensorands in $\mathcal{R}^+$ have positive homogeneous degree;
\item and the second tensorands in $\mathcal{R}^+$ have negative homogeneous degree.
\end{itemize}
\end{cor}

\begin{rem}
The difference between $\mathcal{R}^+$ and $\mathcal{R}^{++}$ is that the latter incorporates the vertical Heisenberg elements.
\end{rem}

\subsection{Vertex representation}\label{Vertex} 
We review the vertex representation defined by Yoshihisa Saito \cite{Saito} and assign a relation to $\Lambda(\ZZ/\ell\ZZ)\otimes\FF[Q]$.

\subsubsection{Mise en place}\label{Mise} 
The \textit{generalized Heisenberg algebra} $\mathcal{H}_\ell$ is generated by
\[\{b_{i,k}, C\}_{i\in\ZZ/\ell\ZZ}^{k\not=0}\]
satisfying the relations
\begin{gather*}
[b_{i,k},b_{j,k'}]=\ddd^{-km_{i,j}}\dfrac{[k]_\qqq[ka_{i,j}]_\qqq}{k}\delta_{k,-k'} \cdot C\\
C\hbox{ is central }
\end{gather*}
We denote by $\mathcal{H}_\ell^{\pm}$ the subalgebra generated by $\{b_{i,\pm k},C\}_{k>0}$. 
$\mathcal{H}^+_\ell$ has a character $\FF v_0$ where $C$ acts as the identity and all the other generators act trivially. 
The induced representation 
\[F_\ell:=\mathrm{Ind}_{\mathcal{H}_\ell^+}^{\mathcal{H}_\ell}\FF v_0\]
is called the \textit{Fock representation of $\mathcal{H}_\ell$}. 
We identify $\FF\otimes\Lambda(\ZZ/\ell\ZZ)$ and $\mathcal{H}^-_\ell$ via the map
\begin{equation}
p_k(i)\mapsto \frac{k}{[k]_\qqq}b_{i,-k}
\label{PowerToBosons}
\end{equation}
This gives an identification of vector spaces $\FF\otimes\Lambda(\ZZ/\ell\ZZ)\cong F_\ell$.
We will occasionally abuse notation and denote $v_0$ by $1$.

We will also need the \textit{twisted group algebra} of the weight lattice $P$ of $\mathfrak{sl}_\ell$. 
Recall our simple roots $\{\alpha_j\}_{j=1}^{\ell-1}$ of $\mathfrak{sl}_\ell$. 
Let $\{h_j\}_{j=1}^{\ell-1}$ be the corresponding simple coroots and $\{\Lambda_p\}_{j=1}^{\ell-1}$ the simple weights. 
Additionally, we define
\begin{gather*}
\alpha_0:=-\sum_{j=1}^{\ell-1}\alpha_j,\,
h_0:=-\sum_{j=1}^{\ell-1}h_j,\\
\Lambda_0:=0,\, \Lambda_{-p}:=-\Lambda_p
\end{gather*}
We have that $\{\alpha_2,\ldots,\alpha_{\ell-1},\Lambda_{\ell-1}\}$ is a basis of $P$. 
The twisted group algebra $\FF\{P\}$ is the $\FF$-algebra generated by $\{e^{\pm\alpha_j}\}_{j=2}^{\ell-1}\cup\{e^{\pm\Lambda_{\ell-1}}\}$ satisfying the relations
\begin{align*}
e^{\alpha_i}e^{\alpha_j}&=(-1)^{\langle h_i,\alpha_j\rangle}e^{\alpha_j}e^{\alpha_i}\\
e^{\alpha_i}e^{\Lambda_{\ell-1}}&=(-1)^{\delta_{i,\ell-1}}e^{\Lambda_{\ell-1}}e^{\alpha_i}
\end{align*}
For $\alpha\in P$ with $\alpha=\sum_{j=2}^{\ell-1}m_j\alpha_j+m_{\ell}\Lambda_{\ell-1}$, we set
\[e^{\alpha}=e^{m_2\alpha_2}\cdots e^{m_{\ell-1}\alpha_{\ell-1}}e^{m_\ell\Lambda_{\ell-1}}\]
For example,
\begin{align}
\begin{split}
e^{\alpha_1}&=e^{-2\alpha_2}e^{-3\alpha_3}\cdots e^{-(\ell-1)\alpha_{\ell-1}}e^{\ell\Lambda_{\ell-1}}\\
e^{\alpha_0}&=e^{\alpha_2}e^{2\alpha_3}\cdots e^{(\ell-2)\alpha_{\ell-1}}e^{-\ell\Lambda_{\ell-1}}
\end{split}
\label{TwistedAlpha}
\end{align}
Define $\FF\{Q\}$ to be the subalgebra of $\FF\{P\}$ generated by $\{e^{\pm\alpha_i}\}_{i=1}^{\ell-1}$. 
We identify the usual group algebra $\FF[Q]$ with $\FF\{Q\}$ as vector spaces via the `identity' map $e^\alpha\mapsto e^\alpha$.
Thus, we have made an identification
\[\FF\otimes\Lambda(\ZZ/\ell\ZZ)\otimes\FF[Q]\cong F_\ell\otimes\FF\{Q\}\]

\subsubsection{Vertex operators} 
For $-\ell+1\le p\le\ell-1$, our representation is on the space
\[W_p:=F_\ell\otimes\FF\{Q\}e^{\Lambda_p}\]
For $v\otimes e^\alpha e^{\Lambda_p}\in W_p$ where
\begin{align*}
v&=b_{i_1,-k_1}\cdots b_{i_N,-k_N} v_0\\
\alpha&=\sum_{j=1}^{\ell-1}m_j\alpha_j
\end{align*}
we define the operators $b_{i,k}$, $e^\beta$, $\partial_{\alpha_i}$, $z^{H_{i,0}}$, and $d$ by
\begin{gather*}
b_{i,k}(v\otimes e^\alpha e^{\Lambda_p}):=(b_{i,k}v)\otimes e^\alpha e^{\Lambda_p},\, e^\beta(v\otimes e^\alpha e^{\Lambda_p}):=v\otimes(e^\beta e^\alpha e^{\Lambda_p}),\\
\partial_{\alpha_i}(v\otimes e^\alpha e^{\Lambda_p}):=\langle h_i,\alpha+\Lambda_p\rangle v\otimes e^\alpha e^{\Lambda_p},\\
z^{H_{i,0}}(v\otimes e^\alpha e^{\Lambda_p}):=z^{\langle h_i,\alpha+\Lambda_p\rangle} \ddd^{\frac{1}{2}\sum_{j=1}^{\ell-1}\langle h_i,m_j\alpha_j\rangle m_{i,j}}v\otimes e^\alpha e^{\Lambda_p},\\
d(v\otimes e^\alpha e^{\Lambda_p}):=-\left(\dfrac{(\alpha,\alpha)}{2}+(\alpha,\Lambda_p)+\sum_{i=1}^Nk_i\right)v\otimes e^\alpha e^{\Lambda_p}
\end{gather*}

For $0\le p\le \ell-1$, the following was defined by Yoshihisa Saito \cite{Saito}.
However, it can clearly be extended to negative values of $p$.
\begin{thm}\label{Vertexrep}
For any $\vec{\ccc}=(\ccc_0,\ldots, \ccc_{\ell-1})\in(\CC^\times)^\ell$, the following formulas endow $W_p$ with an action of $\ddot{U}'$
\begin{align*}
\rho_{p,\vec{\ccc}}(e_i(z))&=\ccc_i\exp\left(\sum_{k>0}\frac{\qqq^{-\frac{k}{2}}}{[k]_\qqq}b_{i,-k}z^k\right)\\
&\times\exp\left(-\sum_{k>0}\frac{\qqq^{-\frac{k}{2}}}{[k]_\qqq}b_{i,k}z^{-k}\right)e^{\alpha_i} z^{1+H_{i,0}},\\
\rho_{p,\vec{\ccc}}(f_i(z))&=\frac{(-1)^{\ell\delta_{i,0}}}{\ccc_i}\exp\left(-\sum_{k>0}\frac{\qqq^{\frac{k}{2}}}{[k]_\qqq}b_{i,-k}z^k\right)\\
&\times\exp\left(\sum_{k>0}\frac{\qqq^{\frac{k}{2}}}{[k]_\qqq}b_{i,k}z^{-k}\right)e^{-\alpha_i} z^{1-H_{i,0}},\\
\rho_{p,\vec{\ccc}}(\psi_i^\pm(z))&=\exp\left(\pm(\qqq-\qqq^{-1})\sum_{k>0}b_{i,\pm k}z^{\mp k}\right)\qqq^{\pm\partial_{\alpha_i}},\\
&\rho_{p,\vec{\ccc}}(\gamma^{\frac{1}{2}})=\qqq^\frac{1}{2},\,\rho_{p,\vec{\ccc}}(\qqq^{d_1})=\qqq^{-d}
\end{align*}
The representation $W_p$ is irreducible and has central charge $(1,0)$.
\end{thm}



Notice that $\qqq^{d_1}$ gives a nonpositive grading on $W_p$, which we will also call the homogeneous grading. 
We call $1\otimes e^{\Lambda_p}=v_0\otimes e^{\Lambda_p}$ the vacuum and denote it by $\mathbb{1}_p$. 

\begin{cor}\label{Annihilate}
Elements of $\ddot{U}'$ of positive homogeneous degree annihilate $\mathbb{1}_p$.
\end{cor}

Finally, we have a difficult computation by Tsymbaliuk \cite{Tsym}:

\begin{lem}\label{Highwt}
For $0\le p\le \ell-1$, the vacuum $\mathbb{1}_p\in W_p$ is an eigenvector for the currents $\{\varpi(\psi_i^{\pm}(z))\}$. 
The eigenvalues are given by
\begin{align*}
\varpi(\psi_i^{\pm}(z))\mathbb{1}_p&=\left(\frac{ z-\ddd^{-\frac{\ell}{2}}\uuu}{\qqq^{-1}z-\ddd^{-\frac{\ell}{2}}\qqq\uuu}\right)^{\delta_{i,p}}\mathbb{1}_p
\end{align*}
where
\begin{equation}
\uuu=\frac{(-1)^{\frac{(\ell-2)(\ell-3)}{2}}}{\qqq\ccc_0\cdots \ccc_{\ell-1}}
\label{uform}
\end{equation}
\end{lem}

\begin{rem}\label{uRemark}
Our expression for $\ddd^{-\frac{\ell}{2}}\uuu$ differs from the analogous parameter $u$ in \cite{Tsym}.
The different power of $\ddd$ arises because the author does not account for the powers of $\ddd$ coming from $z^{H_{i,0}}$.
Specifically, on page 22 of \textit{loc. cit.}, the power of $\ddd$ in $u$ appears in the calculations (here we use our notation) for $p\not=0$,
\begin{align}
\label{uRemarkE}\varpi(b_{p,-1})\mathbb{1}_p&= (-1)^{p+1}\ddd^{p}e_{p,0}\cdots e_{1,0}e_{p+1,0}\cdots e_{\ell-1,0}e_{0,0}\mathbb{1}_p\\
\nonumber\varpi(b_{p,1})\mathbb{1}_p&= (-1)^{\ell+p+1}\ddd^{-p}f_{0,0}f_{\ell-1,0}\cdots f_{p+1,0}f_{1,0}\cdots f_{p,0}\mathbb{1}_p
\end{align}
and
\begin{align*}
\varpi(b_{0,-1})\mathbb{1}_0&= (-1)^{\ell}\ddd^{\ell-1}e_{0,1}e_{\ell-1,0}\cdots e_{2,0}e_{1,-1}\mathbb{1}_0\\
\varpi(b_{0,1})\mathbb{1}_0&= \ddd^{-\ell+1}f_{1,1}f_{2,0}\cdots f_{\ell-1,0}f_{0,-1}\mathbb{1}_0
\end{align*}
All but first and last operators on the right hand sides of the equations contribute positive or negative half-powers of $\ddd$.
For example, in (\ref{uRemarkE}), because $e_{0,0}$ appends $e^{\alpha_0}$ to the vacuum, the $z^{H_{\ell-1,0}}$ part of $e_{\ell-1,0}$ will multiply by $\ddd^{\frac{1}{2}}$.
This continues up to and including $e_{p+1,0}$, yielding a total power of $\ddd^{\frac{\ell-1-p}{2}}$.
At $e_{1,0}$, the contribution is now $\ddd^{-\frac{1}{2}}$, and this continues up to and including $e_{p-1,0}$, yielding $\ddd^{\frac{-p+1}{2}}$.
Altogether, the operators give us a power of
\[
\ddd^{\frac{\ell-1-p-p+1}{2}}=\ddd^{\frac{\ell}{2}-p}
\]
and multiplication with the $\ddd^p$ that already appears in (\ref{uRemarkE}) leaves behind $\ddd^{\frac{\ell}{2}}$.
The other equations are similar.
\end{rem}

\subsubsection{Pairing\protect\footnote{We thank O. Tsymbaliuk for sharing this construction with us.}}\label{Pairing}
Let $\langle v_0| -|v_0\rangle$ denote the matrix element from and to the vacuum $v_0\in F_\ell$.
We consider the pairing on $F_\ell$ defined by
\[
\left( \prod_{m=1}^{n_1}b_{i_m,-k_m},\prod_{m=1}^{n_2}b_{j_m,-k'_m} \right):=\left\langle v_0\left|\left(\prod_{m=1}^{n_1}(-b_{i_m,k_m})\right)\left(\prod_{m=1}^{n_2}b_{j_m,-k'_m}\right)\right|v_0\right\rangle
\]
We extended it $\FF$-linearly on the \textit{second} argument. In the first argument, we invert $\ddd$ and extend only $\CC(\qqq^{\frac{1}{2}})$-linearly.
This way, as operators on $F_\ell$, we have that $b_{i,k}$ is adjoint to $-b_{i,-k}$.
We extend this to $=F_\ell\otimes\FF\{ Q \}$ by setting
\[
\left( v_1\otimes e^{\alpha},v_2\otimes e^{\beta} \right)=\delta_{\alpha,-\beta}\left( v_1,v_2 \right)
\]
Furthermore, for $-\ell+1\le p\le \ell-1$, we can extend this naturally to a nondegenerate pairing
\[
\left( -,- \right):W_p\times W_{-p}\mapsto\FF
\]
The operator $e^{\alpha}$ is self-adjoint---this is certainly obvious up to sign.
To see that the sign is trivial, suppose that $e^{\alpha}e^{\beta}=(-1)^a e^{\alpha+\beta}$.
We then also have $e^{-\alpha-\beta}=(-1)^ae^{-\alpha}e^{-\beta}$, from which self-adjointness follows.
Finally, if we let $z^{H_{i,0}^*}$ denote the operator $z^{H_{i,0}}$ but with the powers of $\ddd$ inverted, then we have that the adjoint of $z^{H_{i,0}}$ is $z^{-H^*_{i,0}}$ (it is important that we invert $\ddd$ in the first argument).
\begin{prop}\label{Adjunct}
For $x\in {'\ddot{U}'}$ , the adjoint of $\rho_{p,\vec{\ccc}}(x)$ under $\left( -,- \right)$ is $(\rho_{-p,\vec{\ccc}}\circ\eta)(x)$.
\end{prop}

\begin{proof}
It is clear that adjunction under $\left( -,- \right)$ is a $\CC(\qqq^{\frac{1}{2}})$-linear algebra antihomomorphism that inverts $\ddd$.
Thus, it suffices to check the proposition on the generators of $'\ddot{U}'$.
This is trivial for $\gamma^{\pm\frac{1}{2}}$ and $\left\{ \psi_i^\pm(z) \right\}_{i\in\ZZ/\ell\ZZ}$.
For $e_i(z)$, we have
\begin{align*}
\left( \rho_{p,\vec{\ccc}}(e_{i}(z)) a, b \right)&= 
\left(\ccc_i\exp\left(\sum_{k>0}\frac{\qqq^{-\frac{k}{2}}}{[k]_\qqq}b_{i,-k}z^k\right)
\exp\left(-\sum_{k>0}\frac{\qqq^{-\frac{k}{2}}}{[k]_\qqq}b_{i,k}z^{-k}\right)e^{\alpha_i} z^{1+H_{i,0}}a,b\right)\\
&= \left(a, \ccc_iz^{1-H_{i,0}^*}e^{\alpha_i}\exp\left( \sum_{k>0}\frac{\qqq^{-\frac{k}{2}}}{[k]_\qqq}b_{i,-k}z^{-k} \right)\exp\left( -\sum_{k>0}\frac{\qqq^{-\frac{k}{2}}}{[k]_\qqq}b_{i,k}z^k \right)b \right)\\
&= \left(a, \ccc_i\exp\left( \sum_{k>0}\frac{\qqq^{-\frac{k}{2}}}{[k]_\qqq}b_{i,-k}z^{-k} \right)\exp\left( -\sum_{k>0}\frac{\qqq^{-\frac{k}{2}}}{[k]_\qqq}b_{i,k}z^k \right)e^{\alpha_i}z^{-1-H_{i,0}^*}b \right)\\
&= \left(  a,(\rho_{-p,\vec{\ccc}}\circ\eta)(e_{i}(z)) b \right)
\end{align*}
The argument for $f_i(z)$ are similar.
\end{proof}

\subsubsection{Revisiting $\hat{h}_n(i)$ and $\hat{e}_n(i)$}\label{Redux} 
As promised, we give a cleaner presentation of $\hat{h}_n(i)$ and $\hat{e}_n(i)$:
\begin{lem}\label{ReduxLem}
Under the identification
\begin{equation}
q=\qqq\ddd,\, t=\qqq\ddd^{-1}
\label{qtqd}
\end{equation}
we have
\begin{align}
\sum_{n\ge0}\hat{h}_n(i)z^n&=\exp\left((\qqq-\qqq^{-1})^{-1}\sum_{n>0}(-\qqq^{-n}b_{i,-n}^\perp+\ddd^{-n}b_{i+1,-n}^\perp)z^n\right)\label{ModHForm}\\
\sum_{n\ge0}\hat{e}_n(i)(-z)^n&=\exp\left((\qqq-\qqq^{-1})^{-1}\sum_{n>0}(\qqq^{n}b_{i,-n}^\perp-\ddd^{-n}b_{i+1,-n}^\perp)z^n\right)\label{ModEForm}
\end{align}

\end{lem}

\begin{proof}
We will only prove the $q$-statement.
Since
\[\sum_{n\ge0}\Phi_q^{-1}(h_n(i))z^n=\exp\left(\sum_{n>0}\Phi_q^{-1}(p_n(i))\frac{z^n}{n}\right)\]
the result follows from using Lemma \ref{InversePhi} and (\ref{HeisPair}) to compute
\[\varphi\left(b_{j,n},\Phi_q^{-1}(p_n(i))\right)=-\delta_{i,j}\left(\dfrac{\qqq^{-n}}{\qqq-\qqq^{-1}}\right)
+\delta_{i+1,j}\left(\dfrac{\ddd^{-n}}{\qqq-\qqq^{-1}}\right)
\qedhere\]
\end{proof}

\subsection{Fock representation}\label{Fock}
One can view $W_p$ as homogeneous bosonic Fock spaces. 
We will review what should be the fermionic Fock spaces, often simply called the \textit{Fock representations} without ambiguity.

\subsubsection{Definition}\label{FockDef} 
The Fock space $\mathcal{F}$ has a basis $\{|\lambda\rangle\}$ indexed by partitions. 
We denote by $\langle\lambda|$ the dual element to $|\lambda\rangle$. 
To define the representations, we will need some notation for various statistics of $\lambda$ and its nodes.
Let $\square=(a,b)$ be a node in $\lambda$.
We will denote by:
\begin{enumerate}
\item $\chi_\square=q^{a-1}t^{b-1}$ the character of the node;
\item $c_\square$ its content and $\bar{c}_\square$ its content modulo $\ell$ (its \textit{color});
\item $d_i(\lambda)$ the number of $i$-nodes (nodes with color $i$) in $\lambda$;
\item $A_i(\lambda)$ and $R_i(\lambda)$ the addable and removable $i$-nodes of $\lambda$, respectively.
\end{enumerate}
Finally, we will abbreviate $a\equiv b\hbox{ mod }\ell$ simply by $a\equiv b$ and use the Kronecker delta function 
\[\delta_{a\equiv b}:=
\begin{cases}
1 & \hbox{if }a \equiv b\\
0 & \hbox{otherwise}
\end{cases}
\]
\begin{thm}\label{FockRep} 
Let $\upsilon\in\FF^\times$.
If we identify the parameters (cf. Lemma \ref{ReduxLem})
\[
q=\qqq\ddd,\, t=\qqq\ddd^{-1}
\]
then for each $p\in\ZZ/\ell\ZZ$, we can define a $'\ddot{U}$-action $\tau^-_{p,\upsilon}$ on $\mathcal{F}$ where the only nonzero matrix elements of the generators are 
\begin{gather*}
\begin{aligned}
\langle\lambda | \tau_{p,\upsilon}^-(e_i(z))|\lambda+\square\rangle&=\delta_{\bar{c}_\square\equiv i-p}(-\ddd)^{d_{i+1}(\lambda)}\delta\left( \frac{ z}{\chi_\square\upsilon}\right)
\frac{\displaystyle\prod_{\blacksquare\in R_{i-p}(\lambda)}\left( \chi_\square-\qqq^2\chi_\blacksquare \right)}
{\displaystyle\prod_{\substack{\blacksquare\in A_{i-p}(\lambda)\\\blacksquare\not=\square}}\left(\chi_\square-\chi_\blacksquare\right)}\\
\langle\lambda+\square |\tau^-_{p,\upsilon}(f_i(z))|\lambda\rangle&=\delta_{\bar{c}_\square\equiv i-p}(-\ddd)^{-d_{i+1}(\lambda)}\delta\left(\frac{ z}{\chi_\square\upsilon}\right)
\frac{\displaystyle\prod_{\substack{\blacksquare\in A_{i-p}(\lambda)\\\blacksquare\not=\square}}\left( \qqq\chi_\square-\qqq^{-1}\chi_\blacksquare \right)}
{\displaystyle\prod_{\blacksquare\in R_{i-p}(\lambda)}\qqq\left( \chi_\square-\chi_\blacksquare \right)}\\
\langle\lambda|\tau^-_{p,\upsilon}(\psi_i^\pm(z))|\lambda\rangle&=
\prod_{\blacksquare\in A_{i-p}(\lambda)}\frac{\left(\qqq z-\qqq^{-1}\chi_{\blacksquare}\upsilon\right)}{\left( z-\chi_\blacksquare\upsilon\right)}
\prod_{\blacksquare\in R_{i-p}(\lambda)}\frac{\left(\qqq^{-1} z-\qqq\chi_\blacksquare\upsilon\right)}{\left( z-\chi_\blacksquare\upsilon\right)},
\end{aligned}\\
\langle\lambda|\tau_{p,\upsilon}^-(\gamma^{\frac{1}{2}})|\lambda\rangle=1,\,\langle\lambda|\tau_{p,\upsilon}^-(\qqq^{d_2})|\lambda\rangle=\qqq^{-|\lambda|+\frac{p(\ell-p)}{2}}
\end{gather*}
If we instead identify the parameters
\[
q=\qqq^{-1}\ddd,\, t=\qqq^{-1}\ddd^{-1}
\]
then we can define a $'\ddot{U}$-action $\tau^+_{p,\upsilon}$ on $\mathcal{F}$ where the only nonzero matrix elements of the generators are
\begin{gather*}
\begin{aligned}
\langle\lambda +\square| \tau_{p,\upsilon}^+(e_i(z))|\lambda\rangle&=\delta_{\bar{c}_\square\equiv i-p}(-\ddd)^{-d_{i+1}(\lambda)}\delta\left(\frac{z}{\chi_\square \upsilon}\right)
\frac{\displaystyle\prod_{\substack{\blacksquare\in A_{i-p}(\lambda)\\\blacksquare\not=\square}}\left(\chi_\square-\qqq^{2}\chi_\blacksquare\right)}
{\displaystyle\prod_{\blacksquare\in R_{i-p}(\lambda)}\left( \chi_\square-\chi_\blacksquare \right)}\\
\langle\lambda |\tau_{p,\upsilon}^+(f_i(z))|\lambda+\square\rangle&=\delta_{\bar{c}_\square\equiv i-p}(-\ddd)^{d_{i+1}(\lambda)}\delta\left(\frac{z}{\chi_\square \upsilon}\right)
\frac{\displaystyle\prod_{\blacksquare\in R_{i-p}(\lambda)}\left( \qqq\chi_\square-\qqq^{-1}\chi_\blacksquare \right)}
{\displaystyle\prod_{\substack{\blacksquare\in A_{i-p}(\lambda)\\\blacksquare\not=\square}}\qqq\left( \chi_\square-\chi_\blacksquare \right)}\\
\langle\lambda|\tau_{p,\upsilon}^+(\psi_i^\pm(z))|\lambda\rangle&=
\prod_{\blacksquare\in A_{i-p}(\lambda)}\frac{\left(\qqq^{-1}z-\qqq\chi_\blacksquare \upsilon\right)}{\left(z-\chi_\blacksquare \upsilon\right)}
\prod_{\blacksquare\in R_{i-p}(\lambda)}\frac{\left(\qqq z-\qqq^{-1}\chi_\blacksquare \upsilon\right)}{\left(z-\chi_\blacksquare \upsilon\right)},
\end{aligned}\\
\langle\lambda|\tau_{p,\upsilon}^+(\gamma^{\frac{1}{2}})|\lambda\rangle=1,\,\langle\lambda|\tau_{p,\upsilon}^+(\qqq^{d_2})|\lambda\rangle=\qqq^{|\lambda|-\frac{p(\ell-p)}{2}}
\end{gather*}
Both representations are irreducible and have central charge $(0,1)$.
\end{thm} 

\begin{proof}
Most of the relations can be checked on a case-by-case basis.
To illustrate the point about the different identification of parameters, observe, for example, that
\begin{align*}
\tau^-_{p,\upsilon}(f_{i+1}(z))\tau^-_{p,\upsilon}(f_{i}(w))&=
\frac{-\ddd^{-1}(w-qz)}{z-tw}\tau^-_{p,\upsilon}(f_{i}(w))\tau^-_{p,\upsilon}(f_{i+1}(z))\\
\tau^+_{p,\upsilon}(e_{i+1}(z))\tau^+_{p,\upsilon}(e_{i}(w))&=
\frac{-\ddd^{-1}(w-qz)}{z-tw}\tau^+_{p,\upsilon}(e_{i}(w))\tau^+_{p,\upsilon}(e_{i+1}(z))
\end{align*}
We emphasize that the relation between the parameters $(q,t)$ and $(\qqq,\ddd)$ differ between the lowest and highest weight representations.
For the relation between $e_i(z)$ and $f_i(z)$, we use Lemma 3.3 of \cite{FJMMSemi}.
\end{proof}

We call $\tau^-_{p,\upsilon}$ the (color $p$) \textit{highest weight} Fock representation and $\tau^+_{p,\upsilon}$ the \textit{lowest weight} Fock representation.
It is useful to note the weight of the vacuum vector $|\varnothing\rangle$ for the Cartan currents:
\begin{equation}
\begin{aligned}
\langle\varnothing|\tau^-_{p,\upsilon}(\psi_i^\pm(z))|\varnothing\rangle&= 
\left( \frac{\qqq z-\qqq^{-1}\upsilon}{z-\upsilon} \right)^{\delta_{i,p}}\\
\langle\varnothing|\tau^+_{p,\upsilon}(\psi_i^\pm(z))|\varnothing\rangle&= 
\left( \frac{\qqq^{-1} z-\qqq\upsilon}{z-\upsilon} \right)^{\delta_{i,p}}
\end{aligned}
\label{FockWt}
\end{equation}

%

\subsubsection{Tsymbaliuk isomorphisms}  
The following beautiful result was proved by Tsymbaliuk:
\begin{thm}[\cite{Tsym}]\label{BoseFermi}
For $0\le p\le \ell-1$, the map
\[\mathbb{1}_p\mapsto |\varnothing\rangle\]
induces an isomorphism $\mathrm{T}_p: W_p\rightarrow\mathcal{F}$ between the representations $\rho_{p,\vec{\ccc}}\circ\varpi$ and $\tau^-_{p,\upsilon}$ of $'\ddot{U}$, where the parameters $\vec{\ccc}$ and $\upsilon$ are related by 
\[
\upsilon=(-1)^{\frac{(\ell-2)(\ell-3)}{2}}\frac{\qqq\ddd^{-\frac{\ell}{2}}}{\ccc_0\cdots \ccc_{\ell-1}}=\qqq^2\ddd^{-\frac{\ell}{2}}\uuu
\]
where $\uuu$ is given in (\ref{uform}).
\end{thm}

\begin{rem}
In \cite{Tsym}, the author claims to show that $\rho_{p,\vec{\ccc}}\circ\varpi$ is isomorphic to the dual of $\tau_{p,\upsilon}^+$ twisted by the antipode $S$.
However, $S(f_i(z))$ and $S(e_i(z))$ act trivially on $\tau_{p,\upsilon}^\pm$, so this dual is no longer irreducible.
Nonetheless, Tsymbaliuk's proof boils down to the highest weight computation of Lemma \ref{Highwt}, and thus we can correctly interpret his result by presenting any irreducible highest weight representation with the same highest weight.
This is satisfied by $\tau_{p,\upsilon}^-$, once we correctly tune the parameter $\upsilon$.
The precise sense in which the highest weight Fock representation $\tau_{p,\upsilon}^-$ is dual to the lowest weight Fock representation $\tau_{p,\upsilon}^+$ is clarified in Appendix \ref{DualVertexApp}: it is the dual twisted by $\varpi^{-1}\circ S\circ\varpi$.
Let us also mention Theorem 3.15 and Remark 3.14 of \cite{TsymBook}, which corrects \cite{Tsym} and appeared after the first version of our paper.
\end{rem}

\begin{cor}\label{BoseFermi2}
For $0\le p\le \ell-1$, the map
\[
\mathbb{1}_{-p}\mapsto |\varnothing\rangle
\]
induces an isomorphism $\mathrm{T}_{-p}:W_{-p}\rightarrow\mathcal{F}$ between the $'\ddot{U}'$-modules $\rho_{-p,\vec{\ccc}}\circ \varpi^{-1}$ and $\tau^+_{p,\upsilon}$, where the parameters $\vec{\ccc}$ and $\upsilon$ are related by
\[
\upsilon=(-1)^{\frac{(\ell-2)(\ell-3)}{2}}\qqq^{-1}\ddd^{-\frac{\ell}{2}}\ccc_0\cdots\ccc_{\ell-1}=\frac{1}{\qqq^2\ddd^{\frac{\ell}{2}}\uuu}
\]
\end{cor}

\begin{proof}
As in \cite{Tsym}, since both representations are irreducible, we only need to show that $\mathbb{1}_{-p}$ satisfies the following:
\begin{itemize}
\item \textit{Eigenvector property:} $\mathbb{1}_{-p}$ is an eigenvector for the currents $\left\{ \varpi^{-1}(\psi_i^\pm(z)) \right\}$ with eigenvalues 
\[
\varpi^{-1}\left( \psi_i^\pm(z) \right)\mathbb{1}_{-p}=\left( \frac{\qqq^{-1}z- \qqq\upsilon}{z-\upsilon} \right)^{\delta_{i,p}}\mathbb{1}_{-p}
\]
(compare with (\ref{FockWt}));
\item \textit{Lowest weight property:} $\mathbb{1}_{-p}$ is annihilated by $\left\{ \varpi^{-1}(f_i(z)) \right\}$.
\end{itemize}
%
First consider the eigenvector property.
Applying Proposition \ref{Adjunct}, we have
\begin{align}
\left( a,\varpi^{-1}(\psi_i^\pm(z))\mathbb{1}_{-p} \right)=\left( \eta\varpi^{-1}(\psi_i^\pm(z))a,\mathbb{1}_{-p} \right)=\left(\varpi(\psi_i^{\mp}(z^{-1}))a,\mathbb{1}_{-p}  \right)\label{tsymnegpair}
\end{align}
By Theorem \ref{BoseFermi}, $\varpi(\psi_i^\pm(z))$ acts diagonally on $W_p$.
Let $a$ vary over an eigenbasis of $W_p$.
Because $\{\psi_{i,0}\}_{i=1}^\ell$ and $\qqq^{d_1}$ commute with $\{\varpi(\psi_i^\pm(z))\}_{i=1}^\ell$, due to weight and grading reasons, we have that (\ref{tsymnegpair}) is zero unless $a$ is a scalar multiple of $\mathbb{1}_p$.
Lemma \ref{Highwt} then gives us
\[
\left( \mathbb{1}_p,\varpi^{-1}(\psi_i^\pm(z))\mathbb{1}_{-p} \right)
=\left.\left( \frac{1- \ddd^{-\frac{\ell}{2}}z\uuu}{\qqq^{-1}-\qqq\ddd^{-\frac{\ell}{2}} z\uuu} \right)^{\delta_{i,p}}\right|_{\ddd\mapsto\ddd^{-1}}
=\left( \frac{\qqq^{-1}z-\qqq \upsilon}{z-\upsilon} \right)^{\delta_{i,p}}
\]
Since $\mathbb{1}_{-p}$ is the right dual element to $\mathbb{1}_p$ under $\left( -,- \right)$ and the pairing is nondegenerate, the eigenvector property holds.

For the lowest weight property, we apply Proposition \ref{Adjunct} to $\varpi^{-1}(f_i(z))$:
\[
\left( a,\varpi^{-1}(f_i(z))\mathbb{1}_{-p} \right)=\left( \eta\varpi^{-1}(f_i(z))a,\mathbb{1}_{-p} \right)=\left( \varpi(f_i(z^{-1}))a,\mathbb{1}_{-p} \right)
\]
Using Theorem \ref{BoseFermi} and considering the formulas in Theorem \ref{FockRep}, we can see that $\varpi(f_i(z))a$ will never contain $\mathbb{1}_{p}$ as a summand.
Thus, this pairing is zero for all $a\in W_p$.
By the nondegeneracy of $\left( -,- \right)$, we must have $\varpi^{-1}(f_i(z))\mathbb{1}_{-p}=0$.
\end{proof}

The Tsymbaliuk isomorphisms are only vacuum-to-vacuum maps and are thus completely non-explicit.
In the case $p=0$, we have two different maps because we propogate them from the vacuum $\mathbb{1}_0$ using $\rho_{0,\vec{\ccc}}\circ\varpi$ on the one hand and $\rho_{0,\vec{\ccc}}\circ\varpi^{-1}$ on the other.
We will distinguish them by $\mathrm{T}_{-0}$ and $\mathrm{T}_{+0}$, respectively.  
The following main result of this paper provides marginally more detail to $\mathrm{T}_{-0}$:
\begin{thm}\label{Main}
Under the Tsymbaliuk isomorphism $\mathrm{T}_{-0}$ between $\rho_{0,\vec{\ccc}}\circ\varpi$ and $\tau_{0,\upsilon}^-$ (Theorem \ref{BoseFermi}) and the identification of $\FF\otimes\Lambda(\ZZ/\ell\ZZ)\cong F_\ell$ given by (\ref{PowerToBosons}), we have
\[\FF\mathrm{T}_{-0}(H_\lambda\otimes e^{\core(\lambda)})=\FF|\lambda\rangle\]
\end{thm}

\noindent In the representations $\tau_{p,\upsilon}^-$, the vertical Heisenberg subalgebra acts diagonally on $\mathcal{F}$ with eigenbasis $\{|\lambda\rangle\}$. 
Our Main Theorem from the Introduction is then a corollary of Theorem \ref{Main}. 

We provide an easy first step:

\begin{prop}\label{CoreMatch}
For an $\ell$-core $\lambda$, 
\[\FF\mathrm{T}_{-0}(1\otimes e^{\lambda})=\FF|\lambda\rangle\subset\mathcal{F}\] 
\end{prop}

\begin{proof}
First observe that for an \textit{arbitrary} partition $\lambda$,
\begin{equation}
\tau^-_{0,\upsilon}(\psi_{i,0})|\lambda\rangle=\qqq^{\delta_{i,0}+(\alpha_i,\core(\lambda))}|\lambda\rangle
\label{CoreWeight}
\end{equation}
This follows easily from considering the Maya diagram of $\lambda$.
The exponent of $\qqq$ in $\psi_{i,0}|\lambda\rangle$ gains a power of $1$ from an addable $i$-node and a power of $-1$ from a removable $i$-node.
On the other hand, an addable $i$-node corresponds to a \textit{white} bead with index $i$ (mod $\ell$) adjacent to a \textit{black} bead of index $i-1$ (mod $\ell$) while a removable $i$-node corresponds to a \textit{black} bead with index $i$ (mod $\ell$) adjacent to a \textit{white} bead of index $i-1$ (mod $\ell$). 
Thus, if $\core(\lambda)=(c_0,\cdots,c_{\ell-1})$, the exponent of $\qqq$ is $\delta_{i,0}+c_{i-1}-c_i=\delta_{i,0}+(\alpha_i,\core(\lambda))$.
Here, the $\delta_{i,0}$ compensates for the extra shift when comparing beads in the quotient Maya diagrams $m_0(\lambda)$ and $m_{\ell-1}(\lambda)$ that are adjacent in the total diagram $m(\lambda)$.

Since $\varpi(\psi_{i,0})=\gamma^{\delta_{i,0}}\psi_{i,0}$ (Proposition \ref{MikiProp}), this shows that for an $\ell$-core $\lambda$, $|\lambda\rangle$ and $1\otimes e^{\lambda}$ have the same weight for the Cartan elements $\{\psi_{i,0}\}$. 
It remains to observe that the lines spanned by each of these vectors are characterized by having those weights \textit{and} being maximal for the $d_2$ and $\varpi(d_2)=\log_\qqq v(D)$ gradings, respectively.
\end{proof}

In order to prove Theorem \ref{Main}, we need to find the analogue of the subspaces
\begin{gather*}
\Span{\hat{h}_\mu :\mu\ge_\ell\lambda}\\
\Span{\hat{e}_{\mu} :\mu\le_\ell\lambda}
\end{gather*}
from Proposition \ref{AltDef} for $\mathcal{F}$. 
This in turn involves understanding the actions of the horizontal Heisenberg elements
\begin{align}\label{EHAbbrev}
\begin{split}
\tilde{e}_n(i)&:=\varpi^{-1}(\hat{e}_n(i))\\
\tilde{h}_n(i)&:=\varpi^{-1}(\hat{h}_n(i))
\end{split}
\end{align}  
on $\mathcal{F}$ via $\tau_{0,\upsilon}^-$.
As stated before, explicitly computing the images of elements under $\varpi^{-1}$ is extremely difficult. 
Fortunately, the shuffle algebra provides a toolkit for understanding elements of the horizontal Heisenberg subalgebra.

\section{Shuffle algebra}\label{Shuffle}

\subsection{Definition and structures} 
In this subsection, we review the shuffle approach to $U_{\qqq,\ddd}(\ddot{\mathfrak{sl}}_\ell)$ and its structures. 
From now on, we identify the parameters as in (\ref{qtqd}):
\[
q=\qqq\ddd,\, t=\qqq\ddd^{-1}
\]

\subsubsection{The shuffle algebra} 
For $\vec{k}=(k_0,\ldots, k_{\ell-1})\in(\ZZ_{\ge 0})^{\ZZ/\ell\ZZ}$, consider the space of rational functions
\[\FF(x_{i,r})_{i\in\ZZ/\ell\ZZ}^{ 1\le r\le k_i}\]
The group 
\[\Sigma_{\vec{k}}:=\prod_{i\in\ZZ/\ell\ZZ}\Sigma_{k_i}\]
acts on it by having the factor $\Sigma_{k_i}$ permute the variables $\{x_{i,r}\}_{r=1}^{k_i}$.
We call $i$ the \textit{color} of $x_{i,r}$, so $\Sigma_{\vec{k}}$ acts by \textit{color-preserving permutations}.
Consider the set of \textit{color-symmetric} rational functions
\[\mathbb{S}_{\vec{k}}:=\left[\FF(x_{i,r})_{i\in\ZZ/\ell\ZZ}^{ 1\le r\le k_i}\right]^{\Sigma_{\vec{k}}}\]
We now take the $(\ZZ_{\ge 0})^{\ZZ/\ell\ZZ}$-graded vector space
\[\mathbb{S}:=\bigoplus_{\vec{k}\in(\ZZ_{\ge 0})^{\ZZ/\ell\ZZ}}\mathbb{S}_{\vec{k}}\]
Unless we say otherwise, an element of $\mathbb{S}$ with $k_i$ variables of color $i$ for all $i$ is assumed to be in $\mathbb{S}_{\vec{k}}$. 
For a degree vector $\vec{k}$, we will use the following notation:
\begin{align*}
|\vec{k}|&:=\sum_{i\in\ZZ/\ell\ZZ}k_i\\
\vec{k}!&:=\prod_{i\in\ZZ/\ell\ZZ}k_i!
\end{align*}

For $i,j\in\ZZ/\ell\ZZ$, we define the \textit{mixing terms}
\[\omega_{i,j}(z,w):=\left\{\begin{array}{ll}
\left(z-\qqq^{2}w\right)^{-1}\left(z-w\right)^{-1} & \hbox{if }i=j\\
\left(\qqq w-\ddd^{-1}z\right) &\hbox{if }i+1=j\\
\left(z-\qqq\ddd^{-1} w\right) &\hbox{if }i-1=j\\
1 &\hbox{otherwise}
\end{array}\right.\]
We endow $\mathbb{S}$ with the \textit{shuffle product} $\star$: for $F\in\mathbb{S}_{\vec{n}}$, $G\in\mathbb{S}_{\vec{m}}$, $F\star G\in\mathbb{S}_{\vec{n}+\vec{m}}$ is the function
\[\frac{1}{\vec{n}!\vec{m}!}\Sym\left(F\left(\{x_{i,r}\}_{r=1}^{ n_i}\right)G\left(\{x_{j,n_j+s}\}_{s=1}^{m_j}\right)
\prod_{i,j\in\ZZ/\ell\ZZ}\,\prod_{\substack{1\le r\le n_i\\ 1\le s\le m_j}}\omega_{i,j}(x_{i,r},x_{j,n_j+s})\right)\]
Here, $\Sym$ denotes the \textit{color symmetrization}: for $f\in\FF(x_{i,r})_{i\in\ZZ/\ell\ZZ}^{ 1\le r\le k_i}$,
\[\Sym(f):=\sum_{(\sigma_0,\ldots,\sigma_{\ell-1})\in\Sigma_{\vec{k}}}f(\{x_{i,\sigma_i(r)}\})\]

We will consider the subspaces $\Sss_{\vec{k}}\subset \mathbb{S}_{\vec{k}}$ of functions $F$ satisfying the following two conditions:
\begin{enumerate}
\item \textit{Pole conditions:} $F$ is of the form
\begin{equation}
F=\frac{f(\{x_{i,r}\})}{\displaystyle \prod_{i\in\ZZ/\ell\ZZ}\,\prod_{\substack{1\le r, r'\le k_i\\r\not= r'}}(x_{i,r}-\qqq^2x_{i,r'})}
\label{PoleCond}
\end{equation}
for a color-symmetric Laurent polynomial $f$.
\item \textit{Wheel conditions:} $F$ has a well-defined finite limit when
\[\frac{x_{i,r_1}}{x_{i+\epsilon,s}}\rightarrow\qqq\ddd^\epsilon\hbox{ and }\frac{x_{i+\epsilon,s}}{x_{i,r_2}}\rightarrow\qqq\ddd^{-\epsilon}\]
for any choice of $i$, $r_1$, $r_2$, $s$, and $\epsilon$, where $\epsilon\in\{\pm 1\}$. 
This is equivalent to specifying that the Laurent polynomial $f$ in the formula (\ref{PoleCond}) evaluates to zero.
\end{enumerate}
Define $\Sss$ to be the direct sum
\[\Sss:=\bigoplus_{\vec{k}\in(\ZZ_{\ge 0})^{\ZZ/\ell\ZZ}}\Sss_{\vec{k}}\]

\begin{prop}
The product $\star$ is associative and $\Sss$ is closed under $\star$.
\end{prop}

\noindent We call $\Sss$ the \textit{shuffle algebra of type $\hat{A}_{\ell-1}$}.
Let $\Sss^+:=\Sss$ and $\Sss^-:=\Sss^{\mathrm{op}}$.
The relation between $\Sss^\pm$ and the quantum toroidal algebra is given by the following result of Negu\cb{t}:

\begin{thm}[\cite{NegutTor}]\label{NegutShuffIso}
$\Sss$ is generated by $\{x_{i,1}^k\}_{i\in\ZZ/\ell\ZZ}^{k\in\ZZ}$ and there exist algebra isomorphisms $\Psi_+:\Sss^+\rightarrow\ddot{U}^+$ and $\Psi_-:\Sss^-\rightarrow\ddot{U}^-$ such that 
\begin{align*}
\Psi_+\left(x_{i,1}^k\right)&=e_{i,k}\\
\Psi_-\left(x_{i,1}^k\right)&=f_{i,k}
\end{align*}
\end{thm}

\begin{rem}\label{ShuffleConvention}
Let $\Sss_{\text{FT}}$ denote the shuffle algebra used in \cite{FeiTsym} and \cite{Tsym}.
Our shuffle algebra $\Sss$ is isomorphic to $\Sss_{\text{FT}}$ via the map
\begin{equation}
\Sss_{\vec{k}}\ni F\mapsto F\prod_{i\in\ZZ/\ell\ZZ}\frac{\displaystyle\prod_{1\le r<r'\le k_i}\left( x_{i,r}-\qqq^2x_{i,r'} \right)\left( x_{i,r}-\qqq^{-2}x_{i,r'} \right)}{\displaystyle\prod_{\substack{1\le r\le k_i\\1\le s\le k_{i+1}}}\left( x_{i+1,s}-x_{i,r} \right)}\in\left(\Sss_{\text{FT}}\right)_{\vec{k}}
\label{FTRenorm}
\end{equation}
On the other hand, let $\Sss_{\text{N\cb{t}}}$ denote the shuffle algebra used in \cite{NegutTor}.
Instead of $(\qqq,\ddd)$, $\Sss_{\text{N\cb{t}}}$ depends on parameters $(\qqq,\bar{\qqq})$.
Our parameter $\ddd$ is
\[\ddd=\qqq^{-1}\bar{\qqq}^{-\frac{2}{\ell}}\]
and $\Sss$ is isomorphic to $\Sss_{\text{N\cb{t}}}$ by
\begin{equation}
\Sss_{\vec{k}}\ni F\mapsto \left(F\prod_{i\in\ZZ/\ell\ZZ}\frac{\displaystyle\prod_{1\le r<r'\le k_i}\left( x_{i,r}-\qqq^2x_{i,r'} \right)\left( \qqq x_{i,r}-\qqq^{-1}x_{i,r'} \right)}{\displaystyle\prod_{\substack{1\le r\le k_i\\1\le s\le k_{i+1}}}\left(  x_{i+1,s}-\qqq\ddd^{-1}x_{i,r} \right)}\right|_{x_{i,r}\mapsto \bar{\qqq}^{\frac{2i}{\ell}}x_{i,r}}
\in\left(\Sss_{\text{N\cb{t}}}\right)_{\vec{k}}
\label{NegCon}
\end{equation}
(cf. Remark 3.4 of \textit{loc. cit.}).
\end{rem}

\subsubsection{Pairing}
Theorem \ref{NegutShuffIso} allows us to represent elements of $\UTor$ as rational functions.
We present here a result of Negu\cb{t} \cite{NegutTor}, which states that on the shuffle side, the bialgebra pairing $\varphi\left( -,- \right)$ of Theorem \ref{TorPairDef} can be presented as a multiple contour integral.
To do this, let us introduce an auxiliary variable $\ppp$.
For $F\in\Sss^\pm_{\vec{k}}$, present $F$ as in the pole conditions (\ref{PoleCond}).
We define
\[
F_\ppp:=\frac{f(\{x_{i,r}\})}{\displaystyle\prod_{i\in\ZZ/\ell\ZZ}\,\prod_{\substack{1\le r,r'\le k_i\\r\not= r'}}(x_{i,r}-\ppp^2x_{i,r'})}
\] 
Similarly, we define the modified mixing term:
\[
\omega_{i,i}^\ppp\left( z,w \right)=\left( z-\ppp^2w \right)^{-1}\left( z-w \right)^{-1}
\]
Finally, we set 
\[
Dx_{i,r}=\frac{dx_{i,r}}{2\pi \sqrt{-1}}
\]

\begin{defn}\label{PairDef}
Define the pairing $\langle F,G\rangle\in\FF$ for $F\in \Sss^+_{\vec{k}}$ and $G\in \Sss^-_{\vec{k}}$ by first symbolically computing the integral
\begin{align*}
\langle F,G\rangle_\ppp&:=\frac{1}{\vec{k}!}\underset{|x_{i,r}|=1}{\oint\cdots\oint}
\frac{(\qqq-\qqq^{-1})^{-|\vec{k}|}F_\ppp G_\ppp\displaystyle\prod_{i\in\ZZ/\ell\ZZ}\prod_{r=1}^{k_i}Dx_{i,r}}
{\displaystyle\prod_{i\in\ZZ/\ell\ZZ}\,\prod_{\substack{1\le r,r'\le k_i\\r\not=r'}}\omega_{i,i}^\ppp\left(x_{i,r},x_{i,r'}\right)\prod_{\substack{i,j\in\ZZ/\ell\ZZ\\i\not=j}}\,\prod_{\substack{1\le r\le k_i\\1\le s\le k_j}}\omega_{i,j}\left(x_{i,r},x_{j,s}\right)}
\end{align*}
Here, the cycle is oriented counterclockwise on the $x_{i,r}$-plane, and we pretend that $\qqq$, $\ddd$, and $\ppp$ are constants satisfying $|\qqq^{-1}|<|\ddd|<1$ and $|\qqq\ppp|=1$. 
We obtain $\langle F,G\rangle$ by specializing $\ppp\mapsto\qqq$, which is possible due to our restrictions on the poles.
If $F\in\Sss^+_{\vec{k}}$ and $G\in\Sss^-_{\vec{l}}$ and $\vec{k}\not=\vec{l}$, then we set $\langle F,G\rangle=0$.
\end{defn}

\begin{rem}
To make this definition slightly more transparent, let us give a naive definition: for each specialization of $\qqq$ and $\ddd$, we define
\begin{align*}
\langle F,G\rangle
&=\frac{1}{\vec{k}!}\oint_C\frac{(\qqq-\qqq^{-1})^{-|\vec{k}|}FG\displaystyle\prod_{i\in\ZZ/\ell\ZZ}\prod_{r=1}^{k_i}Dx_{i,r}}
{\displaystyle\prod_{i\in\ZZ/\ell\ZZ}\,\prod_{\substack{1\le r,r'\le k_i\\r\not=r'}}\omega_{i,i}\left(x_{i,r},x_{i,r'}\right)\prod_{\substack{i,j\in\ZZ/\ell\ZZ\\i\not=j}}\,\prod_{\substack{1\le r\le k_i\\1\le s\le k_j}}\omega_{i,j}\left(x_{i,r},x_{j,s}\right)}
\end{align*}
where the integration cycle $C$ is now a color-symmetric cycle satisfying
\begin{enumerate}
\item $\qqq^{-1}\ddd^{-1} x_{i-1,s}$, $\qqq^{2} x_{i,r'}$, and $\qqq^{-1}\ddd x_{i+1,s}$ are inside the contour for $x_{i,r}$, where $r'\not=r$;
\item $\qqq\ddd x_{i+1,s}$, $\qqq^{-2} x_{i,r'}$, and $\qqq\ddd^{-1} x_{i-1,s}$ are outside the contour for $x_{i,r}$, where $r'\not=r$;
\end{enumerate}
and oriented so that we take the residues contained within.
In short, we want to take an integral where the cycle encloses half of the nonzero poles that appear while excluding the others. 
We will note why we include and exclude these poles in Remark \ref{PairConvert} below.
The definition we use, although strange at first, gives us a ``concrete'' cycle to work with. 
For the privilege of integrating over our explicit torus, we must pay the price of introducing the auxiliary variable $\ppp$.
Our approach here is taken from \cite{NegutTor}, and we thank the author for explaining this to us.
\end{rem}

\begin{prop}[\cite{NegutTor}]\label{Twine}
The maps $\Psi_\pm$ intertwine $\langle -,-\rangle$ and $\varphi(-,-)$. 
\end{prop}

\begin{rem}\label{PairConvert}
This result is Proposition 3.9 of \cite{NegutTor}, although as noted in Remark \ref{ShuffleConvention}, our conventions differ from those of \textit{loc. cit.} to the degree that it may be worth mentioning why our formulas for $\langle-,-\rangle$ are equivalent.
By Theorem \ref{NegutShuffIso}, $\Sss^\pm$ and $\ddot{U}^\pm$ are generated by $\{ x_{i,1}^k \}_{i\in\ZZ/\ell\ZZ}^{k\in\ZZ}$ and $\{ \Psi_{\pm}(x_{i,1}^k) \}_{i\in\ZZ/\ell\ZZ}^{k\in\ZZ}$, respectively.
Therefore, $\langle-,-\rangle$ and its compatibility with $\varphi(-,-)$ is completely determined by:
\begin{enumerate}
\item $\langle x_{i,1}^{k},x_{j,1}^{k'}\rangle=\varphi(e_{i,k},f_{j,k'})$
\item $\langle -, -\rangle$ is a bialgebra pairing.
\end{enumerate}
Point (1) is clear, although we note that our maps $\Psi_{\pm}$ differ from those of \textit{loc. cit.} by some scalars, which accounts for the discrepancy in the factor $(\qqq-\qqq^{-1})^{-|\vec{k}|}$.
For point (2), we have not discussed how the coproduct $\Delta$ is manifested in $\Sss^\pm$. 
This entails enlarging $\Sss^\pm$ to include the Cartain currents $\{ \psi_i^\pm(z) \}_{i\in\ZZ/\ell\ZZ}$, and the coproduct formula is written in terms of the mixing terms $\omega_{i,j}(z,w)$.
One then extends $\langle-,-\rangle$ so that
\begin{equation}
\langle\psi_i^+(z),\psi_j^-(w)\rangle=\varphi(\psi_i^+(z),\psi_j^-(w))=\left.\frac{\qqq^{a_{i,j}}\ddd^{m_{i,j}}z/w-1}{\ddd^{m_{i,j}}z/w-\qqq^{a_{i,j}}}\right|_{\|z\|\gg\|w\|}
\label{CartanPair}
\end{equation}
The above pairing of Cartan currents appears in the proof given in \textit{loc. cit.}, and it is this Laurent series expansion that dictates the residues wrapped in our cycle.
Namely, in order to shrink $\|w\|\ll\|z\|$, we need the poles of (\ref{CartanPair}) outside the contour when we integrate the variable $w$.
For $z$ of appropriate colors, these poles are $\qqq^{-2}z$, $\qqq\ddd^{-1}z$, and $\qqq\ddd z$.
The proof of the bialgebra pairing property then carries on similarly.
Finally, we note that in \textit{loc. cit.}, the cycle involves an auxiliary variable $\bar{\qqq}$, but the variable shift $x_{i,r}\mapsto\bar{\qqq}^{\frac{2i}{\ell}}x_{i,r}$ in (\ref{NegCon}) turns this into the unit circle.
\end{rem}

\subsubsection{Action on the Fock module}
The following result is an adaptation of a result of Negu\cb{t} to our conventions.

\begin{prop}[cf. \cite{NegutCyc} Proposition IV.8]\label{ShuffleFock}
For $F\in \Sss^+_{\vec{k}}$ and $G\in\Sss^-_{\vec{k}}$, the actions of $(\tau_{p,\upsilon}^-\circ\Psi_{+})(F)$ and $(\tau_{p,\upsilon}^-\circ\Psi_-)(G)$ on $\mathcal{F}$ are such that the only nonzero matrix elements involve pairs of partitions $\mu$ and $\lambda$ where $\mu$ adds $k_i$ $(i-p)$-nodes to $\lambda$ for all $i\in\ZZ/\ell\ZZ$. 
In this case, assign to each node $\square\in\mu\backslash\lambda$ its own variable $x_\square$ of color $\bar{c}_\square+p$. 
We then have
\begin{align*}
\langle\lambda|(\tau_{p,\upsilon}^-\circ\Psi_+)(F)|\mu\rangle &=
\frac{(-\ddd)^{|\mu\backslash\lambda|}\left(F\left( \left\{ x_{i,r} \right\} \right)\bigg|_{x_\square\rightarrow\chi_\square \upsilon}\right)}
{\displaystyle\prod_{\square\in\mu\backslash\lambda}
(\chi_\square-1)^{\delta_{\bar{c}_\square=0}\left( 1-\delta_{\square=(1,1)} \right)}
\displaystyle\prod^*_{\substack{\blacksquare\in\mu\\\blacksquare\not=\square}}\omega_{\bar{c}_\square,\bar{c}_\blacksquare}(\chi_\square,\chi_\blacksquare)}\\
\langle\mu|(\tau_{p,\upsilon}^-\circ\Psi_-)(G)|\lambda\rangle
&=
\left(G(\{x_{i,r}\})\bigg|_{x_{\square}\rightarrow \chi_{\square}\upsilon}\right)
\prod_{\square\in\mu\backslash\lambda}
\frac{\left( \qqq\chi_\square-\qqq^{-1} \right)^{\delta_{\bar{c}_{\square}=0}}}
{\chi_\square(\qqq-\qqq^{-1})}
\displaystyle\prod_{\blacksquare\in\lambda}\omega_{\bar{c}_{\blacksquare},\bar{c}_\square}(\chi_{\blacksquare},\chi_\square)
\end{align*}
Here, $\displaystyle\prod^*$ denotes taking the product only over nonzero factors.
\end{prop}

\noindent The proof of this proposition is identical to that of \textit{loc. cit.} once we establish the following lemma:

\begin{lem}\label{MatrixLem}
We have:
\begin{align}
\label{FockEForm}\langle\lambda|\tau_{p,\upsilon}^-(e_i(z))|\lambda+\square\rangle&=
(-\ddd)\delta_{\bar{c}_\square\equiv i-p}\delta\left(\frac{ z}{\chi_\square\upsilon}\right)
\left((\chi_\square -1)^{\delta_{\bar{c}_\square=0}\left(1-\delta_{\square=(1,1)}\right)}
\displaystyle\prod_{\blacksquare\in\lambda}^*\omega_{\bar{c}_\square,\bar{c}_\blacksquare}\left(\chi_\square,\chi_\blacksquare\right)\right)^{-1}\\
\label{FockFForm} \langle\lambda+\square|\tau_{p,\upsilon}^-(f_i(z))|\lambda\rangle&=\delta_{\bar{c}_\square\equiv i-p}\delta\left(\frac{z}{\chi_\square \upsilon}\right)
\frac{(\qqq\chi_\square-\qqq^{-1})^{\delta_{\bar{c}_\square=0}}}
{\chi_\square(\qqq-\qqq^{-1})}
\prod_{\blacksquare\in\lambda}\omega_{\bar{c}_\blacksquare,\bar{c}_\square}\left(\chi_\blacksquare,\chi_\square\right)
\end{align}
\end{lem}

\begin{proof}[Proof of Lemma \ref{MatrixLem}]
We will only consider the case $p=0$ to prevent clutter.
It will be helpful to recall that we identify the parameters by
\[
q=\qqq\ddd,\, t=\qqq\ddd^{-1}
\]
For each $i$, the proposed formula only has factors for boxes of color $i-1$, $i$, and $i+1$. 
We can account for all of these by considering all values of the content $c$ congruent to $i$ modulo $\ell$.
Namely, for each $c$, we consider the nodes of $\lambda$ with content $c-1$, $c$, and $c+1$.
If $c>0$, then starting from the western border, these nodes come in vertically adjacent triples until we reach the outer rim of $\lambda$.
The mixing terms in (\ref{FockEForm}) with $\blacksquare$ set as these triples will cancel, leaving behind $-\ddd$.
We note that in the case factors are excluded because they are zero, they are excluded in canceling pairs.
However, the extra $-\ddd$ does not appear in this case, so we include it to compensate.
Similar considerations appear in the later cases but are handled likewise---therefore we will not explicitly address them.
For (\ref{FockFForm}), there are no such considerations and these triples leave behind $-\ddd^{-1}$.
We attribute these powers of $-\ddd^{\pm 1}$ to the $(i+1)$-node.

If there are no such vertically adjacent triples at the outer rim, then there are only three possible arrangements left:
\begin{enumerate}
\item \textit{A $\Gamma$ of one $i$-node and two $(i-1)$-nodes:} Note that there is no addable or removable $i$-node in this case.
The mixing terms in both (\ref{FockEForm}) and (\ref{FockFForm}) cancel out, leaving nothing behind.
\item \textit{Vertically adjacent pair of $i$- and $(i-1)$-nodes:} Here, the $i$-node $\blacksquare$ is removable.
In (\ref{FockEForm}), there is a leftover factor of 
\[(\chi_\square-\qqq^2\chi_\blacksquare)\]
For (\ref{FockFForm}), there is a leftover factor of 
\[\frac{1}{\qqq(\chi_\square-\chi_\blacksquare)}\]
\item \textit{A lonely $(i-1)$-node:} Let $\blacksquare$ denote this $(i-1)$-node. 
In this case, the space directly above $\blacksquare$ is an addable $i$-node $\blacksquare^\uparrow$.
For (\ref{FockEForm}), the mixing term contribution is 
\[\frac{1}{(\chi_\square-t\chi_\blacksquare)}=\frac{1}{(\chi_\square-\chi_{\blacksquare^\uparrow})}\]
For (\ref{FockFForm}), the contribution is 
\[(\qqq\chi_\square-\ddd^{-1}\chi_\blacksquare)=(\qqq\chi_\square-\qqq^{-1}\chi_{\blacksquare^\uparrow})\]
\end{enumerate}

If $c<0$, we work similarly, but now our triples are horizontally adjacent, starting from the southern border.
As before, the mixing terms for these horizontal triples will cancel and leave behind $-\ddd$ for (\ref{FockEForm}) and $-\ddd^{-1}$ (\ref{FockFForm}).
In the case where factors in $(\ref{FockEForm})$ are zero, the $-\ddd$ is present but there will be a $-\ddd$ missing for the $(i+1)$-node in case (3) below.
The leftovers at the rim are transposes of the $c>0$ case:
\begin{enumerate}
\item \textit{A \rotatebox[origin=c]{180}{$\Gamma$} of one $i$-node and two $(i+1)$-nodes:} Here, (\ref{FockEForm}) cancels to leave behind $\ddd^2$ and (\ref{FockFForm}) cancels to leave behind $\ddd^{-2}$, which we attribute to the two $(i+1)$-nodes.
\item \textit{Horizontally adjacent pair of $i$- and $(i+1)$-nodes:} The $i$-node $\blacksquare$ is removable.
In (\ref{FockEForm}), the mixing terms leave behind 
\[-\ddd(\chi_\square-\qqq^2\chi_\blacksquare)\]
and in (\ref{FockFForm}), we have 
\[\frac{-\ddd^{-1}}{\qqq(\chi_\square-\chi_\blacksquare)}\]
Again, we attribute the factors of $-\ddd^{\pm1}$ to the $(i+1)$-node.
\item \textit{A lonely $(i+1)$-node:} Let $\blacksquare$ denote this $(i+1)$-node.
The space to the right of $\blacksquare$ is an addable $i$-node $\blacksquare_\rightarrow$.
For (\ref{FockEForm}), the mixing term is 
\[\frac{1}{(\qqq\chi_{\blacksquare}-\ddd^{-1}\chi_\square)}=\frac{-\ddd}{(\chi_\square-\chi_{\blacksquare_\rightarrow})}\]
whereas for (\ref{FockFForm}), the mixing term is 
\[(\chi_\blacksquare-t\chi_\square)=-\ddd^{-1}(\qqq\chi_\square-\qqq^{-1}\chi_{\blacksquare_\rightarrow})\]
As before, we attribute the factors of $-\ddd^{\pm1}$ to the single $(i+1)$-node.
\end{enumerate}

Finally, we consider the case $c=0$.
Here, we group adjacent triples into L-shaped arrangements.
Let $\boxdot$ be the $0$-node at the corner of the L.
The cancellation in (\ref{FockEForm}) will leave behind 
\[-\ddd\frac{(\chi_\square-\chi_\boxdot)}{(\chi_\square-\qqq^2\chi_\boxdot)}\]
whereas that in (\ref{FockFForm}) will leave behind
\[-\ddd^{-1}\frac{\qqq^2(\chi_\square-\chi_\boxdot)}{(\qqq^2\chi_\square-\chi_\boxdot)}\]
Starting from the origin, we account for these triples along the content zero line until we reach the boundary.
The terms above will cancel.
For (\ref{FockEForm}), they will leave behind the numerator when $\boxdot=(1,1)$ and the denominator when $\boxdot$ is the outermost node of content zero to fit in such a triple (when that factor is nonzero).
In (\ref{FockFForm}), things are the other way around: the denominator is left behind when $\boxdot=(1,1)$ and the numerator remains for such an outermost node.
If there are no leftovers, then the next space in the content zero line is an addable $i$-node $\blacksquare$ with $\qqq^2\chi_\boxdot=\chi_\blacksquare$.
Otherwise, the leftovers are:
\begin{enumerate}
\item \textit{Horizontally adjacent pair of $0$- and $-1$-nodes:} There is no addable or removable $0$-node in this case.
For (\ref{FockEForm}), there is complete cancellation, whereas for (\ref{FockFForm}), there is a leftover factor of $\qqq$.
\item \textit{Vertically adjacent pair of $0$- and $1$-nodes:} Again, there is no addable or removable $0$-node in this case.
In (\ref{FockEForm}), there is a leftover factor of $-\ddd$, and in (\ref{FockFForm}), there is a leftover factor of $-t=-\qqq\ddd^{-1}$.
We attribute the $-\ddd^{-1}$ to the $(i+1)$-node.
\item \textit{Lonely $0$-node:} In this case, the $0$-node $\blacksquare$ is removable.
For (\ref{FockEForm}), there is a leftover factor of
\[
(\chi_\square-\qqq^2\chi_\blacksquare)
\]
whereas for (\ref{FockFForm}), there is a leftover factor of
\[
\frac{1}{(\chi_{\square}-\chi_{\blacksquare})}=\frac{\qqq}{\qqq(\chi_\square-\chi_\blacksquare)}
\]
\end{enumerate}
The lemma follows from accounting for all the leftover factors.
Note that the restricted product $\prod^*$ in (\ref{FockEForm}) and the extra factor $\chi_\square(\qqq-\qqq^{-1})$ in the denominator of (\ref{FockFForm}) remove the factor where $\square\in A_i(\lambda)$.
\end{proof}

\subsection{\textit{L}-operators}\label{Loperators} 
Our goal for the rest of Section \ref{Shuffle} is to find the shuffle elements corresponding to $\tilde{h}_n(i)$ and $\tilde{e}_n(i)$ of (\ref{EHAbbrev}).
Combining this with Proposition \ref{ShuffleFock}, we will be able to discern how these elements act via $\tau_{0,\upsilon}^-$ on $\mathcal{F}$. 
Before doing this though, we will need to find the shuffle elements for another generating set of the horizontal Heisenberg subalgebra, and that is the goal of this subsection. 
The shuffle elements found here are just slight alterations of those computed in \cite{FeiTsym}.
Our Corollary \ref{LoperatorCalc} is very similar to Theorem 4.2 of \cite{Tsym}, which also uses elements of \cite{FeiTsym}.
However, our proofs are quite different as we apply the $R$-matrix factorization results reviewed in \ref{Rmatrix}.

\subsubsection{Vacuum correlations}\label{Correlation} 
Using bra-ket notation, we will consider the $\FF$-linear functional on $(\ddot{U}^{\ge})'$ given by vacuum-to-vacuum matrix element of the representation $\rho_{-p,\vec{\ccc}}$ for $0\le p\le \ell-1$:
\begin{equation}
x\mapsto \langle\mathbb{1}_{-p}| \rho_{-p,\vec{\ccc}}(x)|\mathbb{1}_{-p}\rangle
\label{VacDef}
\end{equation}
We can encode its restriction to $(\ddot{U}^+)'\simeq\ddot{U}^+$ via \textit{correlation functions}:
\begin{equation}
\left\langle \mathbb{1}_{-p}\left|\rho_{-p,\vec{\ccc}}\left(\overset{\curvearrowright}{\prod_{i=0}^{\ell-1}}\overset{\curvearrowright}{\prod_{r=1}^{k_i}}e_i(x_{i,r})\right)\right|\mathbb{1}_{-p}\right\rangle
\label{CorrelationDef}
\end{equation}
These are generating functions of values of the functional (\ref{VacDef}) evaluated at certain elements of $\ddot{U}^+$.
The correlation function completely determines the restriction of the functional $\langle\mathbb{1}_{-p}|-|\mathbb{1}_{-p}\rangle$ onto $\ddot{U}^+$.

Because of the $e^{\alpha_i}$ terms in the vertex operators, (\ref{CorrelationDef}) is nonzero only when all the $k_i$ are equal to the same value $n$. 
In this case, we can explicitly compute it:

\begin{lem}\label{CorrCompute}
We have
\begin{align}
\label{VacCorr}
\begin{split}
\left\langle \mathbb{1}_{-p}\left|\rho_{-p,\vec{\ccc}}\left( \overset{\curvearrowright}{\prod_{i=0}^{\ell-1}}\overset{\curvearrowright}{\prod_{r=1}^{n}}e_i(x_{i,r})\right)\right|\mathbb{1}_{-p}\right\rangle
&=(-\uuu)^{-n}\frac{(-1)^{\frac{\ell n (n+1)}{2}}\qqq^{n(\ell-1)}}{\ddd^{\frac{\ell n^2}{2}}}\prod_{r=1}^n\frac{x_{0,r}}{x_{p,r}}\\
&\times\frac{\displaystyle \prod_{i\in\ZZ/\ell\ZZ}\prod_{1\le r< r'\le n}(x_{i,r'}-\qqq^2 x_{i,r})(x_{i,r'}-x_{i, r})\prod_{r=1}^n x_{i,r}}
{\displaystyle\prod_{r,s=1}^n\left( \qqq x_{0,r}-\ddd^{-1}x_{\ell-1,s} \right)\prod_{i=1}^{\ell-1}\prod_{r,s=1}^n\left( x_{i,s}-\qqq\ddd^{-1}x_{i-1,r} \right)}
\end{split}
\end{align}
where the rational function is a Laurent series expanded in the region
\begin{align}\label{VacExp}
\begin{split}
\|x_{i,r}\|&\gg\|x_{j,s}\|\hbox{ for }0\le i<j\le \ell-1\\
\end{split}
\end{align}
and $\uuu$ is given by the formula (\ref{uform}).
\end{lem}

\begin{proof}
Writing out the vertex operators gives us
\begin{equation}
\begin{aligned}
&\left\langle \mathbb{1}_{-p}\left|\rho_{-p,\vec{\ccc}}\left( \overset{\curvearrowright}{\prod_{i=0}^{\ell-1}}\overset{\curvearrowright}{\prod_{r=1}^{n}}e_i(x_{i,r})\right)\right|\mathbb{1}_{-p}\right\rangle\\
&=\left\langle\mathbb{1}_{-p}\left|\overset{\curvearrowright}{\prod_{i=0}^{\ell-1}}\overset{\curvearrowright}{\prod_{r=1}^{n}}c_{i}\exp\left(\sum_{k>0}\frac{\qqq^{-\frac{k}{2}}}{[k]_\qqq}b_{i,-k}x_{i,r}^{k}\right)
\exp\left(-\sum_{k>0}\frac{\qqq^{-\frac{k}{2}}}{[k]_\qqq}b_{i,k}x_{i,r}^{-k}\right)e^{\alpha_{i}}x_{i,r}^{1+H_{i,0}}\right|\mathbb{1}_{-p}\right\rangle
\end{aligned}
\label{VacEta}
\end{equation}
The statement about the region of expansion for the Laurent series follows from (\ref{VacEta}) and Corollary \ref{Annihilate}.

To compute the vaccuum correlation, we introduce the usual notion of normal ordering $:-:$ for vertex operators, wherein all Fourier coefficients are products, from left to right, of elements of $\mathcal{H}_\ell^-$, $\mathcal{H}_\ell^+$, $\mathbb{F}\left\{ Q \right\}$, then finally the operators $x_{i,r}^{H_{i,0}}$.
We then have
\begin{align*}
\rho_{-p,\vec{\ccc}}(e_i(z))\rho_{-p,\vec{\ccc}}(e_{j}(w))&=:\rho_{-p,\vec{\ccc}}(e_i(z))\rho_{-p,\vec{\ccc}}(e_{j}(w)):\\
&\times 
\begin{cases}
\left(z- w)(z-\qqq^{-2}w\right) & i=j\vspace{.1cm}\\ 
\displaystyle\frac{1}{\ddd^{\frac{1}{2}}\left( z-\qqq^{-1}\ddd^{-1}w \right)} & i-1=j\vspace{.1cm}\\
\displaystyle\frac{1}{\ddd^{-\frac{1}{2}}\left( z-\qqq^{-1}\ddd w\right)} & i+1=j\vspace{.1cm}\\
1 & \hbox{otherwise}
\end{cases}
\end{align*}
The rational functions above are considered as Laurent series expanded in the region (\ref{VacExp}).

Next, we consider the $\FF\left\{ Q \right\}$ terms.
From (\ref{TwistedAlpha}), we have
\begin{align*}
e^{\alpha_1}&=  e^{\ell\Lambda_{\ell-1}}e^{-(\ell-1)\alpha_{\ell-1}}\cdots e^{-3\alpha_3}e^{-2\alpha_2}\\
e^{\alpha_0}e^{\alpha_1}&= (-1)^{\frac{(\ell-2)(\ell-3)}{2}}e^{-\alpha_{\ell-1}}\cdots e^{-\alpha_3}e^{-\alpha_2}\\
(e^{\alpha_0}e^{\alpha_1}\cdots e^{\alpha_i})e^{\alpha_i}&= -e^{\alpha_i}(e^{\alpha_0}e^{\alpha_1}\cdots e^{\alpha_i})\hbox{ for }i\not=0,\ell-1
\end{align*}
Altogether, we then have
\[
\overset{\curvearrowright}{\prod_{i=0}^{\ell-1}}\overset{\curvearrowright}{\prod_{r=1}^{n}}e^{\alpha_i}=(-1)^{\frac{n(\ell-2)(\ell-3)+(\ell-2)n(n-1)}{2}}=(-1)^{\frac{n(\ell-2)(\ell-3)+\ell n(n-1)}{2}}
\]

The correlation function then becomes
\begin{align*}
&\left\langle \mathbb{1}_{-p}\left|\rho_{-p,\vec{\ccc}}\left( \overset{\curvearrowright}{\prod_{i=0}^{\ell-1}}\overset{\curvearrowright}{\prod_{r=1}^{n}}e_i(x_{i,r})\right)\right|\mathbb{1}_{-p}\right\rangle\\
&=(-1)^{\frac{n(\ell-2)(\ell-3)+\ell n(n-1)}{2}}(\ccc_0\cdots \ccc_{\ell-1})^n\prod_{r=1}^n\frac{x_{0,r}}{x_{p,r}}\\
&\times\frac{\displaystyle\prod_{i\in\ZZ/\ell\ZZ}\prod_{1\le r<r'\le n}(x_{i,r}-x_{i,r'})(x_{i,r}-\qqq^{-2}x_{i,r'})\prod_{r=1}^n x_{i,r}}
{\displaystyle\prod_{r,s=1}^n\ddd^{\frac{1}{2}}(x_{0,r}-\qqq^{-1}\ddd^{-1} x_{\ell-1,s})\prod_{i=1}^{\ell-1}\prod_{r,s=1}^n\ddd^{-\frac{1}{2}}(x_{i-1,r}-\qqq^{-1}\ddd x_{i,s})}\\
&=(-1)^{\frac{\ell n(n-1)}{2}}\frac{\uuu^{-n}}{\qqq^n}\prod_{r=1}^n\frac{x_{0,r}}{x_{p,r}}\\
&\times\frac{\displaystyle\prod_{i\in\ZZ/\ell\ZZ}\prod_{1\le r<r' \le n}(x_{i,r}-x_{i,r'})(x_{i,r}-\qqq^{-2}x_{i,r'})\prod_{r=1}^n x_{i,r}}
{\displaystyle\prod_{r,s=1}^n\ddd^{\frac{1}{2}}(x_{0,r}-\qqq^{-1}\ddd^{-1} x_{\ell-1,s})\prod_{i=1}^{\ell-1}\prod_{r,s=1}^n\ddd^{-\frac{1}{2}}(x_{i-1,r}-\qqq^{-1}\ddd x_{i,s})}\\
&=(-1)^{\frac{\ell n(n+1)}{2}+n}\frac{\qqq^{-\ell n(n-1)+\ell n^2-n} \uuu^{-n}}{\ddd^{\frac{\ell n^2}{2}}}\prod_{r=1}^n\frac{x_{0,r}}{x_{p,r}}\\
&\times\frac{\displaystyle\prod_{i\in\ZZ/\ell\ZZ}\prod_{1\le r<r'\le n}(x_{i,r}-x_{i,r'})(\qqq^2x_{i,r}-x_{i,r'})\prod_{r=1}^n x_{i,r}}
{\displaystyle\prod_{r,s=1}^n(\qqq x_{0,r}-\ddd^{-1} x_{\ell-1,s})\prod_{i=1}^{\ell-1}\prod_{r,s=1}^n(x_{i,s}-\qqq\ddd^{-1}x_{i-1,r})}\qedhere
\end{align*}
\end{proof}

\subsubsection{Dual element in $\Sss^-$}\label{CorrelationShuff}
Composing with the isomorphism $\Psi_+$, we can consider the functional $\langle\mathbb{1}_{-p}|-|\mathbb{1}_{-p}\rangle$ on $\Sss^+$.
Our goal here is to find the dual element to this functional on $\Sss^+$ with respect to the integral pairing $\langle -,-\rangle$. 
Specifically, we define the series of shuffle elements
\[
\sum_{n\ge 0}F_{p,n}\uuu^{-n}
\]
such that
\[
\sum_{n\ge 0}\langle x,F_{p,n}\rangle \uuu^{-n}=\langle\mathbb{1}_{-p}|(\rho_{-p,\vec{\ccc}}\circ\Psi_+)(x)|\mathbb{1}_{-p}\rangle
\]
for any $x\in \Sss^+$.
Here, $\uuu$ is the parameter from (\ref{uform}) and (\ref{VacCorr}). 
$F_{p,n}$ is well-defined because $\varphi$ is nondegenerate when restricted to $\ddot{U}^+\times\ddot{U}^-$ and the isomorphisms $\Psi_\pm$ intertwine $\varphi(-,-)$ and $\langle-,-\rangle$ (Proposition \ref{Twine}).

To compute $F_{p,n}$, note that the correlation function calculation in \ref{Correlation} records the values of $\langle\mathbb{1}_{-p}|-|\mathbb{1}_{-p}\rangle$ on elements of $\ddot{U}^+$ of the form
\[
e_{0,N_{0,1}}\cdots e_{0,N_{0,k_0}}e_{1,N_{1,1}}\cdots e_{1,N_{1,k_1}}\cdots e_{\ell-1,N_{\ell-1,1}}\cdots e_{\ell-1,N_{\ell-1,k_{\ell-1}}}
\]
Under the isomorphism $\Psi_+$ of Theorem \ref{NegutShuffIso}, the element above corresponds to the shuffle element
\begin{equation}
x_{0,1}^{N_{0,1}}\star\cdots\star x_{0,k_{0}}^{N_{0,k_0}}\star x_{1,1}^{N_{1,1}}\star\cdots\star x_{1,k_1}^{N_{1,k_1}}\star\cdots\star x_{\ell-1,1}^{N_{\ell-1,1}}\star\cdots\star x_{\ell-1,k_{\ell-1}}^{N_{\ell-1,k_{\ell-1}}}
\label{ShuffMono}
\end{equation}
As mentioned in \ref{Correlation}, the functional is only nonzero when all $k_i=n$ for some $n$.
In this case, (\ref{ShuffMono}) becomes
\begin{equation}
\Sym\left\{ \left[\prod_{i\in\ZZ\ell\ZZ}\left(\prod_{r=1}^nx_{i,r}^{N_{i,r}}\right)\left(\prod_{1\le r<s\le n}\omega_{i,i}(x_{i,r},x_{i,s})  \right)\right]\left[\prod_{0\le i< j\le\ell-1}\,\prod_{r,s=1}^n\omega_{i,j}(x_{i,r},x_{j,s}) \right]\right\}
\label{ShuffMonoSym}
\end{equation}
Our strategy is straightforward: since $\langle-,-\rangle$ is nondegenerate, we need to find shuffle elements whose pairing with (\ref{ShuffMonoSym}) equals the appropriate coefficient of the series computed in Lemma \ref{CorrCompute}.

\begin{lem}\label{LoperatorShuff}
We have
\begin{align*}
\sum_{n\ge 0} F_{p,n}\uuu^{-n}=\sum_{n\ge 0}\frac{(-1)^{\frac{\ell n(n+1)}{2}}(\qqq^2-1)^{n\ell} }{(-\uuu)^{n}\qqq^n\ddd^{\frac{\ell n^2}{2}}}\prod_{r=1}^n\frac{x_{0,r}}{x_{p,r}}\prod_{i\in\ZZ/\ell\ZZ}\prod_{r=1}^nx_{i,r}
\end{align*}
\end{lem}

\begin{proof}
To compute the pairing $\langle -,-\rangle$, we will need to insert the auxillary variable $\ppp$.
Our candidate for $F_{p,n}$ becomes
\begin{equation*}
\frac{(-1)^{\frac{\ell n(n+1)}{2}}\qqq^{n(\ell-1) }(\qqq-\qqq^{-1})}{(-\uuu)^{n}\ddd^{\frac{\ell n^2}{2}}}\prod_{r=1}^n\frac{x_{0,r}}{x_{p,r}}
\prod_{i\in\ZZ/\ell\ZZ}\left[\left(\prod_{r=1}^nx_{i,r}\right)\left(\prod_{\substack{1\le r,r'\le n\\r\not=r'}}\frac{\omega_{i,i}^\ppp(x_{i,r},x_{i,r'})}{\omega_{i,i}(x_{i,r},x_{i,r'})}\right)\right]
\label{FpnWithP}
\end{equation*}
On the other hand, (\ref{ShuffMonoSym}) becomes
\[
\Sym\left\{ \left[\prod_{i\in\ZZ\ell\ZZ}\frac{\displaystyle\prod_{r=1}^n x_{i,r}^{N_{i,r}}\prod_{\substack{1\le r,r'\le n \\r\not=r'}}\omega_{i,i}^\ppp(x_{i,r},x_{i,r'})}{\displaystyle\prod_{1\le r<r'\le n}\omega_{i,i}(x_{i,r'},x_{i,r})}  \right]\left[\prod_{0\le i< j\le\ell-1}\prod_{r,s=1}^n\omega_{i,j}(x_{i,r},x_{j,s}) \right]\right\}
\]
The resulting pairing is:
\begin{align}
\nonumber&\frac{(-1)^{\frac{\ell n (n+1)}{2}}\qqq^{n(\ell-1)}}{(-\uuu)^{n}\ddd^{\frac{\ell n^2}{2}}(n!)^\ell}
\underset{|x_{i,r}|=1}{\oint\cdots\oint}\left\{
\prod_{r=1}^n\frac{x_{0,r}}{x_{p,r}}\prod_{i\in\ZZ/\ell\ZZ}\left(\prod_{r=1}^n x_{i,r}^{N_{i,r+1}}\right)\left( \prod_{\substack{1\le r,r'\le n\\r\not=r'}}\frac{\omega_{i,i}^\ppp(x_{i,r},x_{i,r'})}{\omega_{i,i}(x_{i,r},x_{i,r'})} \right)
\right.\\
\nonumber&\times
\Sym\left[\left(\prod_{i\in\ZZ/\ell\ZZ}\prod_{1\le r<r'\le n}\omega_{i,i}(x_{i,r'},x_{i,r})^{-1}  \right)\left( \prod_{0\le i<j\le\ell-1}\prod_{r,s=1}^n\omega_{j,i}(x_{j,s},x_{i,r})^{-1} \right) \right]\displaystyle
\\
\nonumber&\times\left.\prod_{i\in\ZZ/\ell\ZZ}\prod_{r=1}^{n}Dx_{i,r}\right\}
\end{align}
Furthermore, we can use the color symmetry of the integration cycle to remove the symmetrization and $(n!)^\ell$ (the other factors are already color-symmetric):
\begin{align}
\nonumber&=\frac{(-1)^{\frac{\ell n (n+1)}{2}}\qqq^{n(\ell-1)}}{(-\uuu)^{n}\ddd^{\frac{\ell n^2}{2}}}
\underset{|x_{i,r}|=1}{\oint\cdots\oint}\left\{
\prod_{r=1}^n\frac{x_{0,r}}{x_{p,r}}\prod_{i\in\ZZ/\ell\ZZ}\prod_{r=1}^n x_{i,r}^{N_{i,r}+1}
\right.\\
&\times
\left(\prod_{i\in\ZZ/\ell\ZZ}\prod_{1\le r<r'\le n}\omega_{i,i}(x_{i,r'},x_{i,r})^{-1}  \right)\left( \prod_{0\le i<j\le\ell-1}\prod_{r,s=1}^n\omega_{j,i}(x_{j,s},x_{i,r})^{-1} \right) \displaystyle
\label{CorrShuffPair2}\\
\label{CorrShuffPair1}&\times\left.\prod_{i\in\ZZ/\ell\ZZ} \prod_{\substack{1\le r,r'\le n\\r\not=r'}}\frac{\omega_{i,i}^\ppp(x_{i,r},x_{i,r'})}{\omega_{i,i}(x_{i,r},x_{i,r'})}\prod_{r=1}^{n}Dx_{i,r}\right\} 
\end{align}

We evaluate the integral by taking iterated residues starting from $x_{\ell-1,n}$.
The only possible poles within its contour are:
\begin{enumerate}
\item at $x_{\ell-1,n}=0$
\item and at $x_{\ell-1,n}=\ppp^{2}x_{\ell-1,r}$ coming from $\omega_{\ell-1,\ell-1}^\ppp(x_{\ell-1,n},x_{\ell-1,r})$ in (\ref{CorrShuffPair1}).
\end{enumerate}
Evaluating at the second kind of pole will result in zero when $\ppp\mapsto\qqq$ due to the $\omega_{i,i}(x_{i,n},x_{i,r})$ factor in (\ref{CorrShuffPair1}).
Thus, we can shrink the contour so that $\|x_{\ell-1,n}\|$ is arbitrarily small without crossing any poles that would contribute to the pairing.
The other variables $x_{\ell-1,r}$ of color $\ell-1$ work similarly, and in fact $\|x_{\ell-1,r}\|$ and $\|x_{\ell-1,s}\|$ do not have to hold any relationship with each other.

After the variables of color $\ell-1$ have been integrated, we move on to color $x_{\ell-2,n}$.
The poles are as before except that there is now a new kind of pole:
\begin{enumerate}
\item[(3)] $x_{\ell-2,n}=\qqq^{-1}\ddd x_{\ell-1,r}$ coming from $\omega_{\ell-1,\ell-2}(x_{\ell-1,r},x_{\ell-2,s})^{-1}$ in (\ref{CorrShuffPair2}).
\end{enumerate}
However, since we can shrink $\|x_{\ell-1,r}\|$ arbitrarily small, we can also shrink $\|x_{\ell-2,n}\|$ arbitrarily small as long as we maintain the relationship $\|x_{\ell-2,n}\|\gg\|x_{\ell-1,r}\|$.
The other variables of color $\ell-2$ work the same way, and in fact we can even continue downwards in color.
Finally, at color $0$, the only novelty is that there is a fourth kind of pole:
\begin{enumerate}
\item[(4)] $x_{0,r}=\qqq^{-1}\ddd^{-1} x_{\ell-1,s}$ coming from $\omega_{\ell-1,0}(x_{\ell-1,s},x_{0,r})^{-1}$ in (\ref{CorrShuffPair2}).
\end{enumerate}
Overall, we can expand the integrand into a Laurent series where $\|x_{j,s}\|\ll\|x_{i,r}\|$ for $0\le i<j\le \ell-1$, and in our order of taking iterated residues, we only need to take the residue at $0$, which coincides with the constant term.

Computing the integral and taking $\ppp\mapsto\qqq$ yields the constant term of the Laurent series:
\begin{align*}
&\frac{(-1)^{\frac{\ell n (n+1)}{2}}\qqq^{n(\ell-1)}}{(-\uuu)^{n}\ddd^{\frac{\ell n^2}{2}}}\prod_{r=1}^n\frac{x_{0,r}}{x_{p,r}}\prod_{i\in\ZZ/\ell\ZZ}\prod_{r=1}^n x_{i,r}^{N_{i,r}+1}\\
&\left.\times\displaystyle\left(\prod_{i\in\ZZ/\ell\ZZ}\prod_{1\le r<r'\le n}\omega_{i,i}(x_{i,r'},x_{i,r})^{-1}  \right)\left( \prod_{0\le i<j\le\ell-1}\prod_{r,s=1}^n\omega_{j,i}(x_{j,s},x_{i,r})^{-1} \right) \right|_{\substack{\|x_{j,s}\|\gg\|x_{i,r}\|\\0\le i<j\le\ell-1}}\\
&=\frac{(-1)^{\frac{\ell n (n+1)}{2}}\qqq^{n(\ell-1)}}{(-\uuu)^{n}\ddd^{\frac{\ell n^2}{2}}}\left(\prod_{i\in\ZZ/\ell\ZZ}\prod_{r=1}^n x_{i,r}^{N_{i,r}}\right)\prod_{r=1}^n\frac{x_{0,r}}{x_{p,r}}\\
&\left.\times\frac{\displaystyle \prod_{i\in\ZZ/\ell\ZZ}\prod_{1\le r< r'\le n}(x_{i,r'}-\qqq^2 x_{i,r})(x_{i,r'}-x_{i, r})\prod_{r=1}^n x_{i,r}}
{\displaystyle\prod_{r,s=1}^n\left( \qqq x_{0,s}-\ddd^{-1}x_{\ell-1,r} \right)\prod_{i=1}^{\ell-1}\prod_{r,s=1}^n\left( x_{i,r}-\qqq\ddd^{-1}x_{i-1,s} \right)}\right|_{\substack{\|x_{j,s}\|\gg\|x_{i,r}\|\\0\le i<j\le\ell-1}} 
\end{align*}
This is precisely the desired coefficient of the correlation function.
\end{proof}
\subsubsection{Proof of Theorem \ref{Factor1}}\label{FactorProof}
In \cite{NegutTor}, the author defines for each $\mu\in\QQ$ the \textit{slope sublagebras} $\mathcal{B}_\mu^+\subset[\ddot{U}^\ge]'_{\kappa=1}$ and $\mathcal{B}_\mu^-\subset[\ddot{U}^\le]'_{\kappa=1}$.
Strictly speaking, the author works with the quotient $\gamma=1$ instead, but switching to the quotient $\kappa=1$ does not significantly change his formulas for the coproduct.
The slope parameter $\mu$ is extended to $\infty$ by setting $\mathcal{B}_{\infty}^\pm$ to be the subalgebra generated by the Cartan currents $\{\psi_i^\pm(z)\}$.
In the proof of Proposition 3.10 of \textit{loc. cit.}, the author establishes the factorizations:
\begin{align}
\begin{split}
[\ddot{U}^{\ge}]_{\kappa=1}'=\overset{\curvearrowright}{\prod_{\mu\in\QQ\cup\left\{ \infty \right\}}}\mathcal{B}_\mu^+\\
[\ddot{U}^{\le}]_{\kappa=1}'=\overset{\curvearrowright}{\prod_{\mu\in\QQ\cup\left\{ \infty \right\}}}\mathcal{B}_\mu^-
\end{split}\label{TorFactor}
\end{align}
Namely, elements on the left can be written as linear combinations of ordered products of elements from the slope subalgebras, where the slope parameter increases as one goes from left to right in the product.

Moreover, the proof of (\ref{TorFactor}) yields a factorization of $\mathcal{R}$ into an ordered product of $R$-matrices for the subalgebras $\left\{ \mathcal{B}_\mu^\pm \right\}$:
\begin{equation}
\mathcal{R}=\left(\overset{\curvearrowright}{\prod_{\mu\in\QQ\cup\left\{ \infty \right\}}}\mathcal{R}_\mu\right)\qqq^{t_\infty}
\label{NegutRMatrixFact}
\end{equation}
Specifically, $\mathcal{B}_\mu^+$ and $\mathcal{B}_\mu^-$ are mutually dual subspaces under $\varphi$, and $\mathcal{R}_\mu$ is the canonical tensor of the restricted pairing, which sits in a suitable completion of $\mathcal{B}_\mu^+\otimes\mathcal{B}_\mu^-$.
Except for the Cartan elements $\left\{ \psi_{i,0} \right\}$, elements in $\mathcal{B}_\mu^+$ have \textit{negative} homogeneous degree when $\mu<0$ and \textit{positive} homogeneous degree for $\mu>0$.  
Likewise but opposite, non-Cartan elements of $\mathcal{B}_\mu^-$ have \textit{positive} homogeneous degree when $\mu<0$ and \textit{negative} homogeneous degree for $\mu>0$.
The product of $R$-matrices in (\ref{NegutRMatrixFact}) over positive slopes yields $\left( 1+\mathcal{R}^{++} \right)$ and the product over negative slopes yields $\left( 1+\mathcal{R}^- \right)$.
It remains to be shown that 
\[\mathcal{R}_0=h\left[(1+\bar{\mathcal{R}}^+)\bar{\mathcal{R}}^0_\mathfrak{gl}(1+\bar{\mathcal{R}}^-)\right]\]
This would follow from showing that $\mathcal{B}_0^+=h(\dot{U}_\mathfrak{gl}^\ge)'$ and $\mathcal{B}_0^-=h(\dot{U}_\mathfrak{gl}^\le)'$ (recall that the right prime means we drop $h(D)=\qqq^{d_2}$).
We note that Corollary \ref{QuasiRFact} (the quasi $R$-matrix case) is obtained by removing $\mathcal{R}_\infty\qqq^{t_\infty}$.

In \cite{NegutTor}, it is obvious that $h(\dot{U}^{\ge}_\mathfrak{sl})'\subset\mathcal{B}_0^+$ and $h(\dot{U}^{\le}_\mathfrak{sl})'\subset\mathcal{B}_0^-$.
$\mathcal{B}_0^+$ is obtained by adjoining elements to $h(\dot{U}^{\ge}_\mathfrak{sl})'$ whose shuffle presentations are given by certain functions $\left\{ G_{n} \right\}_{n>0}$, and likewise $\mathcal{B}_0^-$ is obtained by adjoining elements to $h(\dot{U}^{\le}_\mathfrak{sl})'$ whose shuffle presentations are given by functions $\left\{ G_{-n} \right\}_{n>0}$ (cf. 3.26 of \textit{loc. cit.}).
After making convention changes, we in fact have
\begin{align*}
G_n&= a_n^+F_{0,n}\\
G_{-n}&= a_n^-F_{0,n}
\end{align*}
for some nonzero constants $\{a_n^\pm\}$.
Our goal is to show that $\left\{ \Psi_{+}(F_{0,n}) \right\}$ and $\{ \varpi^{-1}(b_{0, k}^\perp) \}$ generate the same subalgebra and likewise for $\left\{ \Psi_{-}(F_{0,n}) \right\}$ and $\{ \varpi^{-1}(b_{0, -k}^\perp) \}$.
Doing so would allow us to conclude that $h(\dot{U}_\mathfrak{gl}^\ge)'=\mathcal{B}_0^+$ and $h(\dot{U}_\mathfrak{gl}^\le)'=\mathcal{B}_0^-$.
To that end, we turn our attention to the case $p=0$ of Theorem 4.5 from \cite{Tsym}.

Mirroring the aforementioned theorem, consider the functional on $[\ddot{U}^{\le}]'_{\kappa=1}$ given by taking the vacuum correlation $\langle \mathbb{1}_0|-|\mathbb{1}_0\rangle$.
Restricting this to $\ddot{U}^-$, its dual element with respect to $\varphi$ lies in $\ddot{U}^+$.
As computed in Theorem 4.8 of \cite{FeiTsym}, this dual element has shuffle presentation given by 
\[\sum_{n=0}^\infty c_n^+F_{0,n}\uuu^{n}\]
for some nonzero constants $\{c_n^+\}$.
Similarly, we consider the functional on $[\ddot{U}^{\ge}]'_{\kappa=1}$ given by taking the $S$-twisted matrix element $\langle\mathbb{1}_0| S(-)|\mathbb{1}_0\rangle$, i.e. the matrix element in the dual representation.
When restricted to $\ddot{U}^+$, calculations similar to those in \ref{DualCorr} from Appendix \ref{PosMode} show that its dual element in $\ddot{U}^+$ has shuffle presentation
\[
\sum_{n=0}^\infty c_n^-F_{0,n}\uuu^{-n}
\]
for some nonzero constants $\{c_n^-\}$.

Since $\mathcal{R}$ is the canonical tensor for the pairing $\varphi$ once $\kappa=1$, we can compute the series
\[
\sum_{n=0}^\infty c_n^+\Psi_+(F_{0,n})\uuu^{n}
\]
by applying $\langle\mathbb{1}_0|-|\mathbb{1}_0\rangle$ to the \textit{second} factor of $\mathcal{R}$ and dividing by the contribution from $\qqq^{t_\infty}$.
Likewise, we can compute the series
\[
\sum_{n=0}^\infty c_n^-\Psi_-(F_{0,n})\uuu^{-n}
\]
by applying $\langle\mathbb{1}_0| S(-)|\mathbb{1}_0\rangle$ to the \textit{first} factor of $\mathcal{R}$.
We denote the former by ${L}^{\langle\mathbb{1}_0|}_{|\mathbb{1}_0\rangle}$ and the latter by $ ^{|\mathbb{1}_0\rangle}_{\langle\mathbb{1}_0|}{L}$.
These are examples of \textit{matrix elements of $L$-operators}.

The proof of Theorem 4.5 of \cite{Tsym} allows us to write $c_n^\pm\Psi_{\pm}(F_{0, n})$ in terms of the horizontal Heisenberg subalgebra $\varpi^{-1}(\ddot{U}^0)$, although some modifications are necessary since we use a different coproduct. 
For the reader making comparisons, we present
\begin{align*}
\Delta(e_{i,k})&= 1\otimes e_{i,k} + e_{i,k}\otimes\psi_{i,0}\gamma^k +\sum_{r>0}e_{i,k-r}\otimes \psi_{i,r}\gamma^{k-\frac{r}{2}}\\
\Delta(f_{i,k})&= f_{i,k}\otimes 1 +\psi_{i,0}^{-1}\gamma^k\otimes f_{i,k}+\sum_{r>0}\psi_{i,-r}\gamma^{k+\frac{r}{2}}\otimes f_{i,k+r} 
\end{align*}
The main technique is to use the equations
\begin{align*}
\langle 1\otimes \mathbb{1}_0|\mathcal{R}\Delta(x)|1\otimes \mathbb{1}_0\rangle
&=\langle 1\otimes \mathbb{1}_0|\Delta^{op}(x)\mathcal{R}|1\otimes \mathbb{1}_0\rangle\\
\langle\mathbb{1}_0\otimes 1|(S\otimes 1)\left(\mathcal{R}\Delta(x)\right) |\mathbb{1}_0\otimes 1\rangle
&= \langle\mathbb{1}_0\otimes 1|(S\otimes 1)\left(\Delta^{op}(x)\mathcal{R}\right) |\mathbb{1}_0\otimes 1\rangle
\end{align*}
to understand ${L}^{\langle\mathbb{1}_0|} _{|\mathbb{1}_0\rangle}$ and $^{|\mathbb{1}_0\rangle}_{\langle\mathbb{1}_0|}{L}$.

Theorem 4.5 of \cite{Tsym} has three parts, (a), (b), and (c), and we remark on how the modifications play out for each part:
\begin{enumerate}[(a)]
\item Despite our alterations, we still have that ${L}^{\langle\mathbb{1}_0|} _{|\mathbb{1}_0\rangle}$ and $^{|\mathbb{1}_0\rangle}_{\langle\mathbb{1}_0|}{L}$ commute with $\varpi(e_{i,0})$ and $\varpi(f_{i,0})$ for $i\not=0$.
The rest of the proof either does not use the coproduct or is trivial when $p=0$.
\item This part consists of establishing certain $q$-commutator relations.
With our conventions, we have
\begin{align*}
\left[ {L}^{\langle\mathbb{1}_0|} _{|\mathbb{1}_0\rangle}, \varpi(e_{0,0}) \right]_{\qqq}
&= \qqq^2\ddd^{-\frac{\ell }{2}}\uuu\left[ {L}^{\langle\mathbb{1}_0|}_{|\mathbb{1}_0\rangle},\varpi(e_{0,-1})  \right]_{\qqq^{-1}}
=\frac{(-1)^{\frac{(\ell-2)(\ell-3)}{2}}\ddd^{\frac{\ell}{2}}\psi_{0,0}\gamma}{\mathfrak{c}_1\cdots \mathfrak{c}_{\ell-1}}\left( {L}^{\langle e^{-\alpha_0}|}_{|\mathbb{1}_0\rangle}\right)\\
\left[ {L}^{\langle\mathbb{1}_0|}_{|\mathbb{1}_0\rangle}, \varpi(f_{0,0}) \right]_{\qqq}
&= \ddd^{\frac{\ell }{2}}\uuu^{-1}\left[ {L}^{\langle\mathbb{1}_0|}_{|\mathbb{1}_0\rangle},\varpi(f_{0,1})  \right]_{\qqq^{-1}}
=-\frac{(-1)^{\frac{(\ell-2)(\ell-3)}{2}}\qqq \mathfrak{c}_1\cdots \mathfrak{c}_{\ell-1}\gamma^{-1}}{\ddd^{\frac{\ell}{2}}}\left( {L}^{\langle\mathbb{1}_0|}_{| e^{-\alpha_0}\rangle}\right)\psi_{0,0}^{-1}\\
\left[ ^{|\mathbb{1}_0\rangle}_{\langle\mathbb{1}_0|}{L} , \varpi(e_{0,0}) \right]_{\qqq}
&=\ddd^{-\frac{\ell }{2}} \uuu\left[ ^{|\mathbb{1}_0\rangle}_{\langle\mathbb{1}_0|}{L} ,\varpi(e_{0,-1})  \right]_{\qqq^{-1}}
=-\frac{(-1)^{\frac{(\ell-2)(\ell-3)}{2}}\ddd^{\frac{\ell}{2}}\psi_{0,0}\gamma}{\qqq\mathfrak{c}_1\cdots \mathfrak{c}_{\ell-1}}\left( ^{|\mathbb{1}_0\rangle}_{\langle e^{-\alpha_0}|}{L}\right)\\
\left[ ^{|\mathbb{1}_0\rangle}_{\langle\mathbb{1}_0|}{L} , \varpi(f_{0,0}) \right]_{\qqq}
&= \qqq^2\ddd^{\frac{\ell }{2}}\uuu^{-1}\left[ ^{|\mathbb{1}_0\rangle}_{\langle\mathbb{1}_0|}{L} ,\varpi(f_{0,1})  \right]_{\qqq^{-1}}
=\frac{(-1)^{\frac{(\ell-2)(\ell-3)}{2}}\qqq \mathfrak{c}_1\cdots \mathfrak{c}_{\ell-1}\gamma^{-1}}{\ddd^{\frac{\ell}{2}}}\left(^{|e^{-\alpha_0}\rangle}_{\langle\mathbb{1}_0 |}{L}\right)\psi_{0,0}^{-1}
\end{align*}
Recall that our $\uuu$ is different from the one in \cite{Tsym}, as explained in Remark \ref{uRemark}.
\item There is no alteration here---we just point out that as emphasized in Remark \ref{PerpRem}, the $h_{0,r}^\perp$ of \textit{loc. cit.} has negative homogeneous degree.
\end{enumerate}

Once this is done, we have:
\begin{align}
\label{FMiki}
\Psi_+\left(\sum_{n=0}^\infty c_n^+F_{0,n}\uuu^{n}\right)&= \exp\left( \sum_{k>0}\ddd^{\frac{\ell k}{2}}\frac{(1-\qqq^{2k})}{(\qqq-\qqq^{-1})k}\varpi(b_{0,-k}^\perp)\uuu^{k} \right)\\
\Psi_-\left(\sum_{n=0}^\infty c_n^-F_{0,n}\uuu^{-n}\right)&= \exp\left( \sum_{k>0}\ddd^{-\frac{\ell k}{2}}\frac{(1-\qqq^{-2k})}{(\qqq-\qqq^{-1})k}\varpi(b_{0,k}^\perp)\uuu^{-k} \right)
\nonumber
\end{align}
(note the extra negative in (\ref{HeisPair}) when comparing the pairing with the Heisenberg commutator).
This is close to what we want, since
\[
\eta\varpi(b_{0,k}^\perp)=\varpi^{-1}\eta(b_{0,k}^\perp)=-\varpi^{-1}( b_{0,-k}^\perp)
\]
The last step is to understand how the involution $\eta$ is manifested on the shuffle side.
One can check that if $F\in\Sss_{\vec{k}}$, then (abusing notation): 
\begin{equation*}
\eta(F)=\left.F(x_{i,r}^{-1})\prod_{i\in\ZZ\ell\ZZ}\prod_{r=1}^{k_i}(-\ddd)^{k_{i+1}k_i} x_{i,r}^{k_{i+1}+k_{i-1}-2(k_i-1)}\right|_{\ddd\mapsto\ddd^{-1}}
\end{equation*}
It is easy to see then that $\eta(F_{0,n})=F_{0,n}$.\qed

\subsubsection{Dual element in $\ddot{U}^-$}\label{LOpMiki}
Recall the quasi $R$-matrix $\mathcal{R}_\circ$ defined in \ref{QuasiR}---it is the canonical tensor for the restriction of $\varphi$ to $\ddot{U}^+\times\ddot{U}^-$.
As in \ref{FactorProof}, we can find dual elements to the functionals $\langle\mathbb{1}_{-p}|-|\mathbb{1}_{-p}\rangle$ on $\ddot{U}^+$ for $0\le p \le \ell-1$ by applying the functionals to the \textit{first} tensor factor of $\mathcal{R}_\circ$:
\[\langle\mathbb{1}_{-p}\otimes 1|\mathcal{R}_\circ|\mathbb{1}_{-p}\otimes 1\rangle\]
The case of $\langle\mathbb{1}_p|-|\mathbb{1}_p\rangle$ was considered in \cite{Tsym}.
Our approach is different and uses the $R$-matrix factorizations of \ref{Rmatrix}.
However, we still needed the $p=0$ case of \textit{loc. cit.} to establish Theorem \ref{QuasiRFact}.

\begin{thm}\label{LOpBoson}
The dual element of $\langle\mathbb{1}_{-p}|-|\mathbb{1}_{-p}\rangle$ in $\ddot{U}^-$ is
\[\exp\left(\sum_{k>0}\ddd^{-\frac{\ell k}{2}}\frac{\qqq^{-2k}-1}{(\qqq-\qqq^{-1})k}\varpi^{-1}(b_{p,-k}^\perp)\uuu^{-k}\right)\]
\end{thm}

\begin{proof} 
Applying Corollary \ref{QuasiRFact} gives us
\begin{align}
\nonumber&\langle\mathbb{1}_{-p}\otimes 1|\mathcal{R}_\circ|\mathbb{1}_{-p}\otimes 1\rangle\\
\label{Collapse1}&=\left\langle\mathbb{1}_{-p}\otimes 1\left|(1+\mathcal{R}^-)h\left[(1+\bar{\mathcal{R}}^+)\bar{\mathcal{R}}^0_\mathfrak{gl}(1+\bar{\mathcal{R}}^-)\right](1+\mathcal{R}^+)\right|\mathbb{1}_{-p}\otimes 1\right\rangle 
\end{align}
Recalling the properties of $\mathcal{R}^\pm$ from Theorem \ref{QuasiRFact}, the functional kills the $\mathcal{R}^\pm$ terms by Corollary \ref{Annihilate}. 
We thus have that (\ref{Collapse1}) is equal to
\begin{equation}
=\left\langle\mathbb{1}_{-p}\otimes 1\left|h\left[(1+\bar{\mathcal{R}}^+)\bar{\mathcal{R}}^0_\mathfrak{gl}(1+\bar{\mathcal{R}}^-)\right]\right|\mathbb{1}_{-p}\otimes 1\right\rangle
\label{Collapse2}
\end{equation}
By Theorem \ref{MikiAut} and Corollary \ref{BoseFermi2},
\begin{align*}
\langle\mathbb{1}_{-p}| h(\bar{e}_{i,k})=\langle\mathbb{1}_{-p}|\varpi^{-1} v(\bar{e}_{i,k})=\langle\varnothing|(\tau_{p,\upsilon}^+\circ v)(\bar{e}_{i,-k})=0
\end{align*}
where $\langle\varnothing|\in\mathcal{F}^*$ is the dual element to the vacuum. 
Similar moves show that
\[ h(\bar{f}_{i,k})|\mathbb{1}_{-p}\rangle=0\]
These together imply that $h(1+\bar{\mathcal{R}}^\pm)$ collapse as well. 
We proceed onwards after applying Proposition \ref{Factor2} to (\ref{Collapse2}):
\begin{align*}
&=\left\langle\mathbb{1}_{-p}\otimes 1\left|h(\bar{\mathcal{R}}^0_\mathfrak{gl})\right|\mathbb{1}_{-p}\otimes 1\right\rangle\\
&=\left\langle\mathbb{1}_{-p}\otimes 1\left|\exp\left(\sum_{i=0}^{\ell-1}\sum_{k>0}\varpi^{-1}(b_{i,k})\otimes \varpi^{-1}(b_{i,-k}^\perp)\right)\right|\mathbb{1}_{-p}\otimes 1\right\rangle\\
&=\exp\left(\sum_{k>0}\ddd^{-\frac{\ell k}{2}}\frac{\qqq^{-2k}-1}{(\qqq-\qqq^{-1})k}\varpi^{-1}(b_{p,-k}^\perp)\uuu^{-k}\right)
\end{align*}
where in the last equality, we used Corollary \ref{BoseFermi2}. 
\end{proof}

\begin{cor}\label{LoperatorCalc}
We have
\begin{align*}
\Psi_-\left(\sum_{n\ge 0} F_{p,n}\uuu^{-n}\right)&=\exp\left(\sum_{k>0}\ddd^{-\frac{\ell k}{2}}\frac{\qqq^{-2k}-1}{(\qqq-\qqq^{-1})k}\varpi^{-1}(b_{p,-k}^\perp)\uuu^{-k}\right)
\end{align*}
\end{cor}


\subsection{Gordon filtrations}\label{Gordon} 
Let $\ddot{U}^0_-\subset\ddot{U}^0$ denote the subalgebra generated by negative modes. 
In this subsection, we study $\Psi_-^{-1}\varpi^{-1}(\ddot{U}^0_-)$ using \textit{Gordon filtrations}. 
Our constructions here are similar to those in Section 3 of \cite{FeiTsym}.

\subsubsection{Limit conditions}\label{LimitCond} 
We will be concerned with degree vectors specified by certain integer intervals: for $a$ and $b$ with $0\le a\le\ell-1$ and $a\le b$, we define the degree vector $(a;b]$ by
\[(a;b]_i=\left|\{c\in\ZZ: a< c\le b\hbox{ and }c\equiv i\}\right|\]
For any integer $k$, we let $k\delta$ denote the diagonal (or imaginary) degree vector
\[(k,\ldots, k)\]
Notice that any nondiagonal $(a;b]$ partitions $\ZZ/\ell\ZZ$ into two nonempty contiguous subsets $(a;b]^+$ and $(a;b]^-$, where there is some integer $k$ such that for all $i\in(a;b]^+$, $(a;b]_i=k+1$ and for all $i\in(a;b]^-$, $(a;b]_i=k$.

For $\vec{0}\le\vec{k}\le\vec{n}$, $\xi\in\CC^\times$, and $F\in \mathbb{S}_{\vec{n}}$, define
\[F_\xi^{\vec{k}}:=F(\{\xi x_{i,1},\ldots, \xi x_{i,k_i}, x_{i,k_i+1},\ldots, x_{i, n_i}\})\]
On $(\ZZ_{\ge 0})^{\ZZ/\ell\ZZ}$, we let $\vec{k}\le\vec{n}$ to mean $k_i\le n_i$ for all $i$. 
Finally, for each dimension vector $\vec{k}$, set
\begin{align*}
r_0(\vec{k})&:=\sum_{i\in\ZZ/\ell\ZZ}k_i(k_i-1)-k_ik_{i+1}=\sum_{i\in\ZZ/\ell\ZZ}\frac{1}{2}\left( k_i-k_{i+1} \right)^2-k_i\\
r_\infty(\vec{k})&:=\sum_{i\in\ZZ/\ell\ZZ}k_ik_{i+1}-k_i(k_i+1)=\sum_{i\in\ZZ\ell\ZZ}-\frac{1}{2}\left( k_i-k_{i+1} \right)^2-k_i
\end{align*}

\begin{defn}
Let $\Sss(0)_{n\delta}\subset \Sss^{-}_{n\delta}$ be the subspace of functions $F$ satisfying
\begin{enumerate}
\item for all $\vec{k}\le n\delta$
\begin{align*}
\lim_{\xi\rightarrow 0}\xi^{r_0(\vec{k})}F_\xi^{\vec{k}}&<\infty\\
\lim_{\xi\rightarrow\infty}\xi^{r_\infty(\vec{k})}F_\xi^{\vec{k}}&<\infty
\end{align*}
\item for all diagonal $k\delta\le n\delta$,
\[\lim_{\xi\rightarrow 0}\xi^{-k\ell}F_\xi^{k\delta}=\lim_{\xi\rightarrow\infty}\xi^{-k\ell}F_\xi^{k\delta}\]
\item and for all nondiagonal $(a;b]\le n\delta$ with $0\in(a;b]^+$,
\[
\lim_{\xi\rightarrow 0}\xi^{r_0( (a;b])}F_\xi^{(a;b]}=0
\]
\end{enumerate}
We set $\Sss(0):=\bigoplus_{n\delta} \Sss(0)_{n\delta}$.
\end{defn}

\begin{rem}
The extra factors of $\xi$ come from translating the conventions of \cite{FeiTsym} to ours.
Namely, consider the extra normalization factor (\ref{FTRenorm}) in Remark \ref{ShuffleConvention}:
\[
R_{\vec{n}}:=\prod_{i\in\ZZ/\ell\ZZ}\frac{\prod_{r\not=r'}\left( x_{i,r}-\qqq^2x_{i,r'} \right)}{\prod_{\substack{r\le n_i\\s\le n_{i+1}}}\left( x_{i,r}-x_{i+1,s} \right)}\in\mathbb{S}_{\vec{n}}
\]
A total power of $\xi^{r_0(\vec{k})}$ can be extracted from $\left( R_{n\delta} \right)^{\vec{k}}_\xi$, and thus we append a factor of $\xi^{r_0(\vec{k})}$ when taking the limit $\xi\rightarrow 0$.
When taking $\xi\rightarrow\infty$, the growth rate takes account of $\xi^{r_0(\vec{k})}$ as well as linear factors where only one variable is scaled by $\xi$.
The latter contributes a power of
\begin{equation}
\nonumber\sum_{i\in\ZZ/\ell\ZZ}2k_i(n-k_i)-k_i(n-k_{i+1})-k_i(n-k_{i-1})=\sum_{i\in\ZZ/\ell\ZZ}2k_ik_{i+1}-2k_i^2
\end{equation}
and thus overall, the growth rate as $\xi\rightarrow\infty$ is $\xi^{r_\infty(\vec{k})}$.
\end{rem}

\begin{prop}[\cite{FeiTsym} Lemma 3.2]\label{GenWithin}
The functions $\{F_{p,n}\}$ lie in $\Sss(0)$.
\end{prop}

\begin{proof}
For condition (1), note that the total power of $\xi$ we can extract from $\xi^{r_0(\vec{k})}\left( F_{p,n} \right)^{\vec{k}}_\xi$ has exponent
\begin{equation}
k_0-k_p+\frac{1}{2}\sum_{i\in\ZZ/\ell\ZZ}\left( k_i-k_{i+1} \right)^2
\label{ZeroPowerFpn}
\end{equation}
This quantity is always nonnegative since
\begin{equation}
k_0-k_p=\frac{1}{2}\sum_{i=0}^{p-1}\left( k_i-k_{i+1} \right)+\frac{1}{2}\sum_{i=p}^{\ell-1}\left( k_{i+1}-k_i \right)
\label{K0PBreak}
\end{equation}
and $\left| k_i-k_{i+1} \right|$ is bounded above by $\left( k_i-k_{i+1} \right)^2$.
Similarly, the growth rate of $\xi^{r_\infty(\vec{k})}\left( F_{p,n} \right)^{\vec{k}}_\xi$ is $\xi$ to the power of
\[
k_0-k_p-\frac{1}{2}\sum_{i\in\ZZ/\ell\ZZ}\left( k_i-k_{i+1} \right)^2
\]
which is always nonpositive.
Condition (2) is obvious, and for condition (3), note that (\ref{ZeroPowerFpn}) is always positive when $0\in (a,b]_+$.
\end{proof}

\begin{lem}[\cite{FeiTsym} Lemma 3.5]\label{GenWithin2}
$(\Psi_-^{-1}\circ\varpi^{-1})\left( \ddot{U}_-^0 \right)$ is contained in $\Sss(0)$.
\end{lem}

\begin{proof} 
Since $\varpi^{-1}(\ddot{U}_-^0)$ is generated by $\{\Psi_-(F_{p,n})\}$, we just need to check conditions (1)-(3) on the shuffle products $F_{p_1,n_1}\star\cdots\star F_{p_m,n_m}$.
Let 
\begin{align*}
n&:=n_1+\cdots n_m\\
n_{\Sigma \alpha}&:=n_1+\cdots+n_{\alpha-1}+n_\alpha\\
n_{\Sigma 0}&:=0
\end{align*}
Note that the product has variable vector $n\delta$.
Up to scalar, this shuffle product is the symmetrization of
\begin{align*}
&\left(\prod_{i\in\ZZ/\ell\ZZ}\,\prod_{r=1}^{n}x_{i,r}\right)
\left(\prod_{\alpha=1}^m\,\prod_{r=n_{\Sigma (\alpha-1)}+1}^{n_{\Sigma \alpha}}\frac{x_{0,r}}{x_{p_\alpha,r}}\right)\\
&\times\left(\prod_{\substack{1\le \alpha<\beta\le m\\i,j\in\ZZ/\ell\ZZ}}\,\prod_{\substack{n_{\Sigma (\alpha-1)}<r\le n_{\Sigma \alpha}\\ n_{\Sigma (\beta-1)}<s\le n_{\Sigma \beta}}}\omega_{j,i}(x_{j,s},x_{i,r})\right)
\end{align*}
We will consider each summand in the symmetrization separately.
For condition (1), let $\vec{k}^\alpha$ be the vector of $\xi$-scaled variables in $F_{p_\alpha,n_\alpha}$ for a given summand of the symmetrization.
The exponent of $\xi$ that can be extracted from $\xi^{r_0(\vec{k})}\prod_\alpha F_{p_\alpha,n_\alpha}$ is
\begin{equation}
\sum_{\alpha=1}^mk^\alpha_0-k^\alpha_{p_\alpha}+\frac{1}{2}\sum_{i\in\ZZ/\ell\ZZ}\left( k_i-k_{i+1} \right)^2
\label{MNKSum}
\end{equation}
On the other hand, the exponent that can be extracted from the mixing terms for a pair of indices $1\le \alpha<\beta\le m$ is 
\begin{align}
\nonumber&\sum_{i\in\ZZ/\ell\ZZ}k^\alpha_ik^\beta_{i+1}+k^\alpha_ik^\beta_{i-1}-2k^\alpha_ik^\beta_i\\
\nonumber&=\sum_{i\in\ZZ/\ell\ZZ}k^\alpha_i\left( k^\beta_{i+1}-k^\beta_i \right)-k^\alpha_i\left( k^\beta_i-k^\beta_{i-1} \right)\\
\nonumber&=\sum_{i\in\ZZ/\ell\ZZ}\left( k^\alpha_i-k^\alpha_{i+1} \right)\left( k^\beta_{i+1}-k^\beta_i \right)\\
\label{MNSum}&=-\sum_{i\in\ZZ/\ell\ZZ}\left( k^\alpha_i-k^\alpha_{i+1} \right)\left( k^\beta_i-k^\beta_{i+1} \right)
\end{align}
Adding (\ref{MNKSum}) with (\ref{MNSum}) and writing
\[k_i=\sum_{a=1}^mk^\alpha_i\]
we can rewrite the overall total exponent 
\begin{equation}
\sum_{\alpha=1}^m\left\{k^\alpha_0-k^\alpha_{p_\alpha}+\frac{1}{2}\sum_{i\in\ZZ/\ell\ZZ}\left( k^\alpha_i-k^\alpha_{i+1} \right)^2\right\}
\label{TotalExp}
\end{equation}
which is nonnegative (see the proof of Proposition \ref{GenWithin}).
Thus, the limit $\xi\rightarrow 0$ is finite.

For the limit $\xi\rightarrow\infty$, we first focus on the mixing terms. 
In (\ref{MNSum}), we have already accounted for the power of $\xi$ we can cleanly extract from the mixing terms for indices $1\le \alpha<\beta\le m$.
The growth from linear factors with only one $\xi$-scaled variable is $\xi$ to the power of
\begin{align}
\nonumber&\sum_{i\in\ZZ/\ell\ZZ}\left\{k_i^\alpha\left( n_\beta-k_{i+1}^\beta \right)+k_i^\alpha\left( n_\beta-k^\beta_{i+1} \right)-2k_i^\alpha\left( n_\beta-k_i^\beta \right)\right.\\
\nonumber&\left.+k_i^\beta\left( n_\alpha-k_{i+1}^\alpha \right)+k_i^\beta\left( n_\alpha-k_{i+1}^\alpha \right)-2k_i^\beta\left( n_\alpha-k_i^\alpha \right)\right\}\\
\label{InfMixGrowth}&=2\sum_{i\in\ZZ/\ell\ZZ}\left( k_{i}^\alpha-k_{i+1}^\alpha \right)\left( k_i^\beta-k_{i+1}^\beta \right)
\end{align}
For the total growth rate of the entire summand in the shuffle product (times $\xi^{r_\infty(\vec{k})}$) as $\xi\rightarrow\infty$, we add (\ref{InfMixGrowth}) to (\ref{TotalExp}) and $r_\infty(\vec{k})-r_0(\vec{k})$, yielding
\begin{align}
\label{MixGrowth}\sum_{\alpha=1}^m\left\{k_0^\alpha-k_{p_\alpha}^\alpha-\frac{1}{2}\sum_{i\in\ZZ/\ell\ZZ}\left( k^\alpha_i-k^\alpha_{i+1} \right)^2\right\}
\end{align}
which is necessarily nonpositive.
Therefore, the $\xi\rightarrow\infty$ limit exists.

If $(a;b]$ is diagonal, (\ref{TotalExp}) and (\ref{MixGrowth}) imply that the only nonzero summands in either limit have all $\vec{k}^\alpha$ diagonal. 
For condition (2), we also need to check that the mixing terms do not contribute any discrepancies between the two limits. 
Such a discrepancy can only happen in mixing terms involving a non-scaled variable in $F_{p_\alpha,n_\alpha}$ and a scaled variable in $F_{p_\beta,n_\beta}$ and vice versa. 
In the first case, for each non-scaled variable $x_{i,r}$ in $F_{p_\alpha,n_\alpha}$, these mixing terms can be partitioned into triples of the form
\[\omega_{i-1,i}(\xi x_{i-1, s_1},x_{i,r})\omega_{i,i}(\xi x_{i,s_2},x_{i,r})\omega_{i+1,i}(\xi x_{i+1,s_3},x_{i,r})\]
This partitioning is possible because $\vec{k}^\alpha$ and $\vec{k}^\beta$ must be diagonal.
Now observe that the limits of such a product as $\xi\rightarrow 0,\infty$ both equal $\ddd^{-1}$. 
The case where the non-scaled variable is in $F_{p_\beta,n_\beta}$ is similar. 

Finally, for $(a;b]=\vec{k}$, our concern for condition (3) is that (\ref{TotalExp}) results in zero for every $\vec{k}^\alpha$.
This can only occur if for each $\vec{k}^\alpha$, 
\begin{equation}
k_i^\alpha-k_{i+1}^\alpha=\left\{
\begin{array}{ll}
-1\hbox{ or }0 & \hbox{for }0\le i\le p_\alpha-1\\
1\hbox{ or }0 &\hbox{for }p_\alpha\le i\le\ell-1
\end{array}\right.
\label{KAlphaCondition}
\end{equation}
To see this, we break $k_0^\alpha-k_{p_\alpha}^\alpha$ apart as in (\ref{K0PBreak}) from the proof of Proposition \ref{GenWithin} and note that the quadratic terms of (\ref{TotalExp}) will overwhelm everything else once $|k_i^\alpha-k^\alpha_{i+1}|\ge 2$. 
Since
\[\sum_{\alpha=1}^m k_i^\alpha-k_{i+1}^\alpha=k_i-k_{i+1}=(a;b]_i-(a;b]_{i+1}\]
the conditions outlined in (\ref{KAlphaCondition}) can only occur for all $\alpha$ if $0\in(a;b]^-$.
\end{proof}

\subsubsection{Filtrations} 
To enhance Lemma \ref{GenWithin2} into an equality, we will use the \textit{Gordon filtrations}. 
They are defined via certain evaluation maps whose introduction will require some notation. 
For a degree vector $\vec{k}$, we call an unordered list $L=\{(a_1;b_1],\ldots,(a_m;b_m]\}$ a \textit{partition of $\vec{k}$} if $\vec{k}=\sum_{u=1}^m(a_u;b_u]$. 
This is denoted by $L\vdash\vec{k}$. 
We will always index the parts of a partition so that
\[b_u-a_u\ge b_{u+1}-a_{u+1}\]
and in the case of equality, $a_u\ge a_{u+1}$.
As in \ref{FockDef}, we will abbreviate $a\equiv b\hbox{ (mod }\ell)$ by $a\equiv b$.
For any two partitions $L,L'$ of $\vec{k}$, we say $L> L'$ if there is some $v$ such that $b_v-a_v> b'_v-a'_v$ but $b_u-a_u=b'_u-a'_u$ for all $1\le u<v$. 
We call this the \textit{dominance order} on partitions of $\vec{k}$.
Finally, let us call a partition $L$ \textit{even} if all its parts satisfy $b_u-a_u\equiv 0$. 

For $F\in \Sss(0)_{n\delta}$ and $L\vdash n\delta$, we define the evaluation map $\phi_L(F)\in\FF(y_u)$ by first splitting the variables into groups according to the parts of $L$. 
Next, obtain $\phi_L(F)$ by specializing the variables assigned to $(a_u;b_u]$ to
\[q^{-a_u-1}y_u,\ldots,q^{-b_u}y_u\]
where $q^{-c}y_u$ is assigned to a variable with color $\equiv c$. 
This does not depend on our choices because of color symmetry and is well-defined because of the pole conditions. 
The Gordon filtration is then given by
\[\Sss(0)_L:=\bigcap_{L'> L}\ker(\phi_{L'})\]

We will also need a `transposed' version. Our dual evaluation map $\phi_L^*(F)\in\FF(y_u)$ is defined by splitting the variables like before, but now for $(a_u;b_u]$, we specialize the variables to
\[t^{a_u+1}y_u,\ldots, t^{b_u}y_u\]
where $t^{c}y_u$ is now assigned to a variable with color $\equiv c$. 
We define the \textit{dual} Gordon filtration by
\[\Sss(0)_L^*:=\bigcap_{L'> L}\ker(\phi_{L'}^*)\]
Note that both filtrations are nested with respect to dominance order: if $L>L'$, then
\begin{align*}
\begin{split}
\Sss(0)_L\supset \Sss(0)_{L'}\\
\Sss(0)_L^*\supset \Sss(0)_{L'}^*
\end{split}
\end{align*}
%

\begin{lem}[\cite{FeiTsym} Lemma 3.4]\label{EvalMapLem} 
Let $F\in\Sss(0)_{n\delta}$.
For $L=\{(a_1;b_1],\ldots,(a_m;b_m]\}\vdash n\delta$, let $\vec{k}^u=(a_u;b_u]$.
We have the following:
\begin{enumerate}
\item For $F\in \Sss(0)_L$, 
\[\phi_L(F)=\nu \frac{Q_1}{Q_2}\prod_u y_u^{\sum_{i\in\ZZ/\ell\ZZ}k_i^uk_{i+1}^u-k_i^u(k_i^u-1)}\]
where $\nu\in\FF$, $Q_1$ is the product of the linear factors
\begin{enumerate}
\item $\left(q^{-x'}y_v-t^{-1}q^{-x}y_u\right)$ for $u<v$, $a_u< x< b_u$, $a_v< x'\le b_v$, and $x'\equiv x-1$;
\item $\left(q^{-x'}y_v-tq^{-x}y_u\right)$ for $u<v$, $a_u+1< x\le b_u$, $a_v< x'\le b_v$, and $x'\equiv x+1$;
\item $\left(q^{-x'}y_v-q^{-b_u-1}y_u\right)$ for $u<v$, $a_v<x'\le b_v$, and $x'\equiv b_u+1$;
\item $\left(q^{-x'}y_v-q^{-a_u}y_u\right)$ for $u<v$,  $a_v<x'\le b_v$, and $x'\equiv a_u$;
\end{enumerate}
and
\[Q_2=\prod_{u<v}\prod_{\substack{a_u< x\le b_u\\a_v<x'\le b_v\\x'\equiv x}}\left(q^{-x'}y_v-\qqq^2q^{-x} y_u\right)\left( q^{-x}y_u-\qqq^2 q^{-x'}y_v \right)\]
\item For $F\in \Sss(0)_L^*$, 
\[\phi^*_L(F)=\nu \frac{Q^*_1}{Q^*_2}\prod_u y_u^{\sum_{i\in\ZZ/\ell\ZZ}k_i^uk_{i+1}^u-k_i^u(k_i^u-1)}\]
where $\nu\in\FF$, $Q_1^*$ is the product of the linear factors
\begin{enumerate}
\item $\left(t^{x'}y_v-qt^{x}y_u\right)$ for $u<v$, $a_u+1< x\le b_u$, $a_v< x'\le b_v$, and $x'\equiv x-1$;
\item $\left(t^{x'}y_v-q^{-1}t^{x}y_u\right)$ for $u<v$, $a_u< x< b_u$, $a_v< x'\le b_v$, and $x'\equiv x+1$;
\item $\left(t^{x'}y_v-t^{b_u+1}y_u\right)$ for $u<v$, $a_v<x'\le b_u$, and $x'\equiv b_u+1$;
\item $\left(t^{x'}y_v-t^{a_u}y_u\right)$ for $u<v$, $a_v<x'\le b_v$, and $x'\equiv a_u$;
\end{enumerate}
and
\[Q_2^*=\prod_{u<v}\prod_{\substack{a_u< x\le b_u\\a_v< x'\le b_v\\x'\equiv x}}\left(t^{x'}y_v-\qqq^2t^{x}y_u\right)\left( t^{x}y_u-\qqq^2t^{x'}y_v \right)\]
\end{enumerate}
\end{lem}

\begin{proof}
The proof is similar to that of Lemma 3.4 of \cite{FeiTsym}, but we provide it anyways because of our change in conventions.
As in the pole conditions, we write $F\in\Sss_{n\delta}$ as
\[
F\left( \left\{ x_{i,r} \right\} \right)=\frac{f\left( \left\{ x_{i,r} \right\} \right)}
{\displaystyle\prod_{i\in\ZZ/\ell\ZZ}\,\prod_{\substack{1\le r,r'\le n\\r\not=r'}}\left( x_{i,r}-\qqq^2x_{i,r'} \right)}
\]
for some Laurent polynomial $f$.
We will focus first on $f$.
Limit condition (2) implies that $\phi_L(f)$ must be homogeneous with total degree $n^2\ell$.
On the other hand, limit condition (1) implies that the degree in the variable $y_u$ is bounded above by
\[
-r_\infty(\vec{k})+\sum_{i\in\ZZ/\ell\ZZ}k_i^u(k_i^u-1)+2k_i^u(n-k_i^u)=\sum_{i\in\ZZ/\ell\ZZ}k_i^u(n-k^u_{i+1})+k_i(n-k^u_{i-1})+k^u_ik^u_{i+1}
\]

The condition that $F\in\Sss(0)_L$ or $F\in\Sss(0)^*_L$ forces $\phi_L(f)$ to be divisible by $Q_1$ or $Q_1^*$, respectively.
Specifically, the factors of type (a) and (b) arise from the wheel conditions while those of type (c) and (d) arise from membership in $\Sss(0)_L$ or $\Sss(0)^*_L$.
Both $Q_1$ and $Q_1^*$ have total degree
\[
\sum_{1\le u<v\le m}\sum_{i\in\ZZ/\ell\ZZ} k_i^uk_{i+1}^v+k_i^uk_{i-1}^v=n^2\ell-\sum_{u=1}^m\sum_{i\in\ZZ/\ell\ZZ}k_i^uk_{i+1}^u
\]
and degree
\[
\sum_{i\in\ZZ/\ell\ZZ}k_i^u\left( n-k_{i-1}^u\right) +k_i^u\left(n-k_{i+1}^u\right)
\]
in the variable $y_u$.
Thus, the quotients $\phi_L(f)/Q_1$ and $\phi_L^*(f)/Q_1^*$ have total degree 
\[
\sum_{u=1}^m\sum_{i\in\ZZ/\ell\ZZ}k_i^uk_{i+1}^u\]
and degree in $y_u$ bounded above by
\[
\sum_{i\in\ZZ/\ell\ZZ}k_i^uk_{i+1}^u
\]
It follows that $\phi_L(f)/Q_1$ and $\phi_L^*(f)/Q_1^*$ equal a monomial of the form
\begin{equation}
\nu' \prod_uy_u^{\sum_{i\in\ZZ/\ell\ZZ}k_i^uk_{i+1}^u}
\label{Nu1}
\end{equation}
for some $\nu'\in\FF$.
In evaluating the denominator of $F$, one obtains $Q_2$ or $Q_2^*$ times the monomial
\begin{equation}
\nu''\prod_uy_u^{\sum_{i\in\ZZ/\ell\ZZ}k_i^u(k_i^u-1)}
\label{Nu2}
\end{equation}
for some \textit{nonzero} $\nu''\in\FF$.
The lemma follows from dividing (\ref{Nu1}) by (\ref{Nu2}).
\end{proof}

\begin{cor}\label{HorShuff}
$\Psi_-\left(\Sss(0)\right)=\varpi^{-1}(\ddot{U}^0_-) $.
\end{cor}

\begin{proof} 
Recall that if $L'< L$, then $\Sss(0)_{L'}\subset\Sss(0)_L$.
Define the associate graded as
\[\mathrm{gr}\,\Sss(0)_L:=\left.\Sss(0)_L\middle/\bigcup_{L'< L}\Sss(0)_{L'}\right.\]
By the definition of the Gordon filtration, we have
\begin{equation}
\ker\left(\left.\phi_L\right|_{\Sss(0)_L}\right)\supset\bigcup_{L'< L}\Sss(0)_{L'}
\label{KerGordon}
\end{equation}
On the other hand, unless $L$ is minimal in dominance order (i.e. all parts have size 1), there is a class of partitions immediately below $L$ in the order.
Specifically, suppose $L=\{ (a_1;b_1],\ldots, (a_m;b_m] \}$, where we allow empty parts.
Let $(a_{u_d}; b_{u_d}]$ be the last part with $b_{u_d}-a_{u_d}>1$ (the \textit{donor} part) and let $(a_{u_r};b_{u_r}]$ be the first part, possibly empty, with $b_{u_r}-a_{u_r}<b_{u_d}-a_{u_d}-1$ (the \textit{receiver} part).
Consider then a partition $L^\downarrow=\{(a_1^\downarrow;b_1^\downarrow],\ldots,(a_m^\downarrow;b_m^\downarrow]\}$ where
\[
b_u^\downarrow-a_u^\downarrow=
\left\{
\begin{array}{ll}
b_u-a_u & \hbox{for }u\not=u_d, u_r\\
b_{u_d}- a_{u_d}-1 & \hbox{for }u= u_d\\
b_{u_r}- a_{u_r}+1 & \hbox{for }u=u_r
\end{array}
\right.
\]
Such a partition can be constructed out of $L$ itself, but there will be multiple due to color shifts.
It immediately follows from the definition of dominance order that for any such $L^\downarrow$, $L'> L^\downarrow$ implies either $L'=L$ or $L'>L$.
Therefore, in the case where $L$ is not minimal, (\ref{KerGordon}) is an equality.
On the other hand, for minimal $L$, $\phi_L(F)$ is just $F$ written with a different set of variables, so $\ker(\phi_L)=0$ and (\ref{KerGordon}) is trivially an equality as well.
Lemma \ref{EvalMapLem} then implies that $\dim_{\FF}\mathrm{gr}\,\Sss(0)_L\le 1$ since the image of $\phi_L$ restricted to $\Sss(0)_L$ is at most 1-dimensional.

Next, observe that $\left.\phi_L\right|_{\Sss(0)_L}=0$ is trivial if $L$ is not even. 
To see this, let $F\in \Sss(0)_L$ for such an $L$. 
Since $L$ partitions a diagonal vector, there must be some nondiagonal interval $(a_u;b_u]$ with $0\in(a_u;b_u]^+$.
Using the notation from Lemma \ref{EvalMapLem}(1), the definition of $\Sss(0)$ forces
\[\lim_{\xi\rightarrow 0}\left.\frac{\nu Q_1}{Q_2}\right|_{y_u\mapsto \xi y_u}=0\]
This is only possible if $\nu=0$ and thus $\mathrm{gr}\,\Sss(0)_L=0$. 
Overall, we obtain the following dimension bound: 
\[\dim_{\FF}\Sss(0)_{n\delta}\le|\{L\vdash n\delta : L\hbox{ is even}\}|=\dim_{\FF}(\Psi_-^{-1}\circ\varpi^{-1})(\ddot{U}^0_-)_{n\delta}\]
The result then follows from Lemma \ref{GenWithin}.
\end{proof}

\subsubsection{Evaluation functionals} 
It will be useful for us to `downgrade' the evaluation maps into $\FF$-linear functionals.
When $F\in\Sss(0)_L$, we essentially replace $\phi_L(F)$ with the constant $\nu\in\FF$ of Lemma \ref{EvalMapLem} and likewise for $\Sss(0)_L^*$ and $\phi_L^*$.

\begin{defn}\label{EvalFuncDef}
Let $L=\{(a_1;b_1],\ldots,(a_m;b_m]\}\vdash n\delta$. 
Define
\begin{align*}
\begin{split}
N_L&:=\prod_{u=1}^m \prod_{\substack{c,c'\in (a_u,b_u]\\c<c'}}\omega_{\bar{c}',\bar{c}}\left( q^{-c'}y_u,q^{-c}y_u \right) \\
&\times\prod_{1\le u<v\le m} \prod_{\substack{c\in(a_u,b_u]\\c'\in(a_v,b_v]}}\omega_{\bar{c}',\bar{c}}\left( q^{-c'}y_v,q^{-c}y_u \right) \\
N_L^*&:=\prod_{u=1}^m \prod_{\substack{c,c'\in (a_u,b_u]\\c<c'}}\omega_{\bar{c},\bar{c}'}\left( t^{c}y_u,t^{c'}y_u \right) \\
&\times\prod_{1\le u<v\le m} \prod_{\substack{c\in(a_u,b_u]\\c'\in(a_v,b_v]}}\omega_{\bar{c'},\bar{c}}\left( t^{c'}y_v,t^{c}y_u \right)
\end{split}
\end{align*}
The \textit{evaluation functionals} $\rho_L$ and $\rho_L^*$ are defined as
\begin{align*}
\rho_L(F)&:=\left[ \left.\frac{\phi_L(F)}{N_L}\right|_{|y_{u}|\ll|y_{u-1}|} \right]_0\\
\rho_L^*(F)&:=\left[ \left.\frac{\phi_L^*(F)}{N_L}\right|_{|y_{u}|\ll|y_{u-1}|} \right]_0
\end{align*}
where for a Laurent series $f(y_1,\ldots, y_m)$, $[f]_0$ denotes its constant term.
\end{defn}

\begin{cor}\label{FuncFilt}
The evaluation functionals also determine the Gordon filtrations. Namely,
\begin{align*}
\Sss(0)_L&=\bigcap_{L'> L}\ker(\rho_{L'})\\
\Sss(0)_L^*&=\bigcap_{L'> L}\ker(\rho_{L'}^*)
\end{align*}
\end{cor}

\begin{proof}
If $L$ is a partition with a single part, by Lemma \ref{EvalMapLem}, $\phi_L(F)$ and $\phi_{L}^*(F)$ are monomials in $y_1$.
It is easy to see that $N_L$ and $N_L^*$ are also monomials in the same power, so $\ker(\rho_L)=\ker(\phi_L)$ and $\ker(\rho_L^*)=\ker(\phi_L^*)$ in this case, which is the most dominant case. 
By downward induction along dominance order, it suffices to show that if $F\in \Sss(0)_L$, then $\phi_L(F)=0$ if and only if $\rho_L(F)=0$ and likewise for the dual picture. 
Using the notation of Lemma \ref{EvalMapLem}, the factors of $N_L$ and $N_L^*$ involving different variables $y_u$ and $y_v$ are in bijection with the factors in $Q_1/Q_2$ and $Q_1^*/Q_2^*$ in such a manner that $\phi_L(F)/N_L$ and $\phi_L^*(F)/N_L^*$ have degree zero in each variable.
Thus, up to a nonzero scalar coming from taking the appropriate Laurent series expansion, $\rho_L(F)$ picks out the constant $\nu$ from Lemma \ref{EvalMapLem}.
This implies $\phi_L(F)=0$ if and only if $\rho_L(F)=0$ and likewise in the dual case.
\end{proof}

We will take particular interest in two kinds of partitions of $n\delta$. 
For $p\in\ZZ/\ell\ZZ$, the \textit{short partition} has $n$ parts all equal to $(p;p+\ell]$. 
We denote it by $L_p^{1^n}$. 
The \textit{long partition} has a single part given by $(p;p+\ell n]$, and we will denote it by $L_p^{(n)}$. 
We will call their corresponding evaluations the short and long evaluations, respectively.

\subsubsection{Bottom elements} 
When we restrict ourselves to even partitions, the short partitions lie at the bottom of the dominance order.
Therefore, $\Sss(0)_{L_p^{1^n}}$ and $\Sss(0)_{L_p^{1^n}}^*$ are one-dimensional.

\begin{prop}\label{BottomElements}
$\Sss(0)_{L_p^{1^n}}$ is spanned by
\begin{align*}
E_{p,n}&:=
\Sym\left( \prod_{1\le r<s\le n}\left\{\frac{x_{p+1,r}-q^{-1}x_{p,s}}{x_{p+1,r}-tx_{p,s}}\prod_{i,j\in\ZZ/\ell\ZZ}\omega_{i,j}\left( x_{i,r},x_{j,s} \right)\right\}\right.\\
&\times \left.\prod_{r=1}^n\left\{\left( q^{-1}\frac{x_{0,r}}{x_{p+1,r}}-\frac{x_{0,r}}{x_{p,r}} \right)\prod_{i\in\ZZ/\ell\ZZ}x_{i,r}  \right\} \right)
\end{align*}
and $\Sss(0)_{L_p^{1^n}}^*$ is spanned by
\begin{align*}
H_{p,n}&:=
\Sym\left( \prod_{1\le r<s\le n}\left\{\frac{t^{-1}x_{p+1,s}-x_{p,r}}{qx_{p+1,s}-x_{p,r}}\prod_{i,j\in\ZZ/\ell\ZZ}\omega_{i,j}\left( x_{i,r},x_{j,s} \right)\right\}\right.\\
&\times \left.\prod_{r=1}^n\left\{\left( t\frac{x_{0,r}}{x_{p+1,r}}-\frac{x_{0,r}}{x_{p,r}} \right)\prod_{i\in\ZZ/\ell\ZZ}x_{i,r}  \right\} \right)
\end{align*}
\end{prop}

\begin{proof}
First observe that within the symmetrizations, we have
\begin{align*}
\frac{x_{p+1,r}-q^{-1}x_{p,s}}{x_{p+1,r}-tx_{p,s}}&= \qqq\frac{\omega_{p,p+1}\left( x_{p,s},x_{p+1,r} \right)}{\omega_{p+1,p}\left( x_{p+1,r},x_{p,s} \right)}\\
\frac{t^{-1}x_{p+1,s}-x_{p,r}}{qx_{p+1,s}-x_{p,r}}&= \qqq\frac{\omega_{p+1,p}\left( x_{p+1,s},x_{p,r} \right)}{\omega_{p,p+1}\left( x_{p,r},x_{p+1,s} \right)}
\end{align*}
With these factors written using mixing terms, the proof that $E_{p,n},H_{p,n}\in\Sss$ is similar to the proof that $\Sss$ is closed under the shuffle product.
Also, the proof that both elements lie in $\Sss(0)$ is similar to the proof of Lemma \ref{GenWithin2}.
Specifically, up to scalar, $E_{p,n}$ is the symmetrization of products of $\left(q^{-1}F_{p+1,1}-F_{p,1}\right)$ and mixing terms among them, and likewise $H_{p,n}$ is up to scalar the symmetrization of products of $\left( tF_{p+1,1}-F_{p,1} \right)$ and their mixing terms.
Thus, the analysis in the proof of Lemma \ref{GenWithin2} applies to $E_{p,n}$ and $H_{p,n}$.

Next, observe that every summand in $E_{p,n}$ evaluates to zero once $q^{-1}x_{p,r}=x_{p+1,s}$ for any pair of indices $r,s$, from which we obtain $E_{p,n}\in\Sss(0)_{L_p^{1^n}}$.
Similarly, $H_{p,n}$ evaluates to zero once $tx_{p,r}=x_{p+1,s}$ for any pair of indices $r,s$ and thus $H_{p,n}\in\Sss(0)_{L_p^{1^n}}^*$.
Finally, it is not obvious that $E_{p,n}$ and $H_{p,n}$ are nonzero.
This is a corollary of the computations of (\ref{FuncEpn}) and (\ref{FuncHpn}) done in the proof of Theorem \ref{FinalShuffle} below.
We note that these computations are done by brute force and require no theory.
\end{proof}

To conclude this subsection, we note the following consequence of Proposition \ref{ShuffleFock}:

\begin{prop}\label{Adjacency}
In $\mathcal{F}$, $\langle\mu|(\tau_{0,\upsilon}^-\circ\Psi_-)(E_{p,n})|\lambda\rangle\not=0$ only if $\mu\backslash\lambda$ has no horizontally adjacent $p$- and $(p+1)$-nodes.
Similarly, $\langle\mu|(\tau_{0,\upsilon}^-\circ\Psi_-)(H_{p,n})|\lambda\rangle\not=0$ only if $\mu\backslash\lambda$ has no vertically adjacent $p$- and $(p+1)$-nodes.
\end{prop}

\subsection{Functionals}\label{FuncSec}
In this subsection, we show that the shuffle elements corresponding to $\tilde{h}_n(p)$ and $\tilde{e}_n(p)$ are $\FF$-multiples of $H_{p,n}$ and $E_{p,n}$, respectively. 
Since each of the latter span a lowest piece of the Gordon and dual Gordon filtrations, we first need to understand how the evaluation functionals are manifested in the toroidal side of the picture.

\subsubsection{Dual elements of evaluation functionals} 
For an integer $a$, we denote by $\bar{a}$ its class modulo $\ell$. 
In what follows, we will first consider the long partitions, and in this case, we will index variables by nonnegative integers so that $x_{i,r}$ is assigned to $x_a$, where $a$ is the $r$th integer greater than or equal to $p$ with $\bar{a}=i$. 
Recall that ``$\Sym$'' is the color-symmetrizer, and note that our choice of labeling the variables using $\left\{ x_a \right\}$ is not important if those variables are in a color-symmetric function.

We begin with an adaptation of Proposition 3.25 from \cite{NegutTor} in the case $\mu=0$:

\begin{prop}\label{LongDual}
Let
\begin{align*}
R_{p,n}&:=(\qqq-\qqq^{-1})^{n\ell}\Sym\left(\prod_{a=p+2}^{p+n\ell}\left(1-\frac{q^{-1}x_{a-1}}{x_{a}}\right)^{-1}\prod_{p< a<b\le p+n\ell}\omega_{\bar{a},\bar{b}}\left(x_a,x_b\right)\right)\\
R_{p,n}^*&:=(\qqq-\qqq^{-1})^{n\ell}\Sym\left(\prod_{a=p+1}^{p+n\ell-1}\left(1-\frac{ t^{-1}x_{a+1}}{x_{a}}\right)^{-1}\prod_{p< a<b\le p+n\ell}\omega_{\bar{b},\bar{a}}\left(x_b,x_a\right)\right)
\end{align*}
These functions lie in $\Sss$.
Viewing them in $\Sss^+$, we have for $F\in\Sss(0)_{n\delta}$,
\begin{align*}
\rho_{L_p^{(n)}}(F)&=\langle R_{p,n},F\rangle\\
\rho^*_{L_p^{(n)}}(F)&=\langle R^*_{p,n},F\rangle
\end{align*}
\end{prop}

\begin{proof}
It is not too difficult to see that $R_{p,n}$ and $R_{p,n}^*$ lie in $\Sss$.
For the pairing calculation, first consider $R_{p,n}$. 
Setting up the pairing for $F\in\Sss(0)_{n\delta}$, we have:
\[
\frac{1}{(n!)^\ell}\left.\underset{|x_{i,r}|=1}{\oint\cdots\oint}
\frac{\displaystyle(\qqq-\qqq^{-1})^{-n\ell}\left(R_{p,n}\right)_\ppp F_\ppp\prod_{i\in\ZZ/\ell\ZZ}\,\prod_{r=1}^{k_i}Dx_{i,r}}
{\displaystyle\prod_{i\in\ZZ/\ell\ZZ}\,\prod_{\substack{1\le r,r'\le n\\ r\not=r'}}\omega_{i,i}^\ppp(x_{i,r},x_{i,r'})\prod_{i\not=j}\,\prod_{r,s=1}^n\omega_{i,j}(x_{i,r},x_{j,s})}\right|_{\ppp\mapsto\qqq}
\]
Observe that
\begin{align*}
\left(R_{p,n}\right)_\ppp&=(\qqq-\qqq^{-1})^{n\ell}\prod_{i\in\ZZ/\ell\ZZ}\,\prod_{\substack{1\le r,r'\le n\\r\not=r'}}\omega_{i,i}^\ppp\left( x_{i,r},x_{i,r'} \right)\\
&\times\Sym\left(\prod_{a=p+2}^{p+n\ell}\frac{1}{\displaystyle \left(1-\frac{q^{-1}x_{a-1}}{x_a}\right)}\cdot\frac{\displaystyle\prod_{\substack{p< a<b\le p+n\ell\\\bar{a}\not=\bar{b}}}\omega_{\bar{a},\bar{b}}\left(x_a,x_b\right)}{\displaystyle\prod_{\substack{p< a<b\le p+n\ell\\\bar{a}=\bar{b}}}\omega_{\bar{b},\bar{a}}\left( x_{b},x_{a} \right)}\right)
\end{align*}
Since $F_\ppp$ is color-symmetric, we can move it within the symmetrization for $R_{p,n}$, and thus the integrand becomes:
\begin{align*}
&\frac{1}{(n!)^\ell}\Sym\left(\frac{\displaystyle F_\ppp\prod_{a=p+1}^{p+n\ell}Dx_a}{\displaystyle\prod_{a=p+2}^{p+n\ell}\left(1-\frac{q^{-1} x_{a-1}}{x_a}\right)\prod_{p< a<b\le p+n\ell}\omega_{\bar{b},\bar{a}}(x_b,x_a)}\right)
\end{align*}

Let us consider the integral of the summand in the symmetrization corresponding to the identity permutation; it will become apparent that the other summands will yield the same result if we change the order of integration.
For the identity summand, we begin by integrating $x_{p+n\ell}$.
The only possible poles in its contour are:
\begin{enumerate}
\item $x_{p+n\ell}=0$;
\item $x_{p+n\ell}=\ppp^2 x_{p+s\ell}$ for $1\le s\le n-1$ (from the denominator of $F_\ppp$);
\item $x_{p+n\ell}=q^{-1}x_{p+n\ell-1}$ (from the denominator of $R_{p,n}$).
\end{enumerate}
For the residue at $x_{p+n\ell}=0$, note that by the first limit condition for $F$ applied to the $p$th coordinate vector, $F$ is holomorphic at $x_{p+n\ell}=0$.
The rest of the integrand is obviously holomorphic at $x_{p+n\ell}=0$ once we multiply the $x_{p+n\ell}^{-1}$ factor in $D x_{p+1}$ to $\left( 1-q^{-1} x_{p+n\ell-1}/x_{p+n\ell} \right)^{-1}$.
Thus, the residue here is trivial.
Next, the residue at $x_{p+n\ell}=\ppp^2 x_{p+s\ell}$ vanishes once we evaluate $\ppp\mapsto\qqq$ due to the factor $\omega_{i,i}\left( x_{p+n\ell},x_{p+s\ell} \right)$.
Finally, the residue at $x_{p+n\ell}=q^{-1}x_{p+n\ell-1}$ is the first step for our desired evaluation.

Continuing onwards in \textit{descending} order to the final integral, the integrand is a $q$-shift of
\[
\frac{1}{(n!)^\ell}\frac{\phi_L(F)}{ N_L}(x_{p+1})Dx_{p+1}
\] 
As stated before, for the other summands in the symmetrization, we obtain the same result by switching the order of integration.
The proof of Corollary \ref{FuncFilt} shows that $\phi_L(F)/N_L$ is actually a constant when $L$ is a long partition, so the residue yields $\rho_L(F)$.
For $R_{p,n}^*$, the proof is similar except that we integrate the variables in \textit{ascending} order.
\end{proof}

Similar to how for long evaluations, we indexed shuffle variables using natural numbers, for an even partition $L$ with multiple parts, we will index the variables in the natural way using $x_{c_u}$ for the $u$th part.

\begin{prop}\label{DualProduct}
Let $L=\{(a_1;b_1],\ldots, (a_m;b_m]\}\vdash n\delta$ with
\begin{align*}
a_u&= p_u\\
b_u&= p_u+n_u\ell
\end{align*}
for $0\le p_i\le\ell-1$.
For $F\in \Sss(0)_{n\delta}$, 
\begin{align*}
\langle R_{p_1,n_1}\star\cdots\star R_{p_m,n_m}, F \rangle&=\rho_L(F)\\
\langle R_{p_1,n_1}^*\star\cdots\star R_{p_m,n_m}^*, F  \rangle&=\rho_L^*(F)
\end{align*}
\end{prop}

\begin{proof}
We will only prove the $\rho_L$ statement.
As in the proof of Proposition \ref{LongDual}, the integrand can be manipulated into
\begin{align}
\nonumber&\frac{1}{(n!)^\ell}\Sym\left\{\frac{\displaystyle F_\ppp\prod_{u=1}^m\,\prod_{c=a_u+1}^{b_u}Dx_{c_u}}
{\displaystyle\prod_{u=1}^m\left[\prod_{c=a_u+2}^{b_u}\left(1-\frac{q^{-1} x_{(c-1)_u}}{x_{c_u}}\right)\prod_{a_u< c<d\le b_u}\omega_{\bar{d},\bar{c}}(x_{d_u},x_{c_u})\right]}\right.\\
&\times
\left.\frac{1}{\displaystyle\prod_{1\le u<v\le m}\,\prod_{\substack{a_u<c\le b_u\\a_v<d\le b_v}}\omega_{\bar{d},\bar{c}}(x_{d_v},x_{c_u})}
\right\}
\label{ProductDual}
\end{align}
Also as before, the symmetrization only affects the order of integration.
For the summand corresponding to the identity permutation, we begin with the variable $x_{ (b_m)_m}$ and proceed as in the proof of Proposition \ref{LongDual} until we reach $x_{(a_m+1)_m }$.
Arguing again as in the proof of Proposition \ref{LongDual}, one can see from the factors in (\ref{ProductDual}) that the only pole in the contour for $x_{(a_m+1)_m}$ that can contribute a nonzero residue is at $0$. 
Thus, we can deform the contour so that $|x_{(a_m+1)_m}|$ is much smaller than the magnitudes of the other variables and then take the constant term.
Doing the same for the other parts of $L$ in descending order of parts yields $\rho_L(F)$. 
\end{proof}

\subsubsection{Restriction of $\varphi(-,-)$ to $\varpi^{-1}( \ddot{U}^0 )$}
Recall that the quasi $R$-matrix $\mathcal{R}_\circ$ is the canonical tensor for the pairing $\varphi\left( -,- \right)$ restricted to $\ddot{U}^+\times\ddot{U}^-$, which is nondegenerate. 
Furthermore, Theorem \ref{QuasiRFact} and Proposition \ref{Factor2} combine to show that $\varphi\left( -,- \right)$ restricts to a nondegenerate pairing on $\varpi^{-1}( \ddot{U}^0 )$ and that $\varpi^{-1}( \ddot{U}^0 )$ pairs trivially with elements outside of it.
To add more detail, let $\ddot{U}^0_+$ and $\ddot{U}^0_-$ denote the subalgebras generated by the positive and negative modes, respectively.
The formula for $h(\bar{\mathcal{R}}^0_\mathfrak{gl})$ in Proposition \ref{Factor2} and its appearance in $\mathcal{R}_\circ$ as outlined in Corollary \ref{QuasiRFact} imply that the dual element of 
\[
\varpi^{-1}(b_{i_1,\pm k_1})^{n_1}\cdots\varpi^{-1}(b_{i_m,\pm k_m})^{n_m}\in\varpi^{-1}(\ddot{U}^0_\pm)
\]
is
\[
n_1!\varpi^{-1} (b_{i_1,\mp k_1}^\perp)^{n_1}\cdots n_m!\varpi^{-1}( b_{i_m,\mp k_m}^\perp)^{n_m}\in\varpi^{-1}(\ddot{U}^0_\mp)
\]

Consider the following coproduct $\Delta^\pm_H$ on $\varpi^{-1}(\ddot{U}^0_\pm)$:
\[
\Delta^\pm_H\left(\varpi^{-1}(b_{i,\pm k})\right)=\varpi^{-1}(b_{i,\pm k})\otimes 1+1\otimes \varpi^{-1}(b_{i,\pm k})
\]
for $k>0$.
We then also have:
\[
\Delta^\pm_H\left(\varpi^{-1}(b_{i,\pm k}^\perp)\right)=\varpi^{-1}(b_{i,\pm k}^\perp)\otimes 1+1\otimes \varpi^{-1}(b_{i,\pm k}^\perp)
\]
Note that $\Delta^\pm_H$ is only defined on each half of $\varpi^{-1}(\ddot{U}^0)$.
From our observations above about dual elements, it is easy to deduce the following: 

\begin{prop}\label{DrinBialgPair}
The restriction of $\varphi(-,-)$ to $\varpi^{-1}(\ddot{U}^0_+)\times\varpi^{-1}(\ddot{U}^0_-)$ is a bialgebra pairing with respect to the coproducts $\Delta_{H}^\pm$.
\end{prop}
%

By Proposition \ref{LongDual}, we have seen that $\langle R_{p,n},-\rangle$ and $\langle R_{p,n}^*,-\rangle$ define functionals on $\Sss(0)$.
Combining this with Proposition \ref{Twine} and Corollary \ref{HorShuff}, we have that $\varphi(\Psi_+(R_{p,n}),-)$ and $\varphi(\Psi_+(R_{p,n}^*),-)$ define functionals on $\varpi^{-1}(\ddot{U}_-^0)$.
These functionals are nonzero because they determine nonzero graded pieces of the Gordon filtrations.
Therefore, we define the nonzero elements $\Psi_+(R_{p,n})_0$ and $\Psi_+(R^*_{p,n})_0$ to be the images of $\Psi_+(R_{p,n})$ and $\Psi_+(R_{p,n}^*)$ under the composition
\[
\ddot{U}^+\twoheadrightarrow \varpi^{-1}(\ddot{U}^0_-)^*\cong\varpi^{-1}(\ddot{U}^0_+)\hookrightarrow\ddot{U^+}
\]
i.e. we take the components contributing to these functionals on $\varpi^{-1}(\ddot{U}^0_-)$.
Finally, we define the currents
\begin{align}
\label{RCurrentsDef}
\begin{split}
R_p^*(z)&:=\sum_{n\ge 0}\Psi_+(R_{p,n}^*)_0z^{-n}\\
R_p(z)&:=\sum_{n\ge 0}\Psi_+(R_{p,n})_0z^{-n}
\end{split}
\end{align}

\begin{rem}
It is not too difficult to show that $\Psi_+(R_{p,n}^*)=\Psi_+(R_{p,n}^*)_0$ and $\Psi_+(R_{p,n})=\Psi_+(R_{p,n})_0$ by studying the shuffle realization of $\varpi^{-1}(\ddot{U}^0_+)$. 
\end{rem}

\subsubsection{Coproducts}\label{CoproductStrat} 
We observe that by Lemma \ref{ReduxLem}, the currents
\begin{align*}
\sum_{n\ge0}\tilde{h}_n(p)z^n&=\varpi^{-1}\exp\left((\qqq-\qqq^{-1})^{-1}\sum_{n>0}(-\qqq^{-n}b_{p,-n}^\perp+\ddd^{-n}b_{p+1,-n}^\perp)\frac{z^n}{n}\right)\\
\sum_{n\ge0}\tilde{e}_n(p)z^n&=\varpi^{-1}\exp\left((\qqq-\qqq^{-1})^{-1}\sum_{n>0}(\qqq^{n}b_{p,-n}^\perp-\ddd^{-n}b_{p+1,-n}^\perp)\frac{(-z)^n}{n}\right)
\end{align*}
are group-like under the comultiplication $\Delta_{H}^-$ of $\varpi^{-1}(\ddot{U}^0_-)$, i.e.
\begin{align*}
\Delta^-_{H}(\tilde{h}_n(p))&=\sum_{n_1+n_2=n}\tilde{h}_{n_1}(p)\otimes\tilde{h}_{n_2}(p)\\
\Delta^-_{H}(\tilde{e}_n(p))&=\sum_{n_1+n_2=n}\tilde{e}_{n_1}(p)\otimes\tilde{e}_{n_2}(p)
\end{align*}
Recall that by Proposition \ref{DrinBialgPair}, the restriction of $\varphi(-,-)$ to $\varpi^{-1}( \ddot{U}^0_+ )\times\varpi^{-1}( \ddot{U}^0_- )$ is a bialgebra pairing with respect to $\Delta^\pm_{H}$. 
Thus, 
\begin{align}
\label{CoproductEH}
\begin{split}
\varphi\left(\overset{\curvearrowright}{\prod_{u=1}^k}\Psi_+(R^*_{a_u,n_u})_0,\tilde{h}_n(p)\right)
&=\prod_{u=1}^k\varphi\left(\Psi_+(R^*_{a_u,n_u}),\tilde{h}_{n_u}(p)\right)\\
\varphi\left(\overset{\curvearrowright}{\prod_{u=1}^k}\Psi_+(R_{a_u,n_u})_0,\tilde{e}_n(p)\right)
&=\prod_{u=1}^k\varphi\left(\Psi_+(R_{a_u,n_u}),\tilde{e}_{n_u}(p)\right)
\end{split}
\end{align}

Our goal is to prove the required vanishing conditions for $\tilde{h}_n(p)\in\Psi_-(\Sss(0))^*_{L_{p}^{1^n}}$ and $\tilde{e}_n(p)\in\Psi_-(\Sss(0))_{L_{p}^{1^n}}$
By Proposition \ref{DualProduct}, the evaluation functionals $\rho_L$ can be written as a pairing with a product of the elements $\{R_{p,n}\}_{p\in\ZZ/\ell\ZZ}$ and likewise for the dual functionals and $\{R_{p,n}^*\}_{p\in\ZZ/\ell\ZZ}$.
Combining this with (\ref{CoproductEH}), we only need to show that
\[\rho^*_{L_{p'}^{(n)}}\left(\Psi_-^{-1}\tilde{h}_n(p)\right)=\rho_{L_{p'}^{(n)}}\left(\Psi_-^{-1}\tilde{e}_n(p)\right)=0\]
if $p\not=p'$ or $n>1$. 
To that end, we will investigate the behaviour of $R_p(z)$ and $R^*_p(z)$ under $\Delta_{H}^+$.  

\begin{prop}[cf. \cite{NegutTor} Lemma 3.20]\label{Grouplike}
The currents $R_p(z)$ and $R_p^*(z)$ are group-like under the coproduct $\Delta_{H}^+$. 
Specifically,
\begin{align*}
\Delta^+_{H}(\Psi_+(R_{p,n})_0)&=\sum_{n_1+n_2=n}\Psi_+(R_{p,n_1})_0\otimes\Psi_+(R_{p,n_2})_0\\
\Delta^+_{H}(\Psi_+(R^*_{p,n})_0)&=\sum_{n_1+n_2=n}\Psi_+(R^*_{p,n_1})_0\otimes\Psi_+(R^*_{p,n_2})_0
\end{align*}
\end{prop}

\begin{proof}
We will only prove the statement for $R_p(z)$. 
A consequence of the proof of Corollary \ref{FuncFilt} is that 
\[
\langle R_{p,n},F\rangle=\rho_{L_p^{(n)}}(F)=\frac{\phi_{L_p^{(n)}}(F)}{N_{L_p^{(n)}}}
\]
are all equal to the constant
\[\frac{F\left(\{q^{-a}\}_{a=p+1}^{p+n\ell}\right)}{\displaystyle\prod_{p< a<b\le p+n\ell}\omega_{\bar{b},\bar{a}}(q^{-b},q^{-a})}\]
Now, let $F_1\in \Sss(0)_{n_1\delta}$ and $F_2\in \Sss(0)_{n_2\delta}$ with $n_1+n_2=n$. 
Notice that in $F_1\star^{op} F_2=F_2\star F_1$, the mixing terms $\omega_{i,i+1}(z,w)$ vanish when $q^{-1}z=w$. 
This implies that upon evaluation of the functional, a summand of the symmetrization will vanish if $q^{-a-1}$ is plugged into $F_1$ and $q^{-a}$ is plugged into $F_2$.
Therefore, all indices $a$ for $q^{-a}$ plugged into $F_1$ must be \textit{less} than any index $b$ for $q^{-b}$ plugged into $F_2$, giving us
\[\langle R_{p,n}, F_1\star^{op} F_2\rangle
=\frac{F_1\left(\{q^{-a}\}_{a=p+1}^{ p+n_1\ell}\right)F_2\left(\{q^{-b}\}_{b=p+n_1\ell+1}^{p+n\ell}\right)}
{\displaystyle\prod_{p+1< a<a'\le p+n_1\ell}\omega_{\bar{a'},\bar{a}}(q^{-a'},q^{-a})\prod_{p+n_1\ell+1<b<b'\le p+n\ell}\omega_{\bar{b'},\bar{b}}(q^{-b'},q^{-b})}\]
The equality with $\langle R_{p,n_1},F_1\rangle\langle R_{p,n_2},F_2\rangle$ comes from applying again the fact that $\phi_{L_p^{(n)}}(F_1)/N_{L_p^{(n)}}$ is a constant and thus the starting index for $b$ is irrelevant. 
Finally, consider the chain of equalities: 
\begin{align*}
\varphi\left( \Delta_H^+\left( \Psi_+(R_{p,n})_0 \right), \Psi_-(F_1)\otimes \Psi_-(F_2) \right)&=\varphi\left( \Psi_+(R_{p,n})_0, \Psi_-(F_1) \Psi_-(F_2) \right)\\
&=\langle R_{p,n}, F_1\star^{op} F_2\rangle\\
&= \langle R_{p_1,n_1}, F_1\rangle\langle R_{p_2,n_2}, F_2\rangle\\
&= \langle R_{p_1,n_1}\otimes R_{p_2,n_2}, F_1\otimes F_2\rangle\\
&= \varphi\left( \Psi_+(R_{p_1,n_1})\otimes\Psi_+(R_{p_2,n_2}), \Psi_-(F_1)\otimes\Psi_-(F_2) \right)\\
&= \varphi\left( \Psi_+(R_{p_1,n_1})_0\otimes\Psi_+(R_{p_2,n_2})_0, \Psi_-(F_1)\otimes\Psi_-(F_2) \right)
\end{align*}
Varying the degrees $n_1$ and $n_2$ of $F_1$ and $F_2$, we obtain the coproduct formulas by applying the nondegeneracy of $\varphi$ when restricted to $\varpi^{-1}(\ddot{U}^0_+)\times\varpi^{-1}(\ddot{U}^0_-)$.
\end{proof}

\subsubsection{Calculation of functionals} 
By Proposition \ref{Grouplike}, $R_p(z)$ and $R_p^*(z)$ are group-like currents with respect to the coproduct $\Delta_H^+$.
On the other hand, we can also consider the following group-like currents that generate $\varpi^{-1}(\ddot{U}_-^0)$:
\begin{equation*}
\exp\left( \sum_{k>0}\ddd^{-\frac{\ell k}{2}}\frac{\qqq^{-2k}-1}{(\qqq-\qqq^{-1})k}\varpi^{-1}(b_{p,-k}^\perp)z^{-k} \right)\overset{\hbox{\tiny (Cor. \ref{LoperatorCalc})}}{=}\sum_{n=0}^\infty \Psi_-(F_{p,n})z^{-n} 
\end{equation*}
Because $\varphi(-,-)$ restricted to $\varpi^{-1}(\ddot{U}^0_+)\times\varpi^{-1}(\ddot{U}^0_-)$ is a bialgebra pairing with respect to $\Delta_H^+$, we can obtain formulas similar to (\ref{CoproductEH}):
\begin{align*}
\varphi\left( \Psi_+(R_{p,n}^*), \prod_{u=1}^k\Psi_-(F_{p_u,n_u}) \right)&= \prod_{u=1}^k\varphi\left( \Psi_+(R^*_{p,n_u}),\Psi_-(F_{p_u,n_u}) \right)\\
\varphi\left( \Psi_+(R_{p,n}), \prod_{u=1}^k\Psi_-(F_{p_u,n_u}) \right)&= \prod_{u=1}^k\varphi\left( \Psi_+(R_{p,n_u}),\Psi_-(F_{p_u,n_u}) \right)
\end{align*}
We deduce that pairing with the coefficients of $R_p(z)$ and $R_p^*(z)$ is determined by the pairings on the generators $\{\Psi_-(F_{p',n})\}$ of $\varpi^{-1}(\ddot{U}^0_-)$ (cf. Corollary \ref{LoperatorCalc}). 
Since the restriction of $\varphi(-,-)$ to $\varpi^{-1}(\ddot{U}^0_+)\times\varpi^{-1}(\ddot{U}^0_-)$ is nondegenerate, we can present $R_p(z)$ and $R^*_p(z)$ in terms of bosons by finding group-like currents that have the same pairings with $\{\Psi_-(F_{p',n})\}$.

\begin{lem}\label{FuncCalc}
For $1\le i\le \ell$, we have
\begin{align}
\rho_{L_p^{(n)}}\left( F_{p+i,n} \right)&=\ddd^{-\frac{\ell n}{2}}q^{n(p+i)}\prod_{r=1}^n\frac{\qqq^{-2}-q^{(r-1)\ell}}{1-q^{r\ell}}\label{RhoF}\\
\rho^*_{L_p^{(n)}}\left( F_{p+i,n} \right)&=(-1)^{n\ell}\qqq^{n\ell}\ddd^{-\frac{\ell n}{2}}t^{-n(p+i)}\prod_{r=1}^n\frac{1-\qqq^{-2}t^{-(r-1)\ell}}{1-t^{-r\ell}}\label{Rho*F}
\end{align}
\end{lem}

\begin{proof}
Recall that in the proof of Corollary \ref{FuncFilt}, we showed that in the course of evaluating $\rho_{L_p^{(n)}}(F)$ and $\rho_{L_p^{(n)}}^*(F)$, the functions $\phi_{L_{p}^{(n)}}(F)/N_{L_p^{(n)}}$ and $\phi_{L_{p}^{(n)}}^*(F)/N^*_{L_p^{(n)}}$, which \textit{a priori} are functions in the single variable $y_1$, actually lie in $\FF$.
Setting $y_1=1$, we are then tasked with cleaning up
\begin{align*}
\rho_{L_p^{(n)}}\left( F_{p+i,n} \right)&=\frac{\displaystyle (-1)^{\frac{\ell n(n+1)}{2}+n}(\qqq^2-1)^{n\ell }q^{n(-\ell+p+i)}\prod_{a=p+1}^{p+n\ell}q^{-a}}{\displaystyle\qqq^n\ddd^{\frac{\ell n^2}{2}}\prod_{p<a<b\le p+n\ell}\omega_{\bar{b},\bar{a}}\left( q^{-b},q^{-a} \right)}\\
\rho^*_{L_p^{(n)}}\left( F_{p+i,n} \right)&=\frac{\displaystyle (-1)^{\frac{\ell n(n+1)}{2}+n}(\qqq^2-1)^{n\ell }t^{n(\ell-p-i)}\prod_{a=p+1}^{p+n\ell}t^{a}}{\displaystyle\qqq^n\ddd^{\frac{\ell n^2}{2}}\prod_{p<a<b\le p+n\ell}\omega_{\bar{a},\bar{b}}\left( t^{a},t^{b} \right)}
\end{align*}

Let us first consider the computation for $\rho_{L_p^{(n)}}$.
The indexed products can be grouped in the following ways to produce nice cancellation:
\begin{enumerate}
\item \textit{Adjacent $a$ and $a-1$}: The ``adjacent'' mixing term $\omega_{\bar{a},\bar{a}-1}(q^{-a},q^{-a+1})$ together with the $q^{-a}$ in the numerator cancel to leave behind
\[
\frac{q^{-a}}{ q^{-a}-tq^{-a+1}}=\frac{1}{ 1-\qqq^2 }
\]
There are $n\ell-1$ such terms and we leave behind $q^{-p-1}$ in the numerator.
\item \textit{Triples $a-1+k\ell$, $a+k\ell$, and $a+1+k\ell$ with $a$}: We have
\begin{align*}
&\omega_{\bar{a}-1,\bar{a}}(q^{-a+1-k\ell}, q^{-a})^{-1}\omega_{\bar{a},\bar{a}}(q^{-a-k\ell}, q^{-a})^{-1}\omega_{\bar{a}+1,\bar{a}}(q^{-a-1-k\ell}, q^{-a})^{-1}\\
&=\frac{ \left( q^{-a-k\ell}-\qqq^2q^{-a} \right)\left( q^{-a-k\ell}-q^{-a} \right)}
{ \left( q^{-a-1-k\ell}-\qqq\ddd^{-1}q^{-a} \right)
\left( \qqq q^{-a}-\ddd^{-1}q^{-a+1-k\ell} \right)}\\
&= \frac{\left( q^{-a-k\ell}-\qqq^2q^{-a} \right)\left( q^{-a-k\ell}-q^{-a} \right)}
{\left( q^{-a-k\ell}-\qqq^2 q^{-a} \right)\left( \ddd^{-1}q^{-a}-\ddd^{-1}q^{-a-k\ell} \right)}=-\ddd
\end{align*}
There are $\frac{\ell n(n-1)}{2}-n+1$ such terms.
\item \textit{The ends $p-1+n\ell$ and $p+n\ell$ with $p+k\ell$ for $1\le k\le n-1$}: Here, we have
\begin{align*}
&\omega_{p-1,p}\left( q^{-p+1-n\ell},q^{-p-k\ell} \right)^{-1}
\omega_{p,p}\left( q^{-p-n\ell},q^{-p-k\ell} \right)^{-1}\\
&= \frac{\left( q^{-p-n\ell}-\qqq^2 q^{-p-k\ell} \right)\left( q^{-p-n\ell}-q^{-p-k\ell} \right)}
{\left( \qqq q^{-p-k\ell}-\ddd^{-1}q^{-p+1-n\ell} \right)}\\
&= -\qqq^{-1}\frac{\left( q^{-p-n\ell}-\qqq^2q^{-p-k\ell} \right)\left( q^{-p-k\ell}-q^{-p-n\ell} \right)}{q^{-p-k\ell}-q^{-p-n\ell}}\\
&=-\qqq^{-1}q^{-p-n\ell}\left(1-\qqq^2q^{(n-k)\ell}  \right)
\end{align*}
There are $n-1$ of these terms.
\item \textit{The end $p+n\ell$ with $p+1+k\ell$ for $0\le k\le n-1$}: This is just a lone mixing term, which we write as
\[
\frac{1}{ \qqq q^{-p-1-k\ell}-\ddd^{-1}q^{-p-n\ell} }=\frac{-\ddd}{q^{-p-n\ell}\left(1- q^{(n-k)\ell} \right)}
\]
Now there are $n$ of these terms.
\end{enumerate}
Indexing the factors from (3) and (4) in terms of $r=n-k$, these cancellations give us
\begin{align*}
\rho_{L_p^{(n)}}(F_{p+i,n})&= \frac{\displaystyle (-1)^{\frac{\ell n(n+1)}{2}+n}\left( \qqq^2-1 \right)^{n\ell}q^{n\left( -\ell+p+i \right)}(-\ddd)^{\frac{\ell n(n-1)}{2}-n+1+n}q^{-p-1}\prod_{r=1}^{n-1}(-\qqq)^{-1}q^{-p-n\ell}\left( 1-\qqq^2q^{r\ell} \right)}
{\displaystyle\qqq^n\ddd ^{\frac{\ell n^2}{2}}\left( 1-\qqq^2 \right)^{n\ell-1}\prod_{r=1}^nq^{-p-n\ell}\left( 1-q^{r\ell} \right)}\\
&= \frac{\displaystyle(-1)^{ n\ell+n\ell} q^{n\left( p+i \right)-1}\qqq^{-n}q\prod_{ r=0}^{ n-1}\left( 1-\qqq^2q^{ r\ell } \right)}
{\displaystyle\qqq^n\ddd^{\frac{\ell n}{2}}\prod_{r=1}^n\left( 1-q^{r\ell} \right)}
=\frac{\displaystyle q^{n\left( p+i \right)}\prod_{r=1}^{n}\left( 1-\qqq^2q^{(r-1)\ell} \right)}
{\displaystyle \ddd^{\frac{\ell n}{2}} \qqq^{2n}\prod_{r=1}^n\left( 1-q^{r\ell} \right)}
\end{align*}
From here, the final formula (\ref{RhoF}) easily follows.

The derivation for $\rho_{L_p^{(n)}}^*$ is entirely analogous, although we present its details as it still requires care.
Our groupings are similar to as before:
\begin{enumerate}
\item \textit{Adjacent $a$ and $a+1$}: We cancel the mixing term $\omega_{\bar{a},\bar{a}+1}\left( t^{a},t^{a+1} \right)$ with the $t^{a+1}$ in the numerator, giving us
\[
\frac{t^{a+1}}{\qqq t^{a+1}-\ddd^{-1}t^{a}}=\frac{1}{\qqq-\qqq^{-1}}
\]
Now we leave behind a $t^{p+1}$ in the numerator.
\item \textit{Triples $a-1+k\ell$, $a+k\ell$, and $a+1+k\ell$ with $a$}: As before, we have
\begin{align*}
&\omega_{\bar{a},\bar{a}-1}(t^{a}, t^{a-1+k\ell})^{-1}\omega_{\bar{a},\bar{a}}(t^{a}, t^{a+k\ell})^{-1}\omega_{\bar{a},\bar{a}+1}(t^{a}, t^{a+1+k\ell})^{-1}\\
&=\frac{ \left( t^{a}-\qqq^2t^{a+k\ell} \right)\left( t^{a}-t^{a+k\ell} \right)}
{ \left( \qqq t^{a+1+k\ell}-\ddd^{-1}t^{a} \right)\left( t^{a}-\qqq\ddd^{-1}t^{a-1+k\ell} \right)}\\
&= -\ddd\frac{\left( t^{a}-\qqq^2t^{a-k\ell} \right)\left( t^{a}-t^{a-k\ell} \right)}
{\left( t^a-\qqq^2 t^{a-k\ell} \right)\left( t^a-t^{a-k\ell} \right)}=-\ddd
\end{align*}
\item \textit{The ends $p-1+n\ell$ and $p+n\ell$ with $p+k\ell$ for $1\le k\le n-1$}: Now we have
\begin{align*}
&\omega_{p,p-1}\left( t^{p+k\ell},t^{p-1+n\ell} \right)^{-1}\omega_{p,p}\left( t^{p+k\ell},t^{p+n\ell} \right)^{-1}\\
&= \frac{\left( t^{p+k\ell}-\qqq^2 t^{p+n\ell} \right)\left( t^{p+k\ell}-t^{p+n\ell} \right)}
{\left( t^{p+k\ell}-\qqq\ddd^{-1}t^{p-1+n\ell} \right)}\\
&= \frac{\left( t^{p+k\ell}-\qqq^2t^{p+n\ell} \right)\left( t^{p+k\ell}-t^{p+n\ell} \right)}
{\left( t^{p+k\ell}-t^{p+n\ell} \right)}\\
&= t^{p+n\ell}\left( t^{(k-n)\ell}-\qqq^2 \right)
\end{align*}
\item \textit{The end $p+n\ell$ with $p+1+k\ell$ for $0\le k\le n-1$}: We write the lone mixing term as
\[
\frac{1}{ t^{p+1+k\ell}-\qqq\ddd^{-1}t^{p+n\ell}}=\frac{1}{t^{p+1+n\ell}\left(  t^{(k-n)\ell}-1 \right)}
\]
\end{enumerate}
Setting $r=n-k$ for the terms in (3) and (4), we now have 
\begin{align*}
\rho_{L_p^{(n)}}^*(F_{p+i,n})
&=\frac{\displaystyle (-1)^{\frac{\ell n(n+1)}{2}+n} \left( \qqq^2-1 \right)^{n\ell}t^{n(\ell-p-i)} \left( -\ddd \right)^{\frac{\ell n(n-1)}{2}-n+1} t^{p+1}\prod_{r=1}^{n-1}t^{p+n\ell}\left( t^{-r\ell}-\qqq^2 \right)}
{\displaystyle \qqq^n \ddd^{\frac{\ell n^2}{2}} \left( \qqq -\qqq^{-1}\right)^{n\ell-1}\prod_{r=1}^nt^{p+1+n\ell}\left( t^{-r\ell}-1 \right)}\\
&= \frac{\displaystyle (-1)^{n\ell} \qqq^{n\ell-1}\ddd t^{-n(p+i)+1}\prod_{ r=0}^{n-1}\left( t^{-r\ell}-\qqq^2 \right)}
{\displaystyle\qqq^{n}\ddd^{\frac{\ell n}{2}+n}t^n\prod_{r=1}^n\left( t^{-r\ell}-1 \right)}
=\frac{\displaystyle(-1)^{n\ell} \qqq^{n\ell}t^{-n(p+i)}\prod_{r=1}^{n}\left(t^{-(r-1)\ell}-\qqq^2\right)}
{\displaystyle \qqq^{2n}\ddd^{\frac{\ell n}{2}}\prod_{r=1}^n\left( t^{-r\ell}-1 \right)}
\end{align*}
from which (\ref{Rho*F}) follows.
\end{proof}

\begin{cor}\label{DualCurrentsCalc}
In terms of bosons,
\begin{align*}
R_p(z)&=\exp\left((\qqq-\qqq^{-1})\sum_{n>0}\frac{\sum_{i=1}^{\ell}q^{n(p+i)}\varpi^{-1}(b_{p+i,n})}{(1-q^{n\ell})}\cdot\frac{z^{-n}}{n}\right)\\
R_p^*(z)&=\exp\left(-(\qqq-\qqq^{-1})\sum_{n>0}(-\qqq)^{n\ell}\frac{\sum_{i=1}^{\ell}t^{-n(p+i)}\varpi^{-1}(b_{p+i,n})}{1-t^{-n\ell}}\cdot\frac{z^{-n}}{n}\right)
\end{align*}

\end{cor}

\begin{proof}
Recall that $\varpi^{-1}(b_{i,n})$ is dual to $\varpi^{-1}(b_{i,-n}^\perp)$ with respect to the pairing $\varphi(-,-)$.
Our expressions for the currents follow from combining Lemma \ref{FuncCalc}, Theorem \ref{LOpBoson}, and a generalized partition identity (cf. \cite{Mac} Example I.2.5):
\[\sum_{n\ge 0}\prod_{r=1}^n\frac{a-bc^{r-1}}{1-c^r}z^n=\prod_{n=0}^\infty\frac{1-bc^nz}{1-ac^nz}=\exp\left(\sum_{n>0}\frac{a^n-b^n}{1-c^n}\cdot\frac{z^n}{n}\right)\qedhere\]
\end{proof}

\subsubsection{Shuffle elements for $\tilde{h}_n(p)$ and $\tilde{e}_n(p)$}
For the main result of this paper, we only need to know $\Psi_-^{-1}(\tilde{h}_n(p))$ and $\Psi_-^{-1}(\tilde{e}_n(p))$ up to scalar.
However, we compute the precise scalar for future use.

\begin{thm}\label{FinalShuffle}
We have 
\begin{align*}
\tilde{e}_n(p)&=c_{n}\Psi_-(E_{p,n})\\
\tilde{h}_n(p)&=(qt)^{-n}c_{n}\Psi_-(H_{p,n})
\end{align*}
where
\begin{align*}
c_{n}&= \frac{\left( 1-\qqq^2 \right)^{n\ell}}{\prod_{r=1}^n\left( 1-(qt)^{-r} \right)}
\end{align*}
\end{thm}

\begin{proof}
First, we show that $\tilde{e}_n(p)$ lies in $\Psi_-(\Sss(0)_{L_p^{1^n}})$ and $\tilde{h}_n(p)$ lies in $\Psi_-(\Sss(0)_{L_p^{1^n}}^*)$. 
By Corollary \ref{DualCurrentsCalc}, (\ref{ModHForm}), and (\ref{ModEForm}), we have
\[\begin{array}{rcll}
\varphi\left(R_{p}(w),\sum_n \tilde{e}_n(p+i)(-z)^n\right)&=&\left\{\begin{array}{l}
1\\ 
1-\qqq q^{p}\dfrac{z}{w} 
\end{array}\right.&
\begin{array}{ll}
\vspace{2.5pt}\hbox{if }i\not=0\\
\hbox{if }i=0
\end{array}\\
\varphi(R_{p}^*(w),\sum_n \tilde{h}_n(p+i)z^n)&=&\left\{\begin{array}{ll}
1 \\
1-(-1)^\ell\qqq^{\ell-1} t^{-p}\dfrac{z}{w} 
\end{array}\right.&
\begin{array}{ll}
\vspace{2.5pt}\hbox{if }i\not=0\\
\hbox{if }i=0
\end{array}
\end{array}\]
As mentioned in \ref{CoproductStrat}, the formulas in (\ref{CoproductEH}) combined with Proposition \ref{DualProduct} imply that the necessary vanishing conditions hold.
By Proposition \ref{BottomElements}, $\Psi_-^{-1}(e_n(p))$ is an $\FF$-multiple of $E_{p,n}$ and $\Psi_-^{-1}(h_n(p))$ is an $\FF$-multiple of $H_{p,n}$.

To compute the actual scalars, it will be helpful to first note that by (\ref{CoproductEH}), we have
\begin{align}
\varphi\left( \Psi_+(R_{p,1})^n, \tilde{e}_n(p)\right) &=\qqq^{n}q^{np}\label{FuncETilde}\\
\varphi( \Psi_+(R_{p,1}^*)^n, \tilde{h}_n(p)) &=(-1)^{n\ell-n}\qqq^{n(\ell-1)}t^{-np}\label{FuncHTilde}
\end{align}
On the other hand, we can calculate the pairings $\langle R_{p,1}^n, E_{p,n}\rangle$ and $\langle \left(R_{p,1}^*\right)^n, H_{p,n}\rangle$, which by Proposition \ref{DualProduct} are equal to $\rho_{L_p^{1^n}}\left( E_{p,n} \right)$ and $\rho_{L_p^{1^n}}^*\left( H_{p,n} \right)$, respectively.
We begin by marking the following pieces of $E_{p,n}$ and $H_{p,n}$:
\begin{align*}
E_{p,n}&:=
\Sym\left( \prod_{1\le r<s\le n}\left\{\overbrace{\frac{x_{p+1,r}-q^{-1}x_{p,s}}{x_{p+1,r}-tx_{p,s}}}^{(d)}\,\overbrace{\prod_{i,j\in\ZZ/\ell\ZZ}\omega_{i,j}\left( x_{i,r},x_{j,s} \right)}^{(a)}\right\}\right.\\
&\times \left.\prod_{r=1}^n\left\{\overbrace{\left( q^{-1}\frac{x_{0,r}}{x_{p+1,r}}-\frac{x_{0,r}}{x_{p,r}} \right)}^{(b)}\,\overbrace{\prod_{i\in\ZZ/\ell\ZZ}x_{i,r}}^{(c)}  \right\} \right)\\
H_{p,n}&:=
\Sym\left( \prod_{1\le r<s\le n}\left\{\overbrace{\frac{t^{-1}x_{p+1,s}-x_{p,r}}{qx_{p+1,s}-x_{p,r}}}^{(d)}\,\overbrace{\prod_{i,j\in\ZZ/\ell\ZZ}\omega_{i,j}\left( x_{i,r},x_{j,s} \right)}^{(a)}\right\}\right.\\
&\times \left.\prod_{r=1}^n\left\{\overbrace{\left( t\frac{x_{0,r}}{x_{p+1,r}}-\frac{x_{0,r}}{x_{p,r}} \right)}^{(b)}\,\overbrace{\prod_{i\in\ZZ/\ell\ZZ}x_{i,r}}^{(c)}  \right\} \right)
\end{align*}
For $E_{p,n}$, we will evaluate
\[
x_{p+i,r}\mapsto q^{-p-i}y_r
\]
where $1\le i\le \ell$.
Because of the mixing terms from (a), the only permutations in the symmetrization for $E_{p,n}$ that survive upon this evaluation will be the diagonal permutations in $\Sigma_{n\delta}=\Sigma_n^{\times \ell}$.
We will abuse notation and denote this diagonal copy as $\Sigma_n$.
Let us now document how each part of $E_{p,n}$ cancels with $N_{L_p^{1^n}}$ and fares in the limit $|y_s|\ll|y_{s-1}|$:
\begin{enumerate}[(a)]
\item Depending on the summand in the symmetrization, these mixing terms may or may not totally cancel with the factors
\[
\prod_{1\le u<v\le n}\,\prod_{\substack{c\in(a_u,b_u]\\c'\in(a_v,b_v]}}\omega_{\bar{c}',\bar{c}}\left( q^{-c'}y_v,q^{-c}y_u \right)
\]
of $N_{L_p^{1^n}}$.
However, even if they do not cancel, one can easily check that they do totally cancel in the limit $|y_s|\ll |y_{s-1}|$.
Namely, one can group the mixing terms into ``adjacent triples'' in which the limits for (a) and $N_{L_p^{1^n}}$ are both always $-\ddd^{-1}$, canceling out and leaving nothing behind.
\item These factors will always evaluate to $q^{p-\ell}\left( 1-q^\ell \right)$, regardless of the permutation.
\item This will cancel with the factors
\[
\prod_{u=1}^n\,\prod_{\substack{c,c'\in(a_u,b_u]\\c<c'}}\omega_{\bar{c}',\bar{c}}\left( q^{-c'}y_u,q^{-c}y_u \right)
\]
of $N_{L_p^{1^n}}$.
The cancellation is similar to the $q$-version of enumerated items (1) and (4) from the proof of Lemma \ref{FuncCalc}.
In total, the contribution is
\begin{align*}
\frac{(-\ddd)^n q^{n(-p-1)}}{\left( 1-\qqq^2 \right)^{n(\ell-1)}q^{n(-p-\ell)}\left( 1-q^\ell \right)^{n}}=
\frac{(-\ddd)^n q^{n(\ell-1)}}{\left( 1-\qqq^2 \right)^{n(\ell-1)}\left( 1-q^\ell \right)^{n}}
\end{align*}
\item First consider the summand for the identity permutation.
In this case, these factors will evaluate to
\[
\prod_{1\le r<s\le n}\frac{q^{-p-1}y_r-q^{-p-\ell-1}y_s}{q^{-p-1}y_r-tq^{-p-\ell}y_s}
\]
Taking the limit $|y_s|\ll |y_r|$ and then taking the constant term turns each term in the product into $1$.
Now, for $w\in\Sigma_n$, the factor for $r<s$ becomes $(qt)^{-1}$ instead when $w(r)>w(s)$.
Since the contribution of the other parts of $E_{p,n}$ in the computation of $\rho_{L_p^{1^n}}(E_{p,n})$ do not depend on $w$, the contribution of part (4) is
\[
\sum_{w\in\Sigma_n}(qt)^{-\ell(w)}=\prod_{r=1}^n\frac{1-(qt)^{-r}}{1-(qt)^{-1}}=\frac{(-qt)^{n}\prod_{r=1}^n\left( 1-(qt)^{-r} \right)}{\left( 1-\qqq^2 \right)^{n}}
\]
where $\ell(w)$ is the length of $w$.
\end{enumerate}
Altogether, we have
\begin{equation}
\rho_{L_p^{1^n}}(E_{p,n})=\frac{q^{np}\qqq^n\prod_{r=1}^n\left( 1-(qt)^{-r} \right)}{\left( 1-\qqq^2 \right)^{n\ell}}
\label{FuncEpn}
\end{equation}
We obtain $c_{n}$ by dividing (\ref{FuncETilde}) by (\ref{FuncEpn}).

The computation for $H_{p,n}$ is similar, although we note some details.
We now evaluate the variables as
\[
x_{p+i,r}\mapsto t^{p+i}y_r
\]
for $1\le i\le\ell$.
The contributions are:
\begin{enumerate}[(a)]
\item As before, these factors contribute nothing.
\item Each of these factors contributes $t^{-p}\left( t^\ell-1 \right)$.
\item Borrowing the $t$-versions of enumerated items (1) and (4) from the proof of Lemma \ref{FuncCalc}, we are left with
\[
\frac{t^{n(p+1)}}{\left( \qqq-\qqq^{-1} \right)^{n(\ell-1)}t^{n(p+1)}\left(1-t^\ell\right)^n}
=\frac{(-1)^n}{\left( \qqq-\qqq^{-1} \right)^{n(\ell-1)}\left( t^\ell-1 \right)^n}
\]
\item The contribution is the same is in the case of $E_{p,n}$.
\end{enumerate}
Putting this together, we have
\begin{equation}
\rho_{L_p^{1^n}}^*(H_{p,n})=\frac{t^{-np}(-qt)^n\prod_{r=1}^n\left(1- (qt)^{-r} \right)}{\qqq^{-n(\ell-1)}\left( \qqq^{2}-1 \right)^{n\ell}}
\label{FuncHpn}
\end{equation}
We obtain $(qt)^{-n}c_{n}$ by dividing (\ref{FuncHTilde}) by (\ref{FuncHpn}).
\end{proof}

\section{Combinatorics}\label{Combi}

\subsection{Setup}\label{CombInt}
Let us remind the reader of the definitions of $\hat{h}_n(p)$, $\hat{h}_\lambda$, $\hat{e}_n(p)$, and $\hat{e}_\lambda$ given by (\ref{HatBoson}) and the definitions of $\tilde{h}_n(p)$ and $\tilde{e}_n(p)$ given by (\ref{EHAbbrev}).
We define
\begin{align}
\label{HTilde}
\tilde{h}_\lambda&:=\tau_{0,\upsilon}^-(\tilde{h}_{\quot(\lambda)})|\core(\lambda)\rangle\\
\label{ETilde}
\tilde{e}_\lambda&:=\tau_{0,\upsilon}^-(\tilde{e}_{\quot(\lambda)})|\core(\lambda)\rangle
\end{align}
These are the versions of $\hat{h}_\lambda$ and $\hat{e}_\lambda$ in $\mathcal{F}$.
Finally, recall as well that for two partitions $\mu$ and $\nu$, $\mu\ge_\ell\nu$ means $\mu\ge \nu$ and $\core(\mu)=\core(\nu)$. 

We wish to show that for each $\lambda$,
\begin{align}
\label{HTri}
\mathrm{span}\{ \tilde{h}_\mu\, |\, \mu\ge_\ell\lambda \}&= \mathrm{span}\left\{ |\mu\rangle\, |\, \mu \ge_\ell\lambda \right\}\\
\label{ETri}
\mathrm{span}\{ \tilde{e}_\mu\, |\, \mu\le_\ell\lambda \}&= \mathrm{span}\left\{ |\mu\rangle\, |\, \mu \le_\ell\lambda \right\}
\end{align}
By induction, it suffices to show that $\tilde{h}_\lambda$ and $\tilde{e}_\lambda$ are in their respective subspaces on the right hand side.
The intersection of the two subspaces on the left is by definition spanned by $\mathrm{T}_{-0}(H_\lambda\otimes e^{\core(\lambda)})$.
On the other hand, the intersection of the two subspaces on the right is spanned by $|\lambda\rangle$.
Thus, establishing (\ref{HTri}) and (\ref{ETri}) would yield Theorem \ref{Main}.
We note as well that it would give a new existence proof for $H_\lambda$.

Our approach is, in a sense, a qualitative study of the Pieri rules.
Combining Theorem \ref{FinalShuffle} with Proposition \ref{Adjacency}, we have that $\tau_{0,\upsilon}^-\big(\tilde{e}_{n}(p)\big)|\lambda\rangle$ is a linear combination of $|\mu\rangle$ such that $\mu\backslash\lambda$ has $n$ nodes of each color such that no $p$- and $(p+1)$-nodes are \textit{horizontally} adjacent.
Similarly, $\tau_{0,\upsilon}^-\big(\tilde{h}_{n}(p)\big)|\lambda\rangle$ consists of $|\mu\rangle$ where $\mu\backslash\lambda$ has $n$ nodes of each color such that no $p$- and $(p+1)$-nodes \textit{vertically} adjacent.
Contrary to the case of $\ell=1$, these adjacency conditions on their own are quite weak: the $\mu$ that appear are not necessarily less or more dominant than $\lambda$ with the appropriate row or column appended to $\quot(\lambda)$.
However, the conditions acquire force when combined with two considerations:
\begin{enumerate}
\item In (\ref{HTilde}) and (\ref{ETilde}), each $\tilde{h}_{n}(p)$ and $\tilde{e}_n(p)$ adds boxes to $|\core(\lambda)\rangle$ in a certain way to obtain the set of $|\mu\rangle$ that appear in the expressions.
Choosing an ordering of the rows or columns of $\quot(\lambda)$, we obtain a tableau on $\mu$ with special properties.
Moreover, the commutativity of $\tilde{h}_n(p)$ and $\tilde{e}_n(p)$ imply that such $\mu$ must possess such a tableau for \textit{every} ordering of the columns/rows of $\quot(\lambda)$.
Although it seems opaque ate first, this turns out to be a very strong condition.
\item On the other hand, \textit{specific} orderings in which we apply the $\tilde{h}_n(p)$ and $\tilde{e}_n(p)$ can provide concrete constraints for the $\mu$ that appear.
\end{enumerate}
The necessary definitions for point (1) is addressed in \ref{LTableaux} and those for point (2) are addressed in \ref{Orders}.
Finally, we put them all together in \ref{ProofMain} to prove Theorem \ref{Main}.
Our definitions and arguments will only be shown in their entirety for proving (\ref{ETri}); those for (\ref{HTri}) can be obtained by transposing partitions.

\subsection{On $\ell$-tableaux}\label{LTableaux}
Viewing a partition $\mu$ as its set of nodes, recall that a \textit{tableau} on $\mu$ is a function $\mu\rightarrow\mathbb{N}$ that is weakly increasing along each row as one goes right and along each column as one goes up.

\begin{defn}
For $\ell>0$, an \textit{$\ell$-tableau} on $\mu$ is a tableau $T:\mu\rightarrow\ZZ_{\ge 0}$ such that:
\begin{itemize}
\item $T^{-1}(0)=\core(\mu)$;
\item for each label $k\ge 1$, $T^{-1}(k)$ contains an equal number of nodes of each color.
\end{itemize}
For each label $k$, we define the \textit{truncation} to be
\[
\mu(T,k):=\bigcup_{j=0}^k T^{-1}(j)
\]
Note that $\mu(T,k)\subset\mu$ is a partition with the same $\ell$-core.
\end{defn}

\subsubsection{Strictness and content}
Now, let us include the datum of another partition $\lambda$.

\begin{defn}
Let $\mathcal{O}_c$ and $\mathcal{O}_r$ be orderings of the columns and rows of $\quot(\lambda)$, respectively.
For $\mu$ with $\core(\mu)=\core(\lambda)$, we define the following:
\begin{enumerate}
\item \textit{Row-strict $\ell$-tableau of content $(\lambda,\mathcal{O}_c)$}:
This is an $\ell$-tableau $T_c$ on $\mu$ whose max value is the number of columns in $\quot(\lambda)$.
If the $k$th column in the ordering $\mathcal{O}_c$ has length $n$ and color $p$, then:
\begin{itemize}
\item $T^{-1}_c(k)$ has exactly $n$ nodes of each color;
\item no $p$- and $(p+1)$-nodes in $T_c^{-1}(k)$ are \textit{horizontally} adjacent.
\end{itemize}
\item \textit{Column-strict $\ell$-tableau of content $(\lambda,\mathcal{O}_r)$}:
This is an $\ell$-tableau $T_r$ on $\mu$ whose max value is the number of rows in $\quot(\lambda)$.
If the $k$th row in the ordering $\mathcal{O}_r$ has length $n$ and color $p$, then:
\begin{itemize}
\item $T_r^{-1}(k)$ has exactly $n$ nodes of each color labeled by $k$;
\item no $p$- and $(p+1)$-nodes in $T_r^{-1}(k)$ are \textit{vertically} adjacent.
\end{itemize}
\end{enumerate}
\end{defn}

\noindent In the case $\ell=1$ and the usual (right-left and bottom-top) orderings on columns and rows, we obtain the row-strict and column-strict tableaux.

Given an order $\mathcal{O}_c$ on the columns of $\quot(\lambda)$, we denote by $\lambda(\mathcal{O}_c,k)$ the partition with the same $\ell$-core but with $\ell$-quotient obtained by removing all columns $\quot(\lambda)$ that are $>k$ in the order $\mathcal{O}_c$.
For an order $\mathcal{O}_r$, we similarly define $\lambda(\mathcal{O}_r,k)$.

\begin{defn}\label{TabloidDef}
Let $\lambda$ and $\mu$ be partitions with $\core(\lambda)=\core(\mu)$ and $|\lambda|=|\mu|$. 
Suppose $\quot(\lambda)$ has $M$ rows and $N$ columns.
We give the following recursive definitions:
\begin{enumerate}
\item We call $\mu$ \textit{strongly column-strict $\lambda$-tabloidizable} if for every ordering $\mathcal{O}_r$ of the rows of $\quot(\lambda)$, $\mu$ has a column-strict $\ell$-tableau $T_{\lambda,\mathcal{O}_r}$ of content $(\lambda,\mathcal{O}_r)$ such that for any $k<M$, the truncation $\mu(T_{\lambda,\mathcal{O}_r},k)$ is strongly column-strict $\lambda(\mathcal{O}_r,k)$-tabloidizable.
\item We call $\mu$ \textit{strongly row-strict $\lambda$-tabloidizable} if for every ordering $\mathcal{O}_c$ of the columns of $\quot(\lambda)$, $\mu$ has a row-strict $\ell$-tableau $T_{\lambda,\mathcal{O}_c}$ of content $(\lambda,\mathcal{O}_c)$ such that for any $k<N$, the truncation $\mu(T_{\lambda,\mathcal{O}_c},k)$ is strongly row-strict $\lambda(\mathcal{O}_c,k)$-tabloidizable.
\end{enumerate}
The base case $\lambda=\mu=\core(\lambda)$ vacuously satisfies these conditions.
\end{defn}

We summarize part of the discussion from \ref{CombInt} with the following:

\begin{prop}\label{TabProp}
Writing $\tilde{h}_\lambda$ in terms of the Fock basis, $|\mu\rangle$ appears with nonzero coefficient only if $\mu$ is strongly column-strict $\lambda$-tabloidizable.
Likewise, $|\mu\rangle$ appears with nonzero coefficient in the expansion for $\tilde{e}_\lambda$ only if $\mu$ is strongly row-strict $\lambda$-tabloidizable.
\end{prop}

\begin{proof}
For each part, the existence of some column/row-strict $\ell$-tableau for each order is explained in \ref{CombInt}.
The statements about truncations come from considering products of subsets of the $\tilde{h}_n(p)$'s and $\tilde{e}_n(p)$'s appearing in (\ref{HTilde}) and (\ref{ETilde}).
\end{proof}

\subsubsection{Strong $\ell$-dominance}
The conditions given in Definition \ref{TabloidDef} and Proposition \ref{TabProp} for a Fock basis vector $|\mu\rangle$ to appear in the expansion of $\tilde{e}_\lambda$ or $\tilde{h}_\lambda$ do not at first glance have anything to do with dominance order.
Here, we will formulate their relationship.
%

\begin{defn}
For $\lambda=(\lambda_1,\lambda_2,\ldots)$ and $p\in\ZZ/\ell\ZZ$, let
\[
\lambda_i(p):=\#\left\{ \square\in\lambda\,\bigg|\, \square\hbox{ is in row $i$ and has color $p$}\right\}
\]
Given another partition $\mu$ with $|\mu|=|\lambda|$ and $\core(\mu)=\core(\lambda)$, we say that $\lambda$ is \textit{strongly $\ell$-dominant} over $\mu$ if for all $k\in\ZZ_{\ge 1}$ and $p\in\ZZ/\ell\ZZ$,
\[
\lambda_1(p)+\cdots+\lambda_k(p)\ge \mu_1(p)+\cdots\mu_k(p)
\]
This relation is denoted by $\lambda\succeq_\ell\mu$.
\end{defn}

Note that since $\lambda_i=\lambda_i(0)+\cdots+\lambda_i(\ell-1)$, strong $\ell$-dominance implies dominance.
One way to view dominance order is as follows: $\mu\le \lambda$ if and only if it is possible to define a bijection $\mu\backslash\lambda\rightarrow\lambda\backslash\mu$ that sends each node to one that is strictly south of it.
On the other hand, strong $\ell$-dominance imposes that there exists such a bijection that respects color.
We enhance Proposition \ref{TabProp} to the following:
\begin{lem}\label{StrongLem}
We have the following:
\begin{enumerate}
\item If $\mu$ is strongly column-strict $\lambda$-tabloidizable, then $\nu\succeq_\ell\lambda$.
\item If $\mu$ is strongly row-strict $\lambda$-tabloidizable, then $\mu\preceq_\ell\lambda$.
\end{enumerate}
\end{lem}

We will prove Lemma \ref{StrongLem} in \ref{ProofMain} below.
Theorem \ref{Main} follows:

\begin{cor}\label{StrongCor}
$H_\lambda$ exists and moreover, $\mathrm{T}_{-0}\left(H_\lambda\otimes e^{\core(\lambda)}\right)$ is collinear to $|\lambda\rangle$.
\end{cor}

\begin{proof}
Lemma \ref{StrongLem} implies that
\begin{align*}
\tilde{h}_\lambda&\in\mathrm{span}\left\{ |\mu\rangle\, \middle|\, \mu\ge_\ell\lambda \right\}\\
\tilde{e}_\lambda&\in\mathrm{span}\left\{ |\mu\rangle\, \middle|\, \mu\le_\ell\lambda \right\}
\end{align*}
from which (\ref{HTri}) and (\ref{ETri}) follows.
We then have
\begin{equation*}
\mathrm{span}\{ \tilde{h}_\mu\, |\, \mu\ge_\ell\lambda \}\cap\mathrm{span}\{ \tilde{e}_\mu\, |\, \mu\le_\ell\lambda \}
= \mathrm{span}\left\{ |\mu\rangle\, \middle|\, \mu\ge_\ell\lambda \right\}\cap\mathrm{span}\left\{ |\mu\rangle\, \middle|\, \mu\le_\ell\lambda \right\}= \FF|\lambda\rangle
\end{equation*}
In particular, this independently shows that the left-hand-side is nonzero.
It is necessarily spanned by $\mathrm{T}_{-0}\left( H_\lambda\otimes e^{\core(\lambda)} \right)$.
\end{proof}

\subsubsection{Examples}
We end this subsection with some examples to illustrate both the necessity and the strength of the conditions imposed by Definition \ref{TabloidDef}.
Throughout, $\ell=5$ and $\lambda=(12, 9, 9, 2)$:
\begin{equation*}
\begin{tikzpicture}[scale=.5]
\draw (-1,2) node {$\lambda=$};;
\draw (0,0)--(12,0)--(12,1)--(9,1)--(9,3)--(2,3)--(2,4)--(0,4)--(0,0);;
\draw (1,0)--(1,4);;
\draw (2,0)--(2,3);;
\draw (3,0)--(3,3);;
\draw (4,0)--(4,3);;
\draw (5,0)--(5,3);;
\draw (6,0)--(6,3);;
\draw (7,0)--(7,3);;
\draw (8,0)--(8,3);;
\draw (9,0)--(9,1);;
\draw (10,0)--(10,1);;
\draw (11,0)--(11,1);;
\draw (0,1)--(9,1);;
\draw (0,2)--(9,2);;
\draw (0,3)--(2,3);;
\end{tikzpicture}
\end{equation*}
\noindent Note that $\core(\lambda)=(2)$ and $\quot(\lambda)=\left( \varnothing, \varnothing, (1^2), (2^2), \varnothing \right)$, where in the latter, the coordinates correspond to the colors $(0, 1, 2, 3, 4)$.
We will only consider row-strict $\ell$-tableau.
To describe an order on the columns of $\quot(\lambda)$, we will use a parenthesized superscript to indicate color (although it is unnecessary in this example), e.g. $\mathcal{O}_c=(2^{(3)}, 1^{(2)}, 1^{(2)}, 2^{(3)})$.

\begin{enumerate}
\item $\mu=(12,10, 5, 2, 1,1,1)$:
Note that $\mu$ is incomparable to $\lambda$ with respect to dominance order:

\begin{equation*}
\begin{tikzpicture}[scale=.5]
\draw (-1,2) node {$\mu=$};;
\draw (0,0)--(12,0)--(12,1)--(10,1)--(10,2)--(5,2)--(5,3)--(2,3)--(2,4)--(1,4)--(1,7)--(0,7)--(0,0);;
\draw (1,0)--(1,4);;
\draw (2,0)--(2,3);;
\draw (3,0)--(3,3);;
\draw (4,0)--(4,3);;
\draw (5,0)--(5,3);;
\draw (6,0)--(6,2);;
\draw (7,0)--(7,2);;
\draw (8,0)--(8,2);;
\draw (9,0)--(9,2);;
\draw (10,0)--(10,1);;
\draw (11,0)--(11,1);;
\draw (0,1)--(10,1);;
\draw (0,2)--(5,2);;
\draw (0,3)--(2,3);;
\draw (0,4)--(1,4);;
\draw (0,5)--(1,5);;
\draw (0,6)--(1,6);;
\end{tikzpicture}
\end{equation*}

\noindent Even though $\mu\not\le_\ell\lambda$, it (miraculously) possesses an $\ell$-tableau of content $\left( \lambda, \left( 2^{(3)},2^{(3)}, 1^{(2)}, 1^{(2)} \right) \right)$:
\begin{equation*}
\begin{tikzpicture}[scale=.5]
\draw (0,0)--(12,0)--(12,1)--(10,1)--(10,2)--(5,2)--(5,3)--(2,3)--(2,4)--(1,4)--(1,7)--(0,7)--(0,0);;
\draw (1,0)--(1,4);;
\draw (2,0)--(2,3);;
\draw (3,0)--(3,3);;
\draw (4,0)--(4,3);;
\draw (5,0)--(5,3);;
\draw (6,0)--(6,2);;
\draw (7,0)--(7,2);;
\draw (8,0)--(8,2);;
\draw (9,0)--(9,2);;
\draw (10,0)--(10,1);;
\draw (11,0)--(11,1);;
\draw (0,1)--(10,1);;
\draw (0,2)--(5,2);;
\draw (0,3)--(2,3);;
\draw (0,4)--(1,4);;
\draw (0,5)--(1,5);;
\draw (0,6)--(1,6);;
\draw (0.5,0.5) node {$0$};;
\draw (1.5,0.5) node {$0$};;
\draw (2.5,0.5) node {$1$};;
\draw (3.5,0.5) node {$1$};;
\draw (4.5,0.5) node {$1$};;
\draw (5.5,0.5) node {$1$};;
\draw (6.5,0.5) node {$1$};;
\draw (7.5,0.5) node {$2$};;
\draw (8.5,0.5) node {$2$};;
\draw (9.5,0.5) node {$2$};;
\draw (10.5,0.5) node {$2$};;
\draw (11.5,0.5) node {$2$};;
\draw (0.5,1.5) node {$1$};;
\draw (1.5,1.5) node {$1$};;
\draw (2.5,1.5) node {$1$};;
\draw (3.5,1.5) node {$2$};;
\draw (4.5,1.5) node {$2$};;
\draw (5.5,1.5) node {$2$};;
\draw (6.5,1.5) node {$2$};;
\draw (7.5,1.5) node {$2$};;
\draw (8.5,1.5) node {$3$};;
\draw (9.5,1.5) node {$4$};;
\draw (0.5,2.5) node {$1$};;
\draw (1.5,2.5) node {$3$};;
\draw (2.5,2.5) node {$4$};;
\draw (3.5,2.5) node {$4$};;
\draw (4.5,2.5) node {$4$};;
\draw (0.5,3.5) node {$1$};;
\draw (1.5,3.5) node {$3$};;
\draw (0.5,4.5) node {$3$};;
\draw (0.5,5.5) node {$3$};;
\draw (0.5,6.5) node {$4$};;
\end{tikzpicture}
\end{equation*}
\noindent On the other hand, it does not possess an $\ell$-tableau of content $\left( \lambda, \left( 2^{(3)}, 1^{(2)}, 2^{(3)}, 1^{(2)} \right) \right)$.
We illustrate with an attempt below:
\begin{equation*}
\begin{tikzpicture}[scale=.5]
\draw[fill=lightgray] (8,1)--(8,2)--(10,2)--(10,1)--(8,1);;
\draw (0,0)--(12,0)--(12,1)--(10,1)--(10,2)--(5,2)--(5,3)--(2,3)--(2,4)--(1,4)--(1,7)--(0,7)--(0,0);;
\draw (1,0)--(1,4);;
\draw (2,0)--(2,3);;
\draw (3,0)--(3,3);;
\draw (4,0)--(4,3);;
\draw (5,0)--(5,3);;
\draw (6,0)--(6,2);;
\draw (7,0)--(7,2);;
\draw (8,0)--(8,2);;
\draw (9,0)--(9,2);;
\draw (10,0)--(10,1);;
\draw (11,0)--(11,1);;
\draw (0,1)--(10,1);;
\draw (0,2)--(5,2);;
\draw (0,3)--(2,3);;
\draw (0,4)--(1,4);;
\draw (0,5)--(1,5);;
\draw (0,6)--(1,6);;
\draw (0.5,0.5) node {$0$};;
\draw (1.5,0.5) node {$0$};;
\draw (2.5,0.5) node {$1$};;
\draw (3.5,0.5) node {$1$};;
\draw (4.5,0.5) node {$1$};;
\draw (5.5,0.5) node {$1$};;
\draw (6.5,0.5) node {$1$};;
\draw (7.5,0.5) node {$2$};;
\draw (8.5,0.5) node {$3$};;
\draw (9.5,0.5) node {$3$};;
\draw (10.5,0.5) node {$3$};;
\draw (11.5,0.5) node {$3$};;
\draw (0.5,1.5) node {$1$};;
\draw (1.5,1.5) node {$1$};;
\draw (2.5,1.5) node {$1$};;
\draw (3.5,1.5) node {$3$};;
\draw (4.5,1.5) node {$3$};;
\draw (5.5,1.5) node {$3$};;
\draw (6.5,1.5) node {$3$};;
\draw (7.5,1.5) node {$3$};;
\draw (0.5,2.5) node {$1$};;
\draw (1.5,2.5) node {$2$};;
\draw (2.5,2.5) node {$2$};;
\draw (3.5,2.5) node {$3$};;
\draw (0.5,3.5) node {$1$};;
\draw (1.5,3.5) node {$2$};;
\draw (0.5,4.5) node {$2$};;
\end{tikzpicture}
\end{equation*}
\noindent The two gray squares have colors $2$ and $3$, so they would violate the adjacency condition for label $4$.
However, we leave it to the reader to see that there is no way to save the situation by filling the left gray square with any of the previous labels.

\item $\mu=(11,10,9, 1,1)$:
We have that $\mu\le_\ell\lambda$ but $\mu\not\preceq_\ell\lambda$:

\begin{equation*}
\begin{tikzpicture}[scale=.5]
\draw (-1,2) node {$\mu=$};;
\draw (0,0)--(11,0)--(11,1)--(10,1)--(10,2)--(9,2)--(9,3)--(1,3)--(1,5)--(0,5)--(0,0);;
\draw (1,0)--(1,3);;
\draw (2,0)--(2,3);;
\draw (3,0)--(3,3);;
\draw (4,0)--(4,3);;
\draw (5,0)--(5,3);;
\draw (6,0)--(6,3);;
\draw (7,0)--(7,3);;
\draw (8,0)--(8,3);;
\draw (9,0)--(9,2);;
\draw (10,0)--(10,1);;
\draw (11,0)--(11,1);;
\draw (0,1)--(10,1);;
\draw (0,2)--(9,2);;
\draw (0,3)--(1,3);;
\draw (0,4)--(1,4);;
\end{tikzpicture}
\end{equation*}

\noindent It lacks an $\ell$-tableau of content $\left( \lambda, \left( 2^{(3)}, 2^{(3)}, 1^{(2)}, 1^{(2)} \right) \right)$.
An attempt is given below:

\begin{equation*}
\begin{tikzpicture}[scale=.5]
\draw[fill=lightgray] (4,2)--(6,2)--(6,3)--(4,3)--(4,2);;
\draw[fill=lightgray] (8,1)--(10,1)--(10,2)--(8,2)--(8,1);;
\draw (0,0)--(11,0)--(11,1)--(10,1)--(10,2)--(9,2)--(9,3)--(1,3)--(1,5)--(0,5)--(0,0);;
\draw (1,0)--(1,3);;
\draw (2,0)--(2,3);;
\draw (3,0)--(3,3);;
\draw (4,0)--(4,3);;
\draw (5,0)--(5,3);;
\draw (6,0)--(6,3);;
\draw (7,0)--(7,3);;
\draw (8,0)--(8,3);;
\draw (9,0)--(9,2);;
\draw (10,0)--(10,1);;
\draw (11,0)--(11,1);;
\draw (0,1)--(10,1);;
\draw (0,2)--(9,2);;
\draw (0,3)--(1,3);;
\draw (0,4)--(1,4);;
\draw (0.5,0.5) node {$0$};;
\draw (1.5,0.5) node {$0$};;
\draw (2.5,0.5) node {$1$};;
\draw (3.5,0.5) node {$1$};;
\draw (4.5,0.5) node {$1$};;
\draw (5.5,0.5) node {$1$};;
\draw (6.5,0.5) node {$1$};;
\draw (7.5,0.5) node {$2$};;
\draw (8.5,0.5) node {$2$};;
\draw (0.5,1.5) node {$1$};;
\draw (1.5,1.5) node {$1$};;
\draw (2.5,1.5) node {$1$};;
\draw (3.5,1.5) node {$2$};;
\draw (4.5,1.5) node {$2$};;
\draw (5.5,1.5) node {$2$};;
\draw (6.5,1.5) node {$2$};;
\draw (7.5,1.5) node {$2$};;
\draw (0.5,2.5) node {$1$};;
\draw (1.5,2.5) node {$2$};;
\draw (2.5,2.5) node {$2$};;
\draw (3.5,2.5) node {$2$};;
\draw (0.5,3.5) node {$1$};;
\end{tikzpicture}
\end{equation*}

\noindent Both pairs of shaded boxes are pairs of $2$- and $3$-nodes, which cannot be filled by the labels $3$ and $4$.
There is no way to place labels $1$ and $2$ so that either left gray square is filled.

\end{enumerate}

\subsection{Orders}\label{Orders} 
Recall that for a strip, we call its northwesternmost node its \textit{initial} node and its southwesternmost node its \textit{final node}.
These are also the nodes with the greatest and least content, respectively.

\subsubsection{Order on columns}
Suppose we are building up a partition $\lambda$ from $\core(\lambda)$ by adding columns to the $\ell$-quotient. 
Recall that by Remark \ref{Rem12}(1), if we add the columns in an order such that for each component partition of the $\ell$-quotient, longer columns are added before shorter columns, then adding a column of length $n$ corresponds to adding an $n\ell$-strip to the partition. 
We define an order on the columns of $\quot(\lambda)$ that refine the left-to-right order of columns in each component; since longer columns are before shorter columns in each component, we then obtain a tiling of $\lambda\backslash\core(\lambda)$ by strips.

\begin{defn}
The \textit{The initial order on columns} is defined as follows. 
For each $\lambda$, there is a unique ordering on the columns of $\quot(\lambda)$ so that for $r<s$, the \textit{initial} node of the strip corresponding to the $r$th column has content strictly \textit{greater} than that of the \textit{initial} node of the strip corresponding to the $s$th column.
\end{defn}
We illustrate this column order for $\lambda=(4^22^4)$ with $\ell=3$ in Figure \ref{fig:initcolumnorder} below. 

\begin{figure}[h]
\centering
\begin{tikzpicture}[scale=.5]
\draw (-9,0)--(-6,0)--(-6,1)--(-8,1)--(-8,2)--(-9,2)--(-9,0);;
\draw (-9,1)--(-8,1);;
\draw (-7,0)--(-7,1);;
\draw (-8,0)--(-8,1);;
\draw (-7.5,-1) node {$\lambda(0)=\core(\lambda)$};;
\draw (-5,1) node {$\mapsto$};;
\draw [fill=lightgray] (-4,2)--(-4,6)--(-3,6)--(-3,3)--(-2,3)--(-2,1)--(-3,1)--(-3,2)--(-4,2);;
\draw (-4,0)--(-1,0)--(-1,1)--(-2,1)--(-2,3)--(-3,3)--(-3,6)--(-4,6)--(-4,0);;
\draw (-9,1)--(-8,1);;
\draw (-7,0)--(-7,1);;
\draw (-8,0)--(-8,1);;
\draw (-4,5)--(-3,5);;
\draw (-4,4)--(-3,4);;
\draw (-4,3)--(-3,3);;
\draw (-4,2)--(-2,2);;
\draw (-4,1)--(-2,1);;
\draw (-3,3)--(-3,0);;
\draw (-2,1)--(-2,0);;
\draw (-2.5,-1) node {$\lambda(1)$};;
\draw (0,1) node {$\mapsto$};;
\draw [fill=lightgray] (2,6)--(3,6)--(3,3)--(2,3)--(2,6);;
\draw (1,0)--(4,0)--(4,1)--(3,1)--(3,6)--(1,6)--(1,0);;
\draw (1,5)--(3,5);;
\draw (1,4)--(3,4);;
\draw (1,3)--(3,3);;
\draw (1,2)--(3,2);;
\draw (1,1)--(3,1);;
\draw (2,0)--(2,6);;
\draw (3,0)--(3,1);;
\draw (2.5,-1) node {$\lambda(2)$};;
\draw (5,1) node {$\mapsto$};;
\draw [fill=lightgray] (8,1)--(8,2)--(10,2)--(10,0)--(9,0)--(9,1)--(8,1);;
\draw (6,0)--(10,0)--(10,2)--(8,2)--(8,6)--(6,6)--(6,0);;
\draw (6,5)--(8,5);;
\draw (6,4)--(8,4);;
\draw (6,3)--(8,3);;
\draw (6,2)--(8,2);;
\draw (6,1)--(10,1);;
\draw (7,0)--(7,6);;
\draw (8,0)--(8,2);;
\draw (9,0)--(9,2);;
\draw (8,-1) node {$\lambda(3)=\lambda$};;
\draw (-11,-2) node {$\quot(\lambda(i)):$};;
\draw (-7.5,-2) node {$(\varnothing,\varnothing,\varnothing)$};;
\draw (-2.5,-2) node {$(\varnothing,\varnothing,(1^2))$};;
\draw (2.5,-2) node {$(\varnothing,(1),(1^2))$};;
\draw (8,-2) node {$(\varnothing,(1),(2^11^1))$};;
\draw (11,0) node {\hspace{1cm}};;
\end{tikzpicture}
\caption{The initial order on columns for $\lambda=(4^22^4)$ and $\ell=3$. 
Here, $\lambda(r)$ is the partition resulting from the $i$th column addition.}
\label{fig:initcolumnorder}
\end{figure}

The existence and uniqueness of this order comes from translating Remark \ref{Rem12}(1) back to $m(\lambda)$. 
For any partition $\rho$ and each $i\in\ZZ/\ell\ZZ$, $m_i(\rho)$ will contain a maximal semiinfinite string of black beads on the right---we call this the \textit{black sea}. 
Adding a column of length $n$ to the $i$th coordinate of $\quot(\rho)$ that is shorter than all the previous columns in its component corresponds to swapping the first black bead in the black sea of $m_i(\rho)$ with the white bead $n$ places to its left.
This white bead exists because we are adding shorter columns before longer ones within the $i$th coordinate of $\quot(\lambda)$.
We call this black bead the \textit{initial black bead} for the column.
We call the original position of the initial black bead its \textit{starting spot} and its final position its \textit{landing spot}.
The intial column order stipulates that the landing spot, when viewed in the total Maya diagram $m(\lambda)$, is right of those from the previous column additions.

\subsubsection{Critical height}
We will make use of the following:

\begin{defn}
For a partition $\lambda$ that is not an $\ell$-core, consider the strip corresponding to the last column of the initial order on columns of $\quot(\lambda)$.
The height of the northwesternmost node of this strip is called its \textit{critical height} and denoted $\eta_\lambda$.
For $\lambda$ an $\ell$-core, we set $\eta_\lambda=\infty$.
\end{defn}


\subsubsection{Order on rows}
The analogue for rows is straightforward but some things get flipped.

\begin{defn}
The \textit{The final order on rows} is defined as follows.
For each $\lambda$, there is a unique ordering on the rows of $\quot(\lambda)$ so that for $r<s$, the \textit{final} node of the strip corresponding to the $r$th row has content strictly \textit{less} than that of the \textit{final} node of the strip corresponding to the $s$th row.
\end{defn}

For the row orders, each row addition displaces a white bead from $m(\core(\lambda))$ to the right.
We analogously define the \textit{white sea}, the \textit{initial white bead} for the row, and that bead's starting and landing spots.
The final order merely dictates that the landing spot of the initial white bead for a newer row is left of the landing spots of the initial white beads from previous rows. 

\subsection{Proof of Theorem \ref{Main}}\label{ProofMain}
Recall that to prove Theorem \ref{Main}, it suffices to prove Lemma \ref{StrongLem}.
We will only prove Lemma \ref{StrongLem}(2), and we assume it by induction for partitions of size $<|\lambda|$.
The starting point is Definition \ref{TabloidDef} and Proposition \ref{TabProp}.
We need to show for all $k\ge 1$ and $p\in\ZZ/\ell\ZZ$,
\begin{equation}
\lambda_1(p)+\cdots+\lambda_k(p)\ge \mu_1(p)+\cdots+\mu_k(p)
\label{EllDom}
\end{equation}

\subsubsection{Above critical height}
An easy first case of (\ref{EllDom}) is when $k>\eta_\lambda$.
Let $\mathcal{O}_c^I$ denote the initial order on columns of $\quot(\lambda)$.
By Proposition \ref{TabProp}, $\mu$ has a row-strict $\ell$-tableau of content $(\lambda,\mathcal{O}_c^I)$, which we denote by $T_c^I$.
Suppose $\quot(\lambda)$ has $N$ columns and consider the penultimate truncations:
\begin{align*}
\mu'&:=\mu(T_c^I,N-1)\\
\lambda'&:=\lambda(\mathcal{O}_c^I,N-1)
\end{align*}
By Proposiiton \ref{TabProp} and induction, we have $\mu'\preceq_\ell\lambda'$.

If the last column in $\mathcal{O}_c^I$ has length $n$, then for any $k>\eta_\lambda$ and $p\in\ZZ/\ell\ZZ$, we have:
\begin{align*}
\mu_1(p)+\cdots+\mu_k(p)&\le\mu_1'(p)+\cdots+\mu_k'(p)+n\\
&\le\lambda_1'(p)+\cdots+\lambda_k'(p)+n\\
&= \lambda_1(p)+\cdots+\lambda_k(p)
\end{align*}
The first inequality is because $\mu\backslash\mu'$ contains $n$ $p$-nodes, and the second inequality follows from $\mu'\preceq_\ell\lambda'$.
The final equality is because when $k>\eta_\lambda$, all $n$ $p$-nodes of $\lambda\backslash\lambda'$ have already been added below height $k$.

\begin{rem}
For what follows, it is helpful to recall that for each value of the content congruent to $i$ mod $\ell$, we can associate a bead in $m_{i}(\lambda)$ and $m_{i-1}(\lambda)$.
Flipped around, we can associate to each $i$- or $(i+1)$-node a bead in $m_i(\lambda)$.
\end{rem}

\subsubsection{Below critical height with colors in $\quot(\lambda)$}\label{QuotColor}
Now consider (\ref{EllDom}) for $k\le\eta_\lambda$.
We will first consider the case where $p$ is such that the $p$th coordinate of $\quot(\lambda)$ is nonempty.
Thus, we can choose an ordering $\mathcal{O}_c^p$ of the columns of $\quot(\lambda)$ where the final column is the smallest column in the $p$th coordinate.
By Proposition \ref{TabProp}, $\mu$ has a row-strict $\ell$-tableau $T_c^p$ with content $(\lambda,\mathcal{O}_c^p)$.
The penultimate truncations are denoted by:
\begin{align*}
\mu''&:=\mu(T_c^p,N-1)\\
\lambda''&:=\lambda(\mathcal{O}_c^p,N-1)
\end{align*}
We have by Proposition \ref{TabProp} and induction that $\mu''\preceq_\ell\lambda''$.

Let $\square\in\mu\backslash\lambda$ be a $p$-node with height $\le\eta_\lambda$.
We claim that the $(p+1)$-node $\square_{\leftarrow}$ immediately to its left is in $\mu''\backslash\lambda''$.
To see that $\square_{\leftarrow}\not\in\lambda''$, we consider three cases which are illustrated in Figure \ref{fig:WhiteCase} below: 
\begin{enumerate}
\item \textit{$\square$ is not adjacent to $\lambda$}:
This is obvious because $\square_{\leftarrow}\not\in\lambda$.
\item \textit{$\square$ is adjacent to $\lambda\backslash\lambda''$}:
The content of the $p$-node $\square$ corresponds to a bead in $m_p(\lambda)$.
In this case, that bead is white and we have $\square_{\leftarrow}\in\lambda\backslash\lambda''$.
\item \textit{$\square$ is adjacent to the border of $\lambda$ but not $\lambda\backslash\lambda''$}:
By the definition of the initial order, any node lying at or below height $\eta_\lambda$ cannot lie above the strip $\lambda\backslash\lambda''$.
Thus, in this case, $\square$ sits below $\lambda\backslash\lambda''$.
Following the reasoning from case (2), the content of the $p$-node $\square$ corresponds to a bead in the black sea of $m_p(\lambda)$.
This implies that $\square_\leftarrow\in\mu\backslash\lambda$.
\end{enumerate}
For all cases, $\square_\leftarrow\in\mu''$ due to the row-strict $\ell$-tableau condition.

\begin{figure}[h]
\centering
\begin{tikzpicture}
\draw[fill=lightgray] (2,1)--(3,1)--(3,2)--(2,2)--(2,1);;
\draw (2.5,1.5) node {$(2)$};;
\draw (2,1)--(3,1)--(3,2)--(2,2)--(2,1);;
\draw[fill=lightgray] (3,2)--(4,2)--(4,3)--(3,3)--(3,2);;
\draw (3.5,2.5) node {$(1)$};;
\draw[fill=lightgray] (5,0)--(6,0)--(6,1)--(5,1)--(5,0);;
\draw (5.5,0.5) node {$(3)$};;
\draw[->] (4.5,1.25)--(4.5,.1);;
\draw (4.5,1.5) node {black bead};;
\draw[->] (1.5,0.75)--(1.5,1.5)--(1.9,1.5);;
\draw (1.5,0.5) node {white bead};;
\draw[thick] (0,3)--(2,3)--(2,1)--(4,1)--(4,0)--(7,0);;
\draw (0.5,2.5) node {$\lambda$};;
\draw[thick] (0,2)--(1,2)--(1,1)--(2,1);;
\draw (0.5,1.5) node {$\lambda''$};;
\end{tikzpicture}
\caption{The three cases of the $p$-node $\square$ from \ref{QuotColor} are displayed here as gray squares.
In all three cases, the $(p+1)$-node immediately to its left is outside of $\lambda''$.}
\label{fig:WhiteCase}
\end{figure}

Because $\mu''\preceq_\ell\lambda''$, for each $(p+1)$-node $\boxdot\in\mu''\backslash\lambda''$, we can assign its own $(p+1)$-node $\varsigma(\boxdot)\in\lambda''\backslash\mu''$ below it.
We claim that for every $(p+1)$-node $\blacksquare\in\lambda''\backslash\mu''$ at height $\le \eta_\lambda$, the $p$-node $\blacksquare_\rightarrow$ immediately to its right is in $\lambda\backslash\mu$. 
First, to see that $\blacksquare_{\rightarrow}\in\lambda$, we consider three cases, which we illustrate in Figure \ref{fig:BlackCase} below:
\begin{enumerate}
\item \textit{$\blacksquare$ does not sit in the border of $\lambda''$}:
This is trivial because $\blacksquare_\rightarrow\in\lambda''\subset\lambda$.
\item \textit{$\blacksquare$ is adjacent to $\lambda\backslash\lambda''$}:
As in the previous paragraph, the content of the $(p+1)$-node $\blacksquare$ corresponds to a white bead in $m_p(\lambda'')$.
Since $\lambda\backslash\lambda''$ starts at a $p$-node and ends at a $(p+1)$-node, it follows that $\blacksquare_\rightarrow\in\lambda\backslash\lambda''$.
\item \textit{$\blacksquare$ sits in the border of $\lambda''$ but is not adjacent to $\lambda\backslash\lambda''$}:
As in the previous paragraph, the content of the $(p+1)$-node $\blacksquare$ now corresponds to a bead in the black sea of $m_p(\lambda'')$.
This implies that $\blacksquare_\rightarrow\in\lambda$.
\end{enumerate}
In both cases, since $\blacksquare\not\in\mu''$, the row-strict $\ell$-tableau condition forces $\blacksquare_\rightarrow\not\in\mu$.

\begin{figure}[h]
\centering
\begin{tikzpicture}
\draw[fill=lightgray] (0,2)--(1,2)--(1,1)--(0,1)--(0,2);;
\draw (0.5,1.5) node {$(1)$};;
\draw[fill=lightgray] (1,3)--(2,3)--(2,2)--(1,2)--(1,3);;
\draw (1.5,2.5) node {$(2)$};;
\draw[fill=lightgray] (3,1)--(4,1)--(4,0)--(3,0)--(3,1);;
\draw (3.5,0.5) node {$(3)$};;
\draw[thick] (-1,3)--(2,3)--(2,2)--(4,2)--(4,1)--(6,1)--(6,0);;
\draw (-0.5,2.5) node {$\lambda''$};;
\draw[thick] (-1,4)--(3,4)--(3,3)--(5,3)--(5,1);;
\draw (-0.5,3.5) node {$\lambda$};;
\draw[->] (5.5,1.5)--(4.5,1.5)--(4.5,1.1);;
\draw (6.5,1.5) node {black bead};;
\draw[->] (4.5,3.25)--(4.5,2.5)--(2.1,2.5);;
\draw (4.5,3.5) node {white bead};;
\end{tikzpicture}
\caption{The three cases of the $(p+1)$-node $\blacksquare$ from \ref{QuotColor} are displayed here as gray squares.
In all three cases, the $p$-node immediately to its right is inside $\lambda$.}
\label{fig:BlackCase}
\end{figure}

Altogether,  we have a sequence of injective maps 
\[
\square\mapsto\square_{\leftarrow}\mapsto \varsigma(\square_{\leftarrow})\mapsto\varsigma(\square_{\leftarrow})_{\rightarrow}
\]
that sends a $p$-node $\square\in\mu\backslash\lambda$ of height $\le\eta_\lambda$ to a $p$-node $\varsigma(\square_\leftarrow)_{\rightarrow}\in\lambda\backslash\mu$.
The outer two maps preserve height while the middle map decreases height, so the composition decreases height overall.
This implies (\ref{EllDom}) for all $k\le\eta_\lambda$.

\subsubsection{Remaining colors}\label{Remain}
Altogether, we have established (\ref{EllDom}) for all $k$ and all colors whose coordinate in $\quot(\lambda)$ is nonempty.
For color $p$ whose coordinate in $\quot(\lambda)$ is empty, we will prove (\ref{EllDom}) for all $k$ by induction \textit{downwards} along cyclic order, assuming (\ref{EllDom}) holds for $p+1$.
Note that for the $p$ we are considering, $m_p(\lambda)$ is a shifted Maya diagram for the empty partition.
Thus, from left to right, it consists entirely of its white sea followed by its black sea.

Let $\left( \mu\backslash\lambda \right)(p)$ and $\left( \lambda\backslash\mu \right)(p)$ denote the nodes of color $p$ in the respective skew diagrams.
Our goal is to construct a bijection $\varsigma_p:\left( \mu\backslash\lambda \right)(p)\rightarrow\left( \lambda\backslash\mu \right)(p)$ that sends every node to one below it.
We will do this in a piecewise manner:
\begin{enumerate}
\item For any $\square\in(\lambda\backslash\mu)(p)$ with content corresponding to a bead in the \textit{white} sea of $m_p(\lambda)$, we find a $\varsigma^{-1}_p(\square)\in(\mu\backslash\lambda)(p)$ to its north.
\item For any $\blacksquare\in(\mu\backslash\lambda)(p)$ with content corresponding to a bead in the \textit{black} sea of $m_p(\lambda)$, we define a $\varsigma_p(\blacksquare)\in(\lambda\backslash\mu)(p)$ to its south.
\item The remaining nodes in the domain have content corresponding to a bead in the white sea of $m_p(\lambda)$, while the remaining nodes in the codomain have content corresponding to a bead in the black sea of $m_p(\lambda)$.
The latter are necessarily south of the former, so we define $\varsigma_p$ to be any bijection between these two sets of remaining nodes.
\end{enumerate}
By induction on $p$, we can assume that $\varsigma_{p+1}$ exists.

For step (1), we first show that the $(p+1)$-node $\square_{\uparrow}$ immediately above $\square$ lies in $\lambda$.
There are two cases:
\begin{enumerate}[(a)]
\item \textit{$\square$ does not lie in the border of $\lambda$}:
In this case, it is obvious that $\square_{\uparrow}\in\lambda$.
\item \textit{$\square$ lies in the border of $\lambda$}:
Because the bead in $m_p(\lambda)$ for this content is white, it follows that $\square_\uparrow\in\lambda$.
\end{enumerate}
In both cases, $\square\not\in\mu$ implies $\square_{\uparrow}\not\in\mu$, so we have $\square_{\uparrow}\in\lambda\backslash\mu$.
Now, by induction, we have the $(p+1)$-node $\boxdot:=\varsigma_{p+1}^{-1}(\square_{\uparrow})\in\mu\backslash\lambda$ above $\square_{\uparrow}$.
We claim that the $p$-node $\boxdot_{\downarrow}$ immediately below it is not in $\lambda$.
As before, there are two cases:
\begin{enumerate}[(a)]
\setcounter{enumi}{2}
\item \textit{$\boxdot$ is not adjacent to $\lambda$}:
In this case, it is obvious that $\boxdot_{\downarrow}\not\in\lambda$.
\item \textit{$\boxdot$ is adjacent to $\lambda$}:
Here, note that the content of $\boxdot_{\downarrow}$ is at least that of $\square$.
Thus the content of $\boxdot$ corresponds to a white bead in $m_p(\lambda)$.
It follows that $\boxdot_\downarrow\not\in\lambda$.
\end{enumerate}
In both cases, $\boxdot\in\mu$ implies $\boxdot_\downarrow\in\mu$ and thus $\boxdot_\downarrow\in\mu\backslash\lambda$.
We illustrate the two pairs of cases in Figure \ref{fig:Remain1} below.
Note that $\boxdot_\downarrow$ is above $\square$, so we set $\varsigma_p^{-1}(\square)=\boxdot_\downarrow$.
This way to assign values to $\varsigma_p$ is well-defined because $\varsigma_{p+1}$ is well-defined and we are just taking adjacent nodes.
Moreover, $\varsigma_p^{-1}(\square)$ is above $\square$.

\begin{figure}[h]
\centering
\begin{tikzpicture}
\draw[fill=lightgray] (2,1)--(1,1)--(1,2)--(2,2)--(2,1);;
\draw (1.5,1.5) node {(b)};;
\draw[fill=lightgray] (1,0)--(0,0)--(0,1)--(1,1)--(1,0);;
\draw (0.5,0.5) node {(a)};;
\draw[fill=lightgray] (3,3)--(2,3)--(2,4)--(3,4)--(3,3);;
\draw (2.5,3.5) node {(d)};;
\draw[fill=lightgray] (5,2)--(4,2)--(4,3)--(5,3)--(5,2);;
\draw (4.5,2.5) node {(c)};;
\draw[thick] (0,4)--(2,4)--(2,1)--(4,1)--(4,0);;
\draw (0.5,3.5) node {$\lambda$};;
\draw[->] (0.75,2.5)--(1.9,2.5);;
\draw (-.25,2.5) node {white bead};;
\end{tikzpicture}
\caption{The two pairs of cases from step (1) of \ref{Remain}.
Here, both $\square$ and $\boxdot$ are displayed here as gray squares, appropriately labeled by letter.}
\label{fig:Remain1}
\end{figure}

Step (2) is quite similar.
We first show that that the $(p+1)$-node $\blacksquare_{\leftarrow}$ immediately to the left of $\blacksquare$ is not in $\lambda$.
There are two cases:
\begin{enumerate}[(a)]
\item \textit{$\blacksquare$ is not adjacent to $\lambda$}:
Here, it is obvious that $\blacksquare_{\leftarrow}\not\in\lambda$.
\item \textit{$\blacksquare$ is adjacent to $\lambda$}:
Because the content of $\blacksquare$ corresponds to a black bead in $m_p(\lambda)$, it follows that $\blacksquare_{\leftarrow}\not\in\lambda$.
\end{enumerate}
For both cases, $\blacksquare\in\mu$ implies $\blacksquare_{\leftarrow}\in\mu$, so we have $\blacksquare_{\leftarrow}\in\mu\backslash\lambda$.
Recycling notation, consider the $(p+1)$-node $\boxdot:=\varsigma_{p+1}(\blacksquare_{\leftarrow})\in\lambda\backslash\mu$.
We claim that the $p$-node $\boxdot_{\rightarrow}$ directly to the right of $\boxdot$ is in $\lambda$.
The two cases are:
\begin{enumerate}[(a)]
\setcounter{enumi}{2}
\item \textit{$\boxdot$ does not lie in the border of $\lambda$}:
The containment $\boxdot_{\rightarrow}\in\lambda$ is obvious.
\item \textit{$\boxdot$ lies in the border of $\lambda$}:
The content of $\boxdot$ corresponds to a black bead in $m_p(\lambda)$, so $\boxdot_{\rightarrow}\in\lambda$.
\end{enumerate}
Note that $\boxdot_{\rightarrow}\not\in\mu$ if $\boxdot\not\in\mu$, so $\boxdot_{\rightarrow}\in\lambda\backslash\mu$.
The two pairs of cases are illustrated in Figure \ref{fig:Remain2} below.
We set $\varsigma_p(\blacksquare)=\boxdot_{\rightarrow}$.
This assignment is injective and downward because $\varsigma_{p+1}$ is injective and downward and we are taking adjacent nodes.
This completes step (2).

Step (3) is automatic, and thus we finish the proof.\qed

\begin{figure}[h]
\centering
\begin{tikzpicture}
\draw[fill=lightgray] (1,2)--(2,2)--(2,1)--(1,1)--(1,2);;
\draw (1.5, 1.5) node {(d)};;
\draw[fill=lightgray] (0,1)--(1,1)--(1,0)--(0,0)--(0,1);;
\draw (0.5, 0.5) node {(c)};;
\draw[fill=lightgray] (3,3)--(4,3)--(4,2)--(3,2)--(3,3);;
\draw (3.5, 2.5) node {(b)};;
\draw[fill=lightgray] (4,4)--(5,4)--(5,3)--(4,3)--(4,4);;
\draw (4.5, 3.5) node {(a)};;
\draw[thick] (0,4)--(1,4)--(1,2)--(4,2)--(4,0);;
\draw (0.5, 3.5) node {$\lambda$};;
\draw[->] (2.5,0.75)--(2.5,1.9);;
\draw (2.5,0.5) node {black bead};;
\end{tikzpicture}
\caption{The two pairs of cases from step (2) of \ref{Remain}.
Both $\blacksquare$ and $\boxdot$ are displayed as gray squares with appropriate letters attached.}
\label{fig:Remain2}
\end{figure}

\appendix
\renewcommand{\thesubsection}{A.\arabic{subsection}}

\section{Dual vertex representation}\label{DualVertexApp}
In this section of the Appendix, we give details on the dual vertex representation and prove an analogue of the Tsymbaliuk isomorphism.

\subsection{Dual vertex operators}
We begin by writing down the formulas for $\rho_{p,\vec{c}}\circ S$ applied to the generators of $\UTor$.
\begin{align}
\nonumber
(\rho_{p,\vec{\ccc}}\circ S)(e_i(z))&=-\qqq\ccc_iz^{-1+H_{i,0}}e^{\alpha_i}\exp\left(\sum_{k>0}\frac{\qqq^{\frac{k}{2}}}{[k]_\qqq}b_{i,-k}z^k\right)\\
\nonumber
&\times\exp\left(-\sum_{k>0}\frac{\qqq^{\frac{k}{2}}}{[k]_\qqq}b_{i,k}z^{-k}\right) ,\\
\label{AntiVertexF}
(\rho_{p,\vec{\ccc}}\circ S)(f_i(z))&=-\frac{(-1)^{\ell\delta_{i,0}}}{\qqq \ccc_i}z^{-1-H_{i,0}}e^{-\alpha_i}\exp\left(-\sum_{k>0}\frac{\qqq^{-\frac{k}{2}}}{[k]_\qqq}b_{i,-k}z^k\right)\\
\nonumber
&\times\exp\left(\sum_{k>0}\frac{\qqq^{-\frac{k}{2}}}{[k]_\qqq}b_{i,k}z^{-k}\right) ,\\
\nonumber
(\rho_{p,\vec{\ccc}}\circ S)(\psi_i^\pm(z))&=\exp\left(\mp(\qqq-\qqq^{-1})\sum_{k>0}b_{i,\pm k}z^{\mp k}\right)\qqq^{\mp\partial_{\alpha_i}},
\end{align}
\[
(\rho_{p,\vec{\ccc}}\circ S)(\gamma^{\frac{1}{2}})=\qqq^{-\frac{1}{2}},\,(\rho_{p,\vec{\ccc}}\circ S)(\qqq^{d_1})=\qqq^{d}
\]
Note that when applying the antipode, scaling $z$ by $\gamma$ in $\rho_{p,\vec{\ccc}}(e_{i}( z))$ and $\rho_{p,\vec{\ccc}}(f_i( z))$ multiplies $z^{1\pm H_{i,0}}$ by $\qqq^{1\pm\partial_{\alpha_i}}$, which cancels with the $\qqq^{\mp\partial_{\alpha_i}}$ coming from the extra $\psi_i^\pm(\gamma^{\frac{1}{2}}z)^{-1}$.

Let us document how each component of the vertex operators affects a dual vector.
For $v\otimes e^\alpha\in W_p$, we have
\begin{align*}
\langle v\otimes e^{\alpha}|z^{H_{i,0}}&= z^{H_{i,0}} (v\otimes e^\alpha)\langle v\otimes e^\alpha|\\
\langle v\otimes e^{\alpha}|e^{\pm\alpha_i}&= (-1)^{\ell\delta_{i,0}}\langle v\otimes e^{\alpha\mp \alpha_i}|\\
\langle v\otimes e^\alpha| \qqq^{\pm\partial_{\alpha_i}}&=\qqq^{\langle \pm\alpha_i,\alpha\rangle}\langle v\otimes e^\alpha|
\end{align*}
and for $k>0$,
\begin{align*}
\langle v\otimes e^{\alpha}|b_{i,k}&=\left\langle \frac{[k]_\qqq}{k}\left(-\ddd^{-k}[k]_\qqq b_{i-1,-k}+[2k]_{\qqq}b_{i,-k}-\ddd^k[k]_\qqq b_{i+1,-k}\right) v\otimes e^\alpha\right|\\
\langle v\otimes e^{\alpha}|b_{i,-k}&=\left\langle -\frac{b_{i,k}^\perp}{\qqq^k-\qqq^{-k}} v\otimes e^\alpha\right| 
\end{align*}
For $H=z^{H_{i,0}},e^{\pm\alpha_i}, \qqq^{\pm\partial_{\alpha_i}},b_{i,\pm k}$, let us abuse notation and denote
\[
H\langle v\otimes e^\alpha|:=\langle H(v\otimes e^\alpha)|
\]
In this way, we can write the dual representation $\rho_{p,\vec{\ccc}}^*$ as a \textit{left} action:
\begin{align*}
\rho_{p,\vec{\ccc}}^*(e_{i}(z))
&= -(-1)^{\ell\delta_{i,0}}\qqq\ccc_i
\exp\left( -\sum_{k>0}\frac{\qqq^{\frac{k}{2}}}{k}\left( -\ddd^{-k}[k]_\qqq b_{i-1,k} +[2k]_\qqq b_{i,-k}-\ddd^k[k]_\qqq b_{i+1,-k}\right)z^{-k} \right)\\
&\times\exp\left( -\sum_{k>0}\frac{\qqq^{\frac{k}{2}}}{[k]_\qqq\left(\qqq^k-\qqq^{-k}\right)}b_{i,k}^\perp z^k \right)
e^{-\alpha_i}
z^{-1+H_{i,0}}\\
\rho_{p,\vec{\ccc}}^*(f_{i}(z))
&= -\frac{1}{\qqq\ccc_i}
\exp\left( \sum_{k>0}\frac{\qqq^{\frac{k}{2}}}{k}\left( -\ddd^{-k}[k]_\qqq b_{i-1,k} +[2k]_\qqq b_{i,-k}-\ddd^k[k]_\qqq b_{i+1,-k}\right)z^{-k} \right)\\
&\times\exp\left( \sum_{k>0}\frac{\qqq^{\frac{k}{2}}}{[k]_\qqq\left(\qqq^k-\qqq^{-k}\right)}b_{i,k}^\perp z^k \right)
e^{\alpha_i}
z^{-1-H_{i,0}}\\
\rho_{p,\vec{\ccc}}^*(\psi^+_{i}(z))&= 
\exp\left( -(\qqq-\qqq^{-1})\sum_{k>0}\frac{[k]_\qqq}{k}\left(-\ddd^{-k}[k]_\qqq b_{i-1,-k}+[2k]_{\qqq}b_{i,-k}-\ddd^k[k]_\qqq b_{i+1,-k}\right)z^{-k} \right)\qqq^{-\partial_{\alpha_i}}\\
\rho_{p,\vec{\ccc}}^*(\psi^-_{i}(z))&= 
\exp\left(- (\qqq-\qqq^{-1})\sum_{k>0}\frac{1}{\qqq^k-\qqq^{-k}}b_{i,k}^\perp z^{k} \right)\qqq^{\partial_{\alpha_i}},
\end{align*}
\[\rho_{p,\vec{\ccc}}^*(\gamma^{\frac{1}{2}})=\qqq^{-\frac{1}{2}},\,\rho_{p,\vec{\ccc}}^*(\qqq^{d_1})=\qqq^{d}\]
Writing the dual vertex representation will aid in repeating the arguments from 3.3 of \cite{Tsym}.

\subsection{Dual Tsymbaliuk isomorphisms\protect\footnote{We thank O. Tsymbaliuk for suggesting the results presented here.}}
Let $\mathbb{1}_p^*:=\langle\mathbb{1}_p|\in W_p^*$ denote the dual vacuum.
The main result of this subsection is:
\begin{thm}\label{DualTsym1}
For $0\le p\le \ell-1$, the vacuum-to-vacuum map $\mathbb{1}_p^*\mapsto |\varnothing\rangle$ induces an isomorphism $\mathrm{T}_p^{*}:W^*_p\rightarrow\mathcal{F}$ between the representations $\rho_{p,\vec{\ccc}}^*\circ\varpi$ and $\tau_{p,\upsilon}^+$ of $'\ddot{U}$, where the parameters $\vec{\ccc}$ and $\upsilon$ are related by
\[
\upsilon=\frac{(-1)^{\frac{(\ell-2)(\ell-3)}{2}}}{\qqq\ddd^{\frac{\ell}{2}}\ccc_0\cdots\ccc_{\ell-1}}=\ddd^{-\frac{\ell}{2}}\uuu
\]
where $\uuu$ is as in Lemma \ref{Highwt}.
\end{thm}

\begin{proof}
As in the proof of Corollary \ref{BoseFermi2}, we need to show that $\mathbb{1}_p^*$ satisfies the following two properties:
\begin{itemize}
\item \textit{Eigenvector property}: $\mathbb{1}_p^*$ is an eigenvector for $\varpi(\psi^\pm_i(z))$ with eigenvalue
\begin{equation}
\varpi(\psi^\pm_i(z))\mathbb{1}_p^*=\delta_{i,p}\frac{\qqq^{-1}z-\qqq\ddd^{-\frac{\ell}{2}}\uuu}{z-\ddd^{-\frac{\ell}{2}}\uuu}\mathbb{1}_p^*
\label{EigenEq}
\end{equation}
Looking at coefficients, this is equivalent to
\begin{align}
\nonumber
\psi_{i,0}^{\pm 1}\mathbb{1}_p^*&= \qqq^{\mp\delta_{i,p}}\mathbb{1}_p^*\\
\label{CoeffEigen}
\psi_{i,\pm n}\mathbb{1}_p^*&= \mp\delta_{i,p}(\qqq-\qqq^{-1})\ddd^{\mp\frac{\ell n}{2}}\uuu^{\pm n}\mathbb{1}_p^*\hbox{ for }n>0
\end{align}
\item \textit{Lowest weight property}: $\mathbb{1}_p^*$ is annihilated by $\varpi(f_i(z))$ for all $i\in\ZZ/\ell\ZZ$.
\end{itemize}
We follow the strategy from \cite{Tsym}, wherein the proofs of the analogous properties are intertwined.
\begin{enumerate}
\item \textit{Zero modes and $\qqq^{d_2}$}: It is easy to see that the action of $\tau_{p,\upsilon}^-(\psi_{i,0})$ agrees with that of $(\rho_{p,\vec{\ccc}}^*\circ\varpi)(\psi_{i,0})$ for all $i\in\ZZ/\ell\ZZ$.
This is true as well for $\tau_{p,\upsilon}^-(\qqq^{d_2})$ and $(\rho_{p,\vec{\ccc}}^*\circ\varpi)(q^{d_2})$.
Since $\varpi(f_{i,0})=f_{i,0}$ for $i\not=0$, we have that $(\rho_{p,\vec{\ccc}}^*\circ\varpi)(f_{i,0})\mathbb{1}_p^*=0$.
For $\varpi(f_{0,0})$, we use formula (\ref{MikiF0}) from Proposition \ref{MikiProp}:
\begin{equation}
\varpi(f_{0,0})=\ddd^{-1}[e_{\ell-1,0},\cdots[e_{2,0},e_{1,-1}]_{\qqq^{-1}}\cdots]_{\qqq^{-1}}\psi_{0,0}^{-1}\gamma^{-1}\tag{\ref{MikiF0}}
\end{equation}
Note that
\begin{align*}
e_{i,0}\mathbb{1}_p^*&= 0\hbox{ for }i\not=p\\
e_{1,-1}\mathbb{1}_p^*&= 0
\end{align*}
Therefore, the only orderings in the multicommutator (\ref{MikiF0}) that could possibly have nonzero action on $\mathbb{1}_p^*$ must have $e_{p,0}$ on the right when $p\not=0$.
As in \textit{loc. cit.}, this must be followed by $e_{p\pm 1,0}$ and increase/decrease sequentially until arriving at $e_{1,-1}$, which acts by zero.
\item \textit{Action of $b_{i,1}$}: 
Taking logarithms in (\ref{EigenEq}), we need to show 
\begin{equation}
\varpi(b_{i,1})\mathbb{1}_p^*=-\delta_{i,p}\qqq\ddd^{-\frac{\ell}{2}}\uuu\mathbb{1}_p^*
\label{P1Eigen}
\end{equation}
To do this, we use (\ref{PlusI}) and (\ref{Plus0}) from Proposition \ref{MikiProp}:
\begin{align}
\varpi(b_{i,1})&= (-1)^{i+1}\ddd^{-i}[ [\cdots[ [\cdots [f_{0,0},f_{\ell-1,0}]_\qqq,\cdots, f_{i+1,0}]_\qqq, f_{1,0}]_\qqq, \cdots, f_{i-1,0}]_\qqq , f_{i,0}]_{\qqq^2}\tag{\ref{PlusI}}\\
\tag{\ref{Plus0}}
\varpi(b_{0,1})&= (-1)^\ell \ddd^{-(\ell-1)}[ [\cdots [f_{1,1},f_{2,0}]_\qqq,\cdots, f_{\ell-1,0}]_\qqq, f_{0,-1}]_{\qqq^2}
\end{align}
There are two cases:
\begin{enumerate}
\item $p\not=0$:
In this case, every operator in multicommutator (\ref{Plus0}) annihilates $\mathbb{1}_{p}^*$.
On the other hand, only $f_{0,0}$ in (\ref{PlusI}) has nonzero action on $\mathbb{1}_p^*$, from which it follows that the only summand with a possibly nonzero action is 
\[
(-1)^{i+\ell}\ddd^{-i}\qqq^{\ell}f_{i,0}f_{i-1,0}\cdots f_{1,0}f_{i+1,0}f_{i+2,0}\cdots f_{0,0}
\]
If $i\not=p$, then when acting with $f_{p,0}$, only one of $f_{p\pm 1,0}$ must have acted before it, and thus the extra weight $e^{\Lambda_p}$ will force $f_{p,0}$ to act by zero.
For $i=p$, the result is
\begin{align*}
&\frac{(-1)^{p}\ddd^{-\frac{\ell}{2}}}{\ccc_{0}\cdots\ccc_{\ell-1}}\left\langle1\otimes e^{\alpha_p} e^{\alpha_{p-1}}\cdots e^{\alpha_2} e^{\alpha_1}e^{\alpha_{p+1}}e^{\alpha_{p+2}}\cdots e^{\alpha_{\ell-1}}e^{\alpha_0}e^{\Lambda_p}\right|\\
&= \frac{(-1)^{\frac{(\ell-2)(\ell-3)}{2}+1}\ddd^{-\frac{\ell}{2}}}{\ccc_{0}\cdots\ccc_{\ell-1}}\mathbb{1}_p^*=-\qqq\ddd^{-\frac{\ell}{2}}\uuu\mathbb{1}_p^*
\end{align*}
\item $p=0$: Here, every operator in (\ref{PlusI}) annihilates $\mathbb{1}_p^*$.
In (\ref{Plus0}), only $f_{1,1}$ has nonzero action on $\mathbb{1}_p^*$.
Consequently, the only summand in the multicommutator that could have nonzero action is
\[
-\qqq^\ell\ddd^{-(\ell-1)}f_{0,-1} f_{\ell-1,0}\cdots f_{2,0}f_{1,1}
\]
Acting with this expression yields
\begin{align*}
&\frac{(-1)^{\ell+1}\ddd^{-\frac{\ell}{2}}}{\ccc_0\cdots\ccc_{\ell-1}}\langle 1\otimes e^{\alpha_0}e^{\alpha_{\ell-1}}\cdots e^{\alpha_2}e^{\alpha_1}|\\
&=\frac{(-1)^{\frac{(\ell-2)(\ell-3)}{2}+1}\ddd^{-\frac{\ell}{2}}}{\ccc_0\cdots\ccc_{\ell-1}}\mathbb{1}_0^*=-\qqq\ddd^{-\frac{\ell}{2}}\uuu\mathbb{1}_0^*
\end{align*}
\end{enumerate}
\item \textit{Action of $b_{i,-1}$}:
Here, we need
\begin{equation}
\varpi(b_{i,0})\mathbb{1}_p^*=-\delta_{i,p}\frac{\ddd^{\frac{\ell}{2}}}{\qqq\uuu}\mathbb{1}_p^*
\label{P-1Eigen}
\end{equation}
We will now use (\ref{MinusI}) and (\ref{Minus0}):
\begin{align}
\tag{\ref{MinusI}}
\varpi(b_{i,-1})&= (-1)^{i+1}\ddd^i[e_{i,0}, [\cdots,[e_{1,0}, [e_{i+1,0},\cdots,[e_{\ell-1,0},e_{0,0}]_{\qqq^{-1}}\cdots]_{\qqq^{-1}}]_{\qqq^{-1}}\cdots]_{\qqq^{-1}}]_{\qqq^{-2}}\\
\tag{\ref{Minus0}}
\varpi(b_{0,-1})&= (-1)^\ell \ddd^{\ell-1}[ e_{0,1},[e_{\ell-1,0},\cdots,[e_{2,0},e_{1,-1}]_{\qqq^{-1}}\cdots]_{\qqq^{-1}}]_{\qqq^{-2}}
\end{align}
As in (2), we split into two cases:
\begin{enumerate}
\item $p\not=0$:
Every operator in (\ref{Minus0}) will annihilate $\mathbb{1}_p^*$.
In (\ref{MinusI}), only $e_{p,0}$ will have a nonzero action and thus it must be furthest to the right.
Here, the extra $e^{\Lambda_p}$ forces $e_{0,0}$ to act by zero unless $e_{1,0}$ and $e_{\ell-1,0}$ are to its right.
From this, we can deduce that the only chance for a summand to have nonzero action on $\mathbb{1}_p^*$ is when $i=p$, in which case the summand must be
\[
(-1)^{p+\ell}\ddd^p\qqq^{-\ell} e_{0,0} e_{\ell-1,0}\cdots e_{p+1,0}e_{1,0}e_{2,0}\cdots e_{p,0}
\]
Acting with this term yields
\begin{align*}
&(-1)^{p+\ell}\ddd^{\frac{\ell}{2}}\ccc_0\cdots \ccc_{\ell-1}\langle 1\otimes e^{-\alpha_0}e^{-\alpha_{\ell-1}}\cdots e^{-\alpha_{p+1}}e^{-\alpha_1}e^{-\alpha_2}\cdots e^{-\alpha_p}e^{\Lambda_p}|\\
&= (-1)^{\frac{(\ell-2)(\ell-3)}{2}+1}\ddd^{\frac{\ell}{2}}\ccc_0\cdots\ccc_{\ell-1}\mathbb{1}_p^*=-\frac{\ddd^{\frac{\ell}{2}}}{\qqq\uuu}\mathbb{1}_p^*
\end{align*}
\item $p=0$: Like in (2)(b), every operator in (\ref{MinusI}) acts trivally on $\mathbb{1}_p^*$.
In (\ref{Minus0}), only $e_{0,1}$ has a nonzero action, and thus the only summand that might not annihilate $\mathbb{1}_p^*$ is
\[
-\ddd^{\ell-1}\qqq^{-\ell}e_{1,-1}e_{2,0}\cdots e_{\ell-1,0}e_{0,1}
\]
Acting with this gives us
\begin{align*}
&-\ddd^{\frac{\ell}{2}}\langle 1\otimes e^{-\alpha_1}e^{-\alpha_2}\cdots e^{-\alpha_{\ell-1}}e^{-\alpha_0}|\\
&= (-1)^{\frac{(\ell-2)(\ell-3)}{2}+1}\ddd^{\frac{\ell}{2}}\ccc_0\cdots\ccc_{\ell-1}\mathbb{1}_0^*=-\frac{\ddd^{\frac{\ell}{2}}}{\qqq\uuu}\mathbb{1}_0^*
\end{align*}
\end{enumerate}
\item \textit{Bootstrap to other modes}:
To access $\varpi(f_{i,n})$ for $n\not=0$ and $\varpi(\psi_{i,\pm n})$ for $n\ge 2$, we apply $\varpi$ to the the toroidal relations:
\begin{align}
\label{OtherF}
[b_{i,\pm 1},f_{i,n}]&= \mp(\qqq+\qqq^{-1})\gamma^{\frac{1}{2}}f_{i,n\pm 1}\\
\label{OtherPsi}
[e_{i,0},f_{i,\pm m}]&= \pm\frac{1}{\qqq-\qqq^{-1}}\gamma^{\mp\frac{m}{2}}\psi_{i,\pm m}\hbox{ for }m\ge 1
\end{align}
Starting with $f_{i,0}$, it easily follows from (1), (2), (3), and (\ref{OtherF}) that $\varpi(f_i(z))$ annihilates $\mathbb{1}_p^*$ for all $i\in\ZZ/\ell\ZZ$.
Applying this to (\ref{OtherPsi}), we have
\begin{equation}
\varpi(\psi_{i,\pm m})\mathbb{1}_p^*=\mp (\qqq-\qqq^{-1})(\rho_{p,\vec{\ccc}}^*\circ\varpi)(f_{i,\pm m}e_{i,0})\mathbb{1}_{p}^*
\label{PsiEF}
\end{equation}
We compute the right hand side one step at a time:
\begin{enumerate}
\item \textit{Action of $e_{i,0}$}:
For $i\not=0$, we use $\varpi(e_{i,0})=e_{i,0}$ to compute
\begin{equation*}
\varpi(e_{i,0})\mathbb{1}_p^*
=-\delta_{i,p}\qqq\ccc_p\langle 1\otimes e^{-\alpha_p}e^{\Lambda_p}|
\end{equation*}
For $i=0$, we use (\ref{MikiE0}):
\begin{equation}
\varpi(e_{0,0})= \ddd\gamma\psi_{0,0}[ \cdots[f_{1,1},f_{2,0}]_\qqq,\cdots, f_{\ell-1,0}]_\qqq
\tag{\ref{MikiE0}}
\end{equation}
Only $f_{1,1}$ can possibly have nonzero action on $\mathbb{1}_p^*$, and in this case, we must have $p\not=1$.
The only summand with nonzero action is then
\[
\ddd\gamma\psi_{0,0}(-\qqq)^{\ell-2}f_{\ell-1,0}f_{\ell-2,0}\cdots f_{2,0}f_{1,1}
\]
If $p\not=0$, $f_{p,0}$ will act by zero due to the extra $e^{\Lambda_p}$.
For $p=0$, we get
\begin{equation*}
\begin{aligned}
\varpi(e_{0,0})\mathbb{1}_0^*
&=\frac{\ddd}{\ccc_1\cdots\ccc_{\ell-1}}\psi_{0,0}\langle 1\otimes e^{\alpha_{\ell-1}}e^{\alpha_{\ell-2}}\cdots e^{\alpha_2}e^{\alpha_1}|\\
&= \frac{(-1)^{\frac{(\ell-2)(\ell-3)}{2}}\ddd\qqq^{-2}}{\ccc_1\cdots\ccc_{\ell-1}}\langle 1\otimes e^{-\alpha_0}|
\end{aligned}
\end{equation*}
\item \textit{Action of $f_{i,\pm m}$}:
From part (a) above, we only need to consider the case $i=p$.
We can use (\ref{OtherF}) to inductively derive the action of $\varpi(f_{p,\pm m})$ starting from the actions of $\varpi(b_{p,\pm 1})$ and $\varpi(f_{p,0})$ on $\varpi(e_{p,0})\mathbb{1}_p^*$, provided that the latter actions are well-behaved.
In computing $(\rho_{p,\vec{\ccc}}^*\circ\varpi)(b_{p,\pm 1}e_{p,0})\mathbb{1}_p$, only one summand from the expressions (\ref{PlusI}), (\ref{Plus0}), (\ref{MinusI}), and (\ref{Minus0}) will survive.
For $p\not=0$, we have:
\begin{align*}
\varpi(b_{p, 1})\langle 1\otimes e^{-\alpha_p}e^{\Lambda_p}|&= 
(-1)^{p+\ell+1}\ddd^{-p}\qqq^{\ell-2}\\
&\times f_{p-1,0}\cdots f_{1,0}f_{p+1,0} \cdots f_{\ell-1,0}f_{0,0}f_{p,0}
\langle 1\otimes e^{-\alpha_p}e^{\Lambda_p}|\\
&= \frac{(-1)^{p+1}\qqq^{-2}\ddd^{-\frac{\ell}{2}}}{\ccc_0\cdots\ccc_{\ell-1}}\\
&\times\langle 1\otimes e^{\alpha_{p-1}}\cdots e^{\alpha_1}e^{\alpha_{p+1}}\cdots e^{\alpha_{\ell-1}}e^{\alpha_{0}}e^{-\alpha_p}e^{\Lambda_p}|\\
&= \qqq^{-1}\ddd^{-\frac{\ell}{2}}\uuu\langle 1\otimes e^{-\alpha_p}e^{\Lambda_p}|\\
\varpi(b_{p,-1})\langle 1\otimes e^{-\alpha_p}e^{\Lambda_p}|&=
(-1)^{p+\ell+1}\ddd^p\qqq^{-\ell+2}\ccc_0\cdots\ccc_{\ell-1}\\
&\times e_{p,0}e_{0,0}e_{\ell-1,0}\cdots e_{p+1,0}e_{1,0}\cdots e_{p-1,0}
\langle 1\otimes e^{-\alpha_p}e^{\Lambda_p}|\\
&=(-1)^{p+\ell+1}\qqq^2\ddd^{\frac{\ell}{2}}\ccc_0\cdots\ccc_{\ell-1}\\
&\times \langle 1\otimes e^{-\alpha_p}e^{-\alpha_0}e^{-\alpha_{\ell-1}}\cdots e^{-\alpha_{p+1}}e^{-\alpha_1}\cdots e^{-\alpha_{p-1}}e^{-\alpha_p}e^{\Lambda_p}| \\
&= \frac{\qqq\ddd^{\frac{\ell}{2}}}{\uuu}\langle 1\otimes e^{-\alpha_p}e^{\Lambda_p}|\\
\varpi(f_{p, 0})\langle 1\otimes e^{-\alpha_p}e^{\Lambda_p}|&= -\qqq^{-1}\ccc_p^{-1}\mathbb{1}_p^*
\end{align*}
The case $p=0$ is:
\begin{align*}
\varpi(b_{0,1})\langle 1\otimes e^{-\alpha_0}|&= 
\ddd^{-(\ell-1)}\qqq^{\ell-2}f_{\ell-1,0}\cdots f_{2,0}f_{1,1} f_{0,-1}\langle 1\otimes e^{-\alpha_0}|\\
&=\frac{\qqq^{-2}\ddd^{-\frac{\ell}{2}}}{\ccc_0\cdots\ccc_{\ell-1}}
\langle 1\otimes e^{\alpha_{\ell-1}}\cdots e^{\alpha_2}e^{\alpha_1}|\\
&= \qqq^{-1}\ddd^{-\frac{\ell}{2}}\uuu\langle 1\otimes e^{-\alpha_0}|\\
\varpi(b_{0,-1})\langle 1\otimes e^{-\alpha_0}|&=
\ddd^{\ell-1}\qqq^{-\ell+2}e_{0,1}e_{1,-1}e_{2,0}\cdots e_{\ell-1,0}\langle 1\otimes e^{-\alpha_0}|\\
&= (-1)^\ell\ddd^{\frac{\ell}{2}}\qqq^2\ccc_0\cdots\ccc_{\ell-1}
\langle 1\otimes e^{-\alpha_0}e^{-\alpha_1}e^{-\alpha_2}\cdots e^{-\alpha_{\ell-1}}e^{-\alpha_0}|\\
&= \frac{\qqq\ddd^{\frac{\ell}{2}}}{\uuu}\langle 1\otimes e^{-\alpha_0}|\\
\varpi(f_{0,0})\langle 1\otimes e^{-\alpha_0}|&= 
(-1)^\ell\ddd^{-1}\qqq^{-\ell+3}e_{1,-1}e_{2,0}\cdots e_{\ell-1,0}\langle 1\otimes e^{-\alpha_0}|\\
&=(-1)^{\frac{(\ell-2)(\ell-3)}{2}}\ddd^{-1}\qqq^2\ccc_1\cdots\ccc_{\ell-1}\mathbb{1}_0^*
\end{align*}
Combining these equations with (\ref{P1Eigen}), (\ref{P-1Eigen}), and (\ref{OtherF}), we have:
\begin{equation}
\varpi(f_{p,\pm m}e_{p,0})\mathbb{1}_0^*
=\pm(-1)^{\pm m}\ddd^{\mp\frac{m}{2}}\uuu^{\pm m}\mathbb{1}_0^*
\label{FAction}
\end{equation}
\end{enumerate}
Finally, we obtain (\ref{CoeffEigen}) by combining (\ref{FAction}) with (\ref{PsiEF}).\qedhere
\end{enumerate}
\end{proof}

As in \ref{Pairing}, we can define a pairing $W_p^*\times W_{-p}^*\rightarrow\FF$ under which the adjunction antihomomorphism for toroidal operators equals twisting by $\eta$.
With this, we can prove
\begin{cor}\label{DualTsym2}
For $0\le p\le \ell-1$, the vacuum-to-vacuum map $\mathbb{1}_{-p}^*\mapsto |\varnothing\rangle$ induces an isomorphism $\mathrm{T}_{-p}^{*}:W^*_{-p}\rightarrow\mathcal{F}$ between the representations $\rho_{-p,\vec{\ccc}}^*\circ\varpi^{-1}$ and $\tau_{p,\upsilon}^-$ of $'\ddot{U}'$, where the parameters $\vec{\ccc}$ and $\upsilon$ are related by
\[
\upsilon=\frac{1}{\ddd^{\frac{\ell}{2}}\uuu}
\]
\end{cor}

\section{Positive modes}\label{PosMode}
\renewcommand{\thesubsection}{B.\arabic{subsection}}
Here, we carry out the arguments from the bulk of Section 4 for $\Sss^+$ instead of $\Sss^-$.
We maintain the identification of parameters
\[
q=\qqq\ddd,\, t=\qqq\ddd^{-1}
\]
Let
\begin{align}
\nonumber E^+_{p,n}&:=
\Sym\left( \prod_{1\le r<s\le n}\left\{\frac{x_{p+1,r}-q^{-1}x_{p,s}}{x_{p+1,r}-tx_{p,s}}\prod_{i,j\in\ZZ/\ell\ZZ}\omega_{i,j}\left( x_{i,r},x_{j,s} \right)\right\}\right.\\
\nonumber&\times \left.\prod_{r=1}^n\left\{\left( \frac{x_{p,r}}{x_{0,r}}-q\frac{x_{p+1,r}}{x_{0,r}} \right)\prod_{i\in\ZZ/\ell\ZZ}x_{i,r}  \right\} \right)\\
\nonumber H_{p,n}^+&:=
\Sym\left( \prod_{1\le r<s\le n}\left\{\frac{t^{-1}x_{p+1,s}-x_{p,r}}{qx_{p+1,s}-x_{p,r}}\prod_{i,j\in\ZZ/\ell\ZZ}\omega_{i,j}\left( x_{i,r},x_{j,s} \right)\right\}\right.\\
\nonumber&\times \left.\prod_{r=1}^n\left\{\left( \frac{x_{p,r}}{x_{0,r}}-t^{-1}\frac{x_{p+1,r}}{x_{0,r}} \right)\prod_{i\in\ZZ/\ell\ZZ}x_{i,r}  \right\} \right)\\
\label{PosAlt}
\begin{split}
\sum_{n\ge 0}\hat{h}_{-n}^+(i) z^{-n}&:=\exp\left( (\qqq-\qqq^{-1})^{-1}\sum_{n>0}(-b_{i,n}^\perp+t^{-n}b_{i+1,n}^\perp)z^{-n} \right)\\
\sum_{n\ge 0}\hat{e}_{-n}^+(i) (-z)^{-n}&:=\exp\left( (\qqq-\qqq^{-1})^{-1}\sum_{n>0} (b_{i,n}^\perp-q^{n}b_{i+1,n}^\perp)z^{-n}\right)
\end{split}
\end{align}
and
\begin{align*}
\tilde{e}_{-n}^+(p)&:=\varpi^{-1}\hat{e}_{-n}^+(p)\\ 
\tilde{h}_{-n}^+(p)&:=\varpi^{-1}\hat{h}_{-n}^+(p)
\end{align*}
The main result of this appendix is:
\begin{prop}\label{ShufflePres}
We have
\begin{align*}
\tilde{e}^{+}_{-n}(p)&=c_{-n}\Psi_+ (E_{p,n}^+)\\
\tilde{h}^+_{-n}(p)&=c_{-n} \Psi_+(H_{p,n}^+)
\end{align*}
where
\[
c_{-n}= \frac{(-1)^{n\ell}t^{n\ell}\left( 1-\qqq^{-2} \right)^{n\ell}}{\qqq^n\prod_{r=1}^n\left( 1-(qt)^{-r} \right)}
\]
\end{prop}

\subsection{Limit conditions}
Let $\Sss(0)^+_{n\delta}\subset\Sss^+_{n\delta}$ be the subspace of functions satisfying: 
\begin{enumerate}
\item for all $\vec{k}\le n\delta$
\begin{align*}
\lim_{\xi\rightarrow 0}\xi^{r_0(\vec{k})}F_\xi^{\vec{k}}&<\infty\\
\lim_{\xi\rightarrow\infty}\xi^{r_\infty(\vec{k})}F_\xi^{\vec{k}}&<\infty
\end{align*}
\item for all diagonal $k\delta\le n\delta$,
\[\lim_{\xi\rightarrow 0}\xi^{-k\ell}F_\xi^{(a;b]}=\lim_{\xi\rightarrow\infty}\xi^{-k\ell}F_\xi^{k\delta}\]
\item and for all nondiagonal $(a;b]\le n\delta$ with $0\in(a;b]^-$,
\[
\lim_{\xi\rightarrow 0}\xi^{r_0( (a;b])}F_\xi^{(a;b]}=0
\]
\end{enumerate}
We then set $\Sss(0)^+:=\bigoplus_n\Sss(0)^+_{n\delta}$.
Note that the only difference between these limit conditions and those of \ref{LimitCond} lies in (3).

\subsection{Dual vacuum expectations of \textit{L}-operators}\label{DualCorr}
Recall that $\ddot{U}_+^0$ denotes the part of horizontal Heisenberg subalgebra generated by positive modes.
Our next step is to prove that $\Psi_+(\Sss(0)^+)=\varpi^{-1}(\ddot{U}_+^0)$.
To do so, we find a suitable generating set for both by taking the vacuum-to-vacuum matrix element in $W_{-p}^*$ of the right factor of $\mathcal{R}^\circ$.
Thanks to Corollary \ref{DualTsym2}, we obtain
\[\langle 1\otimes \mathbb{1}_{-p}^*|\mathcal{R}^\circ |1\otimes \mathbb{1}_{-p}^*\rangle
=\exp\left(\sum_{k>0}(-1)^{\ell k}\ddd^{\frac{\ell k}{2}}\frac{(\qqq^{2k}-1)}{(\qqq-\qqq^{-1})k}\varpi^{-1}(b_{p,k}^\perp)\uuu^{k}\right)\]
via arguments similar to those in \ref{LOpMiki}.
To find their corresponding shuffle elements, we need to compute the correlation functions for $\langle\mathbb{1}_{-p}^*|-|\mathbb{1}_{-p}^*\rangle$, but now on the currents $f_i(z)$:
\begin{align*}
 \left\langle\mathbb{1}_{-p}^*\left|\overset{\curvearrowright}{\prod_{i=0}^{\ell-1}}\overset{\curvearrowright}{\prod_{r=1}^{n}}\rho_{-p,\vec{\ccc}}^*(f_i(x_{i,r}))\right| \mathbb{1}_{-p}^*\right\rangle
=\left\langle\mathbb{1}_{-p}\left|\overset{\curvearrowleft}{\prod_{i=0}^{\ell-1}}\overset{\curvearrowleft}{\prod_{r=1}^{n}}(\rho_{-p,\vec{\ccc}}\circ S) (f_i(x_{i,r}))\right|\mathbb{1}_{-p}\right\rangle
\end{align*}
Writing this out using (\ref{AntiVertexF}) gives us
\begin{align*}
&\frac{1}{\qqq^{n\ell}}\left\langle\mathbb{1}_{-p}\left|
\overset{\curvearrowleft}{\prod_{i=0}^{\ell-1}}\overset{\curvearrowleft}{\prod_{r=1}^{n}}
\ccc_{i}^{-1}\exp\left(-\sum_{k>0}\frac{\qqq^{-\frac{k}{2}}}{[k]_\qqq}b_{i,-k}x_{i,r}^{k}\right)
\exp\left(\sum_{k>0}\frac{\qqq^{-\frac{k}{2}}}{[k]_\qqq}b_{i,k}x_{i,r}^{-k}\right)
e^{-\alpha_{i}}x_{i,r}^{1-H_{i,0}}
\right|\mathbb{1}_{-p}\right\rangle\\
&=\frac{(-1)^{\frac{n(\ell-2)(\ell-3)+\ell n(n-1)}{2}+n\ell}}{\qqq^{n\ell}(\ccc_0\cdots \ccc_{\ell-1})^n}\\
&\times\frac{\displaystyle \prod_{i\in\ZZ/\ell\ZZ}\prod_{1\le r<r'<n}(x_{i,r'}-x_{i,r})(x_{i,r'}-\qqq^{-2}x_{i,r})\prod_{r=1}^n x_{i,r}}
	{\displaystyle \prod_{r,s=1}^n\ddd^{-\frac{1}{2}}(x_{\ell-1,s}-\qqq^{-1}\ddd x_{0,r})\prod_{i=0}^{\ell-2}\prod_{r,s=1}^n\ddd^{\frac{1}{2}}(x_{i+1,s}-\qqq^{-1}\ddd^{-1} x_{i,r})}
	\cdot\prod_{r=1}^n\frac{x_{p,r}}{x_{0,r}}\\
&=(-1)^{ \frac{\ell n(n-1)}{2}+n\ell}\qqq^{-n(\ell-1)}\uuu^n\\
&\times\frac{\displaystyle \prod_{i\in\ZZ/\ell\ZZ}\prod_{1\le r<r'<n}(x_{i,r'}-x_{i,r})(x_{i,r'}-\qqq^{-2}x_{i,r})\prod_{r=1}^n x_{i,r}}
	{\displaystyle \prod_{r,s=1}^n\ddd^{-\frac{1}{2}}(x_{\ell-1,s}-\qqq^{-1}\ddd x_{0,r})\prod_{i=0}^{\ell-2}\prod_{r,s=1}^n\ddd^{\frac{1}{2}}(x_{i+1,s}-\qqq^{-1}\ddd^{-1} x_{i,r})}
	\cdot\prod_{r=1}^n\frac{x_{p,r}}{x_{0,r}}\\
&=\frac{(-1)^{ \frac{\ell n (n-1)}{2}+n\ell}\qqq^{-n(\ell-1)-\ell n(n-1)+\ell n^2} (-\uuu)^n}{\ddd^{\frac{\ell n^2}{2}}}\\
&\times\frac{\displaystyle\prod_{i\in\ZZ/\ell\ZZ}\prod_{1\le r<r'<n}(x_{i,r'}-x_{i,r})(\qqq^2x_{i,r'}-x_{i,r})\prod_{r=1}^n x_{i,r}}
	{\displaystyle\prod_{r,s=1}^n( x_{0,r}-\qqq\ddd^{-1}x_{\ell-1,s})\prod_{i=0}^{\ell-2}\prod_{r,s=1}^n( \qqq x_{i+1,s}-\ddd^{-1} x_{i,r})}
	\cdot\prod_{r=1}^n\frac{x_{p,r}}{x_{0,r}}\\
&=\frac{(-1)^{\frac{\ell n (n-1)}{2}+n\ell}\qqq^{n} (-\uuu)^n}{\ddd^{\frac{\ell n^2}{2}}}
\cdot
\frac{\displaystyle\prod_{i\in\ZZ/\ell\ZZ}\prod_{1\le r<r'<n}(x_{i,r}-x_{i,r'})(x_{i,r}-\qqq^{2}x_{i,r'})\prod_{r=1}^n x_{i,r}}
	{\displaystyle\prod_{r,s=1}^n(  x_{0,r}-\qqq\ddd^{-1} x_{\ell-1,s})\prod_{i=0}^{\ell-2}\prod_{r,s=1}^n( \qqq x_{i+1,s}-\ddd^{-1} x_{i,r})}
	\cdot\prod_{r=1}^n\frac{x_{p,r}}{x_{0,r}}
\end{align*}
Here, the rational function is expanded in the region
\[
\|x_{i,r}\|\gg\|x_{j,s}\|\hbox{ for }0\le i<j\le\ell-1
\]
Carrying on as in \ref{Loperators} gives us 
\begin{align}
\label{FirstShuffle}
\begin{split}
\sum_{n\ge 0} F_{p,n}^+\uuu^{n}
&:=\Psi_+^{-1}\left(\exp\left(\sum_{k>0}(-1)^{\ell k}\ddd^{\frac{\ell k}{2}}\frac{(\qqq^{2k}-1)}{(\qqq-\qqq^{-1})k}\varpi^{-1}(b_{p,k}^\perp)\uuu^{k}\right)\right)\\
&=\sum_{n\ge 0}(-\uuu)^{n}\frac{(-1)^{\frac{\ell n (n-1)}{2}+n\ell}\qqq^{n}(\qqq-\qqq^{-1})^{n\ell}}{\ddd^{\frac{\ell n^2}{2}}}\prod_{r=1}^n\frac{x_{p,r}}{x_{0,r}}\prod_{i\in\ZZ/\ell\ZZ}\prod_{r=1}^nx_{i,r}
\end{split}
\end{align}
Arguments similar to those in \ref{Gordon} show that $\{F_{p,n}^+\}$ generates $\Sss(0)^+$.

\subsection{Evaluation functionals}
Next, we turn our attention to the shuffle elements $R_{p,n}$ and $R_{p,n}^*$.
We originally viewed them as elements of $\Sss^+$, but we can also view them as elements of $\Sss^-$.
As before, we let $\Psi_-(R_{p,n})_0$ and $\Psi_-(R_{p,n}^*)_0$ denote the summands of $\Psi_-(R_{p,n})$ and $\Psi_-(R_{p,n}^*)_0$ sitting in $\varpi^{-1}(\ddot{U}_-^0)$, respectively. 
Our next step is to use the elements $F_{p,n}^+$ to explicitly write $\Psi_-(R_{p,n})_0$ and $\Psi_-(R_{p,n}^*)_0$ in terms of Heisenberg generators 

Similar to the proof of Corollary \ref{DualCurrentsCalc}, this entails computing the pairings $\langle F_{p,n}, R_{p,n}\rangle$ and $\langle F_{p,n}, R_{p,n}^*\rangle$.
By comparing the scalars in (\ref{FirstShuffle}) with those in the formula for $F_{p,n}$, we can recycle the computations in the proof of Lemma \ref{FuncCalc} to obtain for $1\le i\le \ell$, 
\begin{align*}
\langle F_{p+i,n}^+ , R_{p,n}\rangle
&=q^{n(2\ell-p-i)}\qqq^{2n-n\ell}\ddd^{-\frac{\ell n}{2}}\prod_{r=1}^n\frac{\qqq^{-2}-q^{(r-1)\ell}}{1-q^{r\ell}}\\
&= q^{n(\ell-p-i)}\ddd^{\frac{\ell n}{2}}\prod_{r=1}^n\frac{1-\qqq^2q^{(r-1)\ell}}{1-q^{r\ell}}\\
\langle F_{p+i,n}^+, R^*_{p,n}\rangle
&=(-1)^{n\ell}\qqq^{2n}t^{n(p+i-2\ell)}\ddd^{-\frac{\ell n}{2}}\prod_{r=1}^n\frac{1-\qqq^{-2}t^{-(r-1)\ell}}{1-t^{-r\ell}}\\
&=(-1)^{n\ell}\qqq^{-n\ell}t^{n(p+i-\ell)}\ddd^{\frac{\ell n}{2}}\prod_{r=1}^n\frac{\qqq^2-t^{-(r-1)\ell}}{1-t^{-r\ell}}
\end{align*}
Combining this with equation (\ref{FirstShuffle}) and the generalized partition identity, we deduce 
\begin{align}
\sum_{n\ge 0}\Psi_-(R_{p,n})_0z^n&=\exp\left(-(\qqq-\qqq^{-1})\sum_{n>0}\frac{\sum_{i=1}^{\ell}q^{n(\ell-p-i)}\varpi^{-1}(b_{p+i,-n})}{1-q^{n\ell}}\cdot\frac{z^{n}}{n}\right)\label{Rboson}\\
\sum_{n\ge0}\Psi_-(R_{p,n}^*)_0z^n&=\exp\left((\qqq-\qqq^{-1})\sum_{n>0}(-1)^{n\ell}\qqq^{-n\ell}\frac{\sum_{i=1}^{\ell}t^{n(p+i-\ell)}\varpi^{-1}(b_{p+i,-n})}{1-t^{-n\ell}}\cdot\frac{z^{n}}{n}\right)\label{Rstarboson}
\end{align}

\subsection{Proof of Proposition \ref{ShufflePres}}
From equations (\ref{PosAlt}), (\ref{Rboson}), and (\ref{Rstarboson}) we can compute
\begin{align*}
\varphi\left( \sum_{n\ge0}\tilde{e}_{-n}^+(p+i)(-z)^{-n},\sum_{n\ge0}\Psi_-(R_{p,n})w^n \right)&=
\left\{\begin{array}{ll}
1 &\hbox{if }i\not=0\\
1- q^{-p}\dfrac{w}{z}&\hbox{if }i=0
\end{array}\right.\\
\varphi\left( \sum_{n\ge0}\tilde{h}_{-n}^+(p+i)z^{-n},\sum_{n\ge0}\Psi_-(R^*_{p,n})w^n \right)&=
\left\{\begin{array}{ll}
1 &\hbox{if }i\not=0\\
1- (-1)^{\ell}\qqq^{-\ell}t^{p}\dfrac{w}{z}&\hbox{if } i=0
\end{array}\right.
\end{align*}
Analogues of the coproduct arguments in \ref{CoproductStrat} then imply that $\Psi_+^{-1}(\tilde{e}_n^+(p))$ and $\Psi_+^{-1}(\tilde{h}_n^+(p))$ differ from $E^+_{p,n}$ and $H^+_{p,n}$ by a scalar, respectively.
To compute these scalars, we calculate 
\begin{align*}
\varphi\left( \tilde{e}_{-n}^+(p), \Psi_-(R_{p,1})^n \right)&=q^{-np}\\
\varphi\left( \tilde{h}_{-n}^+(p), \Psi_-(R^*_{p,1})^n \right)&=(-1)^{n\ell-n}\qqq^{-n\ell}t^{np}
\end{align*}
On the other hand, we can recycle the computations from proof of Theorem \ref{FinalShuffle} to obtain
\begin{align*}
\langle E^+_{p,n}, R_{p,1}^n\rangle&=\frac{q^{n(\ell-p)}\qqq^n\prod_{r=1}^n(1-(qt)^{-r})}{(1-\qqq^2)^{n\ell}}\\
\langle H^+_{p,n}, (R^*_{p,1})^n\rangle&=\frac{(-1)^nt^{n(p-\ell)}\prod_{r=1}^n(1-(qt)^{-r})}{\qqq^{n(\ell-1)}(1-\qqq^{-2})^{n\ell}}
\end{align*}
Taking the appropriate quotients, we obtain $c_{-n}$ for both.


\bibliographystyle{plain}
\bibliography{Wreath}

\end{document}